\documentclass [11pt, oneside, fleqn] {report}
\usepackage{epigraph, framed}
\usepackage{amsmath, amsthm, amssymb, amsfonts, dsfont} \numberwithin{equation}{section}
\usepackage{hyperref}
\usepackage[margin=1.3in]{geometry}
\usepackage{epsfig} \usepackage{graphics} 
\usepackage{color}
\usepackage[usenames,dvipsnames,svgnames,table]{xcolor}
\usepackage{xcolor}
\usepackage{mdframed}

\definecolor{theoremColor}{rgb}{0.9,1,0.9}
\definecolor{proofColor}{rgb}{0.9,0.9,0.9}

\begin{document}

\title{ {\Huge \textbf{Summability Calculus}}}
\author{
	{\Large Ibrahim M. Alabdulmohsin}\\[5pt]
	{\small ibrahim.alabdulmohsin@kaust.edu.sa}\\[12pt]
	{\small \textbf{Computer, Electrical, and Mathematical Sciences and Engineering (CEMSE) Division}}\\[5pt]
	{\small \textbf{King Abdullah University of Science and Technology (KAUST)}}\\[5pt]
	{\small Thuwal 23955-6900, Kingdom of Saudi Arabia}
	}
\date{\today}

\maketitle 

\newcommand{\simplefsum}[5]{ #1(#2)=\sum_{#3=#4}^#2 #5(#3)}
\newcommand{\simplefinitesum} {$\simplefsum{f}{n}{k}{a}{g}$}
\newcommand{\convfsum}[5]{#1(#2)=\sum_{#3=#4}^#2 #5(#3, #2)}
\newcommand{\convolutedsum}{$\convfsum{f}{n}{k}{a}{g}$}

\newtheoremstyle{thrmstyle}{12pt}{}{\slshape}{}{\scshape}{.}{ }{}
\theoremstyle{thrmstyle}
\newtheorem{theorem}{Theorem}[section]
\newtheorem{lemma}{Lemma}[section]
\newtheorem{claim}{Claim}[section]
\newtheorem{corollary}{Corollary}[section]
\theoremstyle{definition}
\newtheorem{definit}{Definition}

\hyphenpenalty=1000

\begin{abstract}In this paper, we present the foundations of Summability Calculus, which places various established results in number theory, infinitesimal calculus, summability theory, asymptotic analysis, information theory, and the calculus of finite differences under a single simple umbrella. Using Summability Calculus, any given finite sum of the form $f(n) = \sum_{k=a}^n s_k\, g(k,n)$, where $s_k$ is an arbitrary periodic sequence, becomes immediately \emph{in analytic form}. Not only can we differentiate and integrate with respect to the bound $n$ without having to rely on an explicit analytic formula for the finite sum, but we can also deduce asymptotic expansions, accelerate convergence, assign natural values to divergent sums, and evaluate the finite sum for any $n\in\mathbb{C}$. This follows because the discrete definition of the simple finite sum $f(n) = \sum_{k=a}^n s_k\, g(k,n)$ embodies a \emph{unique natural} definition for all $n\in\mathbb{C}$. Throughout the paper, many established results are strengthened such as the Bohr-Mollerup theorem, Stirling's approximation, Glaisher's approximation, and the Shannon-Nyquist sampling theorem. In addition, many celebrated theorems are extended and generalized such as the Euler-Maclaurin summation formula and Boole's summation formula. Finally, we show that countless identities that have been proved throughout the past 300 years by different mathematicians using different approaches can actually be derived in an elementary straightforward manner using the rules of Summability Calculus. \end{abstract}

\tableofcontents

\addtocontents{lof}{\linespread{2}\selectfont} \listoffigures



\chapter{Introduction} \label{ChaptIntro}
\epigraph{\emph{One should always generalize}}{Carl Jacobi (1804 -- 1851)}

\section{Preliminary Discussion}\label{Sect1dot1}
Generalization has been an often-pursued goal since the early dawn of mathematics. It can be loosely defined as the process of introducing new systems in order to extend consistently the domains of existing operations while still preserving as many prior results as possible. Generalization has manifested in many areas of mathematics including fundamental concepts such as numbers and geometry, systems of operations such as the arithmetic, and even domains of functions as in analytic continuation. One historical example of mathematical generalization that is of particular interest in this paper is extending the domains of special \emph{discrete} functions such as finite sums and products to non-integer arguments. Such process of generalization is quite different from mere interpolation, where the former is meant to preserve some fundamental properties of discrete functions as opposed to mere blind fitting. Consequently, generalization has intrinsic significance that provides deep insights and leads naturally to an evolution of mathematical thought.  

For instance if one considers the discrete power sum function $S_m(n)$ given in Eq \ref{PSDisctrete} below, it is trivial to realize that an infinite number of analytic functions can correctly interpolate it. In fact, let $S$ be one such function, then the sum of $S$ with any function $p(n)$ that satisfies $p(n)=0$ for all $n\in\mathbb{N}$ will also interpolate correctly the discrete values of $S_m(n)$. However, the well-known Bernoulli-Faulhaber formula for the power sum function additionally \emph{preserves the recursive property} of $S_m(n)$ given in Eq \ref{PSRecursive} for all real values of $n$, which makes it a suitable candidate for a generalized definition of power sums. In fact, it is indeed the \emph{unique} family of \emph{polynomials} that enjoys such advantage; hence it is the unique most natural generalized definition of power sums if one considers polynomials to be the simplest of all possible functions. The Bernoulli-Faulhaber formula is given in Eq \ref{BFformula}, where $B_r$ are Bernoulli numbers and $B_1=-\tfrac{1}{2}$.
\begin{equation} \label{PSDisctrete} S_m(n)=\sum_{k=1}^n k^m  \end {equation} 
\begin{equation} \label{PSRecursive} S_m(n)=n^m+S_m(n-1)  \end{equation} 
\begin{equation} \label{BFformula} S_m(n)=\frac{1}{m+1} \sum_{j=0}^n (-1)^j \binom{m+1}{j} B_j n^{m+1-j}  \end{equation}

Looking into the Bernoulli-Faulhaber formula for power sums, it is not immediately obvious, without aid of the finite difference method, why $S_m(n)$ can be a polynomial with degree $m+1$, even though this fact becomes literally trivial using the simple rules of Summability Calculus presented in this paper. Of course, once a correct guess of a closed-form formula to a finite sum or product is available, it is usually straightforward to prove correctness using the method of induction. However, arriving at the right guess itself often relies on intricate ad hoc approaches. 

A second well-known illustrative example to the subject of mathematical generalization is extending the definition of the discrete factorial function to non-integer arguments, which had withstood many unsuccessful attempts by reputable mathematicians such as Bernoulli and Stirling until Euler came up with his famous answer in the early 1730s in a series of letters to Goldbach that introduced his infinite product formula and the Gamma function \cite{Euler1738, Davis1959,MathWorldGamma}. Indeed, arriving at the Gamma function from the discrete definition of factorial was not a simple task, needless to mention proving that it was the \emph{unique} natural generalization of factorials as the Bohr-Miller theorem nearly stated \cite{Krantz1999}. Clearly, a systematic approach is needed in answering such questions so that it can be readily applied  to the potentially infinite list of special discrete functions such as the factorial-like \emph{hyperfactorial} and \emph{superfactorial} functions, defined in Eq \ref{hyperf} and Eq \ref{superf} respectively (for a brief introduction to such family of functions, the reader is kindly referred to \cite{MathWorldHypFactorial} and \cite{MathWordGFunct}). Summability Calculus provides us with the answer. 
\begin{align}
\label{hyperf} \textbf{Hyperfactorial:} &\quad \quad H(n)=\prod_{k=1}^n k^k \\[5pt]
\label{superf} \textbf{Superfactorial:} &\quad \quad S(n)=\prod_{k=1}^n k!
\end{align}

Aside from obtaining exact closed-form generalizations to discrete functions as well as performing infinitesimal calculus, which aids in computing \emph{n}th-order approximation among other applications, deducing asymptotic behavior is a third fundamental problem that could have enormous applications, too. For example, Euler showed that the harmonic sum was asymptotically related to the logarithmic function, which led him to introduce the famous constant $ \lambda$ that bears his name, which is originally defined by Eq \ref{EulerConst} \cite{EulerSumProg,SandifierGC}. Stirling, on the other hand, presented the famous asymptotic formula for the factorial function given in Eq \ref{StirlingAppx}, which is used almost ubiquitously such as in the study of algorithms and data structures \cite{Robbins1955, CLRAlgorithm}. Additionally, Glaisher in 1878 presented an asymptotic formula for the hyperfactorial function given in Eq \ref{hyperfAppx}, which led him to introduce a new constant, denoted $A$ in Eq \ref{hyperfAppx}, that is intimately related to Euler's constant and the Riemann zeta function \cite{Glashier1893,MathWorldGlashConstant,MathWorldHypFactorial}. 
\begin{equation}\label{EulerConst} \lim_{n\to\infty} \Bigl\{ \sum_{k=1}^n \frac{1}{k} - \log{n}\Bigr\}=\lambda \end{equation} 
\begin{equation}\label{StirlingAppx} n! \sim \sqrt{2\pi n} \Bigl(\frac{n}{e}\Bigr)^n \end{equation}
\begin{equation}\label{hyperfAppx} H(n) \sim A \; n^{\frac{n^2+n}{2}+\frac{1}{12}} \,e^{-\frac{n^2}{4}}, \;\; A\approx 1.2824 \end{equation}

While these results have been deduced at different periods of time in the history of mathematics using different approaches, a fundamental question that arises is whether there exists a simple universal calculus that bridges all of these different results together and makes their proofs almost elementary. More crucially, we desire that such calculus yield an elementary approach for obtaining asymptotic behavior to oscillating sums as well, including, obviously, the special important class of alternating sums. The answer to this question is, in fact, in the affirmative and that universal calculus is Summability Calculus.

The primary statement of Summability Calculus is quite simple: given a discrete finite sum of the form $\convfsum{f}{n}{k}{a}{g}$, then such finite sum is in \emph{analytic form}. Not only can we perform differentiation and integration with respect to $n$ without having to rely on an explicit formula, but we can also immediately evaluate the finite sum for fractional values of $n$, deduce asymptotic expressions even if the sum is oscillating, accelerate convergence of the infinite sum, assign natural values to divergent sums, and come up with a potentially infinite list of interesting identities as a result; all without having to \emph{explicitly} extend the domain of $f(n)$ to non-integer values of $n$. To reiterate, this follows because the expression  $\convfsum{f}{n}{k}{a}{g}$ embodies within its discrete definition an immediate natural definition of $f(n)$ for all $n \in\mathbb{C}$ as will be shown in this paper.

Summability Calculus vs. conventional infinitesimal calculus can be viewed in light of an interesting duality. In traditional calculus, infinitesimal behavior at small intervals is used to understand global behavior at all intervals that lie within the same analytic disc. For instance, Taylor series expansion for a function $f(x)$ around a point $x=x_0$ is often used to compute the value of $f(x)$ for some points $x$ outside $x_0$, sometimes even for the entire function's domain. Thus, the behavior of the analytic function at an infinitesimally small interval is sufficient to deduce global behavior of the function. Such incredible property of analytic functions, a.k.a. \emph{rigidity}, is perhaps the cornerstone of traditional calculus that led its wonders. In Summability Calculus, on the other hand, we follow the contrary approach by employing our incomplete knowledge about the global behavior of a function to reconstruct accurately its local behavior at any desired interval.

The aforementioned duality brings to mind the well-known Sampling Theorem. Here, if a function is \emph{bandlimited}, meaning that its non-zero frequency components are restricted to a bounded region in the frequency domain, then discrete samples taken at a sufficiently high rate can be used to represent the function completely \emph{without loss of any information}. In such case, incomplete global information, i.e. discrete samples, can be used to reconstruct the function perfectly at all intervals, which is similar to what Summability Calculus fundamentally entails. Does Summability Calculus have anything to do with the Sampling Theorem? Surprisingly, the answer is yes. In fact, we will use results of Summability Calculus to prove the Sampling Theorem.

Aside from proving the Sampling Theorem using Summability Calculus, there is another more subtle connection between the two subjects. According to the Sampling Theorem, discrete samples can always be used to perfectly reconstruct the \emph{simplest} function that interpolates them, if \emph{bandwidth} is taken as a measure of complexity. As will be shown repeatedly throughout this paper, Summability Calculus also operates implicitly on the simplest of all possible functions that can correctly interpolate discrete values of finite sums and additionally preserve their defining recursive prosperities. So, if the simplest interpolating function, using bandwidth as a measure of complexity, happens to satisfy the defining properties of finite sums, we would expect the resulting function to agree with what Summability Calculus yields. Indeed, we will show that this is the case.

Sampling necessitates interpolation. One crucial link between interpolation and Summability Calculus is polynomial fitting. For example, given discrete samples of the power sum function mentioned earlier in Eq \ref{PSDisctrete}, then such samples can be perfectly fitted using polynomials. Because polynomials can be reconstructed from samples taken at any arbitrarily low sampling rate, their bandwidth must be zero regardless of whether or not polynomials are Fourier transformable in the strict sense of the word. Therefore, if one were to use bandwidth as a measure of complexity, polynomials are among the simplest of all possible functions, hence the Bernoulli-Faulhaber Formula is, by this measure, the \emph{unique most natural generalization} to power sums. In Summability Calculus, polynomial fitting is a cornerstone; it is what guarantees Summability Calculus to operate on the unique most natural generalization to finite sums. It may be counter-intuitive to note that polynomial fitting arises in the context of arbitrary finite sums, but this result will be established in this paper. In fact, we will show that the celebrated Euler-Maclaurin summation formula itself and it various analogues fundamentally arise out of polynomial fitting!

We have stated that Summability Calculus allows us to perform infinitesimal calculus, such as differentiation and integration, and evaluate fractional finite sums without having to explicitly extend the domain of finite sums to non-integer arguments. However, if a finite sum is discrete in nature, what does its derivative mean in the first place? To answer this question, we need to digress to a more fundamental question: what does a finite sum of the form \simplefinitesum\ fundamentally represent? Traditionally, the symbol $\textstyle\sum$ was meant to be used as a shorthand notation for an iterated addition, and was naturally restricted to integers. However, it turns out that if we \textbf{define} a finite sum more broadly by \textbf{two properties}, we immediately have a natural definition that extends domain to the complex plane $\mathbb{C}$. The two properties are: \\[5pt] \, \, (1)\quad $\displaystyle\sum_{k=x}^x g(k) = g(x)$ \\[5pt] \, \, (2) \quad $\displaystyle\sum_{k=a}^b g(k) + \displaystyle\sum_{k=b+1}^x g(k)= \displaystyle\sum_{k=a}^x g(k)$ \\[5pt]

The two properties are two of six axioms proposed in \cite{Muller2011} for defining fractional finite sums. Formally speaking, we say that a finite sum \simplefinitesum\  is a binary operator $a \diamond_g n$ that satisfies the two properties above. Of course, using these two properties, we can easily recover the original discrete definition of finite sums as follows:
\begin{align*}
\sum_{k=a}^n g(k) &= \sum_{k=a}^{n-1} g(k) + \sum_{k=n}^n g(k) \quad  &\textrm{(by property 2)} \\
&= \sum_{k=a}^{n-2} g(k) + \sum_{k=n-1}^{n-1} g(k) +\sum_{k=n}^{n} g(k)  &\textrm{(by property 2)}\\
&=\sum_{k=a}^a g(k) + \sum_{k=a+1}^{a+1} g(k) + \dotsc + \sum_{k=n}^{n} g(k) &\textrm{(by property 2)}\\
&=g(a) + g(a+1) + \dotsc + g(n) &\textrm{(by property 1)} 
\end{align*}

In addition, by property 2, we derive the recurrence identity given in Eq \ref{recur}. Such recurrence identity is extremely important in subsequent analysis.
\begin{equation} \label{recur}
f(n)=g(n)+f(n-1)
\end{equation}

Moreover, we immediately note that the two defining properties of finite sums suggest unique \emph{natural} generalization to all $n \in \mathbb{C}$ if the infinite sum converges. This can be seen clearly from both properties 1 and 2 as shown in Eq \ref{unqSmpConvg}. Here, because  $\textstyle \sum_{k=n+1}^\infty g(k)=g(n+1)+g(n+2) + \dots $ is well-defined for all $n \in \mathbb{C}$, by assumption, the finite sum $\textstyle \sum_{k=a}^n g(k)$ can be \emph{naturally} defined for all $n \in \mathbb{C}$, and its value is given by Eq \ref{unqSmpConvg}. 
\begin{equation} \label{unqSmpConvg}
\sum_{k=a}^n g(k) = \sum_{k=a}^\infty g(k) - \sum_{k=n+1}^\infty g(k)
\end{equation}

This might appear trivial to observe but it is crucial to emphasize that such \lq\lq obvious'' generalization is actually \emph{built on an assumption} that does not necessarily follow from the definition of finite sums. Here, it is assumed that since the infinite sum converges if $n\to\infty$ and $n-a \in\mathbb{N}$, then the sum converges to the same limit as $n\to\infty$ in general. This is clearly only an assumption! For example, if one were to define $\sum_{k=0}^n x^k$ by $\frac{1-x^{n+1}}{1-x}+\sin{2\pi n}$, then both initial condition and recurrence identity hold but the limit $\lim_{n\to\infty} \sum_{k=0}^n x^k$ no longer exists even if $|x|<1$. However, it is indeed quite reasonable to use Eq \ref{unqSmpConvg} as a definition of fractional sums for all $n\in\mathbb{C}$ if the infinite sums converge. After all, why would we add any superficial constraints when they are not needed?

Therefore, it seems obvious that a unique most natural generalization can be defined using Eq \ref{unqSmpConvg} if the infinite sum converges. At least, this is what we would expect if the term \lq\lq natural generalization'' has any meaning. What is not obvious, however, is that a \emph{natural} generalization of finite sums is \emph{uniquely} defined for all $n$ and all $g(k)$. In fact, even a finite sum of the, somewhat complex, form \convolutedsum\ has a unique natural generalization that extends its domain to the complex plane $\mathbb{C}$. However, the argument in the latter case is more intricate as will be shown later.

In addition, we will also show that the study of divergent series and \emph{analytic summability theory} is intimately tied to Summability Calculus. Here, we will assume a generalized definition of infinite sums $\mathfrak{T}$ and describe an algebra valid under such generalized definition. Using $\mathfrak{T}$, we will show how the study of oscillating sums is greatly simplified. For example, if a divergent series $\textstyle\sum_a^\infty g(k)$ is defined by a value $V\in\mathbb{C}$ under $\mathfrak{T}$, then, as far as Summability Calculus is concerned, the series behaves \emph{exactly as if it were a convergent sum}! Here, the term \lq\lq summable'' divergent series, loosely speaking, means that a natural value can be assigned to the divergent series following some reasonable arguments. For example, one special case of $\mathfrak{T}$ is \textbf{Abel summability method}, which assigns to a divergent series $\textstyle\sum_a^\infty g(k)$ the value $\textstyle  \lim_{z\to 1^- } \sum_{k=a}^\infty g(k) z^{k-a}$ if the limit exists. Intuitively speaking, this method appeals to continuity as a rational basis for defining divergent sums, which is similar to the arguments that $\sin{z}/z=1$ and $z\log{z} = 0$ at $z=0$. Abel summability method was used prior to Abel; in fact, Euler used it quite extensively and called it the  \lq\lq generating function method'' \cite{Varad2007}. In addition, Poisson also used it frequently in summing Fourier series \cite{HardyDiverg}. 

Using the generalized definition of sums given by $\mathfrak{T}$, we will show that summability theory generalizes the earlier statement that a finite sum is naturally defined by Eq \ref{unqSmpConvg}, which otherwise would only be valid if the infinite sums converge. Using $\mathfrak{T}$, on the other hand, we will show that the latter equation holds, in general, if the infinite sums exist in $\mathfrak{T}$, i.e. even if they are not necessarily convergent in the strict sense of the word. However, not all divergent sums are defined in $\mathfrak{T}$ so a more general statement will be presented.

With this in mind, we are now ready to see what the derivative of a finite sum actually means. In brief, since we are claiming that any finite sum implies a \emph{unique} natural generalization that extends its domain from a discrete set of numbers to the complex plane $\mathbb{C}$, the derivative of a finite sum is simply the derivative of its unique natural generalization. The key advantage of Summability Calculus is that we can find such derivative of a finite sum without having to find an analytic expression to its unique natural generalization. In fact, we can integrate as well. So, notations such as $ \frac{d}{dt} \sum^t $ and even $ \int \sum^t dt$ will become meaningful. 

In summary, this paper presents the foundations of Summability Calculus, which bridges the gap between various well-established results in number theory, summability theory, infinitesimal calculus, and the calculus of finite differences.  In fact, and as will be shown throughout this paper, it contributes to more branches of mathematics such as approximation theory, asymptotic analysis, information theory, and the study of accelerating series convergence. However, before we begin discussing Summability Calculus, we first have to outline some terminologies that will be used throughout this paper. 

\section{Terminology and Notation} \label{Sect1dot2}
\begin{tabbing} \label{simplFin}
\hspace{1cm}\= $f(n)=\sum_{k=a}^n g(k)$ \; \; \quad\=   $f(n)= \prod_{k=a}^{n} g(k)\; \; \; \; \; \; $ \quad\= (SIMPLE)\\[5pt]
\> $f(n)=\sum_{k=a}^n g(k, n)$ \>$f(n)= \prod_{k=a}^{n} g(k, n)$	\> (CONVOLUTED)
\end{tabbing}

\emph{Simple} finite sums and products are defined in this paper to be any finite sum or product given by the general form shown above, where $k$ is an iteration variable that is independent of the bound $n$. In a convoluted sum or product, on the other hand, the iterated function $g(k)$ depends on \emph{both} the iteration variable $k$ and the bound $n$ as well. Clearly, simple sums and products are special cases of convoluted sums and products.

Because the manuscript is quite lengthy, we will strive to make it as readable as possible. For instance, we will typically use $f(n)$ to denote a discrete function and let $f_G(n): \mathbb{C} \to \mathbb{C}$ denote its unique natural generalization in almost every example. Here, the notation  $ \mathbb{C} \to \mathbb{C}$ is meant to imply that both domain and range are simply connected regions, i.e. \emph{subsets}, of the complex plane. In other words, we will not explicitly state the domain of every single function since it can often be inferred without notable efforts from the function's definition. Statements such as  \lq\lq let $K\subset \mathbb{C} - \{0\}$'' and  \lq\lq $Y=\mathbb{R}+\{\pm \infty\}$'' will be avoided unless deemed necessary. Similarly, a statement such as  \lq\lq $f$ converges to $g$'' will be used to imply that $f$ approaches $g$ whenever $g$ is defined and that both share the same domain. Thus, we will avoid statements such as \lq\lq $f$ converges to $g$ in $\mathbb{C}-\{s\}$ and has a simple pole at $s$'', when both $f$ and $g$ have the same pole. Whereas Summability Calculus extends the domain of a finite sum $\sum_{k=a}^n g(k,n)$ from a subset of integers to the complex plane $n\in\mathbb{C}$, we will usually focus on examples for which $n$ is real-valued.

In addition, the following important definition will be used quite frequently. \\ \hrule 
\begin{definit}\label{asymptOrder}
A function $g(x)$ is said to be \emph{asymptotically of a finite differentiation order} if a non-negative integer $m$ exists such that $g^{m+1}(x) \to 0$ as $x \to \infty$. The minimum non-negative integer $m$ that satisfies such condition for a function $g$ will be called its \emph{asymptotic differentiation order}.
\end{definit} \hrule\vspace{12pt}

In other words, if a function $g(x)$ has an asymptotic differentiation order $m$, then \emph{only the function up to its \emph{m}th derivative matter asymptotically}. For example, any polynomial with degree $n$ has the asymptotic differentiation order $n+1$. A second example is the function $g(x)=\log{x}$, which has an asymptotic differentiation order of zero because its 1\textsuperscript{st} derivative vanishes asymptotically. Non-polynomially bounded functions such as $e^x$ and the factorial function $x!$ have an infinite asymptotic differentiation order.

Finally, in addition to the hyperfactorial and superfactorial functions described earlier, the following functions will be used frequently throughout this paper:
\begin{tabbing}
\hspace{1cm}\= \textbf{Generalized Harmonic Numbers:} \; \quad\=   $H_{m,n}=\sum_{k=1}^n k^{-m}$ \\[5pt]
\> \textbf{Gamma Function:} \>$\Gamma(n)=\int_0^\infty e^{-t} t^{n-1} dt$ \\[5pt]
\> \textbf{Gauss PI Function:} \>$\Pi(n)=\Gamma(n+1)=n!$ \\[5pt]
\> \textbf{Log-Factorial Function:} \>$\varpi(n)=\log{\Pi(n)}=\log{n!}$ \\[5pt]
\> \textbf{Digamma Function:} \>$\psi(n)=\frac{d}{dn} \log{\Gamma(n)}$ \\[5pt]
\> \textbf{Riemann Zeta Function:} \>$\zeta(s)=\sum_{k=1}^\infty k^{-s}$, \, for $s>1$ \\[0pt]
\end{tabbing}

Here, it is worth mentioning that the superfactorial function  $S(n)$ is also often defined using the \emph{double gamma function} or \emph{Barnes G-function} \cite{Choi2000, MathWordGFunct}. All of these definitions are essentially equivalent except for minor differences. In this paper, we will exclusively use the superfactorial function as defined earlier in Eq \ref{superf}. Moreover, we will almost always use the log-factorial function $\log{\Gamma(n+1)}$, as opposed to the log-Gamma function. Legendre's normalization $n!=\Gamma(n+1)$ is unwieldy, and will be avoided to simplify mathematics. It is perhaps worthwhile to note that avoiding such normalization of the Gamma function is not a new practice. In fact, Legendre's normalization has been avoided and even harshly criticized by some 20\textsuperscript{th} century mathematicians such as Lanczos, who described it as \lq\lq void of any rationality'' \cite{Lanczos64}. 

\section{Historical Remarks}\label{Sect1dot3}
It is probably reasonable to argue that the subject of finite sums stretches back in time to the very early dawn of mathematics. In 499 A.D., the great Indian mathematician Aryabhata investigated the sum of arithmetic progressions and presented its closed-form formula in his famous book \emph{Aryabhatiya} when he was 23 years old. In this book, he also stated formulas for sums of powers of integers up to the summation of cubes. Unfortunately, mathematical notations were immature during that era and the great mathematician had to state his mathematical equations using plain words: \lq\lq \emph{Diminish the given number of terms by one then divide by two then} \ldots'' \cite{Volo1977}. In addition, the Greeks were similarly interested as Euclid demonstrated in his Elements Book IX Proposition 35, in which he presents the well-known formula for the sum of a finite geometric series. Between then and now, countless mathematicians were fascinated with the subject of finite sums.

Given the great interest mathematicians have placed in summations and knowing that this paper is all about finite sums,  it should not be surprising to see that a large number of the results presented herein have already been discovered at different points in time during the past 2,500 years. In fact, some of what can be rightfully considered as the building blocks of Summability Calculus were deduced 300 years ago, while others were published as recently as 2010.

The earliest work that is directly related to Summability Calculus is Gregory's quadrature formula, whose original intention was in numerical integration \cite{Jordan1965, Sinha2011}. We will show that Gregory's formula indeed corresponds to the unique most natural generalization of simple finite sums. Later around the year 1735, Euler and Maclaurin independently came up with the celebrated Euler-Macluarin summation formula that also extends the domain of simple finite sums to non-integer arguments \cite{Apostol, SandiferBasel}. The Euler-Maclaurin formula has been widely applied, and was described as one of \lq\lq the most remarkable formulas of mathematics'' \cite{LampretEM2001}. Here, it is worth mentioning that neither Euler nor Maclaurin published a formula for the remainder term, which was first developed by Poisson in 1823 \cite{LampretEM2001}. The difference between the Euler-Macluarin summation formula and Gregory's is that finite differences are used in the latter formula as opposed to infinitesimal derivatives. Nevertheless, they are both formally equivalent, and one can be derived from the other \cite{Sinha2011}. In 1870, George Boole introduced his summation formula, which is the analog of the Euler-Maclaurin summation formula for alternating sums \cite{Borwein2009}. Boole summation formula agrees, at least formally, with the Euler-Macluarin summation formula, and there is evidence that suggests it was known to Euler as well \cite{Borwein1989}. 

In this manuscript, we will derive these formulas, generalize them to convoluted sums, prove their uniqueness for being the most natural generalization to finite sums, generalize them to oscillating sums, and present their counterparts using finite differences as opposed to infinitesimal derivatives. We will show that the Euler-Maclaurin summation formula arises out of polynomial fitting, and show that Boole summation formula is intimately tied with the subject of divergent series. Most importantly, we will show that it is more convenient to use them \emph{in conjunction} with the foundational rules of Summability Calculus.

The second building block of Summability Calculus is asymptotic analysis. Ever since Stirling and Euler presented their famous asymptotic expressions to the factorial and harmonic number respectively, a deep interest in the asymptotic behavior of special functions such as finite sums was forever instilled in the hearts of mathematicians. This includes Glaisher's work on the hyperfactorial function and Barnes work on the superfactorial function, to name a few. Here, Gregory quadrature formula, Euler-Maclaurin summation formula, and Boole summation formula proved indispensable. For example, a very nice application of such asymptotic analysis in explaining a curious observation of approximation errors is illustrated in \cite{Borwein1989}. In this paper, we generalize the Euler-like family of summation formulas to the case of oscillating sums, which makes the task of deducing asymptotic expressions to such finite sums nearly elementary. We will also prove equivalence of all asymptotic formulas.

The third building block of Summability Calculus is summability theory. Perhaps, the most famous summability theorist was again Euler who did not hesitate to use divergent series. According to Euler, the sum of an infinite series should carry a more general definition. In particular, if a sequence of algebraic operations eventually arrive at a divergent series, the value of the divergent series should assume the value of the algebraic expression from which it was derived \cite{Kline1983}. However, Euler's approach was ad-hoc based, and the first systematic and coherent theory of divergent sums was initiated by Cesaro who introduced what is now referred to as the Cesaro mean \cite{HardyDiverg}. Cesaro summability method, although weak, enjoys the advantage of being stable, regular, and linear, which will prove to be crucial for mathematical consistency. In the 20th century, the most well renowned summability theorist was Hardy whose classic book \emph{Divergent Series} remains authoritative.

Earlier, we stated that summable divergent series behave \lq\lq as if they were convergent''. So, in principle, we would expect a finite sum to satisfy Eq \ref{unqSmpConvg} if infinite sums are summable. Ramanujan realized that summability of divergent series might be linked to infinitesimal calculus using such equation \cite{Berndt6}. However, his statements were often imprecisely stated and his conclusions were occasionally incorrect for reasons that will become clear in this paper. In this paper, we state results precisely and prove correctness. We will present the generalized definition of infinite sums $\mathfrak{T}$ in Chapter \ref{Chapter4}, which allows us to integrate the subject of summability theory with the study of finite sums into a single coherent framework. As will be shown repeatedly throughout this paper, this generalized definition of sums is indeed one of the cornerstones of Summability Calculus. It will also shed important insights into the subject of divergent series. For example, we will present a method of deducing analytic expressions and \lq\lq accelerating convergence'' of divergent series. In addition, we will also introduce a new summability method $\Xi$, which is weaker than $\mathfrak{T}$ but it is strong enough to \lq\lq correctly'' sum almost all examples of divergent sums that will referred to in this manuscript. So, whereas $\mathfrak{T}$ is a formal construct that is decoupled from any particular method of computation, results will be deduced using $\mathfrak{T}$ and verified numerically using $\Xi$. 

The fourth building block of Summability Calculus is polynomial approximation. Unfortunately, little work has been made to define fractional finite sums using polynomial approximations except notably for the fairly recent work of M\"uller and Schleicher \cite{Mueller2010}. In their work, a fractional finite sum is defined by using its asymptotic behavior. In a nutshell, if a finite sum can be approximated \emph{in a bounded region} by a polynomial and if the error of such approximation vanishes as this bounded region is pushed towards $\infty$, we might then evaluate a fractional sum by evaluating it asymptotically and propagating values backwards using the recursive property in Eq \ref{recur}. In 2011,  moreover, M\"uller and Schleicher provided an \emph{axiomatic} treatment to the subject of finite sums by proposing six axioms that uniquely extend the definition of finite sums to non-integer arguments, two of which are the defining properties of finite sums listed earlier \cite{Muller2011}. In addition to those two properties, they included linearity, holomorphicity of finite sums of polynomials, and translation invariance (the reader is kindly referred to \cite{Muller2011} for details). Perhaps most importantly, they also included the method of polynomial approximation as a sixth axiom; meaning that if a simple finite sum $\sum_{k=a}^n g(k)$ can be approximated by a polynomial asymptotically, then such approximation is assumed to be valid for \emph{fractional} values of $n$. 

In this paper, we do not pursue an axiomatic foundation of finite sums, although the treatment of M\"uller and Schleicher is a special case of a more general definition that is presented in this paper because their work is restricted to simple finite sums $\sum_{k=a}^n g(k)$, in which $g(n)$ is asymptotically of a finite differentiation order (see Definition \ref{asymptOrder}). In this paper, on the other hand, we will present a general statement that is applicable to all finite sums, including those that cannot be approximated by polynomials, and extend it further to oscillating sums and convoluted sums as well. Finally, we will use such results to establish that the Euler-like family of summation formulas arise out of polynomial fitting. 

The fifth building block of Summability Calculus is the Calculus of Finite Differences, which was first systematically developed by Jacob Stirling in 1730 \cite{Jordan1965}, although some of its most basic results can be traced back to the works of Newton such as Newton's interpolation formula. In this paper, we will derive the basic results of the Calculus of Finite Differences from Summability Calculus and show that they are closely related. For example, we will show that the summability method $\Xi$ introduced in Chapter \ref{Chapter4} is intimately tied to Newton's interpolation formula and present a geometric proof to the Sampling Theorem using the generalized definition $\mathfrak{T}$ and the Calculus of Finite Differences. In fact, we will prove a stronger statement, which we will call the \emph{Half Sampling Theorem}. The Sampling Theorem was popularized in the 20th century by Claude Shannon in his seminal paper \lq\lq Communications in the Presence of Noise'' \cite{ShannonNoise}, but its origin date much earlier. A brief excellent introduction to the Sampling Theorem and its origins is given in \cite{Luke99}.

Finally, Summability Calculus naturally yields a rich set of identities related to fundamental constants such as $\lambda$, $\pi$, and $e$, and special functions such as the Riemann zeta function and the Gamma function. Some of those identities shown in this paper appear to be new while others were proved at different periods of time throughout the past three centuries. It is intriguing to note that while many of such identities were proved at different periods of time by different mathematicians using different approaches, Summability Calculus yields a simple set of tools for deducing them immediately.

\section{Outline of Work}\label{Sect1dot4}
The rest of this manuscript is structured as follows. We will first introduce the foundational rules of performing infinitesimal calculus, such as differentiation, integration, and computing series expansion, on simple finite sums and products, and extend the calculus next to address convoluted sums and products. After that, we introduce the generalized definition of infinite sums  $\mathfrak{T}$  and show how central it is to Summability Calculus. Using $\mathfrak{T}$, we derive foundational theorems for oscillating finite sums that simplify their study considerably, and yield important insights to the study of divergent series and series convergence acceleration. Next, we present a simple method for directly \emph{evaluating} finite sums \convolutedsum\ for all $n \in \mathbb{C}$. Throughout these sections, we will repeatedly prove that we are always dealing with the exact same generalized definition of finite sums, meaning that all results are consistent with each other and can be used interchangeably. Using such established results, we finally extend the calculus to arbitrary discrete functions, which leads immediately to some of the most important basic results in the Calculus of Finite Differences among other notable conclusions, and show how useful it is in the study of finite sums and products. 

\chapter{Simple Finite Sums} \label{Chapter2}
\epigraph{\emph{Simplicity is the ultimate sophistication}}{Leonardo da Vinci (1452 -- 1519)}

We will begin our treatment of Summability Calculus on simple finite sums and products. Even though results of this section are merely special cases of the more general results presented later in which we address the more general case of convoluted sums and products, it is still important to start with these basic results since they are encountered more often, and can serve as a good introduction to what Summability Calculus can markedly accomplish. In addition, the results presented herein for simple finite sums and products are themselves used in establishing the more general case for convoluted sums and products in the following chapter. 

\section{Foundations}\label{Sect2dot1}
Suppose we have the simple finite sum \simplefinitesum\, and let us denote its generalized definition $f_G(n) : \mathbb{C}\to\mathbb{C}$\footnote{Kindly refer to Section \ref{Sect1dot2} for a more accurate interpretation of this notation}. As shown in Section \ref{Sect1dot1}, the two defining properties of finite sums imply that Eq \ref{Eq211} always holds if $f_G(n)$ exists. As will be shown later, the function $f_G(n)$ always exists.
\begin{equation}\label{Eq211}
f_G(n)=\begin{cases} g(n)+f_G(n-1) & \text{for all $n$}\\ g(a) & \text{if $n=a$}
\end{cases}
\end{equation}

From the recurrence identity given in Eq \ref{Eq211}, and because $f_G(n) : \mathbb{C}\to\mathbb{C}$ by assumption, it follows by the basic rules of differentiation that the following identity holds: 
\begin{equation}\label{Eq212}
f_G' (n)=g'(n)+f_G'(n-1)
\end{equation}
Unfortunately, such recurrence identity implicitly requires a definition of \emph{fractional} finite sums, which has not been established yet. Luckily, however, a rigorous proof of Eq \ref{Eq212} can be made using the two defining properties of finite sums as follows:
\begin{align*}
f_G'(n)-f_G'(n-1) &= \lim_{h \to 0} \frac{1}{h} \bigl\{ \sum_{k=a}^{n+h} g(k) - \sum_{k=a}^{n} g(k) - \sum_{k=a}^{n-1+h} g(k) + \sum_{k=a}^{n-1} g(k) \bigr\} \\
&=\lim_{h \to 0} \frac{1}{h} \bigl\{ \sum_{k=a}^{n+h} g(k) - \sum_{k=a}^{n-1+h} g(k)  - \sum_{k=a}^{n} g(k) + \sum_{k=a}^{n-1} g(k) \bigr\} \;  &\textrm{} \\
&=\lim_{h \to 0} \frac{1}{h} \bigl\{ \sum_{k=n+h}^{n+h} g(k) - \sum_{k=n}^{n} g(k) \bigr\} \;  \; \; \; \;\;\;\;\;\;\;\; \textrm{(by property 2)} \\
&=\lim_{h \to 0} \frac{1}{h} \bigl\{ g(n+h) - g(n) \bigr\} \; \;\;\;\;\;\;\;\;\;\;\;\;\;\;\;\;\;\;\;\; \textrm{(by property 1)} \\
&=g'(n)
\end{align*}

Upon combining the recurrence identities in Eq \ref{Eq211} and Eq \ref{Eq212}, Eq \ref{Eq213} follows immediately, where $P(n)$ is an arbitrary periodic function with unit period. Since $a$ is constant with respect to $n$, $f'(a-1)$ is constant as well. However, $f'(a-1)$ is \emph{not} an arbitrary constant and its exact formula will be given later. 
\begin{equation}\label{Eq213}
f_G' (n)=\sum_{k=a}^n g'(k)+f_G'(a-1)+P(n)
\end{equation}

Now, in order to work with the unique most natural generalization to simple finite sums, we choose $P(n)=0$ in accordance with \emph{Occam's razor} principle that favors simplicity. Note that setting $P(n)=0$ is also consistent with the use of \emph{bandwidth} as a measure of complexity as discussed earlier in the Chapter \ref{ChaptIntro}. In simple terms, we should select $P(n)=0$ because the finite sum itself \emph{carries no information} about such periodic function. A more precise statement of such conclusion will be made shortly. Therefore, we always have:
\begin{equation}\label{Eq214}
f_G' (n)=\sum_{k=a}^n g'(k)+c, \; \; \; \; \; \; \; c=f_G'(a-1)
\end{equation}

Also, from the recurrence identity in Eq \ref{Eq211} and upon noting that the initial condition always holds by assumption, we must have $\sum_{k=a}^{a-1}=0$. In other words, any consistent generalization of the discrete finite sum function  \simplefinitesum\ must also satisfy the condition $f_G(a-1)=0$. This can be proved from the defining properties of finite sums as follows:
\begin{equation}\label{Eq215}
\sum_{k=a}^{a-1} g(k) + \sum_{k=a}^{n} g(k) = \sum_{k=a}^{n} g(k) \; \; \; \; \; \; \; \; \; \; \; \; \; \; \; \; \; \; \text{(by property 2)}
\end{equation}

Thus, we always have $\sum_{k=a}^{a-1} g(k)=0$ for any function $g(k)$. This last summation is commonly thought of as a special case of the \emph{empty sum} and is usually defined to be zero by convention. However, we note here that the empty sum is zero by \emph{necessity}, i.e. not a mere convention. In addition, while it is true that $\sum_{k=a}^{a-1} g(k)=0$, it does not necessarily follow that $\sum_{k=a}^{a-b} g(k)=0$ for $b>0$. For example, and as will be shown later, $\sum_{k=1}^{0} \frac{1}{k}=0$ but $\sum_{k=1}^{-1} \frac{1}{k}=\infty$. In fact, using Property 1 of finite sums and the empty sum rule, we always have:
\begin{equation}\label{sumToNegative}
\sum_{k=a}^{a-b} g(k) = - \sum_{a-b+1}^{a-1} g(k)
\end{equation}

The empty sum rule and Eq \ref{sumToNegative} were proposed earlier in \cite{Mueller2010}, in which polynomial approximation is used to define a restricted class of simple finite sums as will be discussed later in Chapter \ref{Chapter6}. Finally, the integral rule follows immediately from the differentiation rule given earlier in Eq \ref{Eq214}, and is given by Eq \ref{Eq216}. In fact, if we denote $h(n)=\sum_{k=a}^n \int^k g(t) \,dt$, then $c_1=-h'(a-1)$. Thus, the integration rule in Eq \ref{Eq216} yields a single arbitrary constant only as expected, which is $c_2$.
\begin{equation}\label{Eq216}
\int^n \sum_{k=a}^{t} g(k) \,dt = \sum_{k=a}^{n} \int^k g(t)\,dt +c_1 n + c_2, \;\;\;\;\;\;\;\; c_1=-\frac{d}{dx}\sum_{k=a}^x \int^k g(t)\,dt \; \; \Big|_{x=a-1}
\end{equation}

Using the above rules for simple finite sums, we can quickly deduce similar rules for simple finite products by rewriting finite products as $\prod_{k=a}^n g(k)=\exp\{\sum_{k=a}^n \log{g(k)}\}$ and using the chain rule, which yields the differentiation rule given in Eq \ref{Eq217}.
\begin{equation}\label{Eq217}
\frac{d}{dn}\prod_{k=a}^n g(k)=\prod_{k=a}^n g(k) \Big(\sum_{k=a}^n \frac{g'(k)}{g(k)} +c\Big), \;\;\;\;\;\;\;\;\;\; \;\;\; c=f_G'(a-1)
\end{equation}

Similar to the case of finite sums, we have $\prod_{k=a}^{a-1}g(k)=1$, which again holds \emph{by necessity}. For example, exponentiation, which can be written as $x^n=\prod_{k=a}^n x$  implies that $x^0=\prod_{k=1}^0 x=1$. Similarly, the factorial function, given by $n!=\prod_{k=1}^n k$ implies that $0!=1$. Table \ref{TableRules} summarizes key results. Again, it is crucial to keep in mind that the empty product rule only holds for $\prod_{k=a}^{a-1}$. For example, $\prod_{k=a}^{a-2}$ may or may not be equal to one.
\begin{table}[b]
\centering\begin{tabular} { |l |c|}
\hline 
\small{\textbf{Rule 1}: Derivative rule for simple finite sums}  &  \small{$f_G' (n)=\sum_{k=a}^n g'(k)+c_1$}\\[7pt]
\small{\textbf{Rule 2}: Integral rule for simple finite sums}  &  \small{$\int^n \sum_{k=a}^{t} g(k) \,dt = \sum_{k=a}^{n} \int^k g(t)\,dt +c_1 n + c_2$}\\[7pt]
\small{\textbf{Rule 3}: Empty sum rule}  &  \small{$\sum_{k=a}^{a-1} g(k)=0$} \\[7pt]
\small{\textbf{Rule 4}: Derivative rule for simple finite products}  & \small{$\frac{d}{dn}\prod_{k=a}^n g(k)=\prod_{k=a}^n g(k) \Big(\sum_{k=a}^n \frac{g'(k)}{g(k)} +c_1\Big)$}\\[7pt]
\small{\textbf{Rule 5}: Empty product rule} &  \small{$\prod_{k=a}^{a-1}g(k)=1$}\\[7pt]
\hline
\end{tabular}
\caption {A summary of foundational rules. In each rule, $c_1$ is a non-arbitrary constant.}
\label{TableRules}
\end{table}

Now, looking into Rule 1 in Table \ref{TableRules}, we deduce that the following general condition holds for some constants $c_r$:
\begin{equation}\label{Eq218}
f^{(r)}(n)=\sum_{k=a}^n g^{(r)}(k) +c_r
\end{equation}
Keeping Eq \ref{Eq218} in mind, we begin to see how \emph{unique} generalization of simple finite sums can be \emph{naturally} defined. To see this, we first note from Eq \ref{Eq218} that we always have:
\begin{equation}\label{Eq219}
f^{(r)}(n)-f^{(r)}(n-1)=g^{(r)}(n), \;\;\;\;\;\;\;\;\;\;\;\;\;\;\; \text{for all $r\ge 0$}
\end{equation}
Eq \ref{Eq219} is clearly a stricter condition than the basic recurrence identity Eq \ref{Eq211} that we started with. In fact, we will now establish that the conditions in Eq \ref{Eq219} are required in order for the function $f(n)$ to be smooth, i.e. infinitely differentiable.

First, suppose that we have a simple finite sum $\simplefsum{f}{n}{k}{0}{g}$. If we wish to find a continuous function $f_G(n)$ defined in the interval $[0, 2]$ that is only required to satisfy the recurrence identity Eq \ref{Eq211} and correctly interpolates the discrete finite sum, we can choose any function in the interval $[0, 1]$ such that $f_G(0)=g(0)$ and $f_G(1)=g(0)+g(1)$. Then, we define the function $f_G(n)$ in the interval $[1, 2]$ using the recurrence identity. Let us see what happens if we do this, and let us examine how the function changes as we try to make $f_G(n)$ smoother.

First, we will define $f_G(n)$ as follows:
\begin{equation}\label{Eq21_10}
f_G(n)=\begin{cases} a_0+a_1 n & \text{if $n \in [0, 1]$}\\ g(n)+f_G(n-1) & \text{if $n>1$}
\end{cases}
\end{equation}

The motivation behind choosing a linear function in the interval $[0, 1]$ is because any continuous function can be approximated by polynomials (see Weierstrass Theorem \cite{Ross1980}). So, we will initially choose a linear function and add higher degrees later to make the function smoother. To satisfy the conditions $f_G(0)=g(0)$ and $f_G(1)=g(0)+g(1)$, we must have:
\begin{equation}\label{Eq21_11}
f_G(n)=\begin{cases} g(0)+g(1) \,n & \text{if $n \in [0, 1]$}\\ g(n)+f_G(n-1) & \text{if $n>1$}
\end{cases}
\end{equation}

Clearly, Eq \ref{Eq21_11} is a continuous function that satisfies the required recurrence identity and boundary conditions. However, $f_G'(n)$ is, in general, discontinuous at $n=1$ because $f_G'(1^-)=g(1)$ whereas $f_G'(1^+)=g(1)+g'(1)$. To improve our estimate such that both $f_G(n)$ and $f_G'(n)$  become continuous throughout the interval $[0, 2]$, we make $f_G(n)$ a polynomial of degree 2 in the interval $[0, 1]$. Here, to make $f_G'(n)$ a continuous function, it is straightforward to see that we must satisfy the condition in Eq \ref{Eq219} for $r=1$. The new set of conditions yields the following approximation:
\begin{equation}\label{Eq21_12}
f_G(n)=\begin{cases} \frac{g(0)}{1}+\bigl(\frac{g(1)}{1}-\frac{g'(1)}{2}\bigr) \,n +\frac{g^{(2)}(1)}{2} n^2& \text{if $n \in [0, 1]$}\\ g(n)+f_G(n-1) & \text{if $n>1$}
\end{cases}
\end{equation}

Now,  $f_G(n)$ and $f_G'(n)$ are both continuous throughout the interval $[0, 2]$, and the function satisfies recurrence and boundary conditions. However, the 2\textsuperscript{nd} derivative is now discontinuous. Again, to make it continuous, we improve our estimate by making $f_G(n)$ a polynomial of degree 3 in the interval $[0, 1]$ and enforcing the condition in Eq \ref{Eq219} for all $r \in \{0, 1, 2\}$. This yields:
\small
\begin{equation}\label{Eq21_13}
f_G(n)=\begin{cases} \frac{g(0)}{1}+\bigl(\frac{g(1)}{1}-\frac{g'(1)}{2}+\frac{g^{(2)}(1)}{12}\bigr) \,\frac{n}{1!} +\bigl(\frac{g'(1)}{1}-\frac{g^{(2)}(1)}{2}\bigr)\frac{n^2}{2!}+\frac{g^{(2)}(1)}{1} \frac{n^3}{3!}& \text{if $n \in [0, 1]$}\\ \small{g(n)+f_G(n-1)} & \text{if $n>1$}
\end{cases}
\end{equation}
\normalsize
Now, we begin to see a curious trend. First, it becomes clear that in order for the function to satisfy the recurrence identity and initial condition in Eq \ref{Eq211} and at the same time be infinitely differentiable, it must satisfy Eq \ref{Eq219}. Second, its \emph{m}th derivative  seems to be given by $f_G^{(m)}(0)=\sum_{r=0}^\infty (-1)^r \frac{b_r}{r!} g^{(r+m-1)}(1)$, where $b_r=\{1, \frac{1}{2}, \frac{1}{12}, \ldots\}$. Indeed this result will be established later, where the constants $b_r$ are Bernoulli numbers!

\begin{figure} [h]
\centering
\includegraphics[scale=0.4]{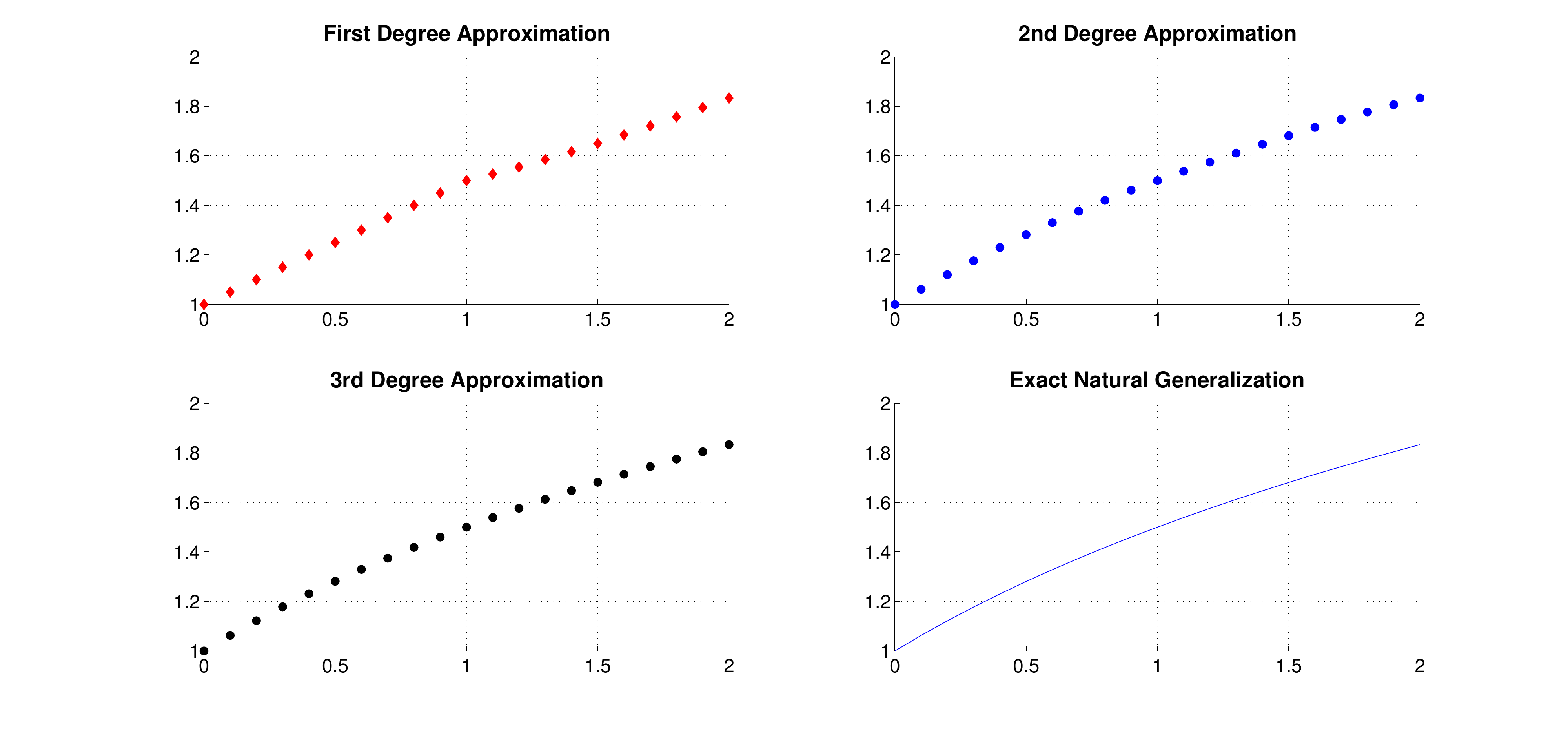}
\caption[The successive polynomial approximation method]{The successive polynomial approximation method applied to $\sum_{k=0}^n \frac{1}{k+1}$. Clearly, the method converges to a unique function. }
\label{figure1}
\end{figure}

Figure \ref{figure1} illustrates the previous process for the simple finite sum $\sum_{k=0}^n \frac{1}{k+1}$. The finite sum enumerates the constants $H_{n+1}$, where $H_n$ is the \emph{n}th harmonic number. One well-known generalization of harmonic numbers to non-integer arguments can be expressed using the well-known digamma function $\frac{d}{dn}\log{\Gamma(n)}$, which is also depicted in the figure. Interestingly, it is clear from the figure that the 2\textsuperscript{nd} degree approximation is already almost indistinguishable from what we would have obtained had we defined $f_G(n)$ in terms of the digamma function. More precisely, the successive approximation method converges the function $\frac{d}{dn} \log{\Gamma(n+2)} - \lambda$, where $\Gamma$ is the Gamma function and $\lambda$ is Euler's constant.

This brings us to the following central statement: \\[6pt]

\hrule
\begin{theorem}[\textbf{Statement of Uniqueness}]\label{StatofUniq}
Given a simple finite sum \simplefinitesum, where $g(k)$ is analytic in the domain $[a, \infty)$, define $p_m(n)$ to be a polynomial of degree $m$, and define $f_{G,m}(n)$ by:\\
\begin{equation*}
f_{G,m}(n)=\begin{cases} p_m(n) \quad &\text{if $n \in [a, a+1]$}\\ g(n)+f_{G,m}(n-1) \quad \quad \quad  &\text{otherwise}
\end{cases}
\end{equation*}
If we require that $f_{G,m+1}(n)$ be \emph{m}-times differentiable in the domain $[a, \infty)$, then the sequence of polynomials $p_m(n)$ is unique. In particular, its limit $f_G(n)=\lim_{m\to\infty} f_{G,m} (n)$ is unique and satisfies both initial condition $f_G(a)=g(a)$ and recurrence identity $f_G^{(r)}(n)=g^{(r)}(n)+f_G^{(r)}(n-1)$. 
\end{theorem} 
\begin{proof} 
By construction of Eq \ref{Eq21_13}.
\end{proof}

\hrule\vspace{12pt}

Theorem \ref{StatofUniq} shows that a \emph{natural} generalization of simple finite sums to the complex plane $\mathbb{C}$ can be \emph{uniquely} defined despite the fact that infinitely many functions qualify to generalize finite sums. According to Theorem \ref{StatofUniq}, while infinitely many functions exist that can correctly satisfy both initial condition and recurrence identity, there is \emph{one and only one} function that can arise out of the successive polynomial approximations method given in the Theorem. Why is this significant? This is because, in a loose sense, that unique function is the simplest possible generalization; as you incorporate additional information about the original finite sum, you gain additional information about its unique natural generalization. 

Of course, the argument of being a \lq\lq natural'' generalization is intricate. One has to agree on what \lq\lq natural'' means in the first place. However, as shown in Theorem \ref{StatofUniq}, a reasonable statement of natural generalization can be made. In addition, we stated earlier in Eq \ref{Eq213} that periodic functions $P(n)$ will not be artificially added because the original finite sum carries no information about such functions. We will show later that such choice of $P(n)=0$ is exactly what the statement of Theorem \ref{StatofUniq} entails. Furthermore, we will show additional properties of unique natural generalization. For instance, if \simplefinitesum\ tends to a limit $V$ as $n\to\infty $  then its unique natural generalization also tends to same limit $V\in\mathbb{C}$. In the latter case, unique natural generalization is given by the \lq\lq natural'' expression in Eq \ref{unqSmpConvg} that was discussed earlier. Moreover, we will show that if the finite sum  \simplefinitesum\ asymptotically behaves as a polynomial, then its unique natural generalization also converges to the same polynomial asymptotically, and others. 

In the sequel, we will derive a complete set of rules for performing infinitesimal calculus on simple finite sums and products without having to know what the generalized definition $f_G(n)$ actually is. We will also show that the foundational rules given in Table \ref{TableRules} indeed correspond to the unique most natural generalization to simple finite sums and products. However, before this is done, we present a few elementary illustrative examples first.

\section{Examples to Foundational Rules}\label{Sect2dot2}
For our first example, we return to the power sum function presented earlier in Eq \ref{PSDisctrete}. The first point we note about this function is that its $(m+2)$\textsuperscript{th} derivative is \textbf{zero}, which follows immediately from  Rule 1 in Table \ref{TableRules}. Thus, the unique most natural generalized definition of the power sum function has to be a polynomial of degree $(m+1)$. Because there exists sufficiently many sample points to the discrete power sum function, infinite to be more specific, that polynomial has to be unique. Of course, those polynomials are given by the Bernoulli-Faulhaber formula.

Second, assume that $a=1$ and let $f_m(n)=\sum_{k=1}^n k^m$ then we have by Rule 2:
\begin{equation}\label{Eq221}
\int_0^n \sum_{k=1}^t k^m \, dt = \frac{1}{m+1} \sum_{k=1}^n k^{m+1} +c_1 n + c_2
\end{equation}

Using the empty sum rule, i.e. Rule 3 in Table \ref{TableRules}, and after setting $n=0$, we have $c_2=0$.  Now, if we let $f_{G,m}(n)$ denotes the polynomial $\sum_{k=1}^n k^m$, we have by Eq \ref{Eq221} the following simple recursive rule for deducing closed-form expressions of power sums:
\begin{equation}\label{cookbook}
f_{G,m+1}(n)=(m+1)\int_0^n f_{G,m}(t)\,dt -(m+1)\Bigl(\int_0^1f_{G,m}(t)\,dt \Bigr) n
\end{equation}

Third, because we always have $\sum_{k=1}^n k^m = \sum_{k=0}^n k^m$, Rule 3 can be used in either case to deduce that $f_{G,m}(0)=f_{G,m}(-1)=0$. That is, $n(n+1)$ is always a proper factor of power sum polynomials. The fact that $n(n+1)$ is always a proper factor was observed as early as Faulhaber himself in his book \emph{Academia Algebrae} in 1631 \cite{Havil03}. The recursive solution of power sums given in Eq \ref{cookbook} is well known and was used by Bernoulli himself \cite{Spivey2006, Bloom1993}.  Due to its apparent simplicity, it has been called the \lq\lq Cook-Book Recipe'' by some mathematicians \cite{Shirali2007}. Nevertheless, its simple two-line proof herein illustrates efficacy of Summability Calculus in deducing closed-form expressions of finite sums.

Our second example is the function $f(n)$ given in Eq \ref{Eq223}. Before we deduce a closed-form expression $f_G(n)$ of this function using Summability Calculus, we know in advance that $f_G(0)=f_G(-1)=0$ because we could equivalently take the lower bound of $k$ inside the summation to be zero without changing the function itself, which is similar to the power sum function discussed earlier. In other words, since $\sum_{k=1}^n k x^k$ and $\sum_{k=0}^n k x^k$ share the same boundary condition and recurrence identity, they correspond to the exact same function. Consequently, the empty sum rule can be applied in either case. 
\begin{equation}\label{Eq223}
f(n)= \sum_{k=1}^n k \,x^k
\end{equation} 

Using Rule 1, we have:
\begin{equation}\label{Eq224}
f_G'(n)= \sum_{k=1}^n x^k + \log{x} \; f_G(n) + c = \frac{1-x^{n+1}}{1-x} + \log{x} \; f_G(n) + c
\end{equation} 

Eq \ref{Eq224} is a first-order linear differential equation whose solution is available in closed-form \cite{Ten85}. Using initial condition $f_G(1)=x$, and after rearranging the terms, we obtain the closed-form expression given Eq \ref{Eq225}. Note that $f_G(0)=f_G(-1)=0$ holds as expected. 
\begin{equation}\label{Eq225}
f_G(n)=\frac{x}{(1-x)^2} \big(x^n (n(x-1)-1)+1\bigr)
\end{equation} 

Our third example is the power function $x^n=\prod_{k=1}^n x$. Because $g(k)=x$ is independent of $k$, we have $g'(k)=0$. Using Rule 4, we deduce that $\frac{d}{dn} x^n = c x^n$, for some constant $c$, which is indeed an elementary result in calculus. Fourth, if we let $f(n)$ be the discrete factorial function and denote its generalization using the PI symbol $\Pi$, as Gauss did, then, by Rule 4, we have $\Pi'(n)/\Pi(n) = \bigr(\sum_{k=1}^n 1/k + c\bigr)$, which is a well-known result where the quantity $\Pi'(n)/\Pi(n)$ is $\psi(n+1)$  and $\psi$ is the digamma function. Here, in the last example, $c$ is Euler's constant. 

For a fifth example, let us consider the function $\sum_{k=0}^n \sin{k}$. If we denote $\beta_1 = f_G'(0)$ and $\beta_2 = f_G^{(2)}(0)$ and by successive differentiation using Rule 1, we note that the generalized function $f_G(n)$ can be formally expressed using the series expansion in Eq \ref{Eq226}. Thus,  $f_G(n)=\beta_1\sin{n} -\beta_2(1-\cos{n})$, where the constants $\beta_1$ and $\beta_2$ can be found using any two values of $f(n)$. This can be verified readily to be correct.  
\begin{equation}\label{Eq226}
f_G(n)=\frac{\beta_1}{1!} n - \frac{\beta_2}{2!} n^2 - \frac{\beta_1}{3!} n^3 + \frac{\beta_2}{4!} n^4 +\frac{\beta_1}{5!} n^5 \dotsm 
\end{equation}

The previous five examples illustrate why the constant $c$ in the derivative rules, i.e. Rule 1 and Rule 4, is not an arbitrary constant. However, if it is not arbitrary, is there a systematic way of deducing its value? We have previously answered this question in the affirmative and we will present a formula for $c$ later. Nevertheless, it is instructive to continue with one additional example that illustrates, what is perhaps, the most straightforward special case in which we could deduce the value of $c$ immediately without having to resort to its, somewhat, involved expression. This example also illustrates the concept of natural generalization of discrete functions and why a finite sum embodies within its discrete definition a natural analytic continuation to all $n\in\mathbb{C}$.

Our final example in this section is the function $H_n=\sum_{k=1}^n 1/k$, which is also known as the Harmonic number. Its derivative is given by Rule 1 and it is shown in Eq \ref{Eq227}. 
\begin{equation}\label{Eq227}
\frac{d}{dn} H_n=-\sum_{k=1}^n \frac{1}{k^2} + c
\end{equation} 

Now, because $\Delta f(n) \to 0$ as $n \to \infty$, we expect its natural generalization $f_G(n)$ to exhibit the same behavior as well (as will be shown later, this statement is, in fact, always correct). Thus, we expect $f'_G(n) \to 0$ as $n \to \infty$. Plugging this condition into Eq \ref{Eq227} yields the unique value $c=\zeta_2$. Of course, this is indeed the case if we generalize the definition of harmonic numbers using the digamma function $\psi$ and Euler's constant $\lambda$ as shown in Eq \ref{Eq228} (for a brief introduction into the digamma and polygamma functions, see \cite{MathWorldPolygamma, AbraPsi, AbraPoly}). Again, using Eq \ref{Eq228}, $H_0=0$  as expected but $H_{-1}=\infty \neq 0$. In this example, therefore, we used both the recursive property of the discrete function as well as one of its visible natural consequences to determine how its derivative should behave at the limit $n \to \infty$. So, in principle, we had a \emph{macro} look at the asymptotic behavior of the discrete function to deduce its \emph{local} derivative at $n=0$, which is the interesting duality we discussed earlier in Chapter \ref{ChaptIntro}. Validity of this approach will be proved rigorously in the sequel.
\begin{equation}\label{Eq228}
H_n=\psi(n+1)+\lambda
\end{equation}
\section{Semi-Linear Simple Finite Sums}\label{Section_23}
In this section, we present fundamental results in Summability Calculus for an important special class of discrete functions that will be referred to as \emph{semi-linear} simple finite sums and their associated products. The results herein are extended in the next section to the general case of all simple finite sums and products. We start with a few preliminary definitions. \\

\hrule
\begin{definit}{\textbf{(Nearly-Convergent Functions)}}\label{definition1}
A function $g(k)$ is called nearly convergent if $\lim_{k\to\infty} g'(k) = 0$ and one of the following two conditions holds:
\begin{enumerate} 
\item $g(k)$ is asymptotically non-decreasing and concave. More precisely, there exists $k_0$ such that for all $k>k_0$,  $g'(k)\ge 0$ and $g^{(2)}(k)\leq 0$. 
\item $g(k)$ is asymptotically non-increasing and convex. More precisely, there exists $k_0$ such that for all $k>k_0$,  $g'(k)\leq 0$ and $g^{(2)}(k)\ge 0$
\end{enumerate}
\end{definit}
\hrule \vspace{12pt}
\hrule
\begin{definit}{\textbf{(Semi-Linear Simple Finite Sums)}}\label{definition2}
A simple finite sum \simplefinitesum\ is called semi-linear if $g(k)$ is nearly-convergent. 
\end{definit}
\hrule \vspace{12pt}

Informally speaking, a function $g(k)$ is nearly convergent if it is both asymptotically monotonic, well shaped, and its rate of change is asymptotically vanishing. In other words, $g(k)$ \emph{becomes almost constant} in the bounded region $k\in(k_0-W, \,k_0+W)$ for any fixed $W\in\mathbb{R}$ as $k_0\to\infty$. Semi-linear simple finite sums are quite common, e.g. $\sum_{k=a}^n k^m$ for $m<1$ and $\sum_{k=a}^n \log^s{k}$ for $a>0$, and Summability Calculus is quite simple in such important cases. Intuitively speaking, because $g(k)$ is almost constant asymptotically in any bounded region, we expect $f_G'(n)$ to be close to $g(n)$ as $n \to\infty$. This is indeed the case as will be shown later. An illustrative example of a nearly-convergent function is depicted in Figure \ref{figure2}.

\begin{figure} [h]
\centering
\includegraphics[scale=0.4]{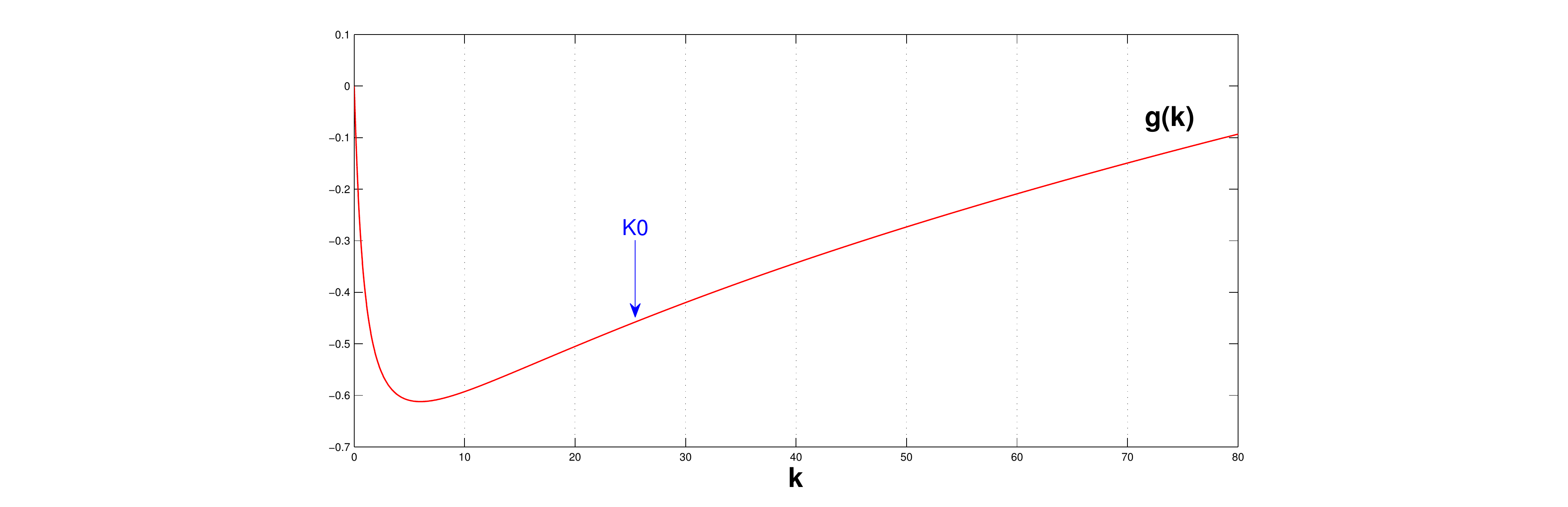}
\caption[Illustration of nearly convergent functions]{Example of a nearly-convergent function}
\label{figure2}
\end{figure}

As stated earlier, Summability Calculus in the case of semi-linear sums is quite simple. Because the function $g(k)$ is nearly-convergent, $g(k)$ becomes by definition almost constant over arbitrary-length intervals. Thus, the simple finite sum \simplefinitesum\ becomes almost a linear function asymptotically over arbitrary-length intervals as well, where the rate of change $f_G'(n)$ approaches $g(n)$; hence the name. Looking into the differentiation rule of simple finite sums, i.e. Rule 1, the non-arbitrary constant $c$ that is \emph{independent of} $n$ and arises out of the differentiation rule should, thus, be given by the limit $\lim_{n\to\infty}\bigl\{g(n)-\sum_{k=a}^n g'(k)\bigr\}$. The following series of lemmas and theorems establish this intuitive reasoning more rigorously.\\ 
\hrule

\begin{lemma}\label{lem21}
If $g(k)$ is nearly convergent, then the limit $\displaystyle\lim_{n\to\infty}\bigl\{g(n)-\sum_{k=a}^n g'(k)\bigr\}$ exists. 
\end{lemma}
\begin{proof}
We will prove the lemma here for the case where $g(k)$ is asymptotically non-decreasing and concave. Similar steps can be used in the second case where $g(k)$ is asymptotically non-increasing and convex. First, let $E_n$ and $D_n$ be given by Eq \ref{EqLemma1_1}, where $k_0$ is defined as in Definition \ref{definition1}.
\begin{equation}\label{EqLemma1_1}
E_n=\sum_{k=k_0}^n g'(k) - g(n),\\ D_n=E_n - g'(n)
\end{equation}

We will now show that the limit $\lim_{n\to\infty}E_n$ exists. Clearly, since $k_0<\infty$, Lemma \ref{lem21} follows immediately afterwords. By definition of $E_n$ and $D_n$, we have:
\begin{equation} \label{EqLemma1_2}
E_{n+1}-E_n=g'(n+1)-\bigl(g(n+1)-g(n)\bigr)
\end{equation}
\begin{equation}\label{EqLemma1_3}
D_{n+1}-D_n=g'(n)-\bigl(g(n+1)-g(n)\bigr)
\end{equation}

Because $g'(n)\ge 0$ by assumption, $D_n\leq E_n$. Thus, $D_n$ is a lower bound on $E_n$. However, concavity of $g(n)$ implies that $g'(n+1)\leq g(n+1)-g(n)$ and $g'(n)\ge g(n+1)-g(n)$. Placing these inequalities into Eq \ref{EqLemma1_2} and Eq \ref{EqLemma1_3} implies that $E_n$ is a non-increasing sequence while $D_n$ is a non-decreasing sequence. Since $D_n$ is a lower bound on $E_n$, $E_n$ converges, which completes proof of the lemma. 
\end{proof}
\hrule\vspace{12pt}

Lemma \ref{lem21} immediately proves not only existence of Euler's constant but also \emph{Stieltje's constants} and \emph{Euler's generalized constants} as well (for definition of these families of constants, the reader can refer to \cite{MathWorldStiel,Havil03}). Notice that, so far, we have not made use of the fact that $\lim_{n\to\infty} g'(n) =0$, but this last property will be important for our next main theorem. \\ \hrule

\begin{theorem}\label{theorem231} 
Let g(k) be analytic in an open disc around each point in the domain $[a, n+1]$. Also, let a simple finite sum be given by \simplefinitesum, where $g^{(m)}(k)$ that denotes the \emph{m}th derivative of $g(k)$ is a nearly-convergent function for all $m\ge 0$, and let $f_G(n)$ be given formally by the following series expansion around $n=a-1$:
\begin{equation}\label{EqTheorem231}
f_G(n)=\sum_{m=1}^\infty \frac{c_m}{m!} (n-a+1)^m,\\ \text{where } c_m=\lim_{n\to\infty} \Bigl\{g^{(m-1)}(n)-\sum_{k=a}^n g^{(m)}(k)\Bigr\} 
\end{equation}
Then, $f_G(n)$ satisfies formally the recurrence identity and initial conditions given in Eq \ref{Eq211}. Thus, $f_G(n)$ is a valid generalization of $f(n)$ to the complex plane $\mathbb{C}$.
\end{theorem}

\begin{proof}
As will be described in details in Chapter \ref{Chapter4}, there exists summability methods for Taylor series expansions such as Lindel\"of and Mittag-Leffler's  summability methods. More specifically, define:
\begin{center}$\vec{\textbf{f}}_{x_0, x}=\Bigl(\frac{f^{(0)}(x_0)}{0!}(x-x_0)^0, \frac{f^{(1)}(x_0)}{1!}(x-x_0)^1, \frac{f^{(2)}(x_0)}{2!}(x-x_0)^2, \dotsm\Bigr)$ \end{center}
That is,  $\vec{\textbf{f}}_{x_0, x}\in\mathbb{C}^\infty$ enumerates the terms of the Taylor series expansion of $f$ around the point $x_0$.
Then, there exists an operator $\mathbb{T}: \mathbb{C}^\infty \to \mathbb{C}$ on infinite dimensional vectors $\vec{\textbf{a}}=(a_0, a_1, a_2, \dots)$, given by $\mathbb{T} (\vec{\textbf{a}})=\displaystyle\lim_{\delta\to 0}\Bigl\{\sum_{j=0}^\infty \,\xi_\delta(j)\, a_j\Bigr\}$, for some constants $\xi_\delta(j)$, which satisfies the following properties:
\begin{enumerate} 
\item If $f(x)$ is analytic in the domain $[x_0, x]$, then we have $\mathbb{T}(\vec{\textbf{f}}_{x_0, x})=f(x)$.
\item $\mathbb{T}(\vec{\textbf{a}} + \alpha \vec{\textbf{b}})=\mathbb{T}(\vec{\textbf{a}}) + \alpha \mathbb{T}(\vec{\textbf{b}}) \quad\quad \text{(i.e. the operator is linear)}$
\item $\lim_{\delta\to 0}\Bigl\{\sum_{j=0}^\infty \,\xi_\delta(j)\, a_j\Bigr\} = \lim_{\delta\to 0}\Bigl\{\sum_{j=0}^\infty \,\xi_\delta(j+\kappa)\, a_j\Bigr\} $ for any fixed $\kappa\ge 0$ if $a_k=\frac{f^{(k)}}{k!}(x-x_0)^k$ for some function $f(x)$ (i.e. the operator is stable for Taylor series expansions). 
\end{enumerate} 
The exact value of $\xi_\delta(j)$ is irrelevant here. We only need to assume that they exist. In addition, third property is not fundamental because it actually follows from the first two properties. Two well-known examples of $\mathbb{T}$ summability methods are the methods of Lindel\"of and Mittag-Leffler \cite{HardyDiverg}. For instance, the  Lindel\"of summability method is given by $\xi_\delta(j)=j^{-\delta j}$. Now, let $c_m^{n_0}$ be given by: 
\begin{equation}\label{Theorem231Eq1}
c_m^{n_0}= g^{(m-1)}(n_0)-\sum_{k=a}^{n_0} g^{(m)}(k)
\end{equation} 
By definition, we have $\lim_{n_0\to\infty} c_m^{n_0}=c_m$. So, the argument is that the function $f_{n_0}(n)$ given formally by Eq \ref{Theorem231Eq2} satisfies the required recurrence identity and initial conditions given in Eq \ref{Eq221} at the limit $n_0\to\infty$. 
\begin{equation}\label{Theorem231Eq2}
f_{n_0}(n)=\sum_{m=1}^\infty \frac{c_m^{n_0}}{m!} (n-a+1)^m
\end{equation} 
To prove this last argument, we first note from Eq \ref{Theorem231Eq2} and the summability method $\mathbb{T}$ that the following holds: 
\begin{align*}
f_{n_0}(n)-f_{n_0}(n-1) &= \lim_{\delta \to 0}  \bigl\{ \sum_{m=1}^\infty \xi_\delta(m) \frac{c_m^{n_0}}{m!} (n-a+1)^m - \sum_{m=1}^\infty \xi_\delta(m) \frac{c_m^{n_0}}{m!} (n-a)^m\bigr\} \\
&=\lim_{\delta \to 0} \bigl\{ \sum_{m=0}^\infty \xi_\delta(m) \frac{K_m^{n_0}}{m!} (n-a)^m \bigr\} 
\end{align*}
\quad\quad\,, where $K_m^{n_0}=\sum_{j=1}^\infty \xi_\delta(j+m) \frac{c_{j+m}^{n_0}}{j!}\bigr\} $

Looking into the previous equation and because $\lim_{\delta \to 0} \xi_\delta(x) = 1$ for all $x\ge 0$, we immediately deduce that the Taylor coefficients of $f_{n_0}(n)-f_{n_0}(n-1)$ must be given by :
\begin{equation}
\frac{d^m}{dn^m} \bigl(f_{n_0}(n)-f_{n_0}(n-1)\bigr)\Big|_{n=a} = \lim_{\delta\to 0}K_m^{n_0}= \lim_{\delta\to 0}\sum_{j=1}^\infty \xi_\delta(j+m) \frac{c_{j+m}^{n_0}}{j!} 
\end{equation}
From now on, and with some abuse of notation, we will simply denote $K_m^{n_0} =\lim_{\delta\to 0}K_m^{n_0}$. So, in order for the recurrence identity to hold, we must have $\lim_{n_0\to\infty} \big(f_{n_0}(n)-f_{n_0}(n-1)\big)=g(n)$. Since Taylor series expansions of analytic functions are unique, we must have:
\begin{equation}\label{Theorem231Eq3}
\lim_{n_0\to\infty} K_m^{n_0} = g^{(m)}(a), \\ \text{for all}\ m\ge 0
\end{equation}

Luckily, because of property 3 of the summability method $\mathbb{T}$, the conditions in Eq \ref{Theorem231Eq3} for $m\ge 0$ are all \emph{equivalent} by definition of $c_m^{n_0}$ so we only have to prove that it holds for $m=0$. In other words, the expression Eq \ref{Theorem231Eq3} is a functional identity. If it holds for $m=0$, then it also holds for all $m$. To prove that it holds for $m=0$, we note that:
\begin{align*}
\lim_{n_0\to\infty} K_0^{n_0} &=\lim_{n_0\to\infty} \lim_{\delta \to 0} \bigl\{ \sum_{j=1}^\infty \xi_\delta(j) \frac{c_{j}^{n_0}}{j!}\bigr\}\\
&=\lim_{n_0\to\infty} \lim_{\delta \to 0} \bigl\{ \sum_{j=1}^\infty \xi_\delta(j) \frac{1}{j!} \bigl(g^{(j-1)}(n_0)-\sum_{k=a}^{n_0} g^{(j)}(k)\bigr)\bigr\}\\
&=\lim_{n_0\to\infty} \lim_{\delta \to 0} \bigl\{ \sum_{j=1}^\infty \xi_\delta(j) \frac{1}{j!} g^{(j-1)}(n_0)\bigr\}\\
&\quad\quad\quad\quad-\lim_{n_0\to\infty} \lim_{\delta \to 0} \bigl\{ \sum_{j=1}^\infty \xi_\delta(j) \frac{1}{j!} \sum_{k=a}^{n_0} g^{(j)}(k)\bigr\}
\end{align*}
In the last step, we split the sum because the summability method $\mathbb{T}$ is linear. Now, because $g(k)$ is analytic by assumption in the domain $[a, n+1]$, its anti-derivative exists and it is also analytic over the same domain. Therefore, we have:
\begin{equation}\label{Theorem231Eq4}
\lim_{\delta \to 0} \bigl\{ \sum_{j=1}^\infty \xi_\delta(j) \frac{1}{j!} g^{(j-1)}(n_0)\bigr\}=\int_{n_0}^{n_0+1} g(t)\,dt
\end{equation}
In addition, we have:
\begin{equation}\label{Theorem231Eq5}
\lim_{\delta \to 0} \bigl\{ \sum_{j=1}^\infty \xi_\delta(j) \frac{1}{j!} g^{(j)}(s)\bigr\}=g(s+1)-g(s)
\end{equation}
In both equations, we used the earlier claim that the summability method $\mathbb{T}$ correctly sums Taylor series expansions under stated conditions. From Eq \ref{Theorem231Eq5}, we realize that:
\begin{equation}\label{Theorem231Eq6}
\lim_{\delta \to 0} \bigl\{ \sum_{j=1}^\infty \xi_\delta(j) \frac{1}{j!} \sum_{k=a}^{n_0} g^{(j)}(k)\bigr\} = g(n_0+1)-g(a)
\end{equation}
This is because the left-hand sum is a telescoping sum. 
Plugging both Eq \ref{Theorem231Eq4} and Eq \ref{Theorem231Eq6} into the last expression for $\lim_{n_0\to\infty} K_0^{n_0}$ yields:
\begin{equation}\label{Theorem231Eq7}
\lim_{n_0\to\infty} K_0^{n_0} = g(a)-\lim_{n_0\to\infty} \Bigl\{g(n_0+1)-\int_{n_0}^{n_0+1} g(t)\,dt \Bigr\}
\end{equation}

Now, we need to show that the second term on the right-hand side evaluates to zero. This is easily shown upon noting that the function $g(k)$ is nearly-convergent. As stated earlier, since $g(k)$ is nearly convergent, then $g^{(m)}(n_0)$ vanish asymptotically for all $m\ge 1$. Since $g(k)$ is also asymptotically monotone by Definition \ref{definition1}, then for any $\epsilon>0$, there exists a constant $N$ large enough such that:
\begin{equation}\label{Theorem231Eq8}
\max_{x\in[n_0, n_0+\tau]} g(x) - \min_{x\in[n_0, n_0+\tau]} g(x) <\epsilon, \;\;\text{for all }\ n_0>N \text{ and } 0\leq\tau<\infty
\end{equation}
Consequently, we can find a constant $N$ large enough such that:
\begin{equation}\label{Theorem231Eq9}
\Big| g(n_0+1)-\int_{n_0}^{n_0+1} g(t)\,dt \Bigr|<\epsilon, \;\;\text{for all }\ n_0>N
\end{equation}
Because $\epsilon$ can be made arbitrary close to zero for which $N$ has to be sufficiently large, we must have:
\begin{equation}\label{Theorem231Eq10}
\lim_{n_0\to\infty} \Bigl\{g(n_0+1)-\int_{n_0}^{n_0+1} g(t)\,dt\Bigr\}=0
\end{equation}
Therefore, we indeed have $\lim_{n_0\to\infty} K_0^{n_0} = g(a)$. As argued earlier, this implies that Eq \ref{Theorem231Eq3} indeed holds. 

With this in mind, we have established formally that $\lim_{n_0\to\infty} \big(f_{n_0}(n)-f_{n_0}(n-1)\big)=g(n)$. Therefore, $f_G(n)$ given by Eq \ref{EqTheorem231} satisfies formally the recurrence identity in Eq \ref{Eq221}. Finally, proving that the initial condition also holds follows immediately. By plugging $n=a-1$ into the formal expression for $f_G(n)$ given by Eq \ref{EqTheorem231}, we note that $f_G(a-1)=0$ so $f_G(a)=g(a)$ by the recurrence identity, which completes proof of the theorem.  
\end{proof} 
\hrule
\begin{lemma}\label{theorem231TwoProperties}
The unique natural generalization given by Theorem \ref{theorem231} satisfies the two defining properties of simple finite sums. 
\end{lemma}
\begin{proof}
The proof will be deferred until Section \ref{sectionSFSGC}. 
\end{proof}
\hrule\vspace{12pt} 

Theorem \ref{theorem231} provides us with a very convenient method of performing infinitesimal calculus such as differentiation, integration, and computing series expansion. As will be illustrated in Section \ref{SectionExampleSL}, it can even be used to deduce asymptotic expressions such as Stirling's approximation. However, one important concern that should be raised at this point is whether the generalization to non-integer arguments given by Theorem \ref{theorem231} is a \emph{natural} generalization and, if so, whether it is equivalent to the \emph{unique natural generalization} given earlier using the successive polynomial approximation method of Theorem \ref{StatofUniq}. This is, in fact, the case. We will now see why Theorem \ref{theorem231} is a very natural generalization by establishing its direct correspondence with \emph{linear fitting}, and prove that its equivalent to the successive polynomial approximation method given in Theorem \ref{StatofUniq} later in Section \ref{sectionSFSGC}. \\ \hrule

\begin{claim}\label{claimLinearFitting}
If a simple finite sum is given by \simplefinitesum, where $g^{(m)}(k)$ is a nearly-convergent function for all $m\ge 0$, then its generalization $f_G(n)$ given by Theorem \ref{theorem231} is the unique most natural generalization of $f(n)$.
\end{claim}
\begin{proof}
As stated earlier, arguments of unique natural generalization are clearly intricate. Several ad hoc definitions or criteria could be proposed for what natural generalization actually entails. We have already presented one such definition in Theorem \ref{StatofUniq}. There, we showed that out of all possible functions that can correctly satisfy initial condition and recurrence identity, there exists a  function that can be singled out as the \emph{unique} natural generalization to the simple finite sum. We will show later in Section \ref{sectionSFSGC} that such unique function is equivalent to the one given by Theorem \ref{theorem231}. 

Aside from this, however, there is, in fact, a special case in which it is universally agreeable what natural generalization should look like, and that is the case of \emph{linear fitting}. More specifically, if a collection of points in the plane can be perfectly interpolated using a straight line \footnote{or a hyperplane in $\mathbb{R}^n$ in case of $n$-dimensional points}, and given lack of additional special requirements, then it is indeed the unique most natural generalization.

In the previous proof of Theorem \ref{theorem231}, we showed that if $g(k)$ is nearly convergent, and given any $\epsilon>0$, then we can always find a constant $N$ large enough such that Eq \ref{Theorem231Eq8} holds. Eq \ref{Theorem231Eq8} implies that the finite sum \simplefinitesum\ grows \emph{almost linearly} around $n$ for sufficiently large $n$ and its derivative can be made arbitrary close to $g(n)$ as $n\to\infty$. We, therefore, expect its natural generalization to exhibit the same behavior and to have its derivative $f_G'(n)$ to be arbitrary close to $g(n)$ for sufficiently large $n$. 

Now, we know by Rule 1 of Table \ref{TableRules} that the following holds for some non-arbitrary constant $c$:
\begin{equation*}
f_G' (n)=\sum_{k=a}^n g'(k)+c
\end{equation*}
Comparing such equation with the fact that $f_G'(n)\to g(n)$ as $n\to\infty$ implies that the constant $c$ of the unique natural generalization must be given by: 
\begin{equation*}
c= \lim_{n\to\infty} \big\{f_G' (n)-\sum_{k=a}^n g'(k)\bigr\} = \lim_{n\to\infty} \big\{g(n)-\sum_{k=a}^n g'(k)\bigr\}
\end{equation*}
However, this is exactly what Theorem \ref{theorem231} states, and, by Lemma \ref{lem21}, the limit exists so the function $f_G(n)$ given by Theorem \ref{theorem231} is indeed the unique function that satisfies such property of natural generalization. 
\end{proof}
\hrule\vspace{12pt} 

Now, we conclude with one last corollary. Here, we will show that if $\sum_{k=a}^{\infty} g(k)$ exists, then an additional statement of uniqueness of natural generalization can be made, which was touched upon earlier in Chapter \ref{ChaptIntro}. \\ \hrule 
\begin{corollary}\label{cor231_1} 
Let \simplefinitesum\ be semi-linear and suppose that $\sum_{k=a}^{\infty} g(k)$ exists. Also, let $f_G(n)$ be its unique natural generalization as given by Theorem \ref{theorem231}, then we have for all $n\in\mathbb{C}$: 
\begin{equation}\label{Eqcor231_1} 
f_G(n)=\sum_{k=a}^{\infty} g(k) - \sum_{k=n+1}^{\infty} g(k)
\end{equation}
\end{corollary}
\begin{proof}
The expression in Eq \ref{Eqcor231_1} is what we would expect from a unique natural generalization to simple finite sums if $\sum_{k=a}^{\infty} g(k)$ exists. This was discussed earlier in the Chapter \ref{ChaptIntro}. To prove that Eq \ref{Eqcor231_1} holds, define: 
\begin{equation}\label{Eqcor231_2} 
f^*(n)=\sum_{k=a}^{\infty} g(k) - \sum_{k=n+1}^{\infty} g(k)
\end{equation}
If we employ the two defining properties of finite sums, we can easily deduce that: 
\begin{equation}\label{Eqcor231_3} 
\sum_{k=n+1}^{\infty} g(k) = g(n+1)+g(n+2)+g(n+3) \dotsm
\end{equation}
Taking derivatives with respect to $n$ of $f^*(n)$ in Eq \ref{Eqcor231_2} reveals that all higher derivatives of $f^*(n)$ agree with the ones given by Theorem \ref{theorem231}. In particular, we have: 
\begin{equation}\label{Eqcor231_4}
\frac{d^r}{dn^r} f^*(n) = \sum_{k=a}^n g^{(r)}(k) - \sum_{k=a}^\infty g^{(r)}(k)  = \frac{d^r}{dn^r} f_G(n)
\end{equation} 
Since both functions share the same higher order derivatives $\frac{d^r}{dn^r} f(n)$ for $r\ge 0$, they have the same Taylor series expansion. By \emph{uniqueness} of Taylor series expansion, the two functions must be identical. 
\end{proof}
\hrule\vspace{12pt} 

\section{Examples to Semi-Linear Simple Finite Sums}\label{SectionExampleSL}
In this section we present many examples that illustrate efficacy of Summability Calculus in handling semi-linear simple finite sums. 
\subsection{Example I: Telescoping Sums}\label{exampleTelescopSum}
We will first start with the elementary simple finite product given by: 
\begin{equation}\label{Eq241_1}
f(n)=\prod_{k=2}^n (1-1/k)=\exp{\bigl\{\sum_{k=2}^n \log{(1-1/k)}}\bigr\}
\end{equation}

Now, we will use Summability Calculus to find an exact expression of the simple finite product. First, if we use Rule 4 for the derivative of simple finite products, we deduce that the derivative of the unique natural generalization $f_G(n)$ is given by Eq \ref{Eq241_2}. To find the exact value of $c$, we rewrite $f(n)$ as given in Eq \ref{Eq241_1}, i.e. we express $f(n)$ as a composition of the exponential function with a semi-linear simple finite sum. So, we can now employ Theorem \ref{theorem231} and the chain rule to find the exact value of $c$ as illustrated in Eq \ref{Eq241_3}. Thus, we conclude that $c = -1$.
\begin{equation}\label{Eq241_2}
f_G'(n)=f_G(n) \Big(\sum_{k=2}^n \frac{1}{k(k-1)} + c\Big)
\end{equation}
\begin{equation}\label{Eq241_3}
c=\lim_{n\to\infty} \Big\{\log{(1-1/n)} - \sum_{k=2}^n \frac{1}{k(k-1)}\Big\} = - \sum_{k=2}^\infty \frac{1}{k(k-1)}=-1
\end{equation}
Plugging $c=-1$ into Eq \ref{Eq241_2} yields the desired expression for $f'_G(n)$. Therefore, we have:
\begin{equation}\label{Eq241_4}
\frac{f_G'(n)}{f_G(n)}=-1+\sum_{k=2}^n \frac{1}{k(k-1)} =- \sum_{k=n+1}^{\infty} \frac{1}{k(k-1)} = -\frac{1}{n}
\end{equation} 
Eq \ref{Eq241_4} is a first-order linear differential equation whose solution is $f_G(n)=1/n$. Indeed, such generalization of the simple finite product $f(n)$ is quite natural because the product in Eq \ref{Eq241_1} is nothing but a telescoping product that evaluates to $1/n$.

Is it always the case that the unique natural generalization of telescoping sums of the form $f(n)=\sum_{k=a}^n g(k)-g(k+1)$ is simply given by $f_G(n)=g(a)-g(n+1)$? In other words, while we know that $\sum_{k=a}^n g(k)-g(k+1) = g(a)-g(n+1)$ holds if $n-a \in \mathbb{N}$, how do we know that, for example, $\sum_{k=1}^{1/2} g(k)-g(k+1)$  is indeed equal to $g(1)-g(\frac{3}{2})$? This can be heuristically argued using the two defining properties of finite sums as follows: 
\begin{align*}
\sum_{k=a}^n g(k)-g(k+1) &= \sum_{k=a}^{n} g(k) - \sum_{k=a}^n g(k+1) \quad  &\textrm{} \\
&= g(a)+\sum_{k=a+1}^{n} g(k) - \sum_{k=a+1}^{n+1} g(k) &\textrm{(by properties 1 and 2)}\\
&=g(a)-\sum_{k=n+1}^{n+1} g(k) &\textrm{(by properties 2)}\\
&=g(a) - g(n+1) &\textrm{(by property 1)} 
\end{align*}

However, we note here that the above argument does not follow from the two defining properties of simple finite sums because we have assumed that $\sum_{k=a}^n g(k+s) = \sum_{k=a+s}^{n+s} g(k)$. This last property, which has been called the \emph{translation invariance} property \cite{Muller2011}, will be indeed established for the unique most natural generalization of simple finite sums later in Section \ref{sectionSFSGC}. So, indeed the unique natural generalization of telescoping sums of the form $\sum_{k=a}^n g(k)-g(k+1)$ is always given by $g(a)-g(n+1)$ even if $n-a \not\in\mathbb{N}$. Fortunately, this is what we would expect if we were indeed dealing with the unique most natural generalization to simple finite sums.  
\subsection{Example II: The Factorial Function}\label{ExampleFactorial}
For a more contrived example, let us consider the log-factorial function given by $\varpi(n)=\sum_{k=1}^n \log{k}$. We can use Summability Calculus to derive a series expansion of $\varpi(n)$ quite easily as follows. 
First, by direct application of Theorem \ref{theorem231}, we find that the derivative at $n=0$ is given by $\lim_{n\to\infty} \{\log{n}-H_n\} = \lambda$, where $H_n$ is the \emph{n}th harmonic number and $\lambda$ is Euler's constant. Thus, we have by Rule 1:
\begin{equation}\label{Eq242_1}
\varpi'(n)=\psi(n+1)=-\lambda + \sum_{k=1}^n \frac{1}{k}
\end{equation}
Because $1/k$ is also nearly convergent by Definition \ref{definition1} and because it converges to zero, we have by Theorem \ref{theorem231}: 
\begin{equation}\label{Eq242_2}
\varpi^{(2)}(n)=\zeta_2 - \sum_{k=1}^n \frac{1}{k^2}
\end{equation}
Here, $\zeta$ is the Riemann zeta function. Continuing in this manner yields the following series expansion:
\begin{equation}\label{Eq242_3}
\varpi(n)=-\lambda n + \sum_{k=2}^\infty (-1)^k \frac{\zeta_k}{k} n^k
\end{equation}
In particular, since $\log{1}=0$, we have the following well-known identity that was first proved by Euler \cite{Sebha2002}:
\begin{equation}\label{Eq242_4}
\lambda = \sum_{k=2}^\infty (-1)^k \frac{\zeta_k}{k}
\end{equation}

In this example, it is important to keep a few points in mind. First, the series expansion in Eq \ref{Eq242_3} is the series expansion of the unique most natural generalization of the log-factorial function, which turned out to be the series expansion of the log-Gamma function $\log{\Gamma{(n+1)}}$. However, our \emph{apriori} assumption of the generalized definition $f_G(n)$ may or may not be equivalent to what Theorem \ref{theorem231} implies. In other words, it was possible, at least in principle, that Theorem \ref{theorem231} would yield a generalized definition of the log-factorial function that is different from the log-Gamma function so we need to exercise caution before equating the two functions. In this particular example, nevertheless, they both turned out to be the same. Second, using the discussion in Claim \ref{claimLinearFitting}, we deduce from this example that the log-gamma function $\log{\Gamma{(n+1)}}$ is the unique smooth function that satisfies recurrence and initial conditions of the log-factorial function and \emph{its higher derivatives do not alternate in sign infinitely many times}. Lemma \ref{factorial_gamma} presents a more precise statement of the latter conclusion. \\ \hrule

\begin{lemma}\label{factorial_gamma} 
Let $f_G(n) : \mathbb{C} \to\mathbb{C}$ be a function that satisfies the following three properties:
\begin{enumerate}
\item $f_G(n) = \log{n} + f_G(n-1)$,  for all $n$ 
\item $f_G(1)=1$ 
\item  For every higher derivative $f_G^{(r)}(n)$ where $r\ge 0$, there exists a constant $n_r$ such that $f_G^{(r)}(n)$ is monotonic for all $n>n_r$. In other words, $f_G^{(r)}(n)$ does not alternate in sign infinitely many times. 
\end{enumerate}
Then, $f_G(n)=\Gamma{(n+1)}$, where $\Gamma$ is the Gamma function. 
\end{lemma}
\begin{proof}
Similar to the proof of Claim \ref{claimLinearFitting}, let $f_G(n)$ be a function that satisfies the recursive property and initial condition stated in the lemma. Because $\log{n}$ is nearly convergent and because $f_G(n)$ is monotonic for $n>n_0$, then its derivative $f_G'(n)$ converges to $\log{n}$ as $n\to\infty$ (see proof of Claim \ref{claimLinearFitting}). However, this implies that its derivative  $f_G'(0)$, given by the constant $c$ in Rule 1, is uniquely determined by Theorem \ref{theorem231} and it is equal to $-\lambda$. Similarly, since $f_G'(n)=-\lambda + H_n$ as stated in Rule 1 of Table \ref{TableRules}, where the summation is now the Harmonic number and it is nearly convergent, and because $f_G'(n)$ is monotonic for all $n>n_1$, $f_G^{(2)}(0)$ is also uniquely determined and is given by $\zeta_2$, and so on. Thus, all higher derivatives at $n=0$ are uniquely determined and are given by Theorem \ref{theorem231}. Therefore, the only possible series expansion of $f_G(n)$ is the one given by Eq \ref{Eq242_3}, which is the series expansion of $\Gamma{(n+1)}$. 
\end{proof}
\hrule\vspace{12pt} 
Lemma \ref{factorial_gamma} presents an alternative statement, different from the \emph{Bohr-Miller Theorem}, as to why the Gamma function is the unique most natural generalization of the factorial function. According to the lemma, the log-Gamma function is the only possible generalization of the log-factorial function if we require that higher derivatives be \emph{monotonic asymptotically}. However, by requiring such smooth asymptotic behavior of the function, its behavior \emph{for all} $n$ is given by $\log{\Gamma{(n+1)}}$.
Historically, there have been many attempts to extend the factorial function into an entire function to avoid Gamma's singularities at negative integers. One well-known example is Hadamard's factorial function, which correctly interpolates the discrete factorial function. However, Hadamard's function does not satisfy the required recurrence identity for all complex values of $n$ \cite{Davis1959}. In fact, it is straightforward to observe that a function that satisfies the recurrence identity $f_G(n)=n \,f_G(n-1)$ for all $n$ has to have singularities at negative integers, since we must have $1!=1\times 0!$ and $0!=0 \times (-1)!$ by recurrence identity, and so on.

\subsection{Example III: Euler's Constant}\label{ExampleEuler}
Suppose that $f(n)$ is given by Eq \ref{Eq243_1} below, where $\Pi{(n)}$ is again Gauss's PI function, which is equal to $\Gamma{(n+1)}$. In Eq \ref{Eq243_1}, let $x$ be independent of $n$. Then, the derivative $f_G'(0)$ is given in Eq \ref{Eq243_2}, where $\psi$ is again the digamma function.
\begin{equation}\label{Eq243_1}
f(n)=\Pi{(n+x)}/\Pi{(x)} = \prod_{k=1}^n (k+x)
\end{equation}
\begin{equation}\label{Eq243_2}
f_G'(0)=\Pi'{(x)}/\Pi{(x)} =\psi(x+1)
\end{equation}
Upon using both Eq \ref{Eq243_2} and the differentiation rule of simple finite products, i.e. Rule 4, we arrive at Eq \ref{Eq243_3}, which can be rearranged as given in Eq \ref{Eq243_4} 
\begin{equation}\label{Eq243_3}
f_G'(n)=f_G(n) \Big(\sum_{k=1}^n \frac{1}{k+x} + \psi{(x+1)} \Big)
\end{equation}
\begin{equation}\label{Eq243_4}
\frac{d}{dn} \log{f_G(n)}=\sum_{k=1}^n \frac{1}{k+x} + \psi{(x+1)}
\end{equation}
Now, we employ Theorem \ref{theorem231} to deduce Taylor series expansion of $\frac{d}{dn} \log{f_G(n)}$, which yields: \footnote{As defined in the Chapter \ref{ChaptIntro}, $H_{m,n}=\sum_{j=1}^n j^{-m}$}
\begin{equation}\label{Eq243_5}
\frac{d}{dn} \log{f_G(n)}=\psi{(x+1)} + \sum_{k=2}^\infty (-1)^k \big(\zeta_k - H_{k, x}\big) \,n^{k-1}
\end{equation}
Now, we integrate both sides with respect to $n$, which yields: 
\begin{equation}\label{Eq243_6}
\log{f_G(n)}=\psi{(x+1)}\, n + \sum_{k=2}^\infty \frac{(-1)^k}{k} \big(\zeta_k - H_{k, x}\big) \,n^k
\end{equation}
Since we derived earlier in Eq \ref{Eq242_1} the identity $\psi(x+1)=-\lambda + H_x$, we plug it into previous equation and set $n=1$ to deduce that: 
\begin{equation}\label{Eq243_7}
\log{(1+x)}=-\big(\lambda - H_x\big) +\frac{\zeta_2 - H_{2,x}}{2} -\frac{\zeta_3 - H_{3,x}}{3} \dotsm  
\end{equation}
Thus, by taking the limit as $x\to\infty$, we recover Euler's famous result again:
\begin{equation}\label{Eq243_8}
\lim_{x\to\infty} \Big\{\log{(1+x)} - \lambda - H_x\Big\} = 0
\end{equation}
Of course, we could alternatively arrive at the same result by direct application of Lemma \ref{lem21}. In addition, if we set $x=0$ in Eq \ref{Eq243_7} and use the empty sum rule, i.e. Rule 3, we arrive at Eq \ref{Eq242_3} again. However, Eq \ref{Eq243_7} is obviously a more general statement. For instance, by setting $x=1$, we arrive at the identity in Eq \ref{Eq243_9}. Additionally, by plugging Eq \ref{Eq242_3} into Eq \ref{Eq243_7}, we arrive at Eq \ref{Eq243_10}. Again, the empty sum rule yields consistent results as expected. 
\begin{equation}\label{Eq243_9}
\log{2} = 1-\lambda +\sum_{k=2}^\infty (-1)^k \frac{\zeta_k-1}{k}
\end{equation}
\begin{equation}\label{Eq243_10}
\log{(1+x)} = \sum_{k=1}^\infty (-1)^{k+1} \frac{H_{k, x}}{k}
\end{equation}
Moreover, we rewrite Eq \ref{Eq243_7} to have: 
\begin{equation}\label{Eq243_11}
\log{(1+x)} =-\lambda + H_x + \frac{1}{2} \sum_{k=x+1}^\infty 1/k^2 - \frac{1}{3} \sum_{k=x+1}^\infty 1/k^3 +\dotsm
\end{equation}
Using the series expansion of $\log{(1+x)}$ for $|x|<1$, we note that Eq \ref{Eq243_11} can be rearranged as follows: 
\begin{equation}\label{Eq243_12}
\log{(1+x)} =-\lambda + H_x + \sum_{k=x+1}^\infty \Big(\frac{1}{k} - \log{(1+\frac{1}{k})}\Big)
\end{equation}
Of course, setting $x=0$ yields Euler's famous identity \cite{MathWorldEulerConstant, Sebha2002}: 
\begin{equation}\label{Eq243_13}
\lambda = \sum_{k=1}^\infty \Big(\frac{1}{k} - \log{(1+\frac{1}{k})}\Big)
\end{equation}

Finally, starting from the Taylor series expansion in Eq \ref{Eq243_6} and setting $x=0$, we have the following identity: 
\begin{equation} \label{Eq243_14}
-\log{\Pi(1/2)} +\log{\Pi(-1/2)}  = \lambda + \sum_{k=1}^\infty \frac{\zeta_{2k+1}}{4^k (2k+1)} = \log{2}
\end{equation}

In this example, we have reproduced some of the most important results related to Euler's constant. The fact that all of these results were proven in only a few pages is an example of how convenient Summability Calculus really is.

\subsection{Example IV: More into the Zeta Function} \label{ExampleZeta}
In this example, we continue our quest to reproduce some of the basic identities related to the Riemann zeta function. For a starting point, we compute the series expansion of the digamma function $\varpi'(n)$, but this time we take the series expansion around $n=1$. To do this, we note from Eq \ref{Eq242_1} that $\varpi'(1)=1-\lambda$. Using Rule 1 and upon successive differentiation of the digamma function using Theorem \ref{theorem231} we have: 
\begin{equation}\label{Eq244_1} 
\varpi'(n) = (1-\lambda)+ \sum_{k=1}^\infty (-1)^{k+1} \big(\zeta_{k+1}-1\big) (n-1)^k 
\end{equation}
Now, if we set $n=2$ in the previous equation, we obtain:
\begin{equation}\label{Eq244_2} 
\sum_{k=2}^\infty (-1)^{k} \big(\zeta_{k}-1\big) =\frac{1}{2} 
\end{equation}
Additionally, if we set $n=0$ in Eq \ref{Eq244_1}, we obtain:
\begin{equation}\label{Eq244_3} 
\sum_{k=2}^\infty \big(\zeta_{k}-1\big) =1
\end{equation}
Of course, we could combine the previous three equations in various ways to deduce additional identities such as:
\begin{equation}\label{Eq244_4} 
\sum_{k=1}^\infty \big(\zeta_{2k}-1\big) =1-\frac{\lambda}{2}
\end{equation}
\begin{equation}\label{Eq244_5} 
\sum_{k=1}^\infty \big(\zeta_{2k+1}-1\big) =\frac{\lambda}{2}
\end{equation}

On the other hand, if we integrate both sides of Eq \ref{Eq244_1} with respect to $n$ and use the fact that $\log{1}=0$, we have:
\begin{equation}\label{Eq244_6} 
\varpi(n)=\log{n!} = (1-\lambda)(n-1)+ \sum_{k=2}^\infty (-1)^k \frac{\zeta_k-1}{k} (n-1)^k
\end{equation}
Setting $n=0$ yields: 
\begin{equation}\label{Eq244_7} 
 \sum_{k=2}^\infty  \frac{\zeta_k-1}{k} =1-\lambda
\end{equation}
On the other hand, if set $n=2$ in Eq \ref{Eq244_6}, we arrive at Eq \ref{Eq243_9} again. Finally, similar techniques can be used to show numerous other results including:
\begin{equation}\label{Eq244_8}
\log{2} = \sum_{k=1}^\infty \frac{\zeta_{2k}-1}{k}
\end{equation}
\begin{equation}\label{Eq244_9}
1-\lambda -\frac{\log{2}}{2} =  \sum_{k=1}^\infty \frac{\zeta_{2k+1}-1}{2k+1}
\end{equation}

It is worth noting that while the method presented herein provides a simple venue for deducing such rich set of identities, some of these identities were proven in the past at different periods of time by different mathematicians using different approaches (see for instance the surveys in \cite{Gourdon2004, Sebha2002}).

\subsection{Example V: One More Function}\label{ExampleOneMoreFunction}
In this example, we use Theorem \ref{theorem231} to derive the Taylor series expansion of the function $f(n)=\sum_{k=1}^n \frac{\log{k}}{k}$. First, we note that the derivative is given by \footnote{Here, $\zeta'_s=-\sum_{k=1}^\infty \frac{\log{k}}{k^s}$.}:
\begin{equation}\label{Eq245_1}
f_G'(n) = -\zeta_2 -\zeta'_2 + \sum_{k=1}^n \frac{1-\log{k}}{k^2}
\end{equation}
Thus, we know that $f_G'(1) = 1-\zeta_2 -\zeta'_2$. Similarly, we have by Theorem \ref{theorem231}: 
\begin{equation}\label{Eq245_2}
f_G^{(2)}(n) = 3\zeta_3 +2\zeta'_3 - \sum_{k=1}^n \frac{3-2\log{k}}{k^2}
\end{equation}
Thus, we have $f_G^{(2)}(1) = 3(\zeta_3-1) +2\zeta'_3$. Continuing in this manner, it is straightforward to prove by induction that $f_G^{(r)}(1) = (-1)^r r!\big(H_r(\zeta_{r+1}-1) +\zeta'_{r+1}\big)$. Therefore, the desired Taylor series expansion is given by: 
\begin{equation}\label{Eq245_3} 
\sum_{k=1}^n \frac{\log{k}}{k}=\sum_{r=1}^\infty (-1)^r \big(H_r(\zeta_{r+1}-1) +\zeta'_{r+1}\big) \big(n-1\big)^r
\end{equation}
Setting $n=0$ and $n=2$ yields the two aesthetic identities Eq \ref{Eq245_4} and Eq \ref{Eq245_5} respectively. 
\begin{equation}\label{Eq245_4} 
\sum_{r=1}^\infty \big(H_r(\zeta_{r+1}-1) +\zeta'_{r+1}\big) = 0
\end{equation}
\begin{equation}\label{Eq245_5} 
\sum_{r=1}^\infty (-1)^r \big(H_r(\zeta_{r+1}-1) +\zeta'_{r+1}\big) = \frac{\log{2}}{2}
\end{equation}

\subsection{Example VI: Asymptotic Expressions} \label{ExampleAsymptotic}
Finally, we conclude this section with two simple proofs to the asymptotic behaviors of the factorial and the hyperfactorial functions, known as Stirling and Glaisher approximations respectively that were mentioned earlier in Eq \ref{StirlingAppx} and Eq \ref{hyperfAppx}. Our starting point will be Eq \ref{Eq243_7}. Replacing $x$ with $n$ and integrating both sides using Rule 2 yields Eq \ref{Eq246_1}. Note that in this particular step, we have employed Summability Calculus \emph{in integrating} with respect to the bounds of finite sums.
\begin{equation}\label{Eq246_1}
(1+n)\log{(1+n)} = c_1 + c_2 n + \log{n!} +\frac{H_n}{2} -\frac{H_{2,n}}{6} +\frac{H_{3,n}}{12} \dotsm   
\end{equation}
Setting $n=0$ and $n=1$ yields $c_1=0$ and $c_2=1$ respectively. Thus, we have: 
\begin{equation}\label{Eq246_2}
(1+n)\log{(1+n)} - n - \log{n!} = \frac{H_n}{2} -\frac{H_{2,n}}{6} +\frac{H_{3,n}}{12} \dotsm   
\end{equation}
Now, knowing that $H_n\sim\log{n} +\lambda$, we have the asymptotic expression for the log-factorial function given in Eq \ref{Eq246_3}, where the error term goes to zero asymptotically. 
\begin{equation}\label{Eq246_3}
 \log{n!}\sim (1+n)\log{(1+n)} - n -\frac{\log{n}}{2} -\frac{\lambda}{2} +\frac{\zeta_2}{6} -\frac{\zeta_3}{12} + \dotsm   
\end{equation}
However, upon using the well-known identity in Eq \ref{Eq246_4}, we arrive at the asymptotic formula for the factorial function given in Eq \ref{Eq246_5}. Clearly, it is straightforward to arrive at Stirling approximation from  Eq \ref{Eq246_5} but the asymptotic expression in  Eq \ref{Eq246_5} is more accurate.
\begin{equation}\label{Eq246_4}
-\frac{\lambda}{2} +\frac{\zeta_2}{6} -\frac{\zeta_3}{12} + \dotsm = \frac{\log{2\pi}}{2} -1   
\end{equation}
\begin{equation}\label{Eq246_5}
n!\sim \frac{1}{e} \sqrt{2\pi(n+1)} \,\Big(\frac{n+1}{e}\Big)^n
\end{equation}

On the other hand, if we start with the log-hyperfactorial function $f(n)=\log{H(n)}=\sum_{k=1}^n k\log{k}$, and differentiate using Rule 1, we obtain: 
\begin{equation}\label{Eq246_6}
\frac{H'(n)}{H(n)} = n + \log{n!} + c
\end{equation}
Unfortunately, $\log{H(n)}$ is not semi-linear so we cannot use Theorem \ref{theorem231} to find the exact value of $c$. These general cases will be addressed in the following section. Nevertheless, since we have the series expansion of the log-factorial function, we can substitute it in the previous equation and integrate with respect to $n$ to arrive at: 
\begin{equation}\label{Eq246_7}
\log{H(n)}= \frac{1}{4} + \frac{(1+n)^2 \big(2\log{(1+n)} -1\big)}{2} - \frac{n+\log{n!}}{2} - \frac{H_n}{6} + \frac{H_{2,n}}{24} - \frac{H_{3,n}}{60} \dotsm  
\end{equation}

Here, we used the facts that $\log{H(1)}=0$ and $\log{H(0)}=0$ by definition of the hyperfactorial function and empty product rule respectively. Thus, an asymptotic expression for the hyperfactorial function is given by: 
\begin{equation}\label{Eq246_8}
H(n)\sim \frac{(1+n)^{(1+n)^2/2} \,\, e^K} {e^{n+n^2/4} \,\, n^{1/6} \sqrt{n!}},\\ K=-\frac{\lambda}{6} +\frac{\zeta_2}{24} -\frac{\zeta_3}{60} +\frac{\zeta_4}{120} -\dotsm
\end{equation}
Again, it is straightforward to proceed from the asymptotic expression in Eq \ref{Eq246_8} to arrive at Glaisher's approximation formula Eq \ref{hyperfAppx} through simple algebraic manipulations but Eq \ref{Eq246_8} is much more accurate.

\section{Simple Finite Sums: The General Case} \label{sectionSFSGC}
In Section \ref{Section_23}, we established Summability Calculus for semi-linear simple finite sums. Needless to mention, many important finite sums and products that arise frequently are not semi-linear including, for example, the log-hyperfactorial function discussed earlier. As illustrated in Section \ref{ExampleAsymptotic}, it is possible to perform calculus on simple finite sums that are not semi-linear by reducing analysis to the case of semi-linear simple finite sums using a few tricks. For example, the log-hyperfactorial function is not semi-linear but its derivative is so we could compute the series expansion of its derivative and integrate subsequently using the rules of Summability Calculus. In principle, therefore, as long as some higher derivative of a simple finite sum is semi-linear, we could perform infinitesimal calculus, albeit using quite an involved approach. Obviously, this is tedious and far from being insightful. 

In this section we extend the results of Section \ref{Section_23} to the general case of simple finite sums and products, which brings to light the celebrated \emph{Euler-Maclurin summation formula}. We will show how Summability Calculus yields unique natural generalizations and present a simple proof of the Euler-Maclurin summation formula by construction. 

Our starting point will be Theorem \ref{StatofUniq}. Here, the successive polynomial approximation method suggests that there exists a sequence of constants $a_r$ such that Eq \ref{Eq25_0} holds formally for all functions $g(k)$. We will assume that this holds for some constants $a_r$ and our first goal is to determine what those constants must be. 
\begin{equation}\label{Eq25_0}
\frac{d}{dn} \sum_{k=a}^n g(k) =\sum_{r=0}^\infty a_r g^{(r)}(n)
\end{equation}
Looking into Table \ref{TableRules}, then Eq \ref{Eq25_0} and Rule 1 both imply that: 
\begin{equation}\label{Eq25_1}
\frac{d^m}{dn^m} \sum_{k=a}^n g(k) =\sum_{r=0}^\infty a_r g^{(r+m-1)}(n)
\end{equation}
However, we know that the following holds for some constants $c_m$.  
\begin{equation}\label{Eq25_2}
\frac{d^m}{dn^m} \sum_{k=a}^n g(k) = \sum_{k=a}^n g^{(m)}(k) + c_m
\end{equation}

Therefore, we can determine value of the constants $c_m$ using: 
\begin{equation}\label{Eq25_3}
c_m = \sum_{r=0}^\infty a_r g^{(r+m-1)}(n) - \sum_{k=a}^n g^{(m)}(k)
\end{equation}

Of course, Eq \ref{Eq25_3} is only a formal expression; the infinite sum $ \sum_{r=0}^\infty a_r g^{(r+m-1)}(n)$ may or may not converge. However, the issue of convergence is not important at this stage because there exists a \emph{unique} sequence $a_r$ that makes the expression in Eq \ref{Eq25_3} hold formally if recurrence identity is to be maintained. In addition, Theorem \ref{StatofUniq} already presents a statement on how to interpret the infinite sum if it diverges, which will be discussed later. \\ \hrule

\begin{theorem} \label{theorem251} 
Let a simple finite sum be given by \simplefinitesum, where $g(k)$ is analytic in an open disc around each point in the domain $[a, \infty)$, and let $f_G(n)$ be given formally by the following series expansion around $n=a-1$, where $B_k$ are Bernoulli numbers:
\begin{equation}\label{theorem251Eq}
f_G(n) = \sum_{k=m}^\infty \frac{c_m}{m!} (n-a+1)^m, \\c_m = \sum_{r=0}^\infty \frac{B_r}{r!} g^{(r+m-1)}(n) - \sum_{k=a}^n g^{(m)}(k)
\end{equation}
Here, $c_m$ are constants that are independent of $n$. Then, $f_G(n)$ satisfies formally the recurrence identity and initial conditions given by Eq \ref{Eq211}. Thus, $f_G(n)$ is a valid generalization of $f(n)$. In fact, $f_G(n)$ is the unique most natural generalization to the simple finite sum.
\end{theorem} 
\begin{proof}
Again, our first point of departure is to question what natural generalization means in the context of extending domains of discrete functions. In the previous case where we had nearly convergent functions, such question was answered by reducing analysis to the case of \emph{linear fitting} since all semi-linear finite sums are asymptotically linear over arbitrary interval lengths as shown in Eq \ref{Theorem231Eq8}. However, this is no longer valid in the general case of simple finite sums and products. So, how do we define it? 

We will establish in Section \ref{Section6dot2} that the Euler-Maclaurin summation formula is the unique natural generalization by showing that it arises out of polynomial fitting. Similar to the argument of linear fitting for semi-linear simple finite sums, polynomial fitting is arguably the unique most natural generalization if it can correctly interpolate data points and given lack of additional special requirements because polynomials are the simplest of all possible functions. In addition, we will show that the Euler-Maclaurin summation formula is indeed the unique natural generalization as dictated by the successive polynomial approximation method of Theorem \ref{StatofUniq}.

Similar to the proof of Theorem \ref{theorem231}, let us express $f_G(n)-f_G(n-1)$ as a Taylor series expansion. However, this time we will take the series expansion around $n=a-1$ starting from Eq \ref{theorem251Eq}. Because this is identical to the proof steps of Theorem \ref{theorem231}, we will omit some details and jump to the important result that Eq \ref{theorem251_1} must hold for all $m\ge 0$ in order to satisfy required recurrence identity.
\begin{equation}\label{theorem251_1} 
K_m= g^{(m)}(a-1), \\ \text{where}\ K_m =\lim_{\delta \to 0} \bigl\{ \sum_{j=1}^\infty \xi_\delta(j) (-1)^{(j+1)}\frac{c_{j+m}}{j!}\bigr\}\\
\end{equation}
Similar to the proof of Theorem \ref{theorem231}, Eq \ref{theorem251_1} is a functional identity; if it holds for $m=0$, then it holds for all $m\ge 0$. So, we will only prove that it holds for $m=0$. As stated in Eq \ref{Eq25_3}, we have for some constants $a_r$: 
\begin{equation}\label{theorem251_2}
c_x = \sum_{r=0}^\infty a_r g^{(r+x-1)}(n) - \sum_{k=a}^n g^{(x)}(k)
\end{equation}
Now, we will prove that $a_r=\frac{B_r}{r!}$ by construction. Plugging Eq \ref{theorem251_2} into Eq \ref{theorem251_1} for $m=0$ yields:
\begin{equation}\label{theorem251_3}
g(a-1)=\lim_{\delta \to 0} \bigl\{ \sum_{j=1}^\infty \xi_\delta(j) \frac{(-1)^{(j+1)}}{j!}\big(  \sum_{r=0}^\infty a_r g^{(r+j-1)}(n) - \sum_{k=a}^n g^{(j)}(k)\big)\bigr\}\\
\end{equation}
By linearity of the summability method $\mathbb{T}$, we have: 
\begin{align*}
g(a-1)=&\lim_{\delta \to 0} \Big\{\frac{\xi_\delta(1)}{1!} \sum_{k=0}^\infty a_k g^{(k)}(n) - \frac{\xi_\delta(2)}{2!} \sum_{k=0}^\infty a_k g^{(k+1)}(n) + \dotsm\Big\}\\
-&\lim_{\delta \to 0} \Big\{\frac{\xi_\delta(1)}{1!}\sum_{k=a}^n g'(k) - \frac{\xi_\delta(2)}{2!}\sum_{k=a}^n g^{(2)}(k) +\frac{\xi_\delta(3)}{3!}\sum_{k=a}^n g^{(3)}(k) - \dotsm\Big\} 
\end{align*}
Because the summability method correctly evaluates Taylor series expansions under stated conditions (see Theorem \ref{theorem231}), we have:
\begin{equation}\label{theorem251_4}
\lim_{\delta \to 0} \Big\{\frac{\xi_\delta(1)}{1!} g'(s) - \frac{\xi_\delta(2)}{2!} g^{(2)}(s) +\frac{\xi_\delta(3)}{3!}g^{(3)}(s) - \dotsm\Big\} =  g(s-1)-g(s)
\end{equation}
Therefore, 
\begin{equation}\label{theorem251_5}
\lim_{\delta \to 0} \Big\{\frac{\xi_\delta(1)}{1!}\sum_{k=a}^n g'(k) - \frac{\xi_\delta(2)}{2!}\sum_{k=a}^n g^{(2)}(k) +\frac{\xi_\delta(3)}{3!}\sum_{k=a}^n g^{(3)}(k) - \dotsm\Big\} = g(a-1)-g(n) 
\end{equation}
Consequently, we have: 
\begin{equation}\label{theorem251_6}
g(n)=\lim_{\delta \to 0} \Big\{\frac{\xi_\delta(1)}{1!} \sum_{k=0}^\infty a_k g^{(k)}(n) - \frac{\xi_\delta(2)}{2!} \sum_{k=0}^\infty a_k g^{(k+1)}(n) + \dotsm\Big\}
\end{equation}
Because $\mathbb{T}$ correctly sums Taylor series expansions under stated conditions, we must have $\lim_{\delta\to 0} \xi_\delta{(j)} = 1$ for any fixed $j<\infty$. For example, the Lindel\"of summation method has $\xi_\delta(j) = j^{-\delta j}$ so the statement is satisfied. Therefore, we arrive at the following  expression: 
\begin{equation}\label{theorem251_7}
g(n)=\frac{a_0}{1!} g(n) + \big(\frac{a_1}{1!}-\frac{a_0}{2!}\big) g'(n) + \big(\frac{a_2}{1!}-\frac{a_1}{2!}+\frac{a_0}{3!}\big) g^{(2)}(n) + \dotsm
\end{equation}
Consequently, the constants $a_r$ satisfy the following conditions if and only if the recurrence identity holds:

$a_0$=1 \quad\quad\quad and \quad \quad $\frac{a_n}{1!}-\frac{a_{n-1}}{2!}+\frac{a_{n-3}}{3!} -\dotsm + (-1)^{(n+1)} \frac{a_0}{(n+1)!} = 0$

The above conditions yield $a_r=\frac{B_r}{r!}$. Thus, the generalized function $f_G(n)$ given by Theorem \ref{theorem251} does indeed satisfy recurrence and initial conditions. The argument that it arises out of polynomial fitting has to await development of Chapter \ref{Chapter6}. In addition, the argument that it is identical to the unique natural generalization given by the successive polynomial approximation method of Theorem \ref{StatofUniq} is easy to establish. First, we note in Eq \ref{Eq21_13} that the constants $b_r=\{1, -1/2, 1/12 , \ldots\}$ in the successive polynomial approximation method must satisfy $\sum_{r=0}^n \frac{b_{n-r}}{r!} = 0$ and $b_0=1$ in order to fulfill boundary conditions. However, these conditions have the unique solution given by $b_r=(-1)^r \frac{B_r}{r!}$. Therefore, the unique natural generalization given by the successive polynomial approximation method of Theorem \ref{StatofUniq} satisfies: 
\begin{equation}\label{theorem251_final} 
f_G'(n) = \sum_{r=0}^\infty (-1)^r \frac{B_r}{r!} g^{(r)}(n+1) 
\end{equation}
Now, we will show that this is identical to the statement of Theorem \ref{theorem251}. Using the Euler-Maclaurin summation formula (see Eq \ref{corollary251_2}), it can be easily shown that the following holds formally:
\begin{equation}\label{theorem251_final_2} 
\sum_{r=0}^\infty (-1)^r \frac{B_r}{r!} g^{(r)}(n+1) = \sum_{r=0}^\infty \frac{B_r}{r!} g^{(r)}(n)
\end{equation}
Therefore, the unique natural generalization given by Theorem \ref{theorem251} and the unique natural generalization of Theorem \ref{StatofUniq} are both identical, which completes proof of this theorem. Because earlier results for the case of semi-linear simple finite sums are just special cases of the more general statement in Theorem \ref{theorem251}, the unique natural generalization of Theorem \ref{theorem231} is also identical to both Theorem \ref{StatofUniq} and Theorem \ref{theorem251}. 

Finally, the proof assumes that $g(k)$ is analytic in the domain $[a-1, \infty)$. This is merely a technical detail in the above proof. Alternative proofs can be made using properties of Bernoulli polynomials, which require that $g(k)$ be analytic in the domain $[a, \infty)$ only (see for instance an elegant proof to the Euler-Maclaurin summation formula in \cite{Apostol}). 
\end{proof}  \hrule\vspace{6pt} 
\begin{corollary}\label{corollary251_1} 
Let $f(n)$ be a simple finite sum given by \simplefinitesum, where $g(k)$ is analytic at an open disc around the point $k=n$, and let $f_G(n)$ be its unique natural generalization, then higher derivatives $f_G^{(m)}(n)$ for $m>0$ are given formally by:
\begin{equation}\label{corollary251_1Eq}
f_G^{(m)}(n) = \sum_{r=0}^\infty \frac{B_r}{r!} g^{(r+m-1)} (n)
\end{equation}
\end{corollary}

\begin{proof}
By Theorem \ref{theorem251} and the differentiation rule of simple finite sums, i.e. Rule 1 in Table \ref{TableRules}. 
\end{proof}  \hrule\vspace{6pt}

\begin{corollary}{\textbf{(The Euler-Maclaurin Summation Formula)}}\label{corollary251_2} 
Let $f(n)$ be a simple finite sum given by \simplefinitesum, where $g(k)$ is analytic in the domain $[a, n]$, and let $f_G(n)$ be its unique natural generalization, then $f_G(n)$ is given formally by:
\begin{equation}\label{corollary251_2Eq}
f_G(n) =\int_a^n g(t)\,dt  + \frac{g(n)+g(a)}{2} + \sum_{r=2}^\infty \frac{B_r}{r!} \big(g^{(r-1)} (n) -g^{(r-1)} (a)\big) 
\end{equation}
\end{corollary}
 
\begin{proof}
By integrating both sides of Eq \ref{corollary251_1Eq}, we have: 
\begin{align}
f_G(n)-f_G(a) &= \int_a^n   \sum_{r=0}^\infty \frac{B_r}{r!} g^{(r)} (n)\\
&=\int_a^n g(t)\,dt  + \frac{g(n)-g(a)}{2} + \sum_{r=2}^\infty \frac{B_r}{r!} \big(g^{(r-1)} (n) -g^{(r-1)} (a)\big) 
\end{align}
Substituting $f_G(a)=g(a)$ yields the desired result. 
\end{proof}  \hrule\vspace{6pt} 

\begin{corollary}{\textbf{(The Translation Invariance Property)}}\label{Corollaryshiftproperty} 
If g(k) is analytic in an open disc around each point in the domain $[a, \infty)$, then for any constant $\mu$, we have:
\begin{equation}\label{Corollaryshiftproperty_1}
\sum_{k=a}^n g(k+\mu) = \sum_{k+\mu}^{n+\mu} g(k)
\end{equation}
\end{corollary}
 
\begin{proof}
By direct application of the Euler-Maclaurin summation formula. 
\end{proof}  \hrule\vspace{6pt} 

\begin{corollary}\label{UniqGenTwoProperties} 
The unique natural generalization of simple finite sums $\sum_{k=a}^n g(k)$ satisfies the two defining properties of simple finite sums.
\end{corollary}
 
\begin{proof}
By direct application of the Euler-Maclaurin summation formula. 
\end{proof}  \hrule\vspace{6pt} 

\begin{corollary}{\textbf{(Evaluating the Euler-Maclaurin Summation Formula)}}\label{corollarySumEMSum} 
Let $f(n)$ be a simple finite sum given by \simplefinitesum, where g(k) is analytic in an open disc around each point in the domain $[a, \infty)$, and let $f_G(n)$ be its unique natural generalization, then $f_G(n)$ can be evaluated using:
\begin{equation}\label{corollarySumEMSum_1}
f_G(n) = g(a)+\lim_{z\to\infty} \sum_{k=1}^z \frac{(n-a)^k}{k!} \sum_{r=0}^{z-k} (-1)^r B_r \frac{g^{(k+r-1)}(a+1)}{r!} 
\end{equation}
In particular, if $g(k)$ is analytic at an open disc around $n$, then $f_G'(n)$ can be evaluated using:
\begin{equation}\label{corollarySumEMSum_2}
f_G'(n) = \lim_{z\to\infty} \sum_{k=1}^z \frac{(n-a)^{(k-1)}}{(k-1)!} \sum_{r=0}^{z-k} (-1)^r B_r \frac{g^{(k+r-1)}(a+1)}{r!} 
\end{equation}
\end{corollary}
 
\begin{proof}
In the proof of Theorem \ref{StatofUniq}, we showed that if $f(n)=\sum_{k=0}^n g(k)$, then: 
\begin{equation}\label{corollarySumEMSum_3}
f_G(n) = g(0)+\lim_{z\to\infty} \sum_{k=0}^z \frac{n^k}{k!} \sum_{r=0}^{z-k} (-1)^{r} B_r \frac{g^{(k+r-1)}(1)}{r!} 
\end{equation}
This last expression was the unique natural generalization given by Theorem \ref{StatofUniq}. However, by the translation invariance property (Corollary \ref{Corollaryshiftproperty}), we have: 
\begin{equation}\label{corollarySumEMSum_4}
\sum_{k=a}^n g(k) = \sum_{k=0}^{n-a} g(k+a) 
\end{equation}
Applying Eq \ref{corollarySumEMSum_3} to the last equation yields the desired result. Note that, unlike the classical Euler-Maclaurin summation formula which diverges almost all the time, the method presented in Eq \ref{corollarySumEMSum_3} is computable if $|n-a|<1$.  
\end{proof}  \hrule\vspace{6pt} 

It is straightforward to observe that Theorem \ref{theorem251} generalizes Theorem \ref{theorem231} that was proved earlier for semi-linear finite sums. Its proof resembles the approach used in proving Theorem \ref{theorem231} and Theorem \ref{theorem231} is clearly a special case of Theorem \ref{theorem251}. So, what does Theorem \ref{theorem251} imply? 

If we look into the differentiation rule of simple finite sums, i.e. Rule 1 in Table \ref{TableRules}, we realize that the constant $c$ that is independent of $n$ can be determined once we know, at least, one value of the derivative of the generalized definition $f_G(n)$. This is because we always have $c = f'_G(n)-\sum_{k=a}^n g'(k)$ for all $n$. In the case of semi-linear simple finite sums, in which $g'(k)$ vanishes asymptotically, the rate of change of the simple finite sum becomes almost linear asymptotically, so we know that the derivative had to get arbitrarily close to $g(n)$ as $n\to\infty$. This was, in fact, the value we were looking for so we used it to determine what $c$ was. Theorem \ref{theorem231} establishes that such approach is correct, meaning that the resulting function exists and satisfies required conditions.

In Theorem \ref{theorem251}, on the other hand, we ask about the general case of simple finite sums in which $g(k)$ is not necessarily semi-linear. In such case, we are still interested in finding at least one value of each higher order derivative $g^{(m)}$ so that the constants $c_m$ in Eq \ref{Eq25_2} can be determined. Corollary \ref{corollary251_1} provides us with a formal answer. However, because the Euler-Maclaurin summation formula typically diverges, we still need a way of interpreting such divergent series. In Chapters \ref{Chapter4} and \ref{Chapter6}, we will present a new summability method $\Xi$, which works reasonably well in estimating the Euler-Maclaurin summation formula even if it diverges. Alternatively, the evaluation method of Corollary \ref{corollarySumEMSum} can be used as well, especially if $|n-a|<1$. For example, if we wish to find the derivative of the log-factorial function $\varpi(n)=\log{n!}$ at $n=0$, results of the evaluation method in Corollary \ref{corollarySumEMSum} for different values of $z$ are depicted in Figure \ref{figureEMEval}. Clearly, it approaches $-\lambda$ as expected. 
\begin{figure} [h]
\centering
\includegraphics[scale=0.4]{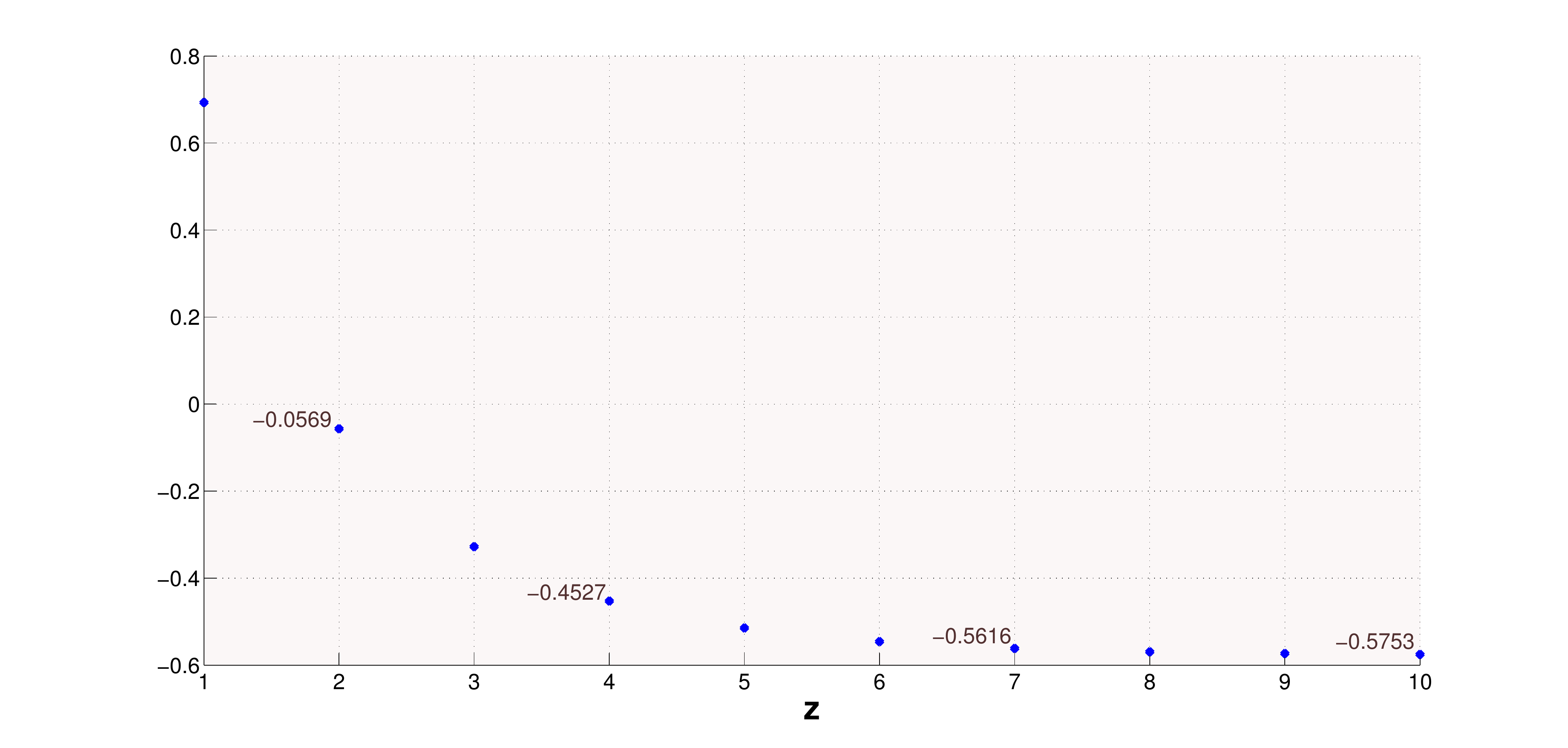}
\caption[Interpreting the divergent Euler-Maclaurin summation formula]{Numerical result of computing $\frac{d}{dn}\log{n!}\big|_{n=0}$ using Corollary \ref{corollarySumEMSum}.}
\label{figureEMEval}
\end{figure}

In the meantime, we note that the Euler-Maclaurin formula is an asymptotic expression; we can clip the infinite sum if $g(k)$ is asymptotically of a finite differentiation order. 
For example, suppose we have \simplefinitesum\ and that some higher derivative $g^{(m+1)}(n)$ vanishes asymptotically as $n\to\infty$, i.e. not necessarily 1\textsuperscript{st} derivative as in the case of semi-linear finite sums. With this in mind, can we deduce an asymptotic expression for the derivative of such simple finite sum $f'_G(n)$? The answer is given by Corollary \ref{corollary251_1}, which states that the derivative approaches $ \sum_{r=0}^m \frac{B_r}{r!} g^{(r)} (n)$. Note that because $g^{(m+1)}(n)$ vanishes asymptotically, we do \emph{not} have to evaluate the entire Euler-Maclaurin summation formula. The last finite sum is the expression we needed in order to compute the value of $c$, in a manner that is similar to what was done earlier for semi-linear simple finite sums\footnote{This will be established in Chapter \ref{Chapter6}.}. Thus, we have: 
\begin{equation}\label{corollary251_3Eq}
f^{(l)}_G(a-1) =\lim_{n\to\infty} \big\{\sum_{r=0}^m \frac{B_r}{r!} g^{(r+l-1)} (n) - \sum_{k=a}^n g^{(l)}(k)\big\}
\end{equation}
Of course, one notable difference between Theorem \ref{theorem251} and Theorem \ref{theorem231}, however, is that we could find the constants $c_m$ in Theorem \ref{theorem251} by choosing \emph{any appropriate value of} $n$ whereas we had to take the limit as $n\to\infty$ in Theorem \ref{theorem231} and \ref{theorem251}. Nevertheless, Eq \ref{corollary251_3Eq} is usually more useful in practice. We will illustrate both approaches in Section \ref{Section2dot6}.

Finally, we conclude this section with a sanity check. Earlier in the foundational rules, we stated that if \simplefinitesum, then we have by Rule 1: $f'_G(n) = \sum_{k=a}^n g'(k) + f_G'(a-1)$. To verify that results so far are consistent with such claim, we note that Corollary \ref{corollary251_1} implies that Eq \ref{bernoulliNumbers_7}  holds if $g^{(r)}(a-1)$ is defined for all $r \ge 0$. 
\begin{equation}\label{bernoulliNumbers_7}
\sum_{k=a}^n g'(k) = \sum_{r=0}^\infty \frac{B_r}{r!} \big(g^{(r)}(n)-g^{(r)}(a-1)\big) 
\end{equation}
Rearranging the terms in Eq \ref{bernoulliNumbers_7}, we have:
\begin{equation}\label{bernoulliNumbers_8} 
 \sum_{r=0}^\infty \frac{B_r}{r!} g^{(r)}(n) = \sum_{k=a}^n g'(k)  +  \sum_{r=0}^\infty \frac{B_r}{r!} g^{(r)}(a-1)
\end{equation}

However, applying Corollary \ref{corollary251_1} to Eq \ref{bernoulliNumbers_8} indeed recovers the statement that $f'_G(n) = \sum_{k=a}^n g'(k) + f_G'(a-1)$.

\section{Examples to the General Case of Simple Finite Sums} \label{Section2dot6}
In this section, we present a few examples for performing infinitesimal calculus on simple finite sums and products using Theorem \ref{theorem251} and its corollaries. 
\subsection{Example I: Extending the Bernoulli-Faulhaber formula} \label{Example261}
In this example, we return to the power sum function given by Eq \ref{PSDisctrete}. If $m$ is a positive integer, we can use Theorem \ref{theorem251} to deduce immediately the Bernoulli-Faulhaber formula for power sums, which is the original proof that Euler provided. However, the Bernoulli-Faulhaber formula does not provide us with a series expansion if $m$ is an arbitrary positive real number. Using Theorem \ref{theorem251}, on the other hand, we can quite easily derive a series expansion of the power sum function $f(n)=\sum_{k=1}^n k^m$ around $n=0$ that generalizes the above formula.

First, we note that $f(0)=0$ by the empty sum rule, and $g^{(r)}(n)=\frac{m!}{(m-r)!} n^{m-r}$. Since $\lim_{n\to\infty} g^{(r)}(n) = 0$ for $r > m$, we can choose $n\to\infty$ in Theorem \ref{theorem251} to find the value of derivatives $c_m = f^{(m)}_G(0)$. In this example, we will derive the series expansion if $m=1/2$. First, we note:
\begin{equation}\label{Example_261_1} 
c_1 = \lim_{n\to\infty} \Big\{\sqrt{n} - \sum_{k=1}^n \frac{1}{2\sqrt{k}}\Big\} \approx 0.7302
\end{equation}
Later in Chapter \ref{Chapter5}, it will be clear that the constant $c_1$ is in fact given by $-\frac{\zeta_{1/2}}{2}$, using Riemann's analytic continuation of the zeta function. Using the differentiation rule:
\begin{equation}\label{Example_261_2} 
f_G'(n) = -\frac{\zeta_{1/2}}{2} + \sum_{k=1}^n \frac{1}{2\sqrt{k}} 
\end{equation}

In addition: 
\begin{equation}\label{Example_261_3} 
c_2= \sum_{k=1}^\infty \frac{1}{4 k^{3/2}} = \frac{\zeta_{3/2}}{4}
\end{equation}

Continuing in this manner, we note:
\begin{equation}\label{Example_261_4} 
c_m= (-1)^m \frac{(2m-3)!}{4^{m-1} (m-2)!} \zeta_{m-1/2} , \\ m\ge0
\end{equation}

Thus, the series expansion is given by: 
\begin{equation}\label{Example_261_5} 
f_G(n) = \sum_{k=1}^n \sqrt{k} = -\frac{\zeta_{1/2}}{2} n + \sum_{m=2}^\infty (-1)^m \frac{(2m-3)!}{4^{m-1} \,m!\, (m-2)!} \zeta_{m-1/2} n^m 
\end{equation}

Unfortunately, the radius of converges of the series in Eq \ref{Example_261_5} is $|n|<1$, which can be deduced using Stirling's approximation. However, we can still perform a sanity check on Eq \ref{Example_261_5} by plotting it for $-1\leq n \leq 1$, which is shown in Figure \ref{FigEx261}. As shown in the figure, $f(0)=0$ and $f(1)=1$ as expected. In addition, since we could take the lower bound $a$ to be zero, we know that $f(-1)=0$ as well by Rule 3, which is indeed the case as shown in the figure. Thus, the series expansion in Eq \ref{Example_261_5}  is indeed reasonable. In fact, we will show later in Chapter \ref{Chapter4} that Eq \ref{Example_261_5}  is correct when we introduce the summability method $\Xi$, which will show, for instance, that $f_G(n)$ given by  Eq \ref{Example_261_5} satisfies  $f_G(2)=\sqrt{1}+\sqrt{2}$ and   $f_G(3)=\sqrt{1}+\sqrt{2}+\sqrt{3}$  thus confirming our earlier analysis. Moreover, using $f(-1)=0$, we have the identity shown in Eq \ref{Example_261_6}, which can be verified numerically quite readily.
\begin{equation}\label{Example_261_6} 
-\frac{1}{2} \zeta_{1/2}= \sum_{m=2}^\infty  \frac{(2m-3)!}{4^{m-1}\,m!\, (m-2)!} \zeta_{m-1/2} 
\end{equation}

\begin{figure} [h] 
\centering
\includegraphics[scale=0.4]{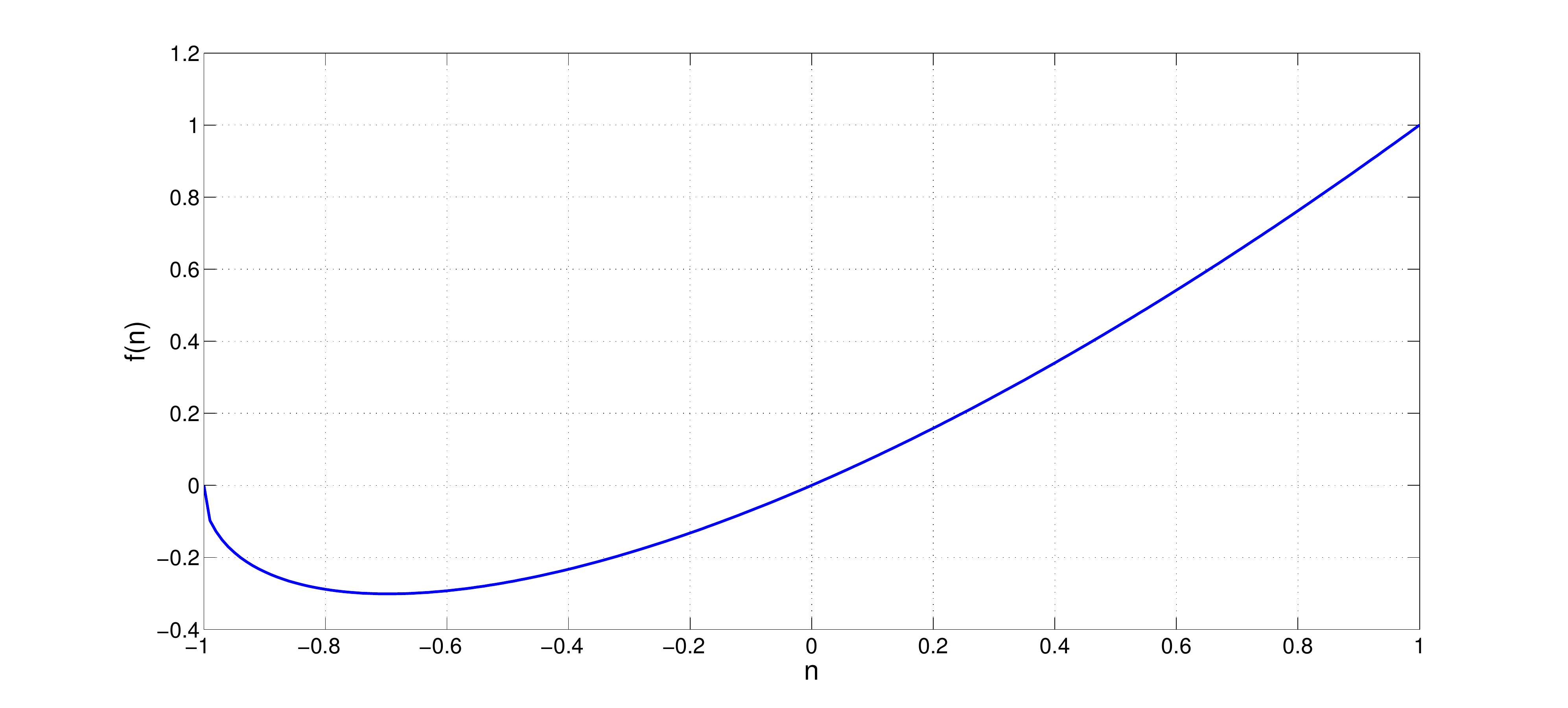}
\caption{The function $\sum_{k=1}^n \sqrt{k}$ \,plotted in the interval $-1\leq n \leq 1$}
\label{FigEx261}
\end{figure}

Earlier, we stated that Theorem \ref{theorem251} provides a more convenient approach for performing infinitesimal calculus on simple finite sums than the classical Euler-Maclaurin summation formula. To see this in the previous example, it is worthwhile to note that the Euler-Maclaurin summation formula does not readily extend the domain of $\sum_{k=1}^n \sqrt{k}$ to negative values of $n$ nor can it be used in deducing Taylor series expansion due to the presence of square roots and divergence, whereas Theorem \ref{theorem251} easily does. Finally, similar techniques can be used to obtain series expansion for different values of $m$.

\subsection{Example II: The Hyperfactorial Function} 
In this example, we return to the log-hyperfactorial function given by $f(n)=\log{H(n)}=\sum_{k=1}^n k \log{k}$. By Theorem \ref{theorem251}, we have:
\begin{equation}\label{Eq262_1} 
f_G'(n)= n + \log{n!} + c, \\ c=\lim_{n\to\infty} \big\{ n \log{n} + \frac{1+\log{n}}{2} - n - \log{n!}\big\}
\end{equation}
Using Stirling's approximation, we have $c=\frac{1-\log{2\pi}}{2}$, which is the missing value of the constant $c$ in Eq \ref{Eq246_6}. Using this value, we have by Theorem \ref{theorem251} the following series expansion: 
\begin{equation}\label{Eq262_2} 
\log{H(n)}= \big(\frac{1-\log{2\pi}}{2}\big) n + \frac{1-\lambda}{2} n^2 + \sum_{k=2}^\infty (-1)^k \frac{\zeta_k}{k(k+1)} n^{k+1}
\end{equation}

\subsection{Example III: The Superfactorial Function} 
In this example, we look into the simple finite sum function $f(n)$ given by $f(n)=\sum_{k=0}^n \varpi'(k)$, where $\varpi(n)$ is the log-factorial function and $\varpi'(n)=\psi(n+1)$. The finite sum $f(n)$ is semi-linear so we can use Theorem \ref{theorem231} directly. Using Rule 1, we have:
\begin{equation}\label{Eq263_1} 
f_G'(n) = c + \sum_{k=0}^n \varpi^{(2)}(k)
\end{equation}
Using Theorem \ref{theorem231}, the value of $c$ is given by: 
\begin{equation}\label{Eq263_2}
c=\lim_{n\to\infty} \Big\{ -\lambda + H_n - \sum_{k=0}^n (\zeta_2-H_{2,k}) = - (1+\lambda)
\end{equation}
Continuing in this manner yields the following series expansion for $f_G(n)$:
\begin{equation}\label{Eq263_3}
f_G(n)=-(1+\lambda) (n+1) + \sum_{k=2}^\infty (-1)^k \zeta_k (n+1)^k 
\end{equation}
By comparing the series expansion in Eq \ref{Eq263_3} with the series expansion of $\varpi'(n)$, we deduce that:
\begin{equation}\label{Eq263_4}
\sum_{k=0}^n \varpi'(k) = (n+1)\big(\varpi'(n+1) -1\big) \Rightarrow \sum_{k=0}^n H_k = (n+1)\big(H_{n+1} -1\big) 
\end{equation}

Moreover, if we take the lower bound of the iteration variable to be $a=1$ in the definition of $f(n)$ and after computing the series expansion around $n=0$ using Theorem \ref{theorem231}, we have the following expression:
\begin{equation}\label{Eq263_5}
\sum_{k=0}^n \varpi'(k) = \big(\zeta_2-1-\lambda\big) n + \sum_{k=2}^\infty (-1)^k \big(\zeta_k-\zeta_{k+1}\big) n^k 
\end{equation}

For instance, setting $n=1$ yields: 
\begin{equation}\label{Eq263_6}
2-\zeta_2 =  \sum_{k=2}^\infty (-1)^k \big(\zeta_k-\zeta_{k+1}\big)
\end{equation}

Using Eq \ref{Eq263_4}, Eq \ref{Eq263_5} and Theorem \ref{theorem251}, we are now ready to find the series expansion of the log-superfactorial function $\log{S(n)}$ given by  $ \sum_{k=1}^n \log{k!}$\;. By integrating both sides of Eq \ref{Eq263_5} using Rule 2, we have: 
\begin{equation}\label{Eq263_7}
\log{S(n)}= c_0 + c_1 n + \frac{\zeta_2-1-\lambda}{2} n^2 + \sum_{k=2}^\infty (-1)^k \frac{\zeta_k-\zeta_{k+1}}{k+1} n^{k+1} 
\end{equation}
Using $n=0$ and Rule 3 yields $c_0=0$. Using $n=1$, on the other hand, yields:
\begin{equation}\label{Eq263_8}
c_1= -\frac{\zeta_2-1-\lambda}{2} + \sum_{k=2}^\infty (-1)^k \frac{\zeta_k-\zeta_{k+1}}{k+1}
\end{equation}

To find a simple expression of $c_1$, we employ Theorem \ref{theorem251} this time, which states that:
\begin{align*}
c_1&= \lim_{n\to\infty} \Big\{ \log{n!} + \frac{\varpi'(n)}{2} - \sum_{k=1}^n \varpi'(k)\Big\}\\
&=\lim_{n\to\infty} \Big\{ \frac{\log{2\pi n}}{2} + n(\log{n} -1) + \frac{H_n -\lambda}{2} - \sum_{k=1}^n \varpi'(k)\Big\}
\end{align*}
Here, we have used Stirling's approximation. Upon using Eq \ref{Eq263_4}, we have: 
\begin{equation}\label{Eq263_9}
c_1= \frac{\log{2\pi}}{2} -\lambda+ \lim_{n\to\infty} \Big\{ (n+1)\big(\log{n}+\lambda-H_n\big) \Big\}
\end{equation}
We know, however, using the Euler-Maclaurin summation formula that the following asymptotic relation holds:
\begin{equation}\label{Eq263_10}
\log{n} + \lambda - H_n \sim -\frac{1}{2n}
\end{equation}

Plugging Eq \ref{Eq263_10} into Eq \ref{Eq263_9} yields the desired value: 
\begin{equation}\label{Eq263_11}
c_1 = \frac{\log{2\pi} -1}{2} - \lambda
\end{equation}

Equating Eq \ref{Eq263_10} with Eq \ref{Eq263_8}, we arrive at the following identity:
\begin{equation}\label{Eq263_11_1}
\frac{\log{2\pi} -1}{2} - \lambda= -\frac{\zeta_2-1-\lambda}{2} + \sum_{k=2}^\infty (-1)^k \frac{\zeta_k-\zeta_{k+1}}{k+1}
\end{equation}

Finally, by plugging  Eq \ref{Eq263_10} into  Eq \ref{Eq263_7}, we arrive at the desired series expansion of the log-superfactorial function:
\begin{equation}\label{Eq263_12}
\log{S(n)}=  \big(\frac{\log{2\pi} -1}{2} - \lambda\big) n + \frac{\zeta_2-1-\lambda}{2} n^2 + \sum_{k=2}^\infty (-1)^k \frac{\zeta_k-\zeta_{k+1}}{k+1} n^{k+1} 
\end{equation}

\subsection{Example IV: Alternating Sums} \label{ExAlterSum}
In this example, we address the case of alternating simple finite sums, and illustrate the general principles using the example $f(n)=\sum_{k=1}^n (-1)^k = \sum_{k=1}^n e^{j\pi k}$. We will first rely solely on the foundational rules given earlier in Table \ref{TableRules}. Using Rule 1, we have:
\begin{equation}\label{Eq264_1} 
f_G'(n) = j\pi f_G(n) + f_G'(0) 
\end{equation} 
Upon successive differentiation of Eq \ref{Eq264_1}, we have: 
\begin{equation}\label{Eq264_2} 
f_G^{(r)}(n) = (j\pi)^r f_G(n) + (j\pi)^{r-1} f_G'(0) 
\end{equation} 
Since $f_G(0)=0$ by Rule 3, we have: 
\begin{equation}\label{Eq264_3} 
f_G^{(r)}(0) = (j\pi)^{r-1} f_G'(0) 
\end{equation} 
Putting this into a series expansion around $n=0$ yields: 
\begin{equation}\label{Eq264_4} 
f_G(n) = f_G'(0) \sum_{k=0}^\infty \frac{(j\pi)^{k-1}}{k!} n^k = \frac{f_G'(0)}{j\pi} (e^{j\pi n} -1) 
\end{equation} 
Using initial condition $f_G(1)=-1$ yields: 
\begin{equation}\label{Eq264_5} 
f_G(n) = \frac{e^{j\pi n} -1}{2},  \\ \text{and } f_G'(0)=\frac{j\pi}{2} 
\end{equation} 

Of course, it is straightforward to show that the closed-form expression in Eq \ref{Eq264_5} is indeed correct in the sense that it satisfies recurrence identity and initial condition of $f(n)$. However, we could have also used Theorem \ref{theorem251} to deduce the same expression. First, we note using Theorem \ref{theorem251} that:

\begin{equation}\label{Eq264_6} 
f_G'(0) = e^{j\pi n} \sum_{r=0}^\infty \frac{B_r}{r!} (j\pi)^r - j\pi \sum_{k=1}^n e^{j\pi k}, \\ \text{ for any value } n
\end{equation} 

We choose $n=1$ and use the formal substitution $\frac{t}{1-e^{-t}} = \sum_{r=0}^\infty \frac{B_r}{r!} t^r$ to obtain:
\begin{equation}\label{Eq264_7} 
f'_G(0) = \frac{j\pi}{2} 
\end{equation} 

Plugging Eq \ref{Eq264_7} into Eq \ref{Eq264_4} yields the same closed-form expression. Of course, such approach for alternating sums is not very insightful. For instance, it is difficult to use the Euler-Maclaurin summation formula to deduce asymptotic expressions for alternating sums, due to the presence of complex values, despite the fact that alternating finite sums are purely real-valued for all integer arguments. We will resolve such difficulty later when we derive the analog of the Euler-Maclaurin summation formula for alternating finite sums, in Chapter \ref{Chapter5}, which will allow us to perform many deeds with ease. For instance, we can easily obtain simple asymptotic expressions for alternating sums and accelerate convergence of alternating series quite arbitrarily. Moreover, we will generalize Summability Calculus to arbitrary oscillating finite sums including, obviously, the special case of alternating finite sums.
\section{Summary of Results} 
In this chapter, a few foundational rules for performing infinitesimal calculus on simple finite sums are deduced using elementary calculus and the two defining properties of simple finite sums, which are given in Table \ref{TableRules}. Such simple rules illustrate how special functions that are defined using simple finite sums are, in fact, as effortless to analyze as elementary functions. Moreover, finite products can be transformed, with aid of exponentiation and the logarithmic function, into a composition of functions that involve simple finite sums; hence the same rules can be employed accordingly. Nevertheless, the foundational rules are insufficient by themselves to perform infinitesimal calculus since they involve non-arbitrary constants. To evaluate such constants, Theorem \ref{theorem231} can be used if the finite sum is semi-linear; otherwise the general approach is to employ Theorem \ref{theorem251}. Of course, both Theorem \ref{theorem231} and Theorem \ref{theorem251} are equivalent for semi-linear finite sums. 

One particular limitation of the approach presented so far, however, is that it does not address the general case of convoluted sums and products. Furthermore, the resulting series expansions that can be deduced using Summability Calculus may have limited radii of convergence, which limits their applicability in \emph{evaluating} finite sums and products for non-integer arguments as illustrated in Example \ref{Example261} earlier. Moreover, the approach presented so far does not yield a convenient method of handling oscillating sums, nor can it provide us with a simple direct method for evaluating finite sums for all complex arguments without having to derive Taylor series expansion. These drawbacks will be resolved in the remaining chapters.
 
\chapter{Convoluted Finite Sums} \label{Chapter3}
\epigraph{\emph{A complex system that works is invariably found to have evolved from a simple system that worked.}}{John Gall in \emph{Systemantics}}
In this chapter, we extend previous results of Chapter 2 to the case of convoluted finite sums. 
\section{Infinitesimal Calculus on Convoluted Sums}
As discussed earlier, convoluted sums and products is a generalization to simple finite sums and products in which the iterated function $g$ is itself a function of the bound $n$. Thus, instead of dealing with finite sums of the form \simplefinitesum, we now have to deal with finite sums of the form \convolutedsum. Such convoluted sums are markedly different. For instance, the recurrence identity is now given by Eq \ref{Eq31_1}, which is clearly more difficult to handle than the recurrence identities in simple finite sums and products. In fact, such complex recurrence property does not offer much insight as to how to define a natural generalization of convoluted sums and products into non-integer arguments in the first place, needless to mention that the two defining properties of simple finite sums no longer hold. However, it is perhaps one remarkable result of Summability Calculus that convoluted sums and products are, in fact, as straightforward to analyze as simple finite sums and products. This includes, for instance, performing infinitesimal calculus, deducing asymptotic behavior, as well as direct evaluation for fractional values of $n$.
\begin{equation}\label{Eq31_1}
f(n)=f(n-1) + g(n,n) + \sum_{k=a}^{n-1} \big(g(k,n)-g(k, n-1)\big)
\end{equation}

The simplicity in analyzing convoluted finite sums comes with the aid of the generic differentiation rule, which is commonly referred to as the Chain Rule in the literature \cite{Hughes98}. However, we will call it here the generic differentiation rule because it readily provides derivatives of any generic binary mathematical operator. Examples related to complex binary operators such as those that involve fractional functional iteration are presented in \cite{IBO1}. In addition, classical differentiation rules, such as the basic sum and product rules, are easily derived from the generic differentiation rule. \\ \hrule
\begin{lemma}{\textbf{(The Generic Differentiation Rule)}}\label{lemmaGeneric} 
Given an arbitrary binary mathematical operator $\diamond$, then: 
\begin{equation}\label{lemmaGenericEq}
\frac{d}{dx} \big( h_1(x) \diamond h_2(x)\big) = \frac{d}{dx} \big( h_1(x) \diamond h_2\big) + \frac{d}{dx} \big( h_1 \diamond h_2(x)\big)
\end{equation}
Here, $h_1(x)$ and $h_2(x)$ are treated as functions of the variable $x$ during differentiation while $h_1$ and $h_2$ are treated as \emph{constants} during differentiation.
\end{lemma}
\begin{proof} 
One direct proof using the definition of differentiation is given in \cite{IBO1}. Alternatively, it follows by the well-known chain rule. 
\end{proof} \hrule\vspace{12pt} 

As stated in \cite{IBO1}, the generic differentiation rule is a very insightful tool. It implies that the derivative of a function $f(x)$, in which the independent variable $x$ appears more than once, can be computed by partitioning all appearances of the independent variable into two groups such that one group is assumed constant during differentiation in one time and the other group is assumed constant during differentiation in a second time. The overall derivative $f'(x)$ is, then, the sum of the two results.

In the context of convoluted sums, the generic differentiation rule, with aid of Summability Calculus on simple finite sums, makes the task of computing derivatives at an arbitrary point $n=n_0$ an elementary task. Here, given that \convolutedsum, we have two appearances of the independent variable $n$: (1) as an upper bound to the finite sum, and (2) as an argument to the function $g$. Using the generic differentiation rule, we can compute the derivative at an arbitrary point $n=n_0$ as follows. First, we treat the upper bound $n$ as a constant $n_0$ and differentiate accordingly, which yields $\sum_{k=a}^n \frac{\partial}{\partial n} g(k, n)$. Second, we treat the second argument of $g$ as constant and differentiate accordingly using Rule 1, which can be done using the earlier results of Chapter \ref{Chapter2} because the sum \emph{is now essentially reduced to a simple finite sum}.

Let us illustrate how above procedure works on the convoluted finite sum $f(n)=\sum_{k=1}^n \frac{1}{k+n}$. In this example, we will also show how the aforementioned method is indeed consistent with the results of Summability Calculus for simple finite sums. Now, given the definition of $f(n)$, suppose we would like to find its derivative at $n=1$. To do this, we first fix the upper bound at $n=1$, which gives us:
\begin{equation*}
\frac{d}{dn} \sum_{k=1}^1 \frac{1}{k+n} = \frac{d}{dn} \frac{1}{1+n} = - \frac{1}{(1+n)^2} \Big|_{n=1} = -\frac{1}{4}
\end{equation*}

Second, we fix the second appearance of $n$ at $n=1$, which yields the simple finite sum $\sum_{k=1}^n \frac{1}{1+k}$. Clearly, the derivative in the latter case is given by:
\begin{equation*}
\frac{d}{dn} \sum_{k=1}^n \frac{1}{1+k} = \sum_{k=n+1}^\infty \frac{1}{(k+1)^2} \Big|_{n=1} = \zeta_2 - \frac{5}{4}
\end{equation*}
Here, we used the rules of Summability Calculus for simple finite sums. So, by the generic differentiation rule, we sum the two last results to deduce that:
\begin{equation*}
\frac{d}{dn} \sum_{k=1}^n \frac{1}{n+k} \Big|_{n=1}=  \zeta_2 - \frac{3}{2}
\end{equation*}

How do we know that this is indeed correct? Or to put it more accurately, why do we know that this corresponds to the derivative of the unique most natural generalization to the convoluted finite sum? To answer this question, we note in this example that we have $f(n)=H_{2n}-H_n$. Therefore, if we were indeed dealing with the unique most natural generalization to such convoluted sum, then such generalization should be consistent with the unique most natural generalization of the Harmonic numbers. To test this hypothesis, we differentiate the expression $f(n)=H_{2n}-H_n$ at $n=1$, which yields by the results of Section \ref{Section_23}: 
\begin{equation*}
\frac{d}{dn} H_{2n}-H_n \Big|_{n=1}=  2\big(\zeta_2-\sum_{k=1}^{2n} \frac{1}{k^2}\big) - \big(\zeta_2-\sum_{k=1}^{n} \frac{1}{k^2}\big) \Big|_{n=1} = \zeta_2 -\frac{3}{2}
\end{equation*}

Thus, the two results agree as desired. We will now prove the general statements related to convoluted finite sums. \\ \hrule 
\begin{lemma}{\textbf{(The Differentiation Rule of Convoluted Sums)}}\label{diffConvoluted} 
Let $f(n)$ be a convoluted sum given by \convolutedsum, and let its unique natural generalization to non-integer arguments be denoted $f_G(n)$, then:
\begin{equation}\label{DiffConvEq}
f_G'(n) = \sum_{k=a}^n \frac{\partial}{\partial n} g(k, n) + \sum_{r=0}^\infty \frac{B_r}{r!} \frac{\partial^r}{\partial m^r} g(m, n) \Big|_{m=n}
\end{equation}
\end{lemma}
\begin{proof} By direct application of the generic differentiation rule in Lemma \ref{lemmaGeneric} and Corollary \ref{corollary251_1}. 
\end{proof} \hrule\vspace{12pt} 

While it is true that Lemma \ref{diffConvoluted} follows directly from the generic differentiation rule and Corollary \ref{corollary251_1}, how do we know that it corresponds to the derivative of the unique most natural generalization of $f(n)$? The answer to this question rests on the earlier claims for simple finite sums. In particular, the generic differentiation rule implies that $f_G'(n)$ is a sum of two terms: (1) the sum of the partial derivatives with respect to $n$, and (2) the derivative rule for simple finite sums. Because the second term used the differentiation rule for simple finite sums given in Theorem \ref{theorem251} that was shown earlier to correspond to the unique most natural generalization, it follows, therefore, that Lemma \ref{diffConvoluted} yields the derivative of the most natural generalization for convoluted sums as well. In other words, we reduced the definition of natural generalization for convoluted sums into the definition of natural generalization for simple finite sums using the generic differentiation rule. Unfortunately, however, Lemma \ref{diffConvoluted} does not permit us to deduce asymptotic behavior of the convoluted sums  . This can be deduced as shown in our next lemma.\\
\hrule 
\begin{lemma}{\textbf{(The  Euler-Maclaurin Summation Formula for Convoluted Sums)}}\label{convolutedEMFormula} 
Let $f(n)$ be a convoluted finite sum given by \convolutedsum, then the following function $f_G(n)$ correctly interpolates the discrete values of $f(n)$.
\begin{align*}\label{convolutedEMFormulaEq}
f_G(n) = &\int_a^n g(t, n) \, dt + \frac{g(n, n)+ g(a, n)}{2} \\
&+ \sum_{r=1}^\infty \frac{B_{r+1}}{(r+1)!} \Big(\frac{\partial^r}{\partial m^r} g(m, n) \Big|_{m=n} - \frac{\partial^r}{\partial m^r} g(m, n) \Big|_{m=a} \Big)  
\end{align*}
\end{lemma}
\begin{proof}
Suppose that $f(n)$ was instead given by $\sum_{k=a}^n g(k, n_0)$ for a fixed $n_0$, then $f(n)$ is a simple finite sum and its generalization is given by the Euler-Maclaurin summation formula, which holds for any value $n_0$ including $n_0 = n$. Plugging $n_0 = n$ into the Euler-Maclaurin  summation formula yields the desired result.
\end{proof} \hrule\vspace{12pt} 

Of course, Lemma \ref{convolutedEMFormula} can be used to deduce asymptotic behavior of convoluted finite sums. However, this does not imply that Lemma \ref{convolutedEMFormula} gives the anti-derivative of the function $f_G'(n)$ as dictated by Lemma \ref{diffConvoluted}. In other words, while Lemma \ref{diffConvoluted} gives the derivative of the unique most natural generalization and Lemma \ref{convolutedEMFormula} yields a potential definition for such generalization, are we indeed dealing with the same function in both cases? The answer is indeed in the affirmative as our next main theorem shows. \\ \hrule
\begin{theorem}\label{theorem311} 
Let $f(n)$ be a convoluted sum given by \convolutedsum, and let its generalization to non-integer arguments be denoted $f_G(n)$, whose value is given by Lemma \ref{convolutedEMFormula}, then $f_G(n)$ is the unique most natural generalization to non-integer arguments of $f(n)$. More precisely, $f_G'(n)$ is the one given by Lemma \ref{diffConvoluted}. 
\end{theorem}
\begin{proof} 
As discussed earlier, we have shown in the previous section that Summability Calculus yields unique natural generalizations in the case of simple finite sums. Using the generic differentiation rule and the earlier results of Summability Calculus on simple finite sums, the derivative of the natural generalization of convoluted sums is unique and it is given by Lemma \ref{diffConvoluted}. Here, we have used the term \lq\lq natural generalization'' for convoluted sums because the generic differentiation rule allows us to reduce analysis of convoluted sums to the case of simple finite sums, in which Summability Calculus yields unique natural generalization. 

Knowing that the derivative $f_G'(n)$ is uniquely determined for the natural generalization of convoluted sums and that $f_G(a)=g(a, a)$ by definition, we have that Lemma \ref{diffConvoluted} implies a unique natural generalization for convoluted sums as well. Now, we need to show that such unique generalization is indeed the one given by Lemma \ref{convolutedEMFormula}. In other words, we need to establish that the differentiation rule given by Lemma \ref{diffConvoluted} implies necessarily that the generalized definition of $f_G(n)$ is the one given by Lemma \ref{convolutedEMFormula}. 

To show this, we first start from Lemma \ref{diffConvoluted} and integrate both sides with respect to $n$ as follows:
\begin{equation}\label{theorem311_1} 
f_G(n) = g(a, a) + \sum_{r=0}^\infty \frac{B_r}{r!} \int_a^n \frac{\partial}{\partial m} g(m, t)    \Big|_{m=t} \,\, dt + \int_a^n \sum_{k=a}^t \frac{\partial}{\partial t} g(k, t)\, dt
\end{equation}

However, $\sum_{k=a}^t \frac{\partial}{\partial t} g(k, t)$ can also be expressed using Eq \ref{theorem311_1} as follows: 
\small
\begin{equation}\label{theorem311_2} 
\sum_{k=a}^t \frac{\partial}{\partial t} g(k, t) = \frac{\partial}{\partial \eta} g(a, \eta) \Big|_{\eta=a} + \sum_{r=0}^\infty \frac{B_r}{r!} \int_a^t \frac{\partial^2}{\partial m\, \partial t_2} g(m, t_2)    \Big|_{m=t_2} \,\, dt_2 + \int_a^t \sum_{k=a}^{t_2} \frac{\partial^2}{\partial t_2^2} g(k, t_2)\, dt_2
\end{equation}
\normalsize

We now plug Eq \ref{theorem311_2} into Eq \ref{theorem311_1}. Repeating the same process indefinitely and by using Cauchy's formula for repeated integration \cite{MathWordCauchyInt}, we have: 
\small
\begin{equation}\label{theorem311_3} 
f_G(n) = \sum_{r=0}^\infty \frac{(n-a)^r}{r!} \frac{\partial^r}{\partial \eta^r} g(a, \eta) \Big|_{\eta=a} + \sum_{r=0}^\infty \frac{B_r}{r!} \frac{\partial^r}{\partial m^r} \int_a^n \sum_{b=0}^\infty \frac{(n-t)^b}{b!} \frac{\partial^b}{\partial t^b} g(m, t) \Big|_{m=t} \, dt  
\end{equation}
\normalsize

Now, we use Taylor's theorem, which formally states that: 
\begin{equation}\label{theorem311_4} 
\sum_{r=0}^\infty \frac{(n-a)^r}{r!} \frac{\partial^r}{\partial \eta^r} g(a, \eta) \Big|_{\eta=a} = g(a, n)
\end{equation}

Here, the last identity holds because the summation is a formal series expansion around $\eta=a$. Similarly, we have:
\begin{equation}\label{theorem311_5} 
\sum_{b=0}^\infty \frac{(n-t)^b}{b!} \frac{\partial^b}{\partial t^b} g(m, t) = g(m, n)
\end{equation}

Plugging both expressions in Eq \ref{theorem311_4}  and Eq \ref{theorem311_5} into Eq \ref{theorem311_3} yields the Euler-Maclaurin summation formula for convoluted sum in Lemma \ref{convolutedEMFormula}, which is the desired result. Therefore, the differentiation rule in Lemma \ref{diffConvoluted} implies necessarily that the natural generalization $f_G(n)$ is the one given by Lemma \ref{convolutedEMFormula}, which completes the proof.
\end{proof} \hrule\vspace{12pt} 

Theorem \ref{theorem311} along with the two lemmas provide us with a complete set of rules for performing infinitesimal calculus on convoluted finite sums that operates implicitly on the unique most natural generalization of those sums. Most importantly, we know by Theorem \ref{theorem311} that whenever a convoluted finite sum can be expressed algebraicly using simple finite sums, then we will still arrive at identical results by working with either expression. This was illustrated earlier on the convoluted finite sum $\sum_{k=a}^n \frac{1}{k+n}$. Here because the finite sum is convoluted, we can use results of this section to perform infinitesimal calculus and deduce asymptotic behavior. However, since the convoluted finite sum can alternatively be expressed as $H_{2n}-H_n$, where the latter expression only involves simple finite sums, we could also use results of the previous section to perform the same operations. Importantly, both approaches are always consistent with each other yielding identical results. More examples are presented in the following section.  
\section{Examples to Convoluted Finite Sums} 
\subsection{Example I: The Factorial Function Revisited} 
Let $f(n)$ be a convoluted finite sum given by $f(n)=\sum_{k=1}^n \log{(1-\frac{k-1}{n})}$. Using Lemma \ref{diffConvoluted}, its derivative is given by: 
\begin{equation}\label{Ex321_1} 
f'_G(n) = \frac{1}{n} \sum_{k=1}^n \frac{k-1}{n-k+1} -\log{n} - \sum_{r=1}^\infty \frac{B_r}{r} 
\end{equation} 
Of course, the infinite sum $\sum_{r=1}^\infty \frac{B_r}{r}$ diverges but it is formally defined by Euler's constant $\lambda$ as discussed earlier in Figure \ref{figureEMEval}. Thus, we have:
\begin{equation}\label{Ex321_2} 
f'_G(n) = \frac{1}{n} \sum_{k=1}^n \frac{k-1}{n-k+1} -\log{n} - \lambda = H_n -\log{n} -(1+\lambda) 
\end{equation}

Eq \ref{Ex321_2} gives a closed-form expression of the derivative of the convoluted sum. We now note, by definition, that $f(n) = \log{n!} - n \log{n}$. By differentiating both sides of the latter expression we arrive at the same result, which illustrates how the differentiation rule in Lemma \ref{diffConvoluted} indeed yields the derivative of the unique most natural generalization to convoluted sums, and that it is always consistent with results of Summability Calculus on simple finite sums.

Now, we use Theorem \ref{theorem311} to deduce that
\begin{equation}\label{Ex321_3} 
f_G(n) = -n + 1 + \frac{\log{n}}{2} - \sum_{r=2}^\infty \frac{B_r}{r(r-1)} \big(1-\frac{1}{n^{r-1}}\big)
\end{equation} 
Taking the derivative of both sides in Eq \ref{Ex321_3} with respect to $n$ and equating it with Eq \ref{Ex321_2} yields:
\begin{equation}\label{Ex321_4} 
H_n - \log{n} -\lambda = \frac{1}{2n} - \sum_{r=2}^\infty \frac{B_r}{r} \frac{1}{n^r} 
\end{equation} 
Of course, last expression could be derived alternatively by direct application of the original Euler-Maclaurin summation formula.

\subsection{Example II: Convoluted Zeta Function}  
In this section, we present another example on how to use the rules of Summability Calculus for simple finite sums in performing infinitesimal calculus for convoluted sums. Our example is the function $\sum_{k=1}^n \frac{1}{k^n}$ and our objective is to find an analytic expression of its derivative for all $n>0$. Using the generic differentiation rule, we first assume that the upper bound $n$ is fixed at $n=n_0$ and differentiate accordingly. This yields $\frac{d}{dn} \sum_{k=1}^{n_0} \frac{1}{k^n} = - \sum_{k=1}^{n_0} \frac{\log{k}}{k^n}$. Second, we assume that the second appearance of $n$ is fixed at $n=n_0$ and differentiate. However, for $n_0 >0$, this is the derivative of a semi-linear finite sum given by Theorem \ref{theorem231}, and it is given by $\frac{d}{dn} \sum_{k=1}^n \frac{1}{k^{n_0}} = n_0 \big(\zeta_{n_0+1} - \sum_{k=1}^n \frac{1}{k^{n_0+1}}\big)$. Therefore, by the generic differentiation rule, the overall derivative $f_G'(n)$ is given by:
\begin{equation}\label{Ex232_1} 
\frac{d}{dn} \sum_{k=1}^n \frac{1}{k^n} = - \sum_{k=1}^{n} \frac{\log{k}}{k^n} + n \big(\zeta_{n+1} - \sum_{k=1}^n \frac{1}{k^{n+1}}\big)
\end{equation}

\subsection{Example III: Numerical Integration}  
The classical approach for approximating definite integrals uses Riemann sums. In this example, we show how Theorem \ref{theorem311} yields higher-order approximations, which, in turn, is used to deduce one of the well-known recursive definitions of Bernoulli numbers. First, we start with the following equation that follows by direct application of Lemma \ref{convolutedEMFormula}: 
\begin{align*}\label{Ex323_1} 
\frac{m}{n} \sum_{k=0}^n f(x_0+\frac{m}{n} k) &= \int_{x_0}^{x_0+m} f(t)\, dt + \frac{m}{n}\Big(\frac{f(x_0)+f(x_0+m)}{2}\Big) \\
&+ \sum_{r=2}^\infty \frac{B_r}{r!} (\frac{m}{n})^{r} \Big(f^{(r-1)}(x_0+m) - f^{(r-1)}(x_0)\Big) 
\end{align*}

Now, we let $n=\frac{x-x_0}{\Delta x}$ and let $m=x-x_0$, and after plugging these two expressions into last equation and rearranging the terms, we arrive at: 
\small
\begin{equation}\label{Ex323_2}
\int_{x_0}^x f(t)\, dt = \sum_{k=0}^{\frac{x-x_0}{\Delta x}} f(x_0 + k \Delta x) - \frac{f(x_0)+f(x)}{2} \Delta x -\sum_{r=2}^\infty \frac{B_r}{r!}\Big(f^{(r-1)}(x) - f^{(r-1)}(x_0)\Big) (\Delta x)^r
\end{equation} 
\normalsize

Clearly, the first term in the right-hand side is the classical approximation of definite integrals but adding additional terms yields higher-order approximations. In particular, if we let $\Delta x = x-x_0$, we, then, have:
\begin{equation}\label{Ex323_3}
\int_{x_0}^x f(t)\, dt =  \frac{f(x_0)+f(x)}{2} (x-x_0) -\sum_{r=2}^\infty \frac{B_r}{r!}\Big(f^{(r-1)}(x) - f^{(r-1)}(x_0)\Big) (x-x_0)^r
\end{equation} 

Using last equation, we can equate coefficients of the Taylor series expansion in both sides, which yields a property of Bernoulli numbers given by: 
\begin{equation}\label{Ex323_4}
\frac{1}{s+1} = \frac{1}{2} - \sum_{r=2}^s \binom{s}{r} \frac{B_r}{s-r+1}
\end{equation} 
Eq \ref{Ex323_4} can, in turn, be rearranged to yield the following well-known recursive definition of Bernoulli numbers:
\begin{equation}\label{Ex323_5}
B_m = 1 - \sum_{r=0}^{m-1} \binom{m}{r} \frac{B_r}{m-r+1} 
\end{equation} 

\subsection{Exmaple IV: Alternating Euler-Maclaurin Sum} 
In Corollary \ref{corollary251_1}, we stated that if \simplefinitesum, then one way of computing $f_G'(a-1)$ is given by: 
\begin{equation}\label{Ex324_1} 
f_G'(a-1)=\sum_{r=0}^\infty \frac{B_r}{r!} g^{(r)}(a-1) 
\end{equation} 
However, it is often the case that $g(a-1)$ is not defined, for which Eq \ref{Ex324_1} becomes meaningless, and this can happen even if $f_G'(a-1)$ is itself well-defined. Of course, we showed earlier that it is easier to compute $f_G'(a-1)$ using Theorem \ref{theorem251} but an alternative method is also possible, which was previously stated in Eq \ref{theorem251_final_2}. In this example, we provide a simple proof to Eq \ref{theorem251_final_2} using Summability Calculus on convoluted sums. First, we note that $\sum_{k=a}^n g(k) = \sum_{k=a}^n g(n-k+a)$ by the \emph{translation invariance property} of Corollary \ref{Corollaryshiftproperty}. That is, we can always convert a simple finite sum into a convoluted finite sum. Then, we can use the differentiation rule of convoluted sums given in Lemma \ref{diffConvoluted}, which yields:
\begin{equation}\label{Ex324_2}
f_G'(n)=\sum_{k=a}^n g'(n-k+a) + \sum_{r=0}^\infty (-1)^r \frac{B_r}{r!} g^{(r)}(a) = \sum_{k=a}^n g'(k) +  \sum_{r=0}^\infty (-1)^r \frac{B_r}{r!} g^{(r)}(a)
\end{equation}
However, upon using Rule 1 of the foundational rules in Table \ref{TableRules}, Eq \ref{Ex324_2} implies that:
\begin{equation}\label{Ex324_3} 
f_G'(a-1)=  \sum_{r=0}^\infty (-1)^r \frac{B_r}{r!} g^{(r)}(a)
\end{equation}

Thus, even for cases in which $g(a-1)$ is not defined, we can use \emph{an alternating Euler-Maclaurin sum} to compute the derivative $f_G'(a-1)$, which is the exact statement of Eq \ref{theorem251_final_2}. To be more precise, we say that if $g^{(r)}(a-1)$ is defined for all $r\ge 0$, then both sides of Eq \ref{theorem251_final_2} are formally equivalent. 

\subsection{Exmaple V: An identity of Ramanujan} 
In one of his earliest works, Ramanujan was interested in simple finite sums of the form $\phi(x, n)= \sum_{k=1}^n \frac{1}{(xk)^3 -xk}$. There, Ramanujan showed that many convoluted sums can be converted to simple finite sums using $\phi(x, n)$ \cite{Berndt2}. For example, he showed that the convoluted finite sum $\sum_{k=1}^n \frac{1}{k+n}$ that was discussed earlier can also be expressed as: 
\begin{equation}\label{Ex325_1}
\sum_{k=1}^n \frac{1}{k+n} = \frac{n}{2n+1} + \sum_{k-1}^n \frac{1}{8k^3-2k}
\end{equation}

With aid of Summability Calculus on convoluted finite sums, we can use Eq \ref{Ex325_1} as a starting point to derive many interesting identities. For instance, if we differentiate both sides with respect to $n$, we obtain:
\begin{equation}\label{Ex325_2}
\sum_{k=1}^n \frac{1}{k^2} - 2 \sum_{k=1}^{2n} \frac{1}{k^2} + \zeta_2 - \frac{n+1}{(2n+1)^2} = \sum_{n+1}^\infty \frac{24k^2-2}{(8k^3-2k)^2}
\end{equation}
Selecting $n=0$, then we have, by the empty sum rule, the following fast-converging series for $\zeta_2$: 
\begin{equation}\label{Ex325_3}
\zeta_2= 1+ \sum_{1}^\infty \frac{24k^2-2}{(8k^3-2k)^2}
\end{equation}
On the other hand, if we integrate both sides of Eq \ref{Ex325_1} and rearrange the terms, we arrive at:
\begin{equation}\label{Ex325_4}
\sum_{k=1}^n \log{(1-\frac{1}{4k^2})} = 2\log{(2n)!} - 4\log{n!} + \log{(2n+1)} - n \log{8}
\end{equation}
Taking the limit as $n\to\infty$ yields by Stirling's approximation: 
\begin{equation}\label{Ex325_5}
\sum_{k=1}^\infty \log{(1-\frac{1}{4k^2})} = -\log{\frac{\pi}{2}}
\end{equation}
Thus, we arrive at a restatement of Wallis formula \cite{SWWallis}: 
\begin{equation}\label{Ex325_6}
\prod_{k=1}^\infty (1-\frac{1}{4k^2}) = \frac{2}{\pi}
\end{equation}

\subsection{Exmaple VI: Limiting Behavior} 
In this example, we look into using the Euler-Maclaurin summation formula for convoluted sums to deduce limiting behavior. Our example will be the function $f(n)=\sum_{k=0}^n (\frac{k}{n})^n$. As shown in \cite{Spivey2006}, we have: 
\begin{equation} \label{EMLB_1}
\lim_{n\to\infty} f(n) = \frac{e}{e-1}
\end{equation} 
To prove this, we note by Lemma \ref{convolutedEMFormula} that: 
\begin{equation} 
f(n) = \frac{n^{n+1}}{n+1} + n^n \sum_{r=0}^n \frac{B_{r+1}}{(r+1)!} \chi_n(r), \\ \text{ where } \chi_n(r) = \prod_{k=1}^r \big(1-\frac{k-1}{n}\big)
\end{equation}
Importantly, we know that the above expression is \emph{exact}. Because $\sum_{r=0}^n \frac{B_r}{r!} = \frac{e}{e-1}$, and $\chi_n(r)=1+O(\frac{1}{n})$, we deduce that Eq \ref{EMLB_1} indeed holds. In the following chapter, the function $\chi_n(r)$ will come up again, in which we show that it defines a simple, yet powerful, summability method for Taylor series expansions. 

\section{Summary of Results}
In this chapter, the earlier results of Summability Calculus on simple finite sums have been extended to address the general case of convoluted sums. Such generalization became possible with aid of the generic differentiation that yields unique derivatives of convoluted sums and products. Because the resulting differentiation rule for convoluted sums uses Summability Calculus on simple finite sums, which was argued earlier to operate on unique most natural generalizations, we deduce uniqueness of natural generalization for convoluted sums and products as well. 

Aside from dealing with derivatives, one potential closed-form expression of the natural generalization itself can be deduced immediately using Euler-Maclaurin summation formula. However, as it turns out, such closed-form solution is indeed more than a mere potential solution. It is shown that the differentiation rule for convoluted sums, which yields the derivative of the unique most natural generalization, is, in fact, the derivative of the Euler-Maclaurin summation formula for convoluted sums, thus bridging the two main results and proving that the Euler-Maclaurin summation formula for convoluted sums is indeed the unique most natural generalization of convoluted sums. 

\chapter{Analytic Summability Theory}\label{Chapter4}
\epigraph{\emph{Errors like straws, upon the surface flow. \\He who would search for pearls, must dive below\ldots}}{John Dryden (1631 -- 1700)}
Infinitesimal calculus is built on the assumption that infinitesimally small errors are simply zero. Euler, for example, is often quoted saying: \lq\lq \emph{To those who ask what the infinitely small quantity in mathematics is, we answer that it is actually zero.}'' Such view appears everywhere in Calculus including in fundamental definitions such as derivatives and integrals. For example, the Cauchy definition of a function derivative is given by $f'(x)=\lim_{h\to\ 0}\big(f(x+h)-f(x)\big)/h$, which implicitly assumes that the error term $O(h)$ is zero because it is infinitesimal at the limit $h\to 0$. 

Clearly, equating infinitesimal quantities with zero is valid if the number of such errors remains finite. However, it turns out that many important results in Calculus are actually derived by repeated application of a specific process \emph{infinitely many times}. For example, the Taylor series expansion and the Euler-Maclaurin summation formula are typically derived by repeated application of integration by parts (see for instance \cite{Sikic90, Apostol, LampretEM2001}). In such cases, if integration by parts induces an infinitesimally small error, then a repeated application of such process \emph{infinitely many times} can lead to incorrect results. We are all very well familiar with such phenomenon. In fact, it is one reason why algebraic derivation of Taylor series  expansion using real analysis fails to warn against series divergence!

Where exactly would infinitesimal calculus fail? To see why calculus sometimes fails because it hides infinitesimal errors, suppose hypothetically that we have an identity, which looks like the following:
\begin{equation}\label{chap4_1}
\lim_{n\to\infty} \sum_{k=0}^n g(k\Delta x)\, \Delta x = F_g + \lim_{n\to\infty} \sum_{k=1}^n h(k\Delta x) \,g(k\Delta x)\, \Delta x
\end{equation} 
Here, if $\Delta x = \frac{x}{n}$, then both sums are the classical Riemann sums that approximate integrals. Now, suppose that $F_g$ is some functional of $g$, and $h$ is independent of $g$. Then, knowing that Riemann sums converge to integrals at the limit $n\to\infty$, one can immediately write: 
\begin{equation}\label{chap4_2}
\int_0^x g(t)\,dt = F_g + \int_0^x h(t)\,g(t)\,dt
\end{equation} 
Starting with the last equation, we can apply it repeatedly infinitely many times, which would yield: 
\begin{equation}\label{chap4_3}
\int_0^x g(t)\,dt = F_g + F_{g\cdot h} + F_{g\cdot h^2} + F_{g\cdot h^3} + \dotsm = \sum_{k=0}^\infty F_{g\cdot h^k}
\end{equation} 

We arrive at a pitfall! The right-hand side might diverge even though the left-hand side is well defined. However, if we look deeper into Eq \ref{chap4_1}, we note that moving from Eq \ref{chap4_1} to Eq \ref{chap4_2} induces an infinitesimal error because the summation in the right-hand side of Eq \ref{chap4_1} actually starts with $k=1$, not $k=0$. In other words, the lower bound of the summation is always \emph{incremented} each time the process is applied. Thus, if, on the other hand, one were to repeatedly apply Eq \ref{chap4_1}, the process will always terminate even for arbitrary large $n$. In infinitesimal calculus, the error is infinitesimal, hence treated as if it were \emph{exactly} zero, and the process is repeated indefinitely leading eventually to incorrect results! We will show that divergence of Taylor series expansion as well as the Euler-Maclaurin summation formula both follow from a very similar process. In fact, the statement of Corollary \ref{corollarySumEMSum} already reveals such phenomenon for the Euler-Maclaurin summation formula.

If infinitesimal errors can cause incorrect results, is there a way to correct them? Luckily, the answer is yes and this can often be done using any of the analytic summability methods such as Abel summation method, Euler summation method, Cesaro means, the Mittag-Leffler summability method, and Lindel\"of summation method. In fact, we will also introduce a new summability method, denoted $\Xi$, that is particularly suited for oscillating sums. Most importantly, we will assume a formal general method of summation, denoted $\mathfrak{T}$, that generalizes the definition of infinite sums and is not tied to any particular method of computation, and show how many popular summability methods fall within $\mathfrak{T}$.

As will be shown repeatedly in the sequel, the formal method of summation $\mathfrak{T}$ is not merely an artificial construct. Using this generalized definition of infinite sums, we will derive the analog of the Euler-Maclaurin summation formula for oscillating sums in Chapter \ref{Chapter5}, as promised earlier, which will allow us to perform many remarkable deeds with ease. In the final section, where we extend Summability Calculus to arbitrary discrete functions, this generalized definition of sums will be used to prove a stronger statement than the Shannon-Nyquist Sampling Theorem, and answers an important question in the calculus of finite differences. Indeed, it will be shown how summability of divergent series is a  fundamental concept in Summability Calculus.

\section{The $\mathfrak{T}$ Definition of Infinite Sums}\label{Section4dot1}

In this section, we present the generalized definition of infinite sums $\mathfrak{T}$. We will establish conditions for which popular summability methods can be used to \emph{compute} the $\mathfrak{T}$ value of divergent series, but it is important to keep in mind that the definition $\mathfrak{T}$ is not restricted to any particular method of computation. \\ \hrule
\begin{definit}{\textbf{(The $\mathfrak{T}$ Definition of Sums)}} \label{TDefinition}
Given an infinite sum $\sum_{k=0}^\infty a_k$, define:
\begin{equation}\label{TdefEq}
h(z)=\sum_{k=0}^\infty a_k\,z^k
\end{equation} 
If $h(z)$ is analytic in the domain $z\in[0,1]$, then the $\mathfrak{T}$ value of the sum $\sum_{k=0}^\infty a_k$ is defined by $h(1)$. 
\end{definit} \hrule\vspace{12pt} 

Note that the formal method of summation $\mathfrak{T}$ not only requires that the function $h(z)$ be analytic at $z=0$ and $z=1$, but it also requires that the function be analytic throughout the line segment $[0, 1]$. In other words, we require that the function $h(z)$ be regular at the origin and that the point $z=1$ falls within the \emph{Mittag-Leffler} star of the function $h(z)$\footnote{The Mittag-Leffler star of a function $f(x)$ is the set of all points $z\in\mathbb{C}$ such that the line segment [0, z] does not pass through a singularity point of $f(x)$ \cite{HardyDiverg}}. Therefore, the $\mathfrak{T}$ value of infinite sums can be immediately computed using \emph{rigidity} of analytic functions. That is, it can be computed by evaluating higher derivatives $f^{(r)}(z_j)$ at a finite sequence of values $\{z_0, z_1, \ldots, z_{n-1}, z_n\}$ such that $z_j$ falls within the radius of convergence of the Taylor series expansion at $z_{j-1}$, and $z_0=0$ and $z_n=1$. In addition, we have the following alternative method of computing the $\mathfrak{T}$ value of infinite sums. \\ \hrule 
\begin{lemma}\label{computingT}
If an infinite sum $\sum_{k=0}^\infty a_k$ has a value in $\mathfrak{T}$ given by $V\in\mathbb{C}$, then $V$ can be computed using either the Mittag-Leffler summability method or the Lindel\"of summation method.
\end{lemma} 
\begin{proof} 
Because the two methods correctly sum any Taylor series expansion in the Mittag-Leffler star \cite{HardyDiverg}. 
\end{proof}
\hrule\vspace{12pt} 

Note that Lemma \ref{computingT} implies that $\mathfrak{T}$ is consistent with but weaker than both the Mittag-Leffler summability method and the Lindel\"of summation method. However, it is important to decouple the generalized definition of infinite sums $\mathfrak{T}$ from the method of computation because we are not really concerned about any peculiar properties of summability methods. Now, we have the following immediate results. \\ \hrule

\begin{lemma}\label{Tproperties}
The $\mathfrak{T}$ definition of infinite sums is regular, linear, and stable. Moreover, if two infinite sums $\sum_{k=0}^\infty a_k$ and $\sum_{k=0}^\infty b_k$ are both defined in $\mathfrak{T}$ by values $V_1,V_2 \in\mathbb{C}$, then the $\mathfrak{T}$ value of their Cauchy product is given by the $V_1\times V_2$. 
\end{lemma} 
\begin{proof}
Follows immediately by Definition \ref{TDefinition}. 
\end{proof}
\hrule\vspace{12pt} 

The motivation behind the use of $\mathfrak{T}$ as a generalized definition of infinite sums is threefold. First, it is inline with Euler's original reasoning that a mathematical expression should assume the value of the algebraic process that led to it, which is the point of view discussed earlier in Chapter \ref{ChaptIntro}. In principle, we opt to select a \emph{natural} assignment of divergent series by using $\mathfrak{T}$ as a definition. 

Second, we know by \emph{uniqueness of Taylor series expansion of analytic functions} that $\mathfrak{T}$ is well-defined. In fact, if there exists two different functions $f_1(x)$ and $f_2(x)$ such that the Taylor series expansion of both functions produces the exact same series at two different points $x_1$ and $x_2$ respectively, i.e. $a_k = \frac{f_1^{(k)}(0)}{k!}x_1^k = \frac{f_2^{(k)}(0)}{k!}x_2^k$, then we can always make the transformation $h_!(z)=f_1(x_1\,z)$ and $h_2(z)=f_2(x_2\,z)$. Doing this yields $h_1(1)=f_1(x_1)$ and $h_2(1)=f_2(x_2)$. However, the Taylor series expansion of both $h_1(z)$ and $h_2(z)$ are identical so the two functions must, in turn, be identical. In particular, we must have $h_1(1)=h_2(1)$ and consequently $f_1(x_1)=f_2(x_2)$. Therefore, the $\mathfrak{T}$ definition is indeed consistent.  

However, in order to guarantee consistency, the conditions of $\mathfrak{T}$ must be valid. For example, the Taylor series expansion of the function $e^{-\frac{1}{x^2}}$ is the zero function, which might suggest that the sum $(0+0+0+\ldots)$ can be assigned the value $e^{-1}$! However, because the function is not analytic at the origin, this assignment is invalid. Similarly, ambiguity can arise if a discontinuity exists in the line segment $[0, 1]$. For example, the infinite sum $1+\frac{2^2}{2}+\frac{2^3}{3} +\ldots$ arises out of the series expansion of $\log(1+x)$ at $x=-2$. However, both $i\pi$ and $-i\pi$ are equally valid substitutes for $\log{(-1)}$ and there is no reason to prefer one over the other. Here, we know that the conditions of $\mathfrak{T}$ are not satisfied so the infinite sum is simply undefined in $\mathfrak{T}$. These phenomena never occur if the conditions of Definition \ref{TDefinition} are satisfied. 

Third, because the $\mathfrak{T}$ definition of infinite sums is regular, linear, and stable, then arithmetic operations remain consistent even if such operations ultimately lead to convergent sums. It might seem implausible at first, but it is even occasionally easier to compute convergent sums by expressing them using divergent sums, where the latter are defined under $\mathfrak{T}$! For example, we will derive an exact value of the convergent sum $\sum_{k=1}^\infty (-1)^k (H_k+\lambda-\log{k})$ in Chapter \ref{Chapter5}, by deriving analytic values of each of the divergent sums $\sum_{k=1}^\infty (-1)^k H_k$, $\sum_{k=1}^\infty (-1)^k \lambda$, and $\sum_{k=1}^\infty (-1)^k \log{k}$.

The generalized definition of sums yields an immediate generalized definition of limits that is interpreted for the space of \emph{sequences} $S=(s_0, s_1, s_2, \ldots)$. Here, whereas $\mathfrak{T}$ is a generalized definition that assigns values to \emph{series}, we can reformulate it such that it is interpreted as a method of assigning limits to infinite \emph{sequences}. Clearly, both are essentially equivalent since assigning a value to an infinite sum $\sum_{k=0}^\infty a_k$ can be thought of as assigning a limit to the infinite sequence $(\sum_{k=0}^0 a_k, \sum_{k=0}^1 a_k, \sum_{k=0}^2 a_k, \ldots)$. \\ \hrule
\begin{definit}{\textbf{(The $\mathfrak{T}$ Sequence Limit)}} \label{TSeqLimit}
Given an infinite sequence $S=(s_0, s_1, s_2, \ldots)$, then the $\mathfrak{T}$ sequence limit is defined by the $\mathfrak{T}$-value of the infinite sum $s_0 + \sum_{k=0}^\infty \Delta s_k$. 
\end{definit}
\hrule\vspace{12pt} 
The $\mathfrak{T}$ definition of sequence limits appears at first to be trivial but it does indeed have interesting consequences. To be more specific, if a finite sum $\sum_{k=a}^n g(k)$ can be expressed in closed-form $f_G(n)$, then the $\mathfrak{T}$ value of the infinite sum $\sum_{k=a}^\infty g(k)$ can be deduced systematically by taking the $\mathfrak{T}$ limit of $f_G(n)$ as $n\to\infty$. The most important special case of $\mathfrak{T}$ sequence limits is stated in the following lemma. \\ \hrule

\begin{lemma}\label{TlimitZero}
If the value of an infinite sum $\sum_{k=0}^\infty a_k$ exists in $\mathfrak{T}$, then the $\mathfrak{T}$ sequence limit of  $\{a_n\}$ must be zero. 
\end{lemma} 
\begin{proof}
By stability of the $\mathfrak{T}$ definition of sums, we have: 
\begin{equation} 
\sum_{k=0}^\infty a_k = a_0 + \sum_{k=1}^\infty a_k
\end{equation}
Because the generalized definition $\mathfrak{T}$ is linear, we have: 
\begin{equation} 
a_0 + \sum_{k=1}^\infty a_k - \sum_{k=0}^\infty a_k = 0 \quad\quad \Rightarrow a_0 + \sum_{k=0}^\infty \Delta a_k = 0 
\end{equation}
However, last equation is exactly the $\mathfrak{T}$ definition of sequence limits so the statement of the lemma follows. 
\end{proof} \hrule\vspace{12pt} 

Lemma \ref{TlimitZero} presents an interesting generalization to the case of the ordinary convergent sums. It will be shown repeatedly throughout the sequel that results are always consistent with such generalized definition of limits. For example, if we look into the Grandi series $\sum_{k=0}^\infty (-1)^k$, we note that $\sum_{k=0}^n (-1)^k = \frac{1}{2} + (-1)^n \frac{1}{2}$. However, because the $\mathfrak{T}$-value of $\sum_{k=0}^\infty (-1)^k$ is $\frac{1}{2}$, since it arises out of the Taylor series expansion of the function $f(x)= (1+x)^{-1}$ at $x=1$, then the $\Xi$  sequence limit of $(-1)^n$ must be zero by Lemma \ref{TlimitZero}. Thus: 
\begin{equation}\label{zerolimitSeq2_1}
\lim_{n\to\infty} \sum_{k=0}^n (-1)^k = \frac{1}{2} + \frac{1}{2} \lim_{n\to\infty} (-1)^n = \frac{1}{2}
\end{equation}

Of course, having a generalized definition of infinite sums and limits is crucial, but coming up with a method of computing such generalized values is also equally important. Luckily, there exists many methods of computing the generalized $\mathfrak{T}$ value of infinite sums, two of which were already listed in Lemma \ref{computingT}. In addition, the $\mathfrak{T}$ value of infinite sums can often be deduced by algebraic construction.  For example, if we return to the Grandi series discussed earlier, we know by stability and linearity that:
\begin{equation}\label{Eq42Discussion_1}
\sum_{k=0}^\infty (-1)^k = 1-\sum_{k=0}^\infty (-1)^k \Rightarrow \sum_{k=0}^\infty (-1)^k=\frac{1}{2}
\end{equation}
Also, performing the Cauchy product of the Grandi series, we have the well-known result:
\begin{equation}\label{Eq42Discussion_2}
\sum_{k=0}^\infty (-1)^k \times \sum_{k=0}^\infty (-1)^k= \sum_{k=1}^\infty (-1)^{k+1} k =\frac{1}{2}\times\frac{1}{2} = \frac{1}{4}
\end{equation}
Eq \ref{Eq42Discussion_2} could equivalently be deduced using the series expansion of $(1+x)^{-2}$ and setting $x=1$ in its Taylor series because the function $(1+x)^{-2}$ is analytic over the line segment $[0, 1]$. Even more, we could have performed the algebraic manipulations in Eq \ref{Eq42Discussion_3} to deduce the same value. It is important to note that all of these different approaches yield consistent results as expected.

\begin{equation}\label{Eq42Discussion_3}
\sum_{k=0}^\infty (-1)^k - \sum_{k=1}^\infty (-1)^{k+1} k= \sum_{k=1}^\infty (-1)^{k+1} k
\end{equation}

However, caution should always be exercised when using algebraic derivation. For instance, one can easily show by stability and linearity that $\sum_{r=0}^\infty r^k=\frac{1}{1-r}$. However, this latter expression is invalid if $r\in[1, \infty)$, because the ray $[1, \infty)$ is not in the Mittag-Leffler star of the function $\frac{1}{1-r}$.

In the following lemma, we show that if an infinite sum is summable to some value $V\in\mathbb{C}$ using Abel summability method or any of the N\"orlund means, including Cesaro summation methods, then the $\mathfrak{T}$ value of the infinite sum is indeed $V$. \\ \hrule 

\begin{lemma}\label{LemmaAbelNorlund}
The $\mathfrak{T}$ definition of infinite sums is consistent with, but more powerful than,  Abel summation and all N\"orlund means including Cesaro summation methods\footnote{The N\"orlund means is a method of assigning limits to infinite sequences. Here, suppose $p_j$ is a sequence of positive terms that satisfies $\frac{p_n}{\sum_{k=0}^n p_k}\to 0$. Then, the N\"orlund mean of a sequence $(s_0, s_1, \ldots)$ is given by $\lim_{n\to\infty} \frac{p_n s_0 + p_{n-1} s_1+\dotsm +p_0 s_n}{\sum_{k=0}^n p_k}$. The limit of an infinite sequence $(s_0, s_1, \ldots)$ is defined by its N\"orlund mean. Therefore, the N\"orlund mean interprets the limit $\lim_{n\to\infty} s_n$ using an \emph{averaging} method.}.
\end{lemma}
\begin{proof}
Suppose we have an infinite sum $\sum_{k=0}^\infty a_k$, define $h(z)=\sum_{k=0}^\infty a_k z^k$. If $\lim_{z\to 1^-} h(z)$ exists, then $h(z)$ is analytic in the domain $[0, 1)$. By Abel's theorem on power series, we have  $\lim_{z\to 1^-} h(z) = h(1)$, if $h(z)$ is analytic at $z=1$. However, assigning a value to the infinite sum $\sum_{k=0}^\infty a_k$ using such limiting expression is the definition of Abel summability method. If such limit exists, i.e. if a series is Abel summable, then its value coincides with the $\mathfrak{T}$ definition of the infinite sum $\sum_{k=0}^\infty a_k$.

Finally, because Abel summation is consistent with and more powerful than all N\"orlund means \cite{HardyDiverg}, the statement of the lemma follows for all N\"orlund means as well including Cesaro summability method.
\end{proof} \hrule\vspace{12pt} 
Lemma \ref{LemmaAbelNorlund} shows that many methods can be used to compute the $\mathfrak{T}$ value of infinite sums. For example, one of the simplest of all N\"orlund means is to look into the average of $q$ consecutive partial sums and see if such $q$-average converges. For example, if $q=1$, then we have ordinary summation. If, on the other hand, $q=2$, we look into the average of each consecutive two partial sums,  a summability method that was proposed by Hutton in 1812 \cite{HardyDiverg}. Applying this to the Grandi series yields the answer $\sum_{k=0}^\infty (-1)^k = \frac{1}{2}$, as expected. Moreover, the case of $q=\infty$, when properly interpreted, leads to Cesaro summation method. Most importantly, Lemma \ref{LemmaAbelNorlund} states that if any of those methods yields a value $V\in\mathbb{C}$, then $V$ is the $\mathfrak{T}$ value of the infinite sum. 

In addition to evaluating sums, we can often conclude that a series is not defined under $\mathfrak{T}$ as the following lemma shows. \\ \hrule 

\begin{lemma}\label{totalRegularity} 
If a divergent sum $\sum_{k=0}^\infty a_k$, where $a_k\in\mathbb{R}$, is defined  in $\mathfrak{T}$ by a value $V\in\mathbb{C}$, then the sum must be oscillating. In other words, the $\mathfrak{T}$ definition of infinite sums is \emph{totally regular}\footnote{Consult the classic book "Divergent Series" by Hardy \cite{HardyDiverg} for a definition of this term.}. 
\end{lemma} 
\begin{proof}
By Lemma \ref{computingT}, if $\sum_{k=0}^\infty a_k$ is defined under $\mathfrak{T}$ to some value $V\in\mathbb{C}$, then it must be summable using Lindel\"of summability method to the same value V. However, Lindel\"of summability method assigns to an infinite sum the value $\lim_{\delta\to 0} \sum_{k=0}^\infty k^{-\delta k} a_k$. Because $\lim_{\delta\to 0} k^{-\delta k} = 1$ and $0<k^{-\delta k}\le 1$, Lindel\"of summability method cannot sum a divergent series that solely consists of  non-negative terms so the statement of the lemma follows.  
\end{proof} \hrule\vspace{12pt}  

Lemma \ref{totalRegularity} does not necessarily limit applicability of the generalized definition $\mathfrak{T}$. For example, whereas $\zeta(s)$ is not directly defined under $\mathfrak{T}$ for $s<1$, it can be \emph{defined} using Eq \ref{zeta_1}, which is known to hold in the half-plane $\mathcal{R}(s)<-1$. Therefore, the definition in Eq \ref{zeta_1} is valid by analytic continuation. 
\begin{equation}\label{zeta_1}
\sum_{k=1}^\infty k^s = \frac{1}{1-2^{1+s}} \sum_{k=1}^\infty (-1)^{k+1} k^s\\ \text{(by definition)}
\end{equation} 

Now, the right-hand side is well-defined under $\mathfrak{T}$. In fact, we can immediately derive a closed-form expression for it. First, define $N_s(x)$ as: 
\begin{equation}\label{zeta_2}
N_s(x) = \sum_{k=1}^\infty (-1)^k\, k^s\, x^k
\end{equation}
Then, $N_s(x)$ satisfies the following recurrence relationship: 
\begin{equation}\label{zeta_3}
N_s(x) = x\, N_{s-1}'(x) 
\end{equation}
Differentiating both sides of Eq \ref{zeta_3} with respect to $x$ yields:
\begin{equation}\label{zeta_4}
N_s'(x) =  N_{s-1}'(x) + x\,  N_{s-1}^{(2)}(x)
\end{equation}
Therefore, upon repeated application of Eq \ref{zeta_3} and Eq \ref{zeta_4}, we have: 
\begin{equation}\label{zeta_5}
N_s(x) =  \sum_{k=0}^s S(s, k)\, x^k\, N_0^{(k)}(x)
\end{equation}
Here, $S(s, k)$ are Stirling Numbers of the Second Kind. A simple proof for Eq \ref{zeta_5} follows by induction upon using the characteristic property of Stirling Numbers of the Second Kind given by $S(n+1, k) = k\,S(n,k)+S(n, k-1)$. Now, knowing that the series expansion of $N_0(x)$ is given by Eq \ref{zeta_6}, and upon using Definition \ref{TDefinition}, we arrive at Eq \ref{zeta_7}. 
\begin{equation}\label{zeta_6}
N_0(x) = -\frac{x}{1+x} = -\frac{1}{2} -\frac{x-1}{4}+\frac{(x-1)^2}{8} -\frac{(x-1)^3}{16} + \dotsm
\end{equation}
\begin{equation}\label{zeta_7}
N_s(1) = \sum_{k=1}^\infty (-1)^k k^s = \sum_{k=0}^s (-1)^k \, S(s,k)\,\frac{k!}{2^{k+1}}
\end{equation}
Note here that we have appealed to the original statement of $\mathfrak{T}$ to define such divergent sums. Therefore, we have: 
\begin{equation}\label{zeta_8}
\sum_{k=1}^\infty  k^s = - \frac{1}{1-2^{1+s}} \sum_{k=0}^s (-1)^{k+1} \, S(s,k)\,\frac{k!}{2^{k+1}}
\end{equation}
Using the identity $B_{1+s} = -(1+s)\zeta(-s)$, where $B_k$ is the \emph{k}th Bernoulli number, we have the following closed-form expression for Bernoulli numbers:
\begin{equation}\label{zeta_9}
B_s = \frac{s}{1-2^{s}} \sum_{k=0}^{s-1} \frac{1}{2^{k+1}} \sum_{j=0}^k (-1)^{j} \binom{k}{j} j^{s-1}
\end{equation}

Now, we show that Euler summation method almost always agrees with the $\mathfrak{T}$ definition of infinite sums. \\ \hrule
 
\begin{lemma}{\textbf{(The Euler Sum) }}\label{EulerSum} 
Given an infinite sum $\sum_{k=a}^\infty (-1)^k\,g(k)$, define the Euler sum by $ (-1)^a \sum_{k=a}^\infty \frac{(-1)^k}{2^{k+1}}  \Delta^k g(k)$. \\
Then, the Euler sum is equivalent to the $\mathfrak{T}$ definition of infinite sums if the following three conditions hold:
\begin{enumerate}
\item The sum exists in $\mathfrak{T}$.
\item The infinite sum is Euler summable
\item The $\mathfrak{T}$ value of the infinite sum $\sum_{k=a}^\infty (-1)^k \Delta^m g(k)$ is $o(2^m)$
\end{enumerate} 
\end{lemma}
\begin{proof}
If the infinite sum $\sum_{k=a}^\infty g(k)$ exists in $\mathfrak{T}$, then it must satisfy the linearity and stability properties. Thus, we always have:
\begin{equation}\label{EulerSum_1} 
\sum_{k=a}^\infty (-1)^k\,g(k) = \sum_{k=a}^\infty (-1)^k \, g(k+1) - \sum_{k=a}^\infty (-1)^k \Delta g(k)
\end{equation}
Because the $\mathfrak{T}$ definition is stable, we have: 
\begin{equation}\label{EulerSum_2} 
\sum_{k=a}^\infty (-1)^k\,g(k) = (-1)^a g(a) -  \sum_{k=a}^\infty (-1)^k \, g(k) - \sum_{k=a}^\infty (-1)^k \Delta g(k)
\end{equation}
Rearranging the terms yields: 
\begin{equation}\label{EulerSum_3} 
\sum_{k=a}^\infty (-1)^k\,g(k) = \frac{(-1)^a g(a)}{2} - \frac{1}{2} \sum_{k=a}^\infty (-1)^k \Delta g(k)
\end{equation}
Here, both divergent sums are interpreted using the definition $\mathfrak{T}$. By repeated application of Eq \ref{EulerSum_3}, we have:
\begin{equation}\label{EulerSum_4} 
\sum_{k=a}^\infty (-1)^k\,g(k) = (-1)^a \sum_{p=0}^m \frac{(-1)^p}{2^{p+1}} \Delta^p g + \frac{(-1)^{m+1}}{2^{m+1}} \sum_{k=a}^\infty (-1)^k \Delta^{m+1}g(k)
\end{equation}
Last equation gives the error term of using Euler summability method, which is given by $\frac{(-1)^{m+1}}{2^{m+1}} \sum_{k=a}^\infty (-1)^k \Delta^{m+1}g(k)$. Therefore, if the error term goes to zero, i.e. the $3^{rd}$ condition in the lemma is satisfied, and the sum $(-1)^a \sum_{p=0}^\infty \frac{(-1)^p}{2^{p+1}} \Delta^p g$ is convergent, i.e. the 2\textsuperscript{nd} condition is satisfied, then the statement of the lemma follows.
\end{proof} \hrule\vspace{12pt} 

Lemma \ref{EulerSum} states that Euler summation method can indeed be often used to compute divergent sums instead of applying the definition $\mathfrak{T}$ directly. For instance, if we return to the Riemann zeta function, then the conditions of Lemma \ref{EulerSum} hold so we always have: 
\begin{equation}\label{zetaEulerSum_1} 
\sum_{k=1}^\infty (-1)^{k+1}\, k^s = \sum_{k=0}^\infty \frac{1}{2^{k+1}} \,\sum_{j=0}^k (-1)^{j} \binom{k}{j} (j+1)^s
\end{equation} 

Using Eq \ref{zetaEulerSum_1}, we deduce a \emph{globally} valid expression for the Riemann zeta function given in Eq \ref{zetaEulerSum_2}. The expression in Eq \ref{zetaEulerSum_2} was proved by Helmut Hasse in 1930 and rediscovered by Sondow in \cite{Sondow94, MathWorldRiemannZeta}. We will provide an alternative globally convergent expression for the Riemann zeta function later in Chapter \ref{Chapter7}.
\begin{equation}\label{zetaEulerSum_2} 
\zeta(s) = \frac{1}{1-2^{1-s}} \sum_{k=0}^\infty \frac{1}{2^{k+1}} \,\sum_{j=0}^k (-1)^{j} \binom{k}{j} \frac{1}{(j+1)^s}
\end{equation}

Consequently, we indeed have a rich collection of methods to compute the $\mathfrak{T}$ value of infinite sums. Unfortunately, most of these methods have serious limitations. For instance, all N\"orlund means are weak, i.e. are limited in their ability to sum divergent series, and Abel summability method alone outperforms them all.  However,  Abel summability method itself is usually insufficient. For instance, it cannot sum any oscillating series that grows exponentially large such as the Taylor series expansion of the logarithmic function and the sum of Bernoulli numbers. On the other hand, the Mittag-Leffler summability method and Lindel\"of summation method are both powerful but they are computationally demanding. They are often extremely slow in convergence and require high-precision arithmetic. 

To circumvent such limitations, we will introduce a new summability methods $\Xi$ that is especially suited to oscillating sums in the following section. The method $\Xi$ is easy to implement in practice and can converge reasonably fast. In addition, it is also powerful and can sum a large number of divergent series. In fact, all of the examples presented in the sequel are $\Xi$ summable. 

\section{The Summability Method $\Xi$}
In this section, we present a new summability method $\Xi$ for Taylor series expansions that is consistent with the $\mathfrak{T}$ definition of infinite sums. The method $\Xi$ is simple to implement in practice. The error term in the approximation method given by $\Xi$ is $O(\frac{1}{n})$; thus it is unfortunately slowly converging. Nevertheless, we will present in Chapter \ref{Chapter5} a method of accelerating convergence to an arbitrary speed. We start with an argument by construction as to why $\Xi$ is a natural method of summing divergent series.

\subsection{A Statement of Summability} 
\begin{claim} \label{claim411}
Given a point $x_0$ and a function $f(x)$ that is $n$-times differentiable throughout the interval $[x_0 , x)$, then the function $\hat f_n(x)$ defined by Eq \ref{EqClaim411}  is a valid approximation to the function $f(x)$. In addition, the approximation error typically improves as $n\to\infty$.
\begin{equation}\label{EqClaim411}
\hat f_n(x) = \sum_{j=0}^n \chi_n(j) \frac{f^{(j)}(x_0)}{j!} (x-x_0)^j, \quad \quad \\ \text{where } \chi_n(j) = \frac{n!}{n^j (n-j)!} = \prod_{k=1}^j \big(1-\frac{k-1}{n}\big)
\end{equation}
\end{claim}
\begin{proof} 
Before we present an argument for Claim \ref{claim411}, three points are worth mentioning. First, Claim \ref{claim411} states that the error term in the summability method will typically vanish for most functions as $n\to\infty$, which is clearly an improvement over classical Taylor series approximation. However, this implies that the approximation in Eq \ref{EqClaim411} may or may not converge to the function $f(x)$ but it does typically converge as will be shown in the proof. Later in Claim \ref{claim412}, we will look into the topic of convergence. We will also deduce a simple asymptotic expression for the error term later in Theorem \ref{theorem441}  

Second, if the Taylor series converges, then $n$ can be taken to be literally infinite, in which case we arrive at the famous Taylor series expansion for analytic functions. This follows in the latter case because convergence becomes essentially independent of $n$, as long as $n$ is sufficiently large. 

Third, because $\chi_n(j) \to 1$ as $n\to\infty$ for  fixed  $j$, the first terms of the polynomial expansion given in Eq \ref{EqClaim411} approach terms of the classical Taylor series expansion as $n$ increases. Such phenomenon is desirable because it implies \emph{regularity}\footnote{As stated earlier, a summability method $\Xi$ is called \emph{regular} if it is consistent with ordinary convergent sums.} of the summability method  but it raises two important observations. First, the partial sum of the polynomial expansion given in Claim \ref{claim411} may increase beyond numerical accuracy if $x$ lies outside the radius of convergence of the Taylor series before the sum converges back again as will be depicted later. Hence, while validity of the new approximation method is established from a mathematical point of view, i.e. in an ideal Platonic sense, the method itself may not be immediately useful in \emph{implementing} numerical computation in some cases due solely to numerical precision limitations.  Also, if the original function $f(x)$ was itself a polynomial with degree $n$, it is true that the \emph{n}th degree polynomial expansion given in Claim \ref{claim411} of $f(x)$ is not equal to $f(x)$, a fact that seems at first sight to be paradoxical, but because as $\chi_n(j) \to 1$ as $n\to\infty$, the polynomial expansion does indeed converge to $f(x)$ as $n$ increases, which is the original assertion of Claim \ref{claim411}.

The argument for Claim \ref{claim411} rests on its derivation by construction. As a starting point, suppose we have a function $f(x)$ that is \emph{n}-time differentiable at a particular point $x_0$ and let $\Delta x$ be a chosen step size such that we wish to approximate the value of the function $f(x_0 + n \Delta x)$ for an arbitrary value of $n$ using \emph{solely} local information about the behavior of the function $f(x)$ at $x_0$. Using the notation $f_j = f(x_0 + j \Delta x)$ and $f_j^{(k)} = f^{(k)}(x_0+j \Delta x)$, we know immediately from the very basic definition of differentiation that the following recurrence holds:
\begin{equation}\label{Claim411_1} 
f_j^{(k)} \approx f_{j-1}^{(k)} + f_{j-1}^{(k+1)} \Delta x
\end{equation}
Furthermore, we know that the approximation error strictly improves as we decrease the step size $\Delta x$. Thus, we obtain the following approximations that can be held with arbitrary accuracy for sufficiently small step size $\Delta x$: 
\begin{equation}\label{Claim411_2} 
f_1 = f_0 + f_0^{(1)} \Delta x
\end{equation} 
\begin{equation}\label{Claim411_3} 
f_3 = f_1 + f_1^{(1)} \Delta x = f_0 + 2 f_0^{(1)} \Delta x + f_0^{(2)} \Delta x^2 
\end{equation}
In Eq \ref{Claim411_3}, we have substituted Eq \ref{Claim411_2} and approximated the derivative $f_1^{(1)}$ using the recurrence in Eq \ref{Claim411_1}. In general, we can show by induction that the following general formula holds:
\begin{equation}\label{Claim411_4}
f_j = \sum_{k=0}^ j \binom{j}{k} f_0^{(k)} \Delta x^k
\end{equation}

To prove that Eq \ref{Claim411_4} holds, we first note that a base case is established for $j=1$ in Eq \ref{Claim411_2}. Suppose that it holds for $j < m$, we will show that such inductive hypothesis implies that Eq \ref{Claim411_4} also holds for $j = m$. First, we note that if Eq \ref{Claim411_4}  holds for $j<m$, then we have:
\begin{equation}\label{Claim411_5}
f_m = f_{m-1} + f_{m-1}^{(1)} \Delta x = \sum_{k=0}^{m-1} \binom{m-1}{k} f_0^{(k)} \Delta x^k + \sum_{k=0}^{m-1} \binom{m-1}{k} f_0^{(k+1)} \Delta x^{k+1}
\end{equation}
In Eq \ref{Claim411_5}, the second substitution for $f_{m-1}^{(1)}$ follows from the same inductive hypothesis because $f_{m-1}^{(1)}$ is simply another function that can be approximated using the same inductive hypothesis. Eq \ref{Claim411_5} can, in turn, be rewritten as:
\begin{equation}\label{Claim411_6}
f_m = \sum_{k=0}^{m-1} \Big[\binom{m-1}{k} + \binom{m-1}{k-1}\Big] f_0^{(k)} \Delta x^k +f_0^{(m)} \Delta x^m
\end{equation}

Upon using the well-known recurrence relation for binomial coefficients, i.e. Pascal's rule, we obtain Eq \ref{Claim411_4}, which is exactly what is needed in order to complete the proof by induction of that equation. In addition, Eq \ref{Claim411_4}  can be rewritten as given in Eq \ref{Claim411_7}  below using the substitution $x=x_0 + n \Delta x$.
\begin{equation}\label{Claim411_7}
f(x) \approx \hat f_n(x) = \sum_{j=0}^n \binom{n}{j} f^{(j)}(x_0) \frac{(x-x_0)^j}{n^j}
\end{equation}
However, the binomial coefficient can be expanded, which yields Eq \ref{EqClaim411}. 

Since $x=x_0 + n \Delta x$, it follows that increasing $n$ while holding $x$ fixed is equivalent to choosing a smaller step size $\Delta x$. Because the entire proof is based solely on the linear approximation recurrence given in Eq \ref{Claim411_1}, the basic definition of differentiation implies that the expected approximation error in Eq \ref{Claim411_1} will typically vanish as $n\to\infty$. This follows because the summability method essentially simulates walking over the function in the domain $[x_0, x]$, which is similar to Euler's approximation method. Furthermore, it follows by construction of the proof, that $\hat f_n(x)$ is indeed a valid approximation to the function $f(x)$ as stated in the claim. 
\end{proof} \hrule\vspace{12pt} 

In simple terms, Claim \ref{claim411} states that the \emph{n}th-order approximation of a function $f(x)$ is given by:
\small
\begin{align*}
f(x)\approx &f(x_0) + \frac{f'(x_0)}{1!}(x-x_0) + \big(1-\frac{1}{n}\big)\frac{f^{(2)}(x_0)}{2!}(x-x_0)^2\\&+ \big(1-\frac{1}{n}\big) \big(1-\frac{2}{n}\big)\frac{f^{(3)}(x_0)}{3!}(x-x_0)^3 \dotsm
\end{align*}
\normalsize
Here, the \emph{n}th-order approximation typically converges to the function $f(x)$ as $n$ tends to infinity if $f(x)$ is analytic in the domain $[x_0, x]$. Comparing the above expression with the classical Taylor series expansion illustrates how infinitesimal errors arise and how they accumulate as the process is repeated infinitely many times. In particular, the use of Taylor series expansion corresponds to the \emph{erroneous} assertion that $\lim_{n\to\infty} \sum_{j=0}^n \chi_n(j) a_j = \sum_{j=0}^\infty \lim_{n\to\infty} \chi_n(j) a_j$. 

Curiously, the summability method can be stated succinctly using the Calculus of Finite Differences. Here, if we let $D_x$ be the \emph{differential operator}, then Claim \ref{claim411} states that:
\begin{equation}\label{summMethodCFD}
f(x+h)=\lim_{n\to\infty} \Big(1+h \frac{D_x}{n}\Big)^n 
\end{equation}
On other other hand, Taylor series expansion states that $f(x+h)=e^{hD_x}$. Of course, both expressions are equivalent for ordinary numbers but they differ for general symbolic operators as illustrated here.

Looking into the proof of Claim \ref{claim411}, it is clear that the $\Xi$ summability method is inline with the $\mathfrak{T}$ definition of infinite sums because it is a method that is built to sum Taylor series expansions in the first place. To show that $\Xi$ is linear and regular, we start with the following lemma.  \\ \hrule 

\begin{lemma}\label{lemmaSumXin}
For all $n\ge 0$, the following equality holds: 
\begin{equation}\label{lemmaSumXinEq}
\sum_{j=0}^n j\, \chi_n(j) = n
\end{equation} 
\end{lemma}
\begin{proof}
By definition: 
\begin{align*}
\sum_{j=0}^n j\, \chi_n(j) &= n! \sum_{j=0}^n \frac{j}{(n-j)!\, n^j} = \frac{n!}{n^n} \sum_{j=0}^n \frac{n-j}{j!} n^j \\
&=n! \big(\frac{n}{n^n} \sum_{j=0}^n \frac{n^j}{j!} -\frac{n}{n^n}  \sum_{j=0}^n \frac{n^j}{j!} + \frac{1}{(n-1)!}\big)\\
&=\frac{n!}{(n-1)!} = n
\end{align*}
\end{proof}\hrule\vspace{12pt} 

\begin{lemma}{\textbf{(The $\Xi$ Sequence Limit)}}\label{XiLimitSeq}
Define $S$ to be the sequence of partial sums given by $S=(\sum_{k=0}^0 a_k, \sum_{k=0}^1 a_k, \sum_{k=0}^2 a_k, \ldots)$. Then the $\Xi$ sequence limit of $S$ given by Eq \ref{XiLimitSeqEq} agrees with the $\Xi$ sum of $\sum_{k=0}^\infty a_k$. 
\begin{equation}\label{XiLimitSeqEq}
\lim_{n\to\infty} \frac{p_n(0) s_0 + p_n(1) s_1+\dotsm + p_n(n) s_n}{\sum_{k=0}^n p_n(k)}, \\ \text{ where } p_n(j)=j\,\chi_n(j)
\end{equation} 
\end{lemma}
\begin{proof} 
We will prove the theorem by induction. First, let us denote $c_n(k)$ to be the sequence of terms that satisfies: 
\begin{equation}\label{XiLimitSeq_1} 
\frac{p_n(0) s_0 + p_n(1) s_1+\dotsm + p_n(n) s_n}{\sum_{k=0}^n p_k} = \sum_{k=0}^n c_n(k) a_k
\end{equation}
Here, $s_j = \sum_{k=0}^j a_k$. Our objective is to prove that $c_n(k)=\chi_n(k)$. To prove this by induction, we first note that $\sum_{k=0}^n p_n(k) = n$ as proved in Lemma \ref{lemmaSumXin}, and a base case is already established since $c_n(0)=1=\chi_n(0)$. Now, we note that: 
\begin{equation}\label{XiLimitSeq_2}
c_n(k)=\frac{p_n(k)+p_n(k+1) + \dotsm + p_n(n)}{n}
\end{equation}
For the inductive hypothesis, we assume that $c_n(k)=\chi_n(k)$ for $k<m$. To prove that such inductive hypothesis implies that $c_n(m)=\chi_n(m)$, we note by Eq \ref{XiLimitSeq_2} that the following holds: 
\begin{equation}\label{XiLimitSeq_3}
c_n(m)=c_n(m-1)-\frac{p_n(m-1)}{n} =(1-\frac{m-1}{n}) \chi_n(m-1)=\chi_n(m)
\end{equation}
Here, we have used the inductive hypothesis and the original definition of $\chi_n(m)$ given in Claim \ref{claim411}. Therefore, we indeed have: 
\begin{equation}\label{XiLimitSeq_4}
\frac{p_n(0) s_0 + p_n(1) s_1+\dotsm + p_n(n) s_n}{\sum_{k=0}^n p_n(k)} = \sum_{k=0}^n \chi_n(k) a_k
\end{equation}
The statement of the theorem follows immediately. 
\end{proof} \hrule\vspace{12pt} 

Lemma \ref{XiLimitSeq} shows that the summability method $\Xi$ is indeed \emph{an averaging method} but it is important to note that it is different from the N\"orland means because the sequence $p_n(k)$ are not independent of the number of terms $n$. The reformulation of the summability method $\Xi$ given by Lemma \ref{XiLimitSeq} allows us to prove the following important statement. \\ \hrule

\begin{corollary}\label{XiregStabLinear}
The summability method $\Xi$ given by Claim \ref{claim411} is linear and regular. 
\end{corollary}
\begin{proof} 
To show that the summability method is linear, we note that: 
\begin{equation}\label{XiregStabLinear_1} 
\sum_{j=0}^n \chi_n \big(\alpha\, a_j + \beta\, b_j\big) = \alpha \sum_{j=0}^n \chi_n(j)\, a_j + \beta \sum_{j=0}^n \chi_n(j)\, b_j
\end{equation}
Because linearity holds for all $n$, it holds at the limit $n\to\infty$. To show regularity, we use the Toeplitz-Schur Theorem. The Toeplitz-Schur Theorem states that any matrix summability method $t_n=\sum_{k=0}^\infty A_{n,k} s_k \,(n=0,1,2,\ldots)$ in which $\lim_{n\to\infty} s_n$ is defined by $\lim_{n\to\infty} t_n$ is regular \emph{if and only if} the following three conditions hold \cite{HardyDiverg}: 
\begin{enumerate}
\item $\sum_{k=0}^\infty |A_{n, k}| < H$, for all $n$ and some constant $H$ that is independent of $n$. 
\item $\lim_{n\to\infty} A_{n, k} = 0$ for each $k$. 
\item $\lim_{n\to\infty} \sum_{k=0}^\infty A_{n,k} = 1$. 
\end{enumerate}

Using Theorem \ref{XiLimitSeq}, we see that the summability method $\Xi$ in Claim \ref{claim411}  is a matrix summability method characterized by $A_{n,k} = \frac{k}{n}\,\chi_n(k)$. Because $A_{n,k}\ge 0$ and $\sum_{k=0}^\infty A_{n,k} = \sum_{k=0}^n A_{n,k} = 1$, both conditions 1 and 3 are immediately satisfied. In addition, for each fixed $k$,  $\lim_{n\to\infty} A_{n, k} = 0$, thus condition 2 is also satisfied. Therefore, the summability method $\Xi$ is regular. In fact, because  $A_{n,k}\ge 0$, $\Xi$ is totally regular. 
\end{proof} \hrule\vspace{12pt} 

\subsection{Convergence}\label{Section4dot2_2}
In this section, we show that the summability method $\Xi$ correctly sums Taylor series expansions if the function $f(x)$ is analytic in the domain $[x_0, x)$ and if the Taylor series expansion is not \lq\lq too rapidly'' diverging. \\ \hrule 

\begin{claim} \label{claim412}
Let $a_k=\frac{f^{(k)}(x_0)}{k!}(x-x_0)^k$. If a function $f(x)$ is analytic in an open disc around each point in the domain $[x_0, x]$ and $a_n=o(\kappa^n)$, where $\kappa \approx 3.5911$ is the solution to the equation $\log{\kappa} - \frac{1}{\kappa} = 1$, then we almost always have $\lim_{n\to\infty} \{f(x) - \hat f_n(x)\} = 0$, where $\hat f_n(x)$ is as defined by Claim \ref{claim411}.
\end{claim}
\begin{proof}
First, it is easy to come up with examples where violating any of the two conditions makes $\Xi$ fail in correctly summing Taylor series expansions. For instance, the Taylor series of $e^{-\frac{1}{x^2}}$ is not regular at the origin so the first condition is violated. Here, because the Taylor series expansion is exactly the zero function, applying $\Xi$ will not yield correct results. Also, applying $\Xi$ to the alternating geometric series $1-x+x^2-x^3+\ldots$ shows that $\Xi$ converges if $0\le x\le \kappa$ and diverges if $x>\kappa$. 

Second, let us define the error terms $E_j^k$ using Eq \ref{theorem411_1}. That is, $E_j^k$ is the error in our approximation of the \emph{k}th derivative of $f(x)$ at the point $x_j=x_0+ j \Delta x$. 
\begin{equation}\label{theorem411_1} 
E_j^k = f_j^k - \hat f_j^k 
\end{equation}
We also note in the argument of Claim \ref{claim411} that $E_n^k$ is alternatively given by: 
\begin{equation}\label{theorem411_2} 
E_n^k = f_n^k - \hat f_{n-1}^k - \hat f_{n-1}^{k+1} \Delta x = f_n^k -f_{n-1}^k - f_{n-1}^{k+1} \Delta x + E_{n-1}^k + E_{n-1}^{k+1} \Delta x
\end{equation}
Applying Eq \ref{theorem411_2} repeatedly $m$ times yields: 
\begin{equation}\label{theorem411_3}
f_n - \hat f_n = E_n^0 = f_n - \sum_{k=0}^m \binom{m}{k} f_{n-m}^k \Delta x^k + \sum_{k=0}^m \binom{m}{k} E_{n-m}^k \Delta x^k 
\end{equation}

Eq \ref{theorem411_3} is intuitive. Here, if we denote  $z=x_0 + (n-m) \Delta x$, then Eq \ref{theorem411_3} basically states that computing $f(x)$ using the Taylor series expansion around $x_0$ is equivalent to computing $f(x)$ using the Taylor series expansion around an intermediate point $z\in[x_0, x]$, except for the fact that our estimates of $f^{(r)}(z)$ are themselves approximated. For example, if $z=x_0$, then $E_0^k=0$ and $m=n$, and we recover the original definition $E_n^0 = f_n-\hat f_n$. 

Now, if we fix our choice of the intermediate point $z$, chosen such that $x$ is \emph{inside the analytic disc of $f$ around the point $z$}, then $m\to\infty$ as $n\to\infty$. Most importantly, $f_n-\sum_{k=0}^m \binom{m}{k} f_{n-m}^k \Delta x^k$ goes to zero as $n\to\infty$, because the summation converges to the classical Taylor series expansion around $z$, and because $x$ is within the analytic disc of $f$ around $z$ by assumption. Consequently, we have the following expression:
\begin{equation}\label{theorem411_4} 
f_n-\hat f_n \sim \sum_{k=0}^m \binom{m}{k} E_{n-m}^k \Delta x^k 
\end{equation}

As stated earlier, Eq \ref{theorem411_4} has a simple intuitive interpretation. If $f(x)$ is computed using the summability method in Claim \ref{claim411} where series expansion is taken around $z$ instead of $x_0$, and if $x$ is within the analytic disc of $f$ around $z$, then overall error asymptotically comes \emph{solely} from errors in our approximation of higher order derivatives $f^{(k)}(z)$. Now, if $z$ itself was within the analytic disc of $f$ around $x_0$, then its higher order derivatives $f^{(k)}(z)$ can be approximated with arbitrary precision using higher order derivatives $f^{(k)}(x_0)$. The process for doing this is exactly what the construction of the summability method in Claim \ref{claim411} performs. Therefore, the error in the approximation of $f(x)$ using the summability method in Claim \ref{claim411} similarly goes to zero in such case. Loosely speaking, we will say that local behavior of $f$ at $x_0$ is used to compute local behavior of $f$ at $z$. In turn, local behavior of $f$ at $z$ is used to compute local behavior of $f$ at $x$. The summability method $\Xi$ becomes, therefore, a method of propagating information about $f$ from $x_0$ to $x$ through the intermediary point $z$.

Such reasoning holds true even if $z$ is not within the analytic disc of $f$ around $x_0$. More specifically, if there exists a sequence of \emph{collinear} points $(x_0, z_1, z_2, \ldots, x)$, such that each point $z_k$ is within the analytic disc of $f$ around $z_{k-1}$, then the summability method propagates information about the local behavior of $f$ around $x_0$ to the neighborhood of $x$ by passing it through the analytic region $(x_0, z_1, z_2, \ldots, x)$, and the overall error term in our approximation of $f(x)$ goes to zero.
 
Therefore, we expect $\Xi$ to correctly sum Taylor series expansions that are not too rapidly diverging. To assess the impact of rapid divergence, we look into the alternating geometric series $\sum_{k=0}^\infty (-z)^k$. The motivation behind the use of the geometric series as a reference is the \emph{Borel-Okada} principle, which essentially states that the geometric series is sufficient in quantifying conditions of convergences of analytic summability methods \cite{tauberianTheory}. In particular, if a summability method can correctly sum the geometric series in its star, then it can sum any Taylor series in its star as well \cite{HardyDiverg}. 

However, for the geometric series, we have the following identity: 
\begin{equation}\label{geoSeries_1}
\sum_{k=0}^n \chi_n(k)\, z^k = - \frac{1}{z} \int_0^n \big(1-\frac{t}{n}\big)^n\, e^{\frac{t}{z}}\,dt + \chi_n(n)\,z^n\,e^{\frac{n}{z}}
\end{equation}

If $\mathcal{R}(z)<0$, the first term in the right hand side converges to $(1-z)^{-1}$. So, in order for the summability method $\Xi$ to correctly sum the alternating geometric series $\sum_{k=0}^\infty (-z)^k$ when $n\to\infty$, we must have $\chi_n(n)\,z^n\,e^{\frac{n}{z}} \to 0$. However, this happens only if $|z|<\kappa$, where $\kappa$ is as defined in the claim. Therefore, we indeed have that some too rapidly diverging Taylor series expansions are not $\Xi$ summable, but most divergent series of interest in this manuscript are $\Xi$ summable as will be shown throughout the sequel.
\end{proof} \hrule\vspace{12pt} 

Before we analyze the asymptotic behavior of the error term of the summability method $\Xi$, we will illustrate how the approximation method works. Our first example is the logarithmic function $f(x)=\log{x}$ expanded around $x_0=1$ whose Taylor radius of convergence is $|x|<1$. In Figure \ref{Figure4_2_1}, the behavior of the approximation $\hat f_n(x)$  is depicted for $n=100$ and $x=3$, which falls outside the radius of convergence of the Taylor expansion of the logarithmic function. The horizontal axis in the plot corresponds to the index $0\leq k \leq n$ while the vertical axis corresponds to the the absolute value of the associated partial sum of the summability method, given by $|\sum_{j=1}^k (-1)^{j+1}\chi_n(j) \frac{(x-x_0)^j}{j}|$. As shown in the figure, the partial sum initially grows very rapidly because the Taylor series diverges; in fact, it grows up to an order of $10^6$. Nevertheless, it converges eventually to the value of the original function, where the relative approximation error in this particular example is less than 0.2\%, which is accurate to a remarkable degree given the fact that we used a polynomial of degree 100 for a point $x=3$ that lies outside the radius of convergence of the corresponding Taylor's series.
\begin{figure} 
\centering
\includegraphics[scale=0.4]{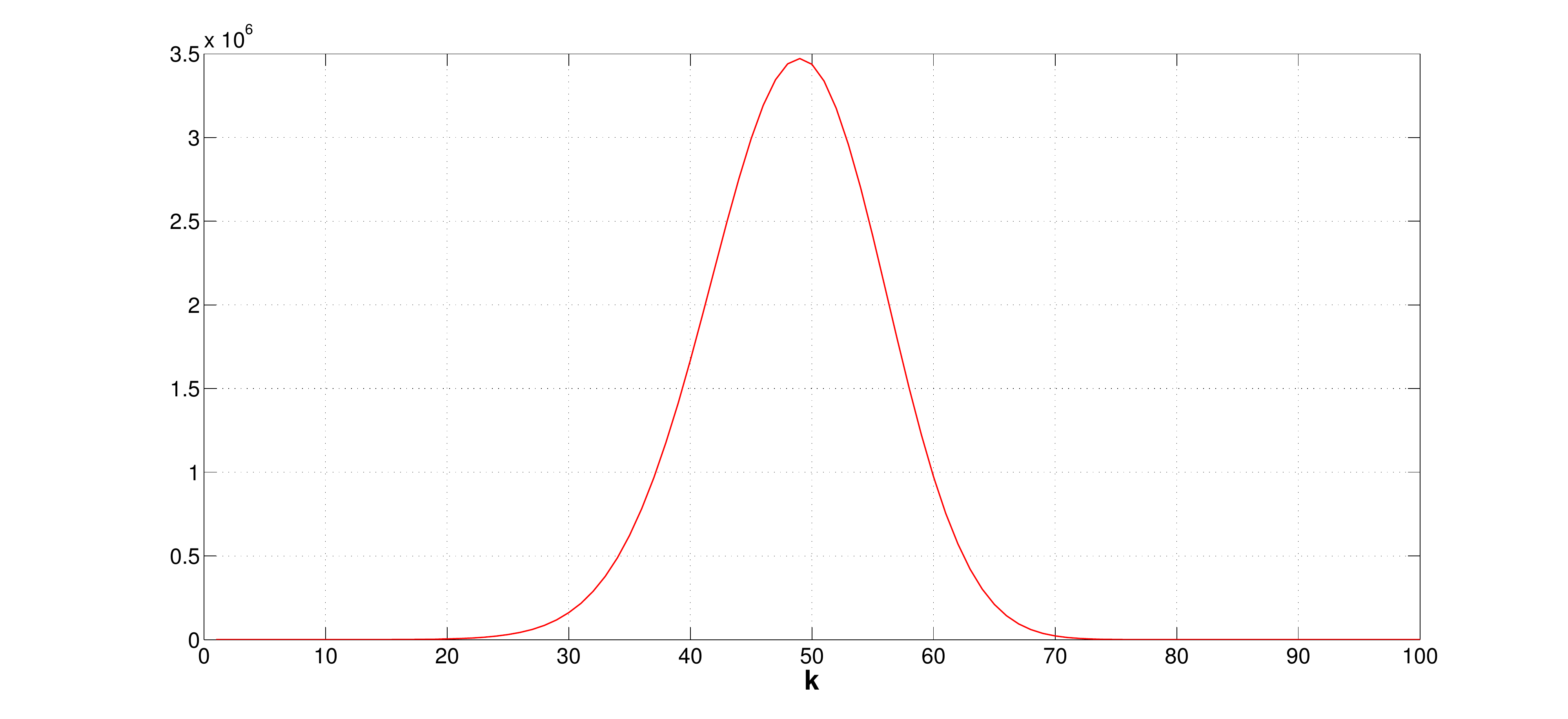}
\caption[The summability method $\Xi$ applied to the logarithmic function]{$|\sum_{j=1}^k (-1)^{j+1}\chi_n(j) \frac{(x-x_0)^j}{j}|$ plotted against $k$ for $n=100$ and $x-x_0=2$}
\label{Figure4_2_1}
\end{figure}

Because a complex-valued function $f(x)$ satisfies the Cauchy-Riemann equations, which essentially state that higher derivatives $f^{(m)}(x)$ are the same regardless of the direction of the step size $\Delta x$ over the complex plane $\mathbb{C}$, Claim \ref{claim411} and Claim \ref{claim412} also hold for complex-valued functions. For example, if $f(x)=(1-x)^{-1}$, and if we use its well-known series expansion around $x_0=1$, then the approximation $\hat f_n(x)$ converges to $f(x)$ as $n\to\infty$ in a larger disc than the radius of convergence of the classical Taylor series. For example, if $x=-1+i$, where $i=\sqrt{-1}$, we have $\hat f_{30}(x)=0.3941+i\, 0.2010$ and $\hat f_{50}(x)=0.3965 + i\, 0.2006$. By contrast, we know that the exact value is given by $f(x)=0.4+i \,0.2$. Clearly, $\hat f_n(x)$ indeed approaches $f(x)$ as $n\to\infty$ even though $x=-1+j$ falls outside the radius of convergence of the classical Taylor series.

Moreover, we can immediately deduce using Theorem \ref{theorem231} that the following series expansion holds for the derivative of the log-factorial function $\varpi'(n)$:
\begin{equation}\label{Eq42_1}
\varpi'(n)=\psi(n+1)=-\lambda + \sum_{k=2}^\infty (-1)^k \zeta_k n^{k-1}
\end{equation}
Plugging $n=1$ into both sides of the last equation yields:
\begin{equation}\label{Eq42_2}
\sum_{k=2}^\infty (-1)^k \zeta_k =1
\end{equation}
However, the left-hand sum diverges. Using $\Xi$, on the other hand, the left-hand sum is summable and its sum indeed converges to 1. Moreover, if we take the Grandi series in Eq \ref{GrandiSeries}, which arises in the series expansion of $(1+x)^{-1}$ using $x=1$, the summability method $\Xi$ assigns a value of $\frac{1}{2}$, which is what we would expect by plugging $x=1$ into the original function.
\begin{equation}\label{GrandiSeries}
\sum_{k=0}^\infty (-1)^k =\frac{1}{2}
\end{equation}
Becauase the summability method $\Xi$ is consistent under linear combinations, we can combine both results in Eq \ref{Eq42_2} and Eq \ref{GrandiSeries} to deduce that Eq \ref{Eq42_4} holds; a result that was proved earlier in Eq \ref{Eq244_2}. 
\begin{equation}\label{Eq42_4}
\sum_{k=2}^\infty (-1)^k (\zeta_k-1)=\frac{1}{2}
\end{equation}

Interestingly, the summability method $\Xi$ is not restricted to Taylor series expansions. For example, let us consider the divergent sum $\sum_{k=1}^\infty \sin{(k\theta)}$. As Hardy showed, the divergent sum should be naturally assigned the value given in Eq \ref{Eq42Discussion_4}. Simple numerical checks using the summability method of Claim \ref{claim411} confirm that this is indeed correct. 
\begin{equation}\label{Eq42Discussion_4}
\sum_{k=1}^\infty \sin{(k\theta)} = \frac{1}{2}\cot{\frac{\theta}{2}}
\end{equation}

Moreover, one illustrative example of regularity of the summability method is the series $\sum_{k=0}^\infty 1/k!$, which is equal to the natural logarithmic base $e$. Here, applying the summability method yields:
\begin{equation}\label{Eq42Discussion_5}
\sum_{k=0}^\infty \frac{1}{k!} = \lim_{n\to\infty} \Big\{\sum_{k=0}^n \chi_n(k) \frac{1}{k!}\Big\} = \lim_{n\to\infty} \sum_{k=0}^n \binom{n}{k}\, n^{-k}
\end{equation}
Using the Binomial Theorem, we recover the original definition of $e$: 
\begin{equation}\label{Eq42Discussion_6}
e = \lim_{n\to\infty} (1+\frac{1}{n})^n 
\end{equation}

Of course, the summability method itself cannot be immediately used to deduce closed-form expressions of divergent sums. For example, the sum $\sum_{k=0}^\infty (-1)^k \log{k!}$ is summable to -0.1129 through direct numerical computation but the summability method itself cannot be used to determine an analytic expression of the latter constant. Later, we will provide an answer to this question by introducing the analog of the Euler-Maclaurin summation formula for alternating sums, which shows that   $\sum_{k=0}^\infty (-1)^k \log{k!}$ is summable to $\frac{1}{4} \log{\frac{2}{\pi}}$. 

Finally, we conclude this section with two examples: one in summing Taylor series expansion of simple finite sums, and one in summing the Euler-Maclaurin summation formula. Many other examples will be presented throughout rest of the paper. To mitigate the impact of numerical precision errors, calculations in the following examples were carried out using the \emph{Multi-Precision Math} package (\textbf{mpmath}). \textbf{mpmath} is a Python library for arbitrary-precision floating-point arithmetic and includes built-in support for many special functions such as the Factorial and the Riemann zeta function \cite{mpmath}. 

For our first example, let us return to the sum of square roots function $f(x)= \sum_{k=1}^x \sqrt{k}$ whose series expansion was derived earlier in Eq \ref{Example_261_5}, also rewritten here in Eq \ref{Eq43Sqrt_1} below, and employ the summability method of Claim \ref{claim411} to compute its values outside the radius of convergence. Here, we note that we can alternatively compute the function for $0\leq x < 1$ using its original Taylor series and employ the recursive property $f(x)=\sqrt{x}+f(x-1)$ to compute the function for different values of $x$ that fall outside the radius of convergence. However, our intention here is to verify validity/correctness of the series expansion in Eq \ref{Eq43Sqrt_1}. Figure \ref{figureSSqrt} displays results for $x=2$ and $x=3$. Here, by definition, we expect the series expansion to converge to $1+\sqrt{2}$ if $x=2$ and to $1+\sqrt{2}+\sqrt{3}$ if $x=3$. As shown in the figure, this is indeed the case. 
\begin{equation}\label{Eq43Sqrt_1} 
f_G(x) = \sum_{k=1}^x \sqrt{k} = -\frac{\zeta_{1/2}}{2} x + \sum_{m=2}^\infty (-1)^m \frac{(2m-3)!}{4^{m-1} \,m!\, (m-2)!} \zeta_{m-1/2} \,x^m 
\end{equation}

\begin{figure} [h] 
\centering
\includegraphics[scale=0.4]{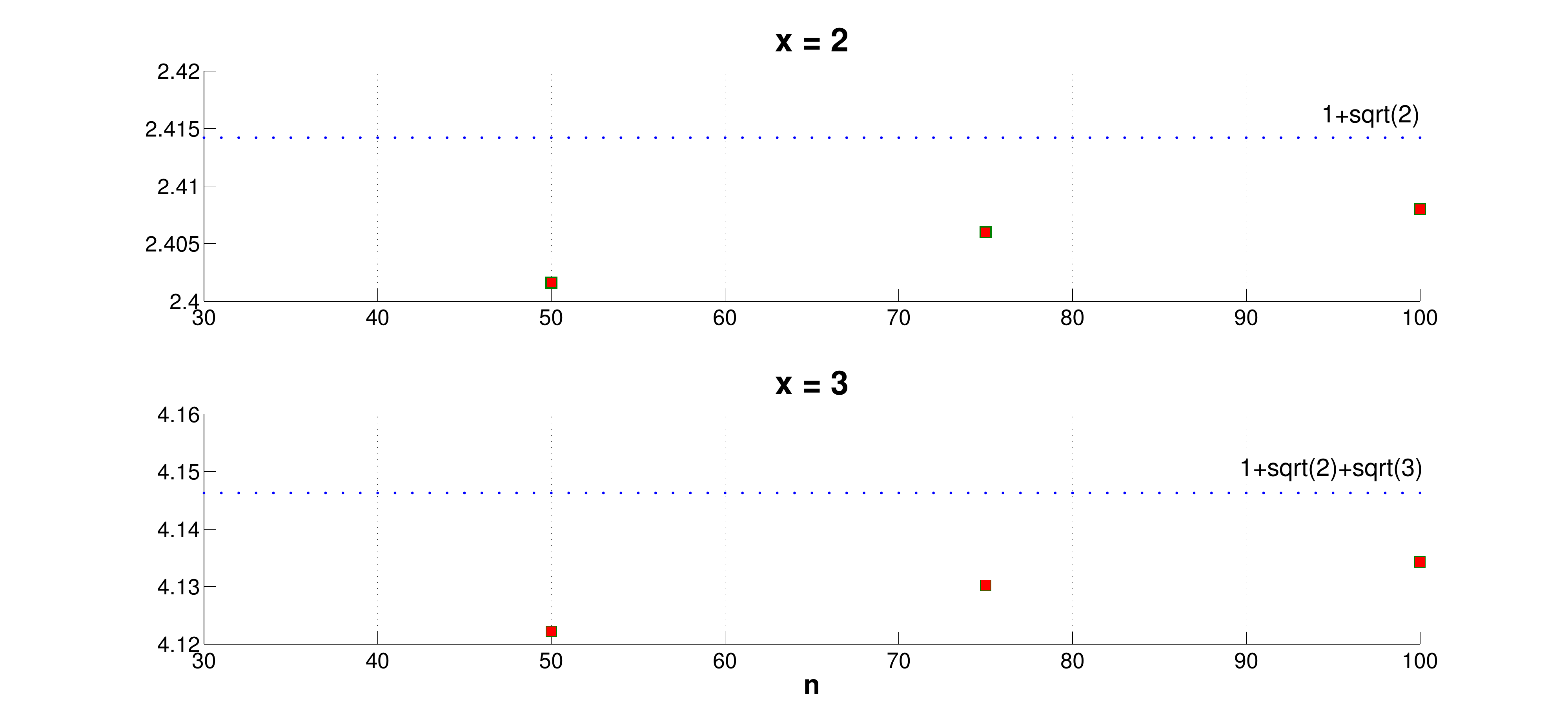}
\caption[The summability method $\Xi$ applied to the sum of square roots function]{$\Xi$ applied to sum of square roots function. Values of $\Xi$ are highlighted in red.}
\label{figureSSqrt}
\end{figure}

For our last example, we look into using the summability method $\Xi$ in evaluating the Euler-Maclaurin summation formula when it diverges. Here, if we apply Corollary \ref{corollary251_1} to the Harmonic sum, we obtain $\zeta_2 = \sum_{k=0}^\infty B_k$, where $B_k$ are again Bernoulli numbers. Clearly, the sum diverges but using $\Xi$ with small values of $n$ shows that the sum of all Bernoulli numbers can be rightfully assigned the value $\zeta_2$. 

To see that the sum of all Bernoulli numbers should be naturally assigned the value $\zeta_2$, it is instructive to alternatively consider the function $f(x)$ whose series expansion around $x=0$ is given in Eq \ref{Eq43Bern_1}. Using the asymptotic expression for Bernoulli numbers given by Euler \cite{Apostol1973, SandiferBernoulli}, it is straightforward to see that the series expansion in Eq \ref{Eq43Bern_1} converges for $x=0$ only and diverges elsewhere. Thus, without aid of summability theory, we cannot examine behavior of this function. Using the summability method $\Xi$, however, the function can be approximated despite the fact that it has been defined using a series expansion that diverges everywhere except at $x=0$.  
\begin{equation}\label{Eq43Bern_1}
f(x) = \sum_{k=0}^\infty B_k \, x^k 
\end{equation}

To show the behavior of the function in Eq \ref{Eq43Bern_1}, we use $\Xi$ for different values of $n$, which is shown in Table \ref{Tab43Bern}. Clearly, the function is well defined even though we only know its series expansion that is almost everywhere divergent! Using the Euler-Maclaurin Summation Formula, it can be shown that the series expansion in Eq \ref{Eq43Bern_1} corresponds to the function $f(x)= \frac{1}{x}\varpi^{(2)}(\frac{1}{x}) +x$ , where $\varpi$ is the log-factorial function \footnote{In MATLAB, the function is given by the command: \ttfamily{1/x*psi(1,1/x+1)+x}}. Thus, we can contrast the estimated values in Table \ref{Tab43Bern} with the actual values of the original function, which is given in the last row of the table. Clearly, they both agree as expected \footnote{Note that because the series expansion is divergent for all $x>0$, the condition of being analytic in $[x_0, x)$ is no longer valid so the summability method is not expected to work. However, it works reasonably well in this particular example. As a result, even for cases in which the limit $n\to\infty$ of $\Xi$ does not exist, we can still find approximate values of divergent sums using small values of $n$.}. Note that we could use this last result to define the divergent sum of all Bernoulli numbers by $\sum_{k=0}^\infty B_k = \zeta_2$, a result that was first stated by Euler \cite{Peng2002}. The interpretation of this last equation follows from the results in the previous section. In particular, we could say that whenever we perform algebraic operations such as additions and multiplications on the infinite sum of all Bernoulli numbers and arrive eventually at a convergent series, then we could compute the exact value of that convergent series by substituting for the sum of all Bernoulli numbers its value $\zeta_2$.

\definecolor{Gray}{gray}{0.9}
\definecolor{lightGreen}{rgb}{0.85,1,0.85}

\begin{table}[h] 
\centering\begin{tabular} { |l|c|c|c|c|c|c|c|c|c|}
\hline 
\rowcolor{Gray}
 \textbf{x=}  &  \textbf{-1} & \textbf{-0.7} & \textbf{-0.5} & \textbf{-0.2} & \textbf{ 0 } & \textbf{0.2} & \textbf{0.5} & \textbf{0.7} & \textbf{1}\\[7pt]
\hline
\textbf{n=20}  &  0.6407 & 0.7227 & 0.7882 & 0.9063 & 1.000 & 1.1063 & 1.2882 & 1.4227 & 1.6407\\[7pt]
\hline
\textbf{n=25}  &  0.6415 & 0.7233 & 0.7885 & 0.9064 & 1.000 & 1.1064 & 1.2885 & 1.4233 & 1.6415\\[7pt]
\hline
\textbf{n=30}  &  0.6425 & 0.7236 & 0.7888 & 0.9064 & 1.000 & 1.1064 & 1.2888 & 1.4236 & 1.6425\\[7pt]
\hline
\rowcolor{lightGreen}
\textbf{Exact}  &  0.6449 & 0.7255 & 0.7899 & 0.9066 & 1.000 & 1.1066 & 1.2899 & 1.4255 & 1.6449\\[7pt]
\hline
\end{tabular}
\caption {Computed values of $\sum_{k=0}^\infty B_k\,x^k$ using $\Xi$ summability method for different choices of $n$. Clearly, the function is well defined even though its series expansion is everywhere divergent except at $x=0$.}\label{Tab43Bern}
\end{table}

\subsection{Analysis of the Error Term} \label{Section4dot4}
In this section, we analyze the asymptotic behavior of the error term in the summability method given in Claim \ref{claim411} for analytic functions. We will show why the error term in the summability method is $O(\frac{1}{n})$ and provide a simple expression for its asymptotic value. Next, we provide an interpretation of the asymptotic expression of the error term and present a few numerical examples that illustrate its accuracy. To do this, we start with the following lemma.\\ \hrule

\begin{lemma}\label{lemma441} 
If  $[x_0, x]$ falls in the interior of the region of convergence of $\Xi$ for $f(x)$, then: 
\begin{equation}\label{lemma441Eq} 
\lim_{n\to\infty} \Big\{\sum_{k=1}^n \chi_n(k) \Big[f(x)-\sum_{j=0}^{k-1} \frac{f^{(j)}(x_0)}{j!} (x-x_0)^j\Big] \Big\}= (x-x_0)\,f'(x)
\end{equation}
\end{lemma}
\begin{proof}
It is straightforward to see that if the limit in Eq \ref{lemma441Eq} exists, then Eq \ref{lemma441Eq} must hold. To see this, define $g$ to be a functional of $f(x)$ such that: 
\begin{equation}\label{lemma441_1} 
g_f=\lim_{n\to\infty} \Big\{\sum_{k=1}^n \chi_n(k) \Big[f(x)-\sum_{j=0}^{k-1} \frac{f^{(j)}(x_0)}{j!} (x-x_0)^j\Big] \Big\}
\end{equation}
Now, we differentiate both sides with respect to $x$, which yields: 
\begin{align*}\label{lemma441_2} 
\frac{d}{dx} g_f&=\lim_{n\to\infty} \Big\{\sum_{k=1}^n \chi_n(k) \Big[ \frac{f^{(k)}(x_0)}{(k-1)!} (x-x_0)^{(k-1)}+f'(x)-\sum_{j=0}^{k-1} \frac{f^{(j+1)}(x_0)}{j!} (x-x_0)^j\Big] \Big\}\\
&=g_{f'} + \lim_{n\to\infty} \sum_{k=1}^n \chi_n(k) \frac{f^{(k)}(x_0)}{(k-1)!} (x-x_0)^{(k-1)}\\
&=g_{f'} + f'(x)
\end{align*}
Therefore, we have: 
\begin{equation}\label{lemma441_3}
\frac{d}{dx} g_f = f'(x)+ g_{f'}
\end{equation} 
Since $g_f(x_0)=0$, the solution is given by: 
\begin{equation}\label{lemma441_4}
g_f(x) = (x-x_0)\,f'(x)
\end{equation} 
The above derivation assumes that the limit exists. To prove that Eq \ref{lemma441Eq} holds under stated conditions, we note that $f(x)-\sum_{j=0}^{k-1} \frac{f^{(j)}(x_0)}{j!} (x-x_0)^j$ is the error term of the Taylor series expansion, which is \emph{exactly} given by: 
\begin{equation}\label{lemma441_5}
f(x)-\sum_{j=0}^{k-1} \frac{f^{(j)}(x_0)}{j!} (x-x_0)^j = \int_{x_0}^x \frac{f^{(k)}(t)}{(k-1)!} (x-t)^{k-1} \,dt
\end{equation} 
Upon using last expression, we have that: 
\small
\begin{equation}\label{lemma441_6}
g_f = \lim_{n\to\infty} \sum_{k=1}^{n} \chi_n(k) \int_{x_0}^x \frac{f^{(k)}(t)}{(k-1)!} (x-t)^{k-1} \,dt = \int_{x_0}^x  \lim_{n\to\infty}\sum_{k=1}^{n} \chi_n(k) \frac{f^{(k)}(t)}{(k-1)!} (x-t)^{k-1} \,dt
\end{equation} 
\normalsize
Here, the justification for exchanging sums with integrals is because $t\in [x_0, x]$ and the line segment $[x_0, x]$ is in the \emph{interior} of the star-like region of convergence of $\Xi$ for the function $f(x)$, so we have uniform convergence for $t\in [x_0, x]$. Since we have: 
\begin{equation}\label{lemma441_7}
 \lim_{n\to\infty}\sum_{k=1}^{n} \chi_n(k) \frac{f^{(k)}(t)}{(k-1)!} (x-t)^{k-1} = f'(x), \\ \text{ for all } t\in[x_0, x]
\end{equation} 
Plugging Eq \ref{lemma441_7} into Eq \ref{lemma441_6} yields the desired result.
\end{proof} \hrule

\begin{theorem}\label{theorem441} 
If a function $f(x)$ is analytic in an open disc around each point in the domain $[x_0, x]$ and $[x_0, x]$ falls in the interior of the region of convergence of $\Xi$ for $f(x)$, then the error term of the summability method is asymptotically given by:
\begin{equation}\label{theorem441ErrorTerm}
f(x)-\hat f_n(x) \sim \frac{f^{(2)}(x) (x-x_0)^2}{2n}
\end{equation}
\end{theorem}
\begin{proof}
Because the function $f(x)$ is analytic around each point in the domain $[x_0, x]$, define $\epsilon>0$ to be the distance between the line segment $[x_0, x]$ to the nearest singularity point of $f(x)$. In the argument of Claim \ref{claim411}, we have used the following linear approximations, where $f_j^{(k)} = f^{(k)}(x_0+j\Delta x)$: 
\begin{equation}\label{XiErrTherm_1}
f_j^{(k)} \approx f_{j-1}^{(k)} + f_{j-1}^{(k+1)} \Delta x  
\end{equation}

Because the distance from $x_0+j\Delta x$ to the nearest singularity point of $f(x)$ is at least $\epsilon$, then selecting $\Delta x<\epsilon$ or equivalently $n>\frac{x-x_0}{\epsilon}$ implies that the error term of the linear approximation is \emph{exactly} given by Eq \ref{XiErrTherm_2}. This follows from the classical result in complex analysis that a analytic function is equal to its Taylor series representation within its radius of convergence, where the radius of convergence is, at least, equal to the distance to the nearest singularity. 
\begin{equation}\label{XiErrTherm_2}
E_j^k = f_j^{(k)} - f_{j-1}^{(k)} - f_{j-1}^{(k+1)} \Delta x = \sum_{m=2}^\infty \frac{f_{j-1}^{(k+m)}}{m!} \Delta x^m
\end{equation}

Because higher derivatives at $x_0$ are computed exactly, we have $E_0^k =0$ . Now the linear approximation method was applied recursively in the argument of Claim  \ref{claim411}. Visually speaking, such repeated process mimics the expansion of a binary pyramid, depicted in Figure \ref{Fig441Pyramid}, whose nodes $(j, k)$ correspond to the linear approximations given by Eq \ref{XiErrTherm_1} and the two children of each node $(j, k)$ are given by $(j-1, k)$ and $(j-1, k+1)$ as stated in the equation. It follows, therefore, that the number of times a linear approximation is used is equal to the number of paths from root to the respective node in the binary pyramid, where root is $(n, 0)$. It is a well-known result that the number of such paths is given by $\binom{n-j}{k}$.
\begin{figure} [h] 
\centering
\includegraphics{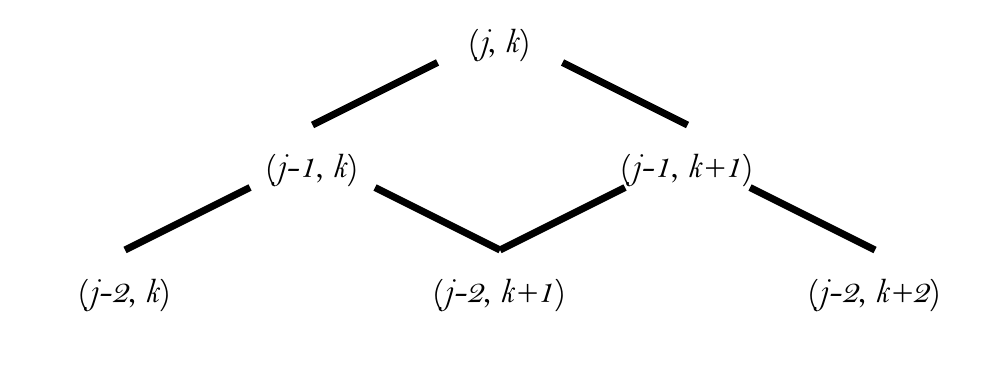}
\caption[A depiction of the recursive proof of the summability method $\Xi$]{A depiction of the recursive proof of the summability method in Claim \ref{claim411}}
\label{Fig441Pyramid}
\end{figure}
Consequently, the error term of the approximation method given in Claim \ref{claim411} is \emph{exactly} given by:
\begin{equation}\label{XiErrTherm_3}
f(x)-\hat f_n(x) = \sum_{j=1}^n \binom{n-j}{0} E_j^0 +  \sum_{j=1}^{n-1} \binom{n-j}{1} E_j^1 \Delta x+  \sum_{j=1}^{n-2} \binom{n-j}{2} E_j^2 \Delta x^2 \dotsm
\end{equation}
Define $e_k$ as a weighted average of the errors $E_j^k$ that is given by: 
\begin{equation}\label{XiErrTherm_4}
e_k = \frac{\sum_{j=1}^{n-k} \binom{n-j}{k} E_j^k}{\sum_{j=1}^{n-k} \binom{n-j}{k}} 
\end{equation}
Then, we have:
\begin{equation}\label{XiErrTherm_5}
f(x)-\hat f_n(x) = e_0 \sum_{j=1}^n \binom{n-j}{0} +  e_1 \sum_{j=1}^{n-1} \binom{n-j}{1} \Delta x+  e_2 \sum_{j=1}^{n-2} \binom{n-j}{2} \Delta x^2 \dotsm
\end{equation}
Now, we make use of the identity $\sum_{j=1}^{n-k} \binom{n-j}{k} = \binom{n}{k+1}$, and substitute $\Delta x = \frac{x-x_0}{n}$, which yields: 
\begin{equation}\label{XiErrTherm_6}
f(x)-\hat f_n(x) = \frac{n}{x-x_0} \sum_{k=1}^n \chi_n(k) \frac{e_{k-1}}{k!} (x-x_0)^k
\end{equation}
Interestingly, the right-hand side is of a similar form to the summability method itself in which the function $\chi_n(k)$ appears again. Now, define $e_k^*$ by:
\begin{equation}\label{XiErrTherm_7}
e_k^* = \frac{e_k}{\Delta x^2} =  \sum_{j=1}^{n-k} \frac{\binom{n-j}{k}}{\binom{n}{k+1}} \sum_{m=0}^\infty \frac{f_{j-1}^{(k+m+2)}}{(m+2)!} \Delta x^{m}
\end{equation}
This yields:
\begin{equation}\label{XiErrTherm_8}
f(x)-\hat f_n(x) = \frac{x-x_0}{n} \sum_{k=1}^n \chi_n(k) \frac{e_{k-1}^*}{k!} (x-x_0)^k
\end{equation}

In the last equation, we begin to see what the error term is $O(\frac{1}{n})$. This is because $e_k^*$ asymptotically lies between $\min_{z\in[x_0,x]} f^{(k+2)}(z)$ and $\max_{z\in[x_0,x]} f^{(k+2)}(z)$ when $n\to\infty$ as shown in Eq \ref{XiErrTherm_7}. So, we expect the infinite sum $ \sum_{k=1}^\infty \frac{e_{k-1}^*}{k!} (x-x_0)^k$ to be summable to a value $V\in\mathbb{C}$ using the summability method $\Xi$. 

Now, we examine the weighted average error terms $e_k^*$ at the limit $n\to\infty$. First, we note that the expression for $e_k^*$ given by Eq \ref{XiErrTherm_7} is exact and that the infinite summation converges because $n$ is chosen large enough such that $\Delta x$ is within the radius of convergence of the series by assumption. Therefore, it follows that $e_k^*$ is asymptotically given by Eq \ref{XiErrTherm_9}. Thus, $e_k^*$ is asymptotically a convex combination of $\frac{f_{j}^{(k+2)}}{2}$, iterated over $j$. The discrete probability distribution $p_n(k)=\binom{n-j}{k}/\binom{n}{k+1}$ approaches a probability density function at the limit $n\to\infty$ when $\frac{j}{n}$ is fixed at a constant $z$. Using Stirling's approximation, such probability density is given by Eq \ref{XiErrTherm_10}. For example, if $k=0$, the probability density function is uniform as expected.
\begin{equation}\label{XiErrTherm_9}
e_k^* \sim \frac{1}{2} \sum_{j=1}^{n-k} \frac{\binom{n-j}{k}}{\binom{n}{k+1}} f_{j-1}^{(k+2)} + O(n^{-2})
\end{equation}
\begin{equation}\label{XiErrTherm_10}
\rho(z)=(1+k)(1-z)^k,\\ \text{ where } z=\frac{j}{n}, \\ \text {and } 0\leq z \leq 1
\end{equation}
Using the probability density function in Eq \ref{XiErrTherm_10}, $e_k^*$ is asymptotically given by: 
\begin{equation}\label{XiErrTherm_11}
e_k^* \sim \frac{1+k}{2} \int_0^1 (1-t)^k f^{(k+2)} \big(x_0 + t(x-x_0)\big)\,dt
\end{equation}
Doing integration by parts yields the following recurrence identity:
\begin{equation}\label{XiErrTherm_12}
e_k^* \sim -\frac{1+k}{2(x-x_0)} f^{(k+1)}(x_0) + \frac{1+k}{x-x_0} e_{k-1}^* 
\end{equation}
Also, we have by direct evaluation of the integral in Eq \ref{XiErrTherm_10} when $k=0$:
\begin{equation}\label{XiErrTherm_13}
e_0^* \sim \frac{f'(x)-f'(x_0)}{2(x-x_0)} 
\end{equation}
Using Eq \ref{XiErrTherm_12} and Eq \ref{XiErrTherm_13}, we have:
\begin{equation}\label{XiErrTherm_14}
e_k^* \sim \frac{1}{2}  \frac{(1+k)!}{(x-x_0)^{k+1}} f'(x) - \frac{1}{2} \frac{(1+k)!}{(x-x_0)^{k+1}} \sum_{m=0}^k \frac{f^{(m+1)}(x_0)}{m!} (x-x_0)^m +o(1)
\end{equation}
The summation in the right-hand side of the last equation is itself the \emph{k}th degree Taylor polynomial for $f'(x)$ expanded around $x_0$. Plugging Eq \ref{XiErrTherm_14} into Eq \ref{XiErrTherm_8} yields: 
\begin{equation}\label{XiErrTherm_15} 
f(x)-\hat f_n(x) \sim \frac{x-x_0}{2n} \sum_{k=1}^n \chi_n(k) \big(f'(x)-\sum_{m=0}^{k-1} \frac{f^{(m+1)}(x_0)}{m!} (x-x_0)^m \big) + o(1/n)
\end{equation}
Using Lemma \ref{lemma441}, we arrive at the desired result:
\begin{equation}\label{XiErrTherm_Final} 
f(x)-\hat f_n(x) \sim \frac{f^{(2)}(x)\, (x-x_0)^2}{2n} + o(1/n)
\end{equation}
\end{proof} \hrule\vspace{12pt} 

To test the asymptotic expression given by Theorem \ref{theorem441}, suppose $f(x)=(1+x)^{-1}$ and suppose that we want to evaluate the function at $x=2$ using the summability method $\Xi$ applied to the Taylor expansion around $x_0=0$. The exact error term by direct application of the summability method when $n=40$ is 0.0037. The expression in Theorem \ref{theorem441} predicts a value of 0.0037, which is indeed the exact error term up to 4 decimal places.  For a second example, if $f(x)=\log{(1+x)}$, $x=3$, and $n=30$, then the exact error term is -0.0091 whereas Theorem \ref{theorem441} estimates the error term to be -0.0094.

The asymptotic expression for the error term given in Theorem \ref{theorem441} presents interesting insights as well as important precautions. First, we can immediately know if the summability method approaches the function from above or below depending on whether the function $f(x)$ is concave or convex at $x$. Such conclusion becomes intuitive if we keep in mind that the summability method essentially mimics the process of walking over the function using successive linear approximations. Thus, once we get close to $x$, first order linear approximation overestimates the value of the function if it is concave and underestimates it if convex, which is intuitive given the definition of convexity. Importantly, this implies that for sufficiently large $n$, we will obtain either successive upper bounds or successive lower bounds to the infinite sum as $n$ increases.

Second, in order for the asymptotic expression given by Theorem \ref{theorem441} to be accurate, it is crucial to keep indices correct inside the summation. For example, suppose $f(x)=(1+x)^{-2}$, whose Taylor series expansion is given by $f(x)=\sum_{k=1}^\infty (-1)^{k+1} k \, x^{k-1}$. If we wish to evaluate $f(x)$, then we know that the following holds for any fixed value of $m\ge 0$: 
\begin{equation}\label{PostTher441_1} 
f(x) = \lim_{n\to\infty}\big\{\sum_{k=1}^n (-1)^{k+1} \chi_n(k+m)\, k\, x^{k-1} \big\}
\end{equation}

However, while the definition in Eq \ref{PostTher441_1} is correct for any fixed value of $m\ge 0$, the error term given by Theorem \ref{theorem441} is only accurate when $m=0$. For example, if $n=100$, $x=1$ and $m=0$, the exact error term of the summability method is 0.0019, and the error estimate given by Theorem \ref{theorem441} is also 0.0019. However, if we choose $m=5$, the error term is also small but it is 0.0142 now, which is clearly different. This is because choosing $m=5$ is essentially equivalent to applying the summability method $\Xi$ for the \emph{different} function $f(x)=x^5 (1+x)^{-2}$ at $x=1$. Clearly, both assign the same value to the infinite sum, which is $f(1)=\frac{1}{4}$ in both cases but the value of the second derivative is now different. Applying Theorem \ref{theorem441} to the function $f(x)=x^5 (1+x)^{-2}$ for $n=100$ and $x=1$ yields an error estimate of 0.0145, which agrees well with the exact error when $m=5$. 

In general, we have the following corollary. \\ \hrule 

\begin{corollary}\label{ErrorTermStab_m}
Define $\hat f_{n,m}(x)$ by $\sum_{k=0}^n \chi_n(k+m) \frac{f^{(k)}(x_0)}{k!} (x-x_0)^k$. For any function $f(x)$ that is analytic in an open disc around each point in the domain $[x_0, x]$ and $[x_0, x]$ falls in the interior of the region of convergence of $\Xi$ for $f(x)$, we have: 
\begin{equation}\label{ErrorTermStab_m_1}
f(x) = \lim_{n\to\infty} \Big\{ \sum_{k=0}^n \chi_n(k+m) \frac{f^{(k)}(x_0)}{k!} (x-x_0)^k \Big\}
\end{equation}
In addition, the error term is asymptotically given by: 
\begin{equation}\label{ErrorTermStab_m_2}
f(x) -\hat f_n(x) \sim \frac{m(m-1) f(x) + 2m(x-x_0)f'(x) + (x-x_0)^2 f^{(2)}(x)}{2n}
\end{equation}
\end{corollary}
\begin{proof}
Define $h(z)=f(z(x-x_0)+x_0)$, whose series expansion around $z=0$ is given by:
\begin{equation}\label{ErrorTermStab_m_3}
h(z)=\sum_{k=0}^\infty \frac{f^{(k)}(x_0)}{k!} (x-x_0)^k z^k
\end{equation}
Applying the summability method to the Taylor series expansion of $f(x)$ is equivalent to applying it for the series expansion of $g(z)$ at $z=1$. In addition, shifting $\chi_n$ by $m$ as given in Eq \ref{ErrorTermStab_m_1} is equivalent to applying the summability method for the new function $z^m h(z)$. Therefore, both Eq \ref{ErrorTermStab_m_1} and \ref{ErrorTermStab_m_2} follow by Theorem \ref{theorem441}. 
\end{proof} \hrule\vspace{12pt} 

Interestingly, Eq \ref{ErrorTermStab_m_2} shows that the optimal choice of $m$ that yields fastest convergence may not be $m=0$. If we differentiate the right-hand side of Eq \ref{ErrorTermStab_m_2} with respect to $m$, we see that the optimal choice $m^*$ that yields fastest convergence is given by: 
\begin{equation}\label{ErrorTermStab_m_4}
m^* = \frac{1}{2} - \frac{f'(x)}{f(x)} (x-x_0) 
\end{equation}
For example, if we return to the function $f(x)=(1+x)^{-2}$ where $x=1$, Eq \ref{ErrorTermStab_m_4} suggests that choosing $m=1$ yields faster convergence than $m=0$. This can be confirmed numerically quite readily. For example, for $n=100$ and $m=0$, the error term is 0.0019. On the other hand, for $n=100$ and $m=1$, the error term is only -6.2339e-04. The error term predicted by Corollary \ref{ErrorTermStab_m} if $m=1$ is -6.0e-04, which is very close to the actual value as expected.

\section{Summary of Results} 
In this chapter, a generalized definition of infinite sums $\mathfrak{T}$ is presented, which is regular, linear, stable, and respects the Cauchy product. To compute the $\mathfrak{T}$ value of infinite sums, various methods can be employed such as Abel and Lindel\"of summation methods. In addition, a new summability method $\Xi$ is proposed in this chapter that is simple to implement in practice and is powerful enough to sum most divergent series of interest. 

The key advantage of using the $\mathfrak{T}$ definition of infinite sums is the fact that it admits a consistent algebra that includes performing linear combinations and even multiplication of divergent series. So, in principle, results on divergent series can be used to deduce new results on convergent series. In the following Chapter, we will use the $\mathfrak{T}$ definition of infinite sums to deduce the analog of the Euler-Maclaurin summation formula for oscillating sums, which will simplify the study of oscillating sums considerably and shed important insight into the subject of divergent series. 

\chapter{Oscillating Finite Sums} \label{Chapter5}
\epigraph{\emph{One cannot escape the feeling that these mathematical formulas have an independent existence and an intelligence of their own, that they are wiser than we are, wiser even than their discoverers.}}{Heinrich Hertz (1857 -- 1894)}

In this section, we use the $\mathfrak{T}$ definition of infinite sums, presented in the previous section, to deduce a family of series that can be thought of as the analog of the Euler-Maclaurin summation formula for oscillating sums. Such results can be used to perform many remarkable deeds with ease. For instance, they can be used to derive analytic expressions for summable divergent sums, obtain asymptotic expressions of oscillating sums, and even accelerate convergence of oscillating series by arbitrary orders of magnitude. They also shed interesting insight into the subject of divergent series. For instance, we will show the notable fact that, as far as the foundational rules of Summability Calculus are concerned, summable divergent series behave exactly as if they were convergent. We will first discuss the important special case of alternating sums, in which we present an analog of the Euler-Maclaurin summation formula that is similar to, but slightly different from, Boole's summation formula. After that, we will generalize results to a broader class of oscillating sums using the notion of \emph{preiodic sign sequences}. 

\section{Alternating Sums} \label{Section5dot1}
We will begin our treatment of alternating sums with the following lemma. \\ \hrule
\begin{lemma}\label{lemma511} 
Given an alternating infinite sum of the form $\sum_{k=a}^\infty (-1)^k g(k)$, where $g(k)$ is analytic in the domain $[a-1, \infty)$, and suppose that the infinite sum is defined in $\mathfrak{T}$ by a value $V\in\mathbb{C}$. Then $V$ is alternatively given by:
\begin{equation}\label{lemma51_1Eq}
V=(-1)^a \sum_{r=0}^\infty \frac{N_r}{r!} g^{(r)}(a-1), \\ \text{ where } N_r = \sum_{k=1}^\infty (-1)^{k+1}\, k^r 
\end{equation}
Here, both infinite sums are also interpreted using the generalized definition $\mathfrak{T}$.  
\end{lemma}
\begin{proof}
First, we note that $V$ is given by: 
\begin{equation}\label{lemma51_1_1}
V=\lim_{\delta\to 0} \big\{\sum_{k=a}^\infty (-1)^k\, \xi_\delta(k-a)\,g(k)\big\}
\end{equation}
Again, $\xi_\delta(j)$ is either given by Lindel\"of or Mittag-Leffler summability methods. However, since $g(k)$ is analytic in the domain $[a-1, \infty)$, then we can also rewrite $V$ using:
\begin{equation}\label{lemma51_1_2}
V=\lim_{\delta_1\to 0} \big\{\sum_{k=a}^\infty (-1)^k\, \xi_{\delta_1}(k-a)\,\lim_{\delta_2\to 0} \{\sum_{j=0}^\infty \frac{g^{(j)}(a-1)}{j!} \xi_{\delta_2}(j)\,(k-a+1)^j\}\big\}
\end{equation}
Therefore, we have: 
\begin{equation}\label{lemma51_1_3}
V=(-1)^a \lim_{\delta_2\to 0} \big\{\sum_{r=0}^\infty  \xi_{\delta_2}(r)\, \frac{g^{(r)}(a-1)}{r!}\lim_{\delta_1\to 0} \{\sum_{j=1}^\infty  \xi_{\delta_1}(j)\,(-1)^{j+1}\,j^r\}\big\}
\end{equation}
By definition of $N_r$, we have: 
\begin{equation}\label{lemma51_1_4}
V=(-1)^a\lim_{\delta\to 0} \big\{\sum_{r=0}^\infty  \xi_\delta(r)\, \frac{N_r}{r!} g^{(r)}(a-1) \big\}
\end{equation}
Thus, statement of the lemma holds. 
\end{proof} \hrule\vspace{12pt} 
Of course, Lemma \ref{lemma511} is not directly useful. It simply provides us with an alternative method of computing the $\mathfrak{T}$ value of divergent sums that are already directly summable using any analytic summability method such as Lindel\"of method, the Mittag-Leffler methods, or the summability method $\Xi$ of Claim \ref{claim411}. However, Lemma \ref{lemma511} can be used to deduce the analog of the Euler-Maclaurin summation formula for alternating sums, as shown next. \\ \hrule 

\begin{theorem}\label{theorem511} 
Given a simple finite sum of the form $f(n)=\sum_{k=a}^n (-1)^k g(k)$, where $\sum_{k=a}^\infty (-1)^k g(k)$ is defined in $\mathfrak{T}$ by a value $V\in\mathbb{C}$, then its unique natural generalization $f_G(n)$ is formally given by: 
\begin{equation}\label{theorem511_Eq1}
f_G(n) = (-1)^a g(a) + \sum_{r=0}^\infty \frac{N_r}{r!} \big[(-1)^n g^{(r)}(n) - (-1)^a g^{(r)}(a)\big]
\end{equation}
In addition, we have for all $n\in\mathbb{C}$: 
\begin{equation}\label{theorem511_Eq2}
f_G(n) = \sum_{k=a}^\infty (-1)^k g(k) - \sum_{k=n+1}^\infty (-1)^k g(k)
\end{equation}
\end{theorem}
\begin{proof}
By Lemma \ref{lemma511}, we have Eq \ref{theorem511_1}, where both sums are interpreted using the generalized definition $\mathfrak{T}$.
\begin{equation}\label{theorem511_1}
\sum_{k=n+1}^\infty (-1)^k g(k) = (-1)^{n+1} \sum_{r=0}^\infty \frac{N_r}{r!} g^{(r)}(n)
\end{equation}
Because the definition $\mathfrak{T}$ is stable, we have: 
\small
\begin{equation}\label{theorem511_2}
\sum_{k=a}^\infty (-1)^k g(k) = (-1)^a\,g(a) +\sum_{k=a+1}^\infty (-1)^k g(k) = (-1)^a\,g(a)-(-1)^a\,\sum_{r=0}^\infty \frac{N_r}{r!} g^{(r)}(a)
\end{equation}
\normalsize
Now, we will show that Eq \ref{theorem511_Eq2} implies Eq \ref{theorem511_Eq1} and that it indeed corresponds to the unique most natural generalization to the simple alternating finite sum. First, starting from Eq \ref{theorem511_Eq2}, and upon using both Eq \ref{theorem511_1} and Eq \ref{theorem511_2}, we deduce that Eq \ref{theorem511_Eq1} holds as expected by subtraction.

Second, we can show that the generalized function $f_G(n)$ given by Eq \ref{theorem511_Eq1} is equivalent to the Euler-Maclaurin summation formula, which was argued earlier to be the unique natural generalization of finite sums. To see this, we note that by Eq \ref{theorem511_Eq1}: 
\begin{equation}\label{theorem511_3}
\sum_{k=a}^n  g(k) = g(a) +\sum_{r=0}^\infty \frac{N_r}{r!}\big[(-1)^n \frac{d^r}{dn^r}\big(e^{i\pi n}\,g(n)\big) - (-1)^a \frac{d^r}{da^r}\big(e^{i\pi a}\,g(a)\big)\big]
\end{equation}
Here, $i=\sqrt{-1}$. However, we have using the Calculus of Finite Differences:
\begin{equation}\label{theorem511_4}
\frac{d^r}{dx^r}\big(e^{i\pi x}\,g(x)\big) = e^{i\pi x} \big(i\pi + D_x\big)^r
\end{equation}
Here, $D_x$ is the differential operator of the function $g(x)$, i.e. $D_x^r=g^{(r)}(x)$, and the expression $(i\pi+D_x)^r$ is to be interpreted by formally applying the binomial theorem on $D_x$. For example, $(i\pi+D_x)^2=(i\pi)^2 D_x^0 + 2i\pi D_x^1 + D_x^2$. Therefore, Eq \ref{theorem511_3} can be rewritten as: 
\begin{equation}\label{theorem511_5}
\sum_{k=a}^n  g(k) = g(a) +\sum_{r=0}^\infty \frac{N_r}{r!}(i\pi+D_n)^r - \sum_{r=0}^\infty \frac{N_r}{r!}(i\pi+D_a)^r
\end{equation}
Later in Eq \ref{Nr_2}, we will show that $\frac{e^x}{1+e^x}=\sum_{r=0}^\infty \frac{N_r}{r!}\,x^r$. Using the Calculus of Finite Differences, we, thus, have: 
\begin{equation}\label{theorem511_6}
\sum_{k=a}^n  g(k) = g(a) +\frac{e^{i\pi+D_n}}{1+e^{i\pi+D_n}} - \frac{e^{i\pi+D_a}}{1+e^{i\pi+D_a}} = g(a)-\frac{e^{D_n}}{1-e^{D_n}}+\frac{e^{D_a}}{1-e^{D_a}}
\end{equation}
Using the series expansion of the function $\frac{e^x}{1+e^x}$, we deduce that Eq \ref{theorem511_7} holds, where $B_r$ are Bernoulli numbers. 
\begin{equation}\label{theorem511_7}
\sum_{k=a}^n  g(k) = g(a) + \sum_{r=0}^\infty \frac{B_r}{r!}(D^{r-1}_n - D^{r-1}_a)
\end{equation}
However, this last equation is precisely the Euler-Maclaurin summation formula. Thus, the generalized function $f_G(n)$ to alternating simple finite sums given by Eq \ref{theorem511_Eq1} and Eq \ref{theorem511_Eq2} is indeed the unique most natural generalization of alternating sums, which completes proof of the theorem \footnote{The summation formula of Theorem \ref{theorem511} is similar to, but different from, Boole's summation formula. Here, Boole's summation formula states that $\sum_{k=a}^{n} (-1)^k\,g(k) = \frac{1}{2} \sum_{k=0}^\infty \frac{E_k(0)}{k!} \big((-1)^n f^{(k)}(n+1)+(-1)^a f^{(k)}(a)\big)$, where $E_k(x)$ are the Euler polynomials \cite{Borwein2009, Borwein1989, Jordan1965}.}\footnote{Incidently, the summation formula for alternating sums given in this theorem can be alternatively deduced from the Euler-Maclaurin summation formula directly upon using the identity $\sum_{k=a}^n (-1)^k\,g(k) = \sum_{k=a}^n g(k) - 2\sum_{k=a}^{\frac{n}{2}} g(2k)$. However, using analytic summability theory yields a more general approach that can be easily extended to oscillating sums as will be shown later.}. 
\end{proof} \hrule\vspace{12pt} 

As shown earlier, the constants $N_r$ are given by the closed-form expression in Eq \ref{zeta_7}, or alternatively by $N_r=\frac{B_{r+1}}{r+1}(1-2^{1+r})$. For example, the first ten terms are $N_r=\{\frac{1}{2}, \frac{1}{4}, 0, -\frac{1}{8}, 0, \frac{1}{4}, 0, -\frac{17}{16}, 0, \frac{31}{4}, \ldots\}$. A generating function for the constants $N_r$ can be deduced from Lemma \ref{lemma511} if we let $g(k)=e^{kx}$, which yields: 
\begin{equation}\label{Nr_1}
\sum_{k=1}^\infty (-1)^k e^{kx} = \frac{1}{1+e^x}-1 = -\sum_{r=0}^\infty \frac{N_r}{r!}\,x^r
\end{equation}
Therefore, the generating function of $N_r$ is given by: 
\begin{equation}\label{Nr_2}
\frac{e^x}{1+e^x} = \sum_{r=0}^\infty \frac{N_r}{r!}\,x^r
\end{equation}

The constants $N_r$ can be thought of as the alternating analog of Bernoulli numbers. However, it was argued earlier that the defining property of Bernoulli numbers is the fact that they are the unique solutions to the functional equation: 
\begin{equation}\label{Nr_3} 
f'(n)=\sum_{r=0}^\infty \frac{B_r}{r!} [f^{(r)}(n)-f^{(r)}(n-1)]
\end{equation} 
From Theorem \ref{theorem511}, we can immediately deduce a similar defining property of the constants $N_r$. In particular, if we let $a=n-1$, we have that $N_r$ are the  solutions to the following functional equation:
\begin{equation}\label{Nr_4} 
f(n)=\sum_{r=0}^\infty \frac{N_r}{r!} [f^{(r)}(n)+f^{(r)}(n-1)]
\end{equation} 
By formally applying the Taylor series expansion of both sides of Eq \ref{Nr_4}, we arrive at the following recursive method for computing $N_r$:
\begin{equation}\label{Nr_5} 
N_r=\frac{1}{4} - \frac{1}{2}\sum_{k=1}^{r-1} \binom{r}{k} N_k \quad\quad\text{ if } r>0, \\ \text{ and } N_0=\frac{1}{2}
\end{equation} 

One immediate implication of Theorem \ref{theorem511} is that it gives a closed-form formula for alternating power sums. For example, we have:
\begin{equation}\label{Nr_6} 
\sum_{k=1}^n (-1)^k = -1 + \frac{1+(-1)^n}{2} = \frac{(-1)^n-1}{2}
\end{equation} 
\begin{equation}\label{Nr_7} 
\sum_{k=1}^n (-1)^k k = -\frac{1}{4} + (-1)^n \frac{2n+1}{4}
\end{equation} 
Note that in both examples, the $\mathfrak{T}$ sequence limit of $\{(-1)^n\}_n$ and $\{(-1)^n\, n\}_n$ as $n\to\infty$ is zero by Lemma \ref{TlimitZero} so the value of  the divergent sums is consistent with the $\mathfrak{T}$ sequence limit in these equations. In general, we have for all integers $s\ge 0$ the following analog of Bernoulli-Faulhaber formula:
\begin{equation}\label{Nr_8} 
\sum_{k=1}^n (-1)^k k^s = -N_s+ (-1)^n \sum_{r=0}^s \binom{s}{r} N_r\, n^{s-r}
\end{equation}  \\ \hrule

\begin{corollary}\label{corollary511} 
Suppose $g(n)$ is asymptotically of a finite differentiation order $m$, i.e. $g^{(m+1)}(n)\to 0$ as $n\to\infty$, and suppose that the value of $\sum_{k=a}^\infty (-1)^k\,g(k)$ exists in $\mathfrak{T}$ and is given by $V\in\mathbb{C}$, then $V$ can be alternatively evaluated using the following limiting expression:
\begin{equation}\label{cor511_1} 
V = \lim_{n\to\infty} \big\{\sum_{k=a}^n (-1)^k g(k) - (-1)^n \sum_{r=0}^m \frac{N_r}{r!} g^{(r)}(n)\big\}
\end{equation}
Taking higher order terms of the expression $\sum_{r=0}^m \frac{N_r}{r!} g^{(r)}(n)$ improves speed of convergence. Alternatively, and under stated conditions, the simple finite sum $\sum_{k=a}^n (-1)^k g(k)$ is \emph{asymptotically} given by the following expression, where error term vanishes as $n$ goes to infinity.
\begin{equation}\label{cor511_2} 
\sum_{k=a}^n (-1)^k g(k) \sim V + (-1)^n \sum_{r=0}^m \frac{N_r}{r!} g^{(r)}(n)
\end{equation}
\end{corollary}
\begin{proof} 
Follows immediately from Lemma \ref{lemma511} and Theorem \ref{theorem511}. 
\end{proof} \hrule\vspace{12pt} 

Corollary \ref{corollary511} provides us with a simple method of obtaining asymptotic expressions to alternating sums, assigning natural values to divergent alternating sums, and accelerating convergence of alternating sums as well. It can even allow us to \emph{derive analytic expressions} of divergent sums in many cases. For example, suppose we would like to apply the generalized definition $\mathfrak{T}$ to the alternating sum $\sum_{k=1}^\infty (-1)^{k+1} \log{k}$. Numerically, if we use the summability method $\Xi$ where we choose $n=100$ or $n=1000$, we get -0.2264 and -0.2259 respectively. Using higher order approximations show that the divergent sum is summable to around -0.2258 (up to 4 decimal places). Using Corollary \ref{corollary511}, we can derive the exact value of such divergent sum as follows. First, we note that $\log{n}$ is asymptotically of a finite differentiation order zero, thus we have:
\begin{align*} 
\sum_{k=1}^\infty (-1)^{k+1} \log{k} &= \lim_{n\to\infty}\big\{\sum_{k=1}^{2n} (-1)^{k+1}\,\log{k} + \frac{\log{(2n)}}{2}\big\} \\ 
&= \lim_{n\to\infty}\big\{\sum_{k=1}^{2n} \log{k} -2\sum_{k=1}^n \log{(2k)}+ \frac{\log{(2n)}}{2}\big\} \\
&= \lim_{n\to\infty}\big\{ \log{(2n)!} - 2n\log{2} - 2\log{n!} +  \frac{\log{(2n)}}{2}\big\} \\ 
&=\frac{1}{2} \log{\frac{2}{\pi}}
\end{align*}
Here, we used Stirling approximation in the last step. Indeed, we have $\frac{1}{2}\log{\frac{2}{\pi}} = -0.2258$. Similarly, if we apply Corollary \ref{corollary511} to the divergent sum $\sum_{k=1}^n (-1)^{k+1} \log{k!}$, whose asymptotic differentiation order is $m=1$, we obtain Eq \ref{Pcor511_1}, which can be confirmed numerically quite readily using the summability method $\Xi$. 
\begin{equation}\label{Pcor511_1}
\sum_{k=1}^\infty (-1)^{k+1} \log{k!} = \frac{1}{4} \log{\frac{2}{\pi}}
\end{equation}

As stated earlier, we can employ Corollary \ref{corollary511} to accelerate \lq\lq convergence'' of alternating sums, \emph{including alternating divergent sums} for which the generalized definition $\mathfrak{T}$ is implied. For example, suppose we would like to compute $\sum_{k=1}^\infty (-1)^k \log{k}$ using the summability method $\Xi$ of Claim \ref{claim411}. If we use $n=100$, we obtain a figure that is accurate to 3 decimal places only, which is expected given the fact that the summability method converges only linearly. If, on the other hand, we compute the divergent sum using Corollary \ref{corollary511}, where $m=1$ and $n=100$, we obtain a figure that is accurate to 7 decimal places! Here, we chose $m=1$ instead of $m=0$ to accelerate convergence. Moreover, applying Corollary \ref{corollary511} to the convergent sum $\sum_{k=1}^\infty \frac{(-1)^{k+1}}{k^2}$ using $n=100$ and $m=1$ yields a value of 0.8224670334, which is accurate to 10 decimal places! This is quite remarkable given the fact that we have only used 100 terms in such slowly converging sum. Of course, choosing higher values of $m$ accelerates convergence even more \footnote{It is worth mentioning that many algorithms exist for accelerating convergence of alternating series, some of which can sometimes yield several digits of accuracy per iteration. Internestingly, some of these algorithms such as the one proposed in \cite{Cohen2000} were also found to be capable of \lq\lq correctly'' summing some divergent alternating series.}. 

One example where Corollary \ref{corollary511} can be employed to obtain asymptotic expressions is the \emph{second factorials} function given by $f(n) = \sum_{k=1}^n k^{k^2}$. Similar to earlier approaches, we will first find an asymptotic expression of $\log{f(n)}$. However, this is easily done using the Euler-Maclaurin summation formula but its leading constant is unfortunately unknown in closed-form. To find an exact expression of that constant, we use the fact that $\sum_{k=1}^\infty (-1)^k k^2\,\log{k} = 7\zeta_3/(4\pi^2)$, which can be deduced immediately from the well-known analytic continuation of the Riemann zeta function. Here, we know by Eq \ref{zeta_1} that the aforementioned assignment indeed holds when interpreted using the generalized definition $\mathfrak{T}$. Now, we employ Corollary \ref{corollary511}, which yields the asymptotic expression of the second factorials function given in Eq \ref{secondFactAsympt}. Here, the ratio of the two sides of Eq \ref{secondFactAsympt} goes to unity as $n \to\infty$. The asymptotic expression in Eq \ref{secondFactAsympt} is mentioned in \cite{Choi2000}. 
\begin{equation}\label{secondFactAsympt}
 \sum_{k=1}^n k^{k^2} \sim e^{\frac{\zeta_3}{4\pi^2}}\,n^{\frac{2n^3+3n^2+n}{6}}\,e^{\frac{n}{12}-\frac{n^3}{9}}
\end{equation}

The formal expressions for the $\mathfrak{T}$ value of divergent sums can be used to deduce the Euler summation formula. \\ \hrule
\begin{lemma}{\textbf{(The Euler Sum Revisited)}} \label{EulerAgain}
Let $V_1=\sum_{k=a}^\infty (-1)^k g(k)$, when interpreted by $\mathfrak{T}$. Also, let  $V_2=\sum_{k=a}^\infty (-1)^k \sum_{j=a}^k g(k)$. If both $V_1$ and $V_2$ exist, then $V_1=2V_2$. 
\end{lemma} 
\begin{proof} 
We have by Theorem \ref{theorem511} that $V_1$ is given by: 
\begin{equation}\label{EulerAgain_1} 
V_1 = \sum_{k=a}^n (-1)^k\,g(k) - (-1)^n \sum_{r=0}^\infty \frac{N_r}{r!}\,g^{(r)}(n)
\end{equation}
Similarly, we have: 
\begin{equation}\label{EulerAgain_2} 
V_2 = \sum_{k=a}^n (-1)^k\,G(k) - (-1)^n \sum_{r=0}^\infty \frac{N_r}{r!}\,G^{(r)}(n), \\ G(k)=\sum_{j=a}^k g(j)
\end{equation}
However, by know by the basic rules of Summability Calculus that the following holds for all $r\ge 0$: 
\begin{equation}\label{EulerAgain_3} 
G^{(r)}(n) = \sum_{k=a}^n g^{(r)}(k) + G^{(r)}(a-1)
\end{equation}
Plugging Eq \ref{EulerAgain_3} into Eq \ref{EulerAgain_2} yields: 
\begin{equation}\label{EulerAgain_4} 
V_2 = \sum_{k=a}^n (-1)^k\,G(k) - (-1)^n\sum_{r=0}^\infty \frac{N_r}{r!} G^{(r)}(a-1) - (-1)^n \sum_{r=0}^\infty \frac{N_r}{r!} \sum_{j=a}^n g^{(r)}(j)
\end{equation}
We know by Lemma \ref{lemma511} that: 
\begin{equation}\label{EulerAgain_5} 
(-1)^a \sum_{r=0}^\infty \frac{N_r}{r!} G^{(r)}(a-1) = V_2
\end{equation}
Therefore, Eq \ref{EulerAgain_4} can be rewritten as: 
\begin{equation}\label{EulerAgain_6} 
V_2 = \sum_{k=a}^n (-1)^k\,G(k) - (-1)^{n+a} V_2 - (-1)^n \sum_{r=0}^\infty \frac{N_r}{r!} \sum_{j=a}^n g^{(r)}(j)
\end{equation}
The expression holds for all $n$. We set $n=a$ to obtain: 
\begin{equation}\label{EulerAgain_7} 
2V_2 = (-1)^a\,g(a) - (-1)^a \sum_{r=0}^\infty \frac{N_r}{r!} g^{(r)}(a) = V_1
\end{equation}
Therefore, we have the desired result. 
\end{proof} \hrule\vspace{12pt} 

Of course, the Euler summation formula easily follows from Lemma \ref{EulerAgain}. One example of Lemma \ref{EulerAgain} was already demonstrated for the $\sum_{k=1}^\infty (-1)^{k+1}\log{k!}$, where we showed in Eq \ref{Pcor511_1} that the $\mathfrak{T}$ value of such divergent sum is $\frac{\log{(2/\pi)}}{4}$. However, comparing this with the $\mathfrak{T}$ value of the divergent sum $\sum_{k=1}^\infty (-1)^{k+1} \log{k}$ given earlier, we see that the Lemma holds as expected. In fact, we immediately deduce that $\sum_{k=1}^\infty (-1)^{k+1} \log{S(k)} = \frac{\log{(2/\pi)}}{8}\approx -0.0564$, where $S(k)=k!\,S(k-1)$ is the superfactorial function. Again, this can be confirmed numerically using $\Xi$. For instance, using $\Xi$ with $n=50$ yields $\sum_{k=1}^\infty (-1)^{k+1} \log{S(k)} \approx -0.0571$ whereas choosing $n=100$ yields $\sum_{k=1}^\infty (-1)^{k+1} \log{S(k)} \approx -0.0568$. Both are very close to the correct value $-0.0564$. 

In addition, since $\sum_{k=1}^\infty (-1)^{k+1}\frac{1}{k} = \log{2}$, we deduce from Lemma \ref{EulerAgain} that: 
\begin{equation} \label{AlternHarmonicSum}
\sum_{k=1}^\infty (-1)^{k+1} H_k = \frac{\log{2}}{2}
\end{equation}
Therefore, we have by linearity of $\mathfrak{T}$: 
\begin{equation} \label{AlternHarmonicSum_2}
\sum_{k=1}^\infty (-1)^{k+1} (H_k-\log{k}-\lambda) = \frac{\log{\pi}-\lambda}{2}
\end{equation}
Here, we have used the values of divergent sums to compute a convergent sum! 

Now, Theorem \ref{theorem511} can be readily generalized to convoluted alternating sums as the following theorem shows. \\ \hrule 
\begin{theorem}\label{theorem512}
Given a convoluted finite sum of the form $f(n)=\sum_{k=a}^n (-1)^k\,g(k,n)$, then its unique natural generalization $f_G(n)$ is formally given by:
\begin{equation}\label{theorem512Eq} 
f_G(n) = (-1)^a\,g(a,n) + \sum_{r=0}^\infty \frac{N_r}{r!}\Big[(-1)^n \frac{\partial^r}{\partial t^r} g(t,n)\Big|_{t=n} - (-1)^a\frac{\partial^r}{\partial t^r} g(t,n)\Big|_{t=a}\Big]
\end{equation}
\end{theorem}
\begin{proof}
Similar to the proof of Lemma \ref{convolutedEMFormula} and Theorem \ref{theorem311}. 
\end{proof} \hrule\vspace{12pt} 

Theorem \ref{theorem512} gives a simple method for deducing closed-form expressions and/or asymptotic expressions of alternating convoluted sums. For example, suppose $f(n)$ is given by the convoluted alternating sum $f(n)=\sum_{k=0}^n (-1)^k \log{(1+\frac{k}{n})}$, and suppose that we wish to find its asymptotic value as $n$ tends to infinity. Using Theorem \ref{theorem512}, we have: 
\begin{equation}\label{theorem512_1}
\sum_{k=0}^n (-1)^k\,\log{(1+\frac{k}{n})} \sim (-1)^n \frac{\log{2}}{2}
\end{equation}
For a second example, we have by Theorem \ref{theorem512}: 
\begin{equation}\label{theorem512_2}
\sum_{k=0}^n (-1)^k\,(k+n) = \frac{n}{2}+(-1)^n\,n+\frac{(-1)^n-1}{4}
\end{equation}
Note that the same expression in the last equation can be alternatively derived by splitting the sum into two alternating sums $\sum_{k=0}^n (-1)^k\,(k+n) = \sum_{k=0}^n (-1)^k\,k + n\sum_{k=0}^n (-1)^k$ , and by using the results in Eq \ref{Nr_6} and Eq \ref{Nr_7}.

Obviously, the approach employed so far for alternating sums can be readily extended to finite sums of the form $\sum_{k=a}^n e^{i\theta k}\,g(k)$ , which is stated in the following theorem. \\ \hrule 
\begin{theorem}\label{theorem513}
Given a simple finite sum of the form $f(n)=\sum_{k=a}^n e^{i\theta k}\,g(k)$, where $0<\theta <2\pi$, then its unique natural generalization $f_G(n)$ is formally given by:
\small
\begin{equation}\label{theorem513Eq1}
f_G(n)=e^{i\theta a}\,g(a) - \sum_{r=0}^\infty \frac{\Theta_r}{r!}\big[e^{i\theta n} g^{(r)}(n) - e^{i\theta a} g^{(r)}(a)\big], \\ \text{ where } \Theta_r=\sum_{k=1}^\infty e^{i\theta k}k^r
\end{equation}
\normalsize
Here, infinite sums are interpreted using the generalized definition $\mathfrak{T}$. In addition, we also have: 
\begin{equation} \label{theorem513Eqa1_1} 
f_G(n) = \sum_{k=a}^\infty  e^{i\theta k}\,g(k) -  \sum_{k=n+1}^\infty  e^{i\theta k}\,g(k)
\end{equation}
Also, $\sum_{k=a}^\infty e^{i\theta k}\,g(k)$ is formally given by: 
\small
\begin{equation}\label{theorem513Eq2}
\sum_{k=a}^\infty e^{i\theta k}\,g(k) = e^{i\theta(a-1)} \sum_{r=0}^\infty \frac{\Theta_r}{r!} g^{(r)}(a-1) = \sum_{k=a}^n e^{i\theta k}g(k) + \sum_{r=0}^\infty \frac{\Theta_r}{r!} e^{i\theta n} g^{(r)}(n)
\end{equation}
\normalsize
For instance, if $g(n)$ is asymptotically of a finite differentiation order $m$, then Eq \ref{theorem513Eq2} is interpreted as a limiting expression as $n\to\infty$, which is similar to Corollary \ref{corollary511}. In general, infinite sums are interpreted using the generalized definition $\mathfrak{T}$.
\end{theorem}
\begin{proof} 
Similar to the proofs of Theorem \ref{theorem511} and Corollary \ref{corollary511}. 
\end{proof} \hrule\vspace{12pt} 

If $\theta=\pi$, we have $\Theta_r=-N_r$. To find a generating function for $\Theta_r$ in general, we set $g(k)=e^{kx}$, which yields the generating function: 
\begin{equation}\label{ThetaGenFunction}
\sum_{k=1}^\infty e^{i\theta k} e^{kx} = \frac{e^{x+i\theta}}{1-e^{x+i\theta}} = \sum_{r=0}^\infty \frac{\Theta_r}{r!} x^r
\end{equation}

Of course, Theorem \ref{theorem513} can be extended to convoluted sums as well. In addition, by setting $a=n-1$ in Eq \ref{theorem513Eq1}, we deduce that the constants $\Theta_r$  are the unique solutions to the functional equation:
\begin{equation}\label{ThetaGenFunction_2}
f(n) = - \sum_{r=0}^\infty \frac{\Theta_r}{r!} \big(f^{(r)}(n)-e^{-i\theta} f^{(r)}(n-1)\big)
\end{equation}

\section{Oscillating Sums: The General Case} \label{Section5dot2}
More than two centuries ago, Daniel Bernoulli suggested that periodic oscillating sums should be interpreted probabilistically. For instance, he stated that the Grandi series $\sum_{k=0}^\infty (-1)^k$  should be assigned the value of $\frac{1}{2}$ because the partial sums are equal to one if $n$ is even and equal to zero otherwise; therefore the \emph{expected value} of the infinite sum is $\frac{1}{2}$. Using the same reasoning, Bernoulli argued that the divergent sum $1+0-1+1+0-1\ldots$ should be assigned the value of $\frac{2}{3}$ \cite{HardyDiverg, SandiferDiverg}. Remarkably, such values are indeed consistent with many natural summability methods such as Abel summability method and the summability method $\Xi$ given in Claim \ref{claim411}. Almost a century later, Cesaro captured and expanded the same principle by proposing his well-known summability method in 1890 that is given by Eq \ref{cesaroSum} \cite{HardyDiverg}. The intuition behind Cesaro summability method, however, is not Bernoulli's probabilistic interpretation, but the notion of what is currently referred to as the \emph{Cesaro mean}. Simply stated, if an infinite sum $\sum_{k=0}^\infty a_k$ converges to some value $V$, then the partial sums $\sum_{k=0}^j a_k$  approach $V$ for sufficiently large $j$. Thus, $\sum_{j=0}^n \sum_{k=0}^j a_k$ will approach $n\,V$ as $n\to\infty$, which makes Eq \ref{cesaroSum} a \emph{very natural definition} for divergent sums as well if the limit exists.
\begin{equation}\label{cesaroSum} 
\sum_{k=0}^\infty a_k = \lim_{n\to\infty} \big\{\frac{1}{n} \sum_{j=1}^n \sum_{k=0}^j a_k \big\} 
\end{equation}

In this section, we will generalize results of the previous section, which will allow us to prove that the intuitive reasoning of D. Bernoulli is indeed correct when infinite sums are interpreted using the generalized definition $\mathfrak{T}$. Of course, we have already established that $\mathfrak{T}$ is consistent with Cesaro summability method by proving that it is consistent with Abel summation.  Nonetheless, results of this section will allow us to deduce an alternative \emph{direct} proof to Bernoulli's probabilistic interpretation.

However, we will restrict our attention in this section to oscillating sums of the form $\sum_{k=a}^n s_k\, g(k)$, where $g(k)$ is analytic in the domain $[a, \infty)$ and $s_k$ is an arbitrary \emph{periodic} sequence of numbers that will be referred to as \emph{the sign sequence}. Common examples of sign sequences include $(1,1,1,\dots)$ in ordinary sums, $(1,-1,1,-1,\ldots)$ in the case of alternating sums, and $(e^{i2\pi/n}, e^{i4\pi/n}, e^{i6\pi/n}, \ldots)$ in harmonic analysis, and so on. It is important to keep in mind that the sign sequence is assumed to be periodic, but it does not have to be analytic.

Looking into the proofs of Lemma \ref{lemma511} and Theorem \ref{theorem511}, we note that if a similar approach is to be employed for oscillating sums of the form $\sum_{k=a}^n s_k\, g(k)$ then the sum $\sum_{k=a}^\infty s_k\,(k-a+1)^r$ has to exit in $\mathfrak{T}$ for all $r \ge 0$. Thus, our first point of departure is to ask for which periodic sign sequences $s_k$ do the sums $\sum_{k=a}^\infty s_k\,(k-a+1)^r$ exist in $\mathfrak{T}$ for all integers $r \ge 0$. Lemma \ref{lemma521} provides us with the complete answer. \\ \hrule 

\begin{lemma}\label{lemma521} 
Given a periodic sign sequence $s_k$ with period $p$, i.e. $s_{k+p} = s_k$, then the sums $\sum_{k=a}^\infty s_k\,(k-a+1)^r$ are well-defined in $\mathfrak{T}$ for all integers $r \ge 0$ if and only if $\sum_{k=a}^{a+p-1} s_k = 0$. That is, the sum of values of the sign sequence $s_k$ for any cycle of length $p$ is equal to zero.
\end{lemma}
\begin{proof}
Without loss of generality, we will set $a=1$ to simplify notations. To prove that the condition $\sum_{k=1}^{p} s_k = 0$ is necessary, we set $r=0$, which yields by stability: 
\begin{equation}\label{lemma521_1} 
\sum_{k=1}^\infty s_k = \sum_{k=1}^{p} s_k + \sum_{k=p+1}^\infty s_k = \sum_{k=1}^{p} s_k + \sum_{k=1}^\infty s_k
\end{equation}
Therefore, if $\sum_{k=1}^\infty s_k$ is summable to some value $V\in\mathbb{C}$, then $\sum_{k=1}^p s_k$ must be equal to zero. To show that the condition is sufficient, we first note by Taylor's theorem that: 
\begin{equation}\label{lemma521_2} 
\frac{x-1}{x^p-1} = 1-x+x^p-x^{p+1} + x^{2p} - x^{2p-1} \dotsm 
\end{equation}
Therefore, if we define a function $f(x)$ by: 
\begin{equation}\label{lemma521_3} 
f(x) = \frac{x-1}{x^p-1} \big(s_1 + (s_1+s_2)\,x + \dotsm + (s_1+s_2+\ldots+s_{p-1})\,x^{p-2}\big) 
\end{equation}
Then, the Taylor series expansion of $f(x)$ is given by: 
\begin{equation}\label{lemma521_4} 
f(x) = \sum_{k=0}^\infty s_{k+1}\,x^k
\end{equation}

It is straightforward to see that the condition $\sum_{k=a}^{a+p-1} s_k$  is also satisfied in the generating function $f(x)$. Because $f(x)$ is analytic in an open disc around each point in the domain $[0, 1]$, then applying the summability method to  $\sum_{k=1}^\infty s_k$  yields $f(1)$ by definition of $\mathfrak{T}$. In general, for all integers $r \ge 0$, applying the summability method to $\sum_{k=1}^\infty s_k\,k^r$ yields some linear combination of higher derivatives $f^{(m)}(1)$, therefore, they are also well-defined in $\mathfrak{T}$. 
\end{proof} \hrule 

\begin{corollary}\label{DBernoulliProb}
Given a periodic sequence $s_k$ with period $p$, where $\sum_{k=a}^{a+p-1} s_k = 0$. Then the sum $\sum_{k=a}^\infty s_k$  is well-defined in $\mathfrak{T}$. Moreover, its $\mathfrak{T}$ value agrees with Bernoulli's probabilistic interpretation.  
\end{corollary}
\begin{proof}
We have established in Lemma \ref{lemma521} that the sum is well-defined in $\mathfrak{T}$. We have also shown that applying the generalized definition $\mathfrak{T}$ to the sum $\sum_{k=a}^\infty s_k$  yields the value $f(1)$, where $f(x)$ is defined by Eq \ref{lemma521_3}. Using l'H\"ospital's rule, the assigned value is indeed given by:
\begin{equation}\label{DBernoulliProb_1}
\sum_{k=a}^\infty s_k = \frac{s_1 + (s_1+s_2) + \dotsm + (s_1+s_2+\ldots+s_p)}{p}
\end{equation}
Here, we have used the fact that $\sum_{k=1}^{p} s_k=0$, which follows by assumption. Therefore, the assigned value is indeed the average or expected value of the partial sums as Bernoulli suggested.
\end{proof} \hrule\vspace{12pt} 

One important advantage of Eq \ref{DBernoulliProb_1} is that it allows us to obtain \emph{exact values} of divergent sums $\sum_{k=a}^\infty s_k$ that satisfy the conditions of Lemma \ref{lemma521}, as opposed to approximating them by methods of computation such as Abel summability method or $\Xi$. Clearly, we desire a similar simple approach to obtain exact expressions of all divergent sums of the form $\sum_{k=a}^\infty s_k\,(k-a+1)^r$ that satisfy the conditions of Lemma \ref{lemma521}. Lemma \ref{lemma522} provides us with the answer. \\ \hrule

\begin{lemma}\label{lemma522} 
Given a periodic sign sequence $s_k$ with period $p$, where $\sum_{k=a}^{a+p-1} s_k=0$, let $S_r$ be given by $S_r=\sum_{k=a}^\infty s_k\,(k-a+1)^r$, where the divergent sum is interpreted using $\mathfrak{T}$. Then, $S_r$ can be \emph{exactly} computed using the recursive equation:
\begin{equation}\label{lemma522Eq}
S_r= -\frac{p^r}{r+1}\sum_{m=0}^{r-1} \binom{r+1}{m}\,p^{-m}\,S_m - \frac{1}{p(r+1)}\sum_{k=a}^{a+p-1} s_k\,(k-a+1)^{r+1}
\end{equation}
Here, the base case $S_0$ is given by Corollary \ref{DBernoulliProb}. 
\end{lemma}
\begin{proof} 
Again, the proof rests on the fact that the sums $\sum_{k=a}^\infty s_k\,(k-a+1)^r$ are stable, which yields:
\begin{equation}\label{lemma522_1} 
\sum_{k=a}^\infty s_k\,(k-a+1)^r = \sum_{k=a}^{a+p-1}s_k\,(k-a+1)^r + \sum_{k=a}^\infty s_k\,(k-a+1+p)^r
\end{equation}
Expanding the factors $(k-a+1+p)^r$ using the Binomial Theorem and rearranging the terms yields the desired result. 
\end{proof} \hrule\vspace{12pt} 

Now, knowing that any periodic sequence $s_k$ with period $p$ that satisfy the condition of Lemma \ref{lemma521} imply that $\sum_{k=a}^\infty s_k\,(k-a+1)^r$ exist in $\mathfrak{T}$ for all integers $r \ge 0$, and given that we can compute the exact values of $\sum_{k=a}^\infty s_k\,(k-a+1)^r$ for all $r\ge 0$ using Corollary \ref{DBernoulliProb} and Lemma \ref{lemma522}, we are ready to generalize main results of the previous section. 

First, given an arbitrary periodic sign sequence $s_k$, we compute the average $\tau = \sum_{k=a}^{a+p-1} s_k$, and split the sum $\sum_{k=a}^n s_k\,g(k)$ into $ \tau \sum_{k=a}^n g(k) + \sum_{k=a}^n (s_k-\tau)\,g(k)$. The first term is a direct simple finite sum that can be analyzed using the earlier results of Summability Calculus in Chapter \ref{Chapter2}. The second term, on the other hand, satisfies the condition $\sum_{k=a}^{a+p-1}(s_k-\tau) = 0$ so Lemma \ref{lemma521}, Corollary \ref{DBernoulliProb} and Lemma \ref{lemma522} all hold.  \\ \hrule

\begin{theorem}\label{theorem521} 
Given a simple finite sum of the form $f(n)=\sum_{k=a}^n s_k\,g(k)$, where $s_k=(s_0, s_1, \ldots)$ is a periodic sign sequence with period $p$ that satisfies the condition $\sum_{k=a}^{a+p-1}s_k=0$. Then, for all $a\in\mathbb{N}$ and $n\ge a\in\mathbb{N}$, $f(n)$ is alternatively given by:
\begin{equation}\label{theorem521Eq1}
f(n) = s_a\,g(a) - \sum_{r=0}^\infty \frac{1}{r!}\big[S_r(1+n\bmod{p})\,g^{(r)}(n) - S_r(1+a\bmod{p})\,g^{(r)}(a)\big]
\end{equation}
Here, $S_r(x)=\sum_{k=0}^\infty s_{k+x}\, k^r$ can be exactly computed using Corollary \ref{DBernoulliProb} and Lemma \ref{lemma522}. In addition, if $g(n)$ is asymptotically of a finite differentiation order $m$, then:
\begin{equation}\label{theorem521Eq2}
\sum_{k=a}^\infty s_k\,g(k) = \lim_{x\to\infty} \Big\{\sum_{k=a}^{px}s_k\,g(k) + \sum_{r=0}^m \frac{S_r(1)}{r!}g^{(r)}(px)\Big\}
\end{equation}
\end{theorem}
\begin{proof}
Similar to the proofs of Theorem \ref{theorem511} and Corollary \ref{corollary511}. Note that we cannot consider $f(n)$ given by Eq \ref{theorem521Eq1} to be the natural generalization $f_G(n)$ since it is so far only defined for $n \in\mathbb{N}$ due to the use of congruence. This limitation will be resolved shortly using the Discrete Fourier Transform (DFT).  
\end{proof} \hrule\vspace{12pt} 

Again, Theorem \ref{theorem521} allows us to obtain asymptotic expressions of oscillating sums and accelerate convergence. For instance, suppose $s_k=(0, 1, 0, -1, 0, 1, 0, -1, 0, \ldots)$, i.e. $p=4$ and $s_0=0$. Then, $S_0(1)$ can be computed using Corollary \ref{DBernoulliProb}, which yields:
\begin{equation}\label{theorem521_1} 
S_0(1)=1+0-1+0+1+0-1+0+\dotsm = \frac{1+1+0+0}{4}=\frac{1}{2}
\end{equation}
Similarly, $S_1(1)$ can be computed using Lemma \ref{lemma521}, which yields: 
\begin{equation}\label{theorem521_2} 
S_1(1)=\sum_{k=1}^\infty s_k\,k = -2S_0(1)-\frac{1\times 1 + 0 \times 4 - 1 \times 9 + 0 \times 16}{8} = 0
\end{equation}
Therefore, if we wish to obtain a method of evaluating the convergent sum $\sum_{k=1}^\infty \frac{s_k}{k} = \frac{1}{1}-\frac{1}{3}+\frac{1}{5}-\ldots=\frac{\pi}{4}$  with a cubic-convergence speed, we choose $m=1$ in Eq \ref{theorem521Eq2}. For instance, choosing $n=100$ yields:
\begin{equation}\label{theorem521_3} 
\sum_{k=1}^\infty \frac{s_k}{k} \approx \sum_{k=1}^{100} \frac{s_k}{k} + \frac{S_0(1)}{100} - \frac{S_1(1)}{100^2} = 0.785398
\end{equation}

Here, we obtain a figure that is accurate up to 6 decimal places. This is indeed quite remarkable given that we have only used 100 terms. Otherwise, if we compute the sum directly, we would need to evaluate approximately 500,000 terms to achieve the same level of accuracy! 

Finally, to obtain the unique natural generalization $f_G(n)$ for simple finite sums of the form $f(n)=\sum_{k=a}^n s_k\,g(k)$ in which $s_k$ is a periodic sign sequence, we use the Discrete Fourier Transform (DFT) as the following lemma shows. \\ \hrule 

\begin{lemma}\label{DFTs_k}
Given a periodic sign sequence $s_k=(s_0, s_1, \ldots)$ with period $p$, then $s_k$ can be generalized to all complex values of $k$ using the Discrete Fourier Transform (DFT). More precisely, we have:
\begin{equation}\label{DFT_1}
s_k = \sum_{m=0}^{p-1} \nu_m\,e^{i\frac{mk}{p}}, \\ \text{ where } \quad\nu_m = \frac{1}{p} \sum_{k=0}^{p-1} s_k\,e^{-i\frac{mk}{p}} 
\end{equation}
\end{lemma}
\begin{proof} 
By direct application of the Discrete Fourier Transform (DFT). Intuitively, $\nu_m$ is the \emph{projection} of the sequence $s_k$ on the sequence $\{e^{i\frac{mk}{p}}\}_k$  when both are interpreted as infinite-dimensional vectors and $\{e^{i\frac{mk}{p}}\}_k$ are orthogonal for different values of $m$. In addition, since $s_k$ is periodic, the basis are complete. Note that $\nu_0$ is the average value of the of sign sequence $s_k$, i.e. the DC component in engineering terminology. 
\end{proof} \hrule
\begin{corollary}\label{corollary522} 
Given a simple finite sum of the form $f(n)=\sum_{k=a}^n s_k\,g(k)$, where $s_k=(s_0, s_1, \ldots)$ is a periodic sign sequence with period $p$, then we have: 
\begin{equation}\label{theorem522_1} 
f(n) = \nu_0 \sum_{k=a}^n g(k) + \nu_1 \sum_{k=a}^n e^{i\frac{k}{p}}\,g(k) + \dotsm + \nu_{p-1} \sum_{k=a}^n e^{i\frac{(p-1)k}{p}}\,g(k)
\end{equation} 
Here, $\nu_m$ are defined in Lemma \ref{DFTs_k}. In addition, the unique natural generalization $f_G(n)$ for all $n\in\mathbb{C}$ is given by the sum of unique natural generalizations to all terms in the previous equation (see Section \ref{sectionSFSGC} and Theorem \ref{theorem513}). 
\end{corollary}
\begin{proof} 
By Lemma \ref{DFTs_k}. 
\end{proof} \hrule\vspace{12pt} 

\section{Infinitesimal Calculus on Oscillating Sums}\label{Section5dot3}
The general rules for performing infinitesimal calculus on simple finite sums of the form $\sum_{k=a}^n z(k)$  were presented earlier in Chapter \ref{Chapter2} that used the Euler-Maclaurin summation formula, which are also formally applicable to oscillating sums (see for instance the Example \ref{ExAlterSum}). In this section, nevertheless, we will show that if an oscillating infinite sum $\sum_{k=a}^n z(k)$  is well-defined in $\mathfrak{T}$, then the function $f(n)$ behaves with respect to infinitesimal calculus exactly as if it were a convergent function as $n\to\infty$, regardless of whether or not $\lim_{n\to\infty} f(n)$ exists in the classical sense of the word. In other words, $f(n)$ behaves as if it were a convergent function if the generalized $\mathfrak{T}$ limit of $f(n)$ as $n\to\infty$ exists. 

In his notebooks, Ramanujan captured the basic idea that ties summability theory to infinitesimal calculus. In his words, every series has a constant $c$, which acts \lq\lq like the center of gravity of the body.'' However, his definition of the constant $c$ was imprecise and frequently led to incorrect conclusions  \cite{Berndt6}. Assuming that such constant exists, Ramanujan deduced that the fractional sum can be defined by: 
\begin{equation}\label{Ramanujan_1} 
\sum_{k=a}^n z(k) = \sum_{k=a}^\infty z(k) - \sum_{k=n+1}^\infty z(k) = c_z(a) - c_z(n+1) 
\end{equation}
Here, $c_z(x)$ is the constant of the series $\sum_{k=x}^\infty z(k)$. Now, Ramanujan reasoned that: 
\begin{align*}\label{Ramanujan_2} 
\frac{d}{dn}\sum_{k=a}^n z(k) &= \frac{d}{dn}\big(\sum_{k=a}^\infty z(k) - \sum_{k=n+1}^\infty z(k)\big) \\
&= -\frac{d}{dn} c_z(n+1) = -\sum_{k=n+1}^\infty z'(k) \\
&= \sum_{k=a}^{n} z'(k) -  \sum_{k=a}^\infty z'(k) = \sum_{k=a}^{n} z'(k) - c_{z'}(a)
\end{align*}
Therefore, we recover Rule 1 of Summability Calculus. Such reasoning would have been correct if a precise and consistent definition of the constant $c$ existed. However, a consistent definition of $c$ for all series cannot be attained. For instance, suppose we have the function $f(n)=\sum_{k=1}^n (z(k)+\alpha)$ for some constant $\alpha$, then the above reasoning implies that $f'(n)$ is independent of $\alpha$, which is clearly incorrect because $f'(n)= \alpha + \frac{d}{dn} \sum_{k=1}^n z(k)$. 

In this section, we will show that Ramanujan's idea is correct but \emph{only if} the infinite sum $\sum_{k=a}^\infty z(k)$ is summable to a value $V\in\mathbb{C}$ using a summability method that is both regular, stable, and linear. Here, we will restrict analysis to the case of infinite sums that are well-defined in $\mathfrak{T}$. Now, introducing a constant $\alpha$ as in $f(n)=\sum_{k=1}^n (z(k)+\alpha)$ violates the summability condition so the method is well-defined. Therefore, we will indeed have that $\sum_{k=a}^n z(k) = \sum_{k=a}^\infty z(k) - \sum_{k=n+1}^\infty z(k)$. Moreover, we also have $\frac{d}{dn} \sum_{k=a}^n z(k) = \sum_{k=a}^n z'(k) - \sum_{k=a}^\infty z'(k)$, where all infinite sums are interpreted using the generalized definition $\mathfrak{T}$.  

We will begin our discussion with simple finite sums of the form $\sum_{k=a}^n e^{i\theta k} g(k)$  and state the general results afterwards. \\ \hrule

\begin{lemma}\label{Lemma531_1}
Given a simple finite sum of the form $f(n)=\sum_{k=a}^n e^{i\theta k} g(k)$, where the infinite sum $\sum_{k=a}^\infty e^{i\theta k} g(k)$ is defined by a value $V\in\mathbb{C}$ in $\mathfrak{T}$, then we have: 
\begin{equation}\label{Lemma531_1_1} 
\sum_{k=a}^n e^{i\theta k} g(k) = \sum_{k=a}^\infty e^{i\theta k} g(k) - \sum_{k=n+1}^\infty e^{i\theta k} g(k)
\end{equation} 
\begin{equation}\label{Lemma531_1_2} 
\frac{d}{dn} \sum_{k=a}^n e^{i\theta k} g(k) = \sum_{k=a}^n \frac{d}{dk} \big(e^{i\theta k} g(k)\big) - \sum_{k=a}^\infty \frac{d}{dk} \big(e^{i\theta k} g(k)\big)
\end{equation} 
\end{lemma}
\begin{proof}
Eq \ref{Lemma531_1_1} was already established in Theorem \ref{theorem513}. To prove Eq \ref{Lemma531_1_2}, we first note by Theorem \ref{theorem513} that: 
\begin{equation}\label{Lemma531_1_3}
f_G(n)=e^{i\theta a}\,g(a) - \sum_{r=0}^\infty \frac{\Theta_r}{r!}\big[e^{i\theta n} g^{(r)}(n) - e^{i\theta a} g^{(r)}(a)\big]
\end{equation}
In Theorem \ref{theorem513}, we have also shown that $\sum_{k=a}^\infty e^{i\theta k}\,g(k)$, interpreted using $\mathfrak{T}$, is formally given by:
\begin{equation}\label{Lemma531_1_4}
\sum_{k=a}^\infty e^{i\theta k}\,g(k) = \sum_{k=a}^n e^{i\theta k}g(k) + \sum_{r=0}^\infty \frac{\Theta_r}{r!} e^{i\theta n} g^{(r)}(n)
\end{equation}
Now, using the differentiation rule of simple finite sums, i.e. Rule 1 in Table \ref{TableRules}, we have:
\begin{equation}\label{Lemma531_1_5}
\frac{d}{dn} \sum_{k=a}^n e^{i\theta k} g(k) = \sum_{k=a}^n \frac{d}{dk} \big(e^{i\theta k} g(k)\big) +c
\end{equation}
On the other hand, differentiating both sides of Eq \ref{Lemma531_1_3} formally yields:
\begin{equation}\label{Lemma531_1_6}
f_G'(n) = -\sum_{r=0}^\infty \frac{\Theta_r}{r!} \frac{d}{dn}\big(e^{i\theta n} g^{(r)}(n)\big)
\end{equation}
Equating Eq \ref{Lemma531_1_5} and Eq \ref{Lemma531_1_6} yields:
\begin{equation}\label{Lemma531_1_7}
c= -\sum_{r=0}^\infty \frac{\Theta_r}{r!} \frac{d}{dn}\big(e^{i\theta n} g^{(r)}(n)\big) - \sum_{k=a}^n \frac{d}{dk} \big(e^{i\theta k} g(k)\big)
\end{equation}
Comparing last equation with Eq \ref{Lemma531_1_4} implies that: 
\begin{equation}\label{Lemma531_1_8}
c= -\sum_{k=a}^\infty \frac{d}{dk} \big(e^{i\theta k} g(k)\big)
\end{equation}
Plugging Eq \ref{Lemma531_1_8} into Eq \ref{Lemma531_1_5} yields the desired result. 
\end{proof} \hrule

\begin{lemma}\label{lemma531} 
Given a simple finite sum of the form $f(n)=\sum_{k=a}^n e^{i\theta k} g(k)$, where the infinite sum $\sum_{k=a}^\infty e^{i\theta k} g(k)$ is well-defined in $\mathfrak{T}$, then the unique natural generalization $f_G(n)$ is formally given by the series expansion: 
\begin{equation}\label{lemma531_Eq}
f_G(n) = \sum_{r=1}^\infty \frac{c_r}{r!} (n-a+1)^r, \\ \text{ where } c_r = -  \sum_{k=a}^\infty \frac{d^r}{dk^r} \big(e^{i\theta k} g(k)\big)
\end{equation}
Here, the infinite sums $c_r$ are interpreted using the generalized definition $\mathfrak{T}$. 
\end{lemma}
\begin{proof} 
Follows immediately from Lemma \ref{Lemma531_1}. 
\end{proof} \hrule\vspace{12pt} 

Note that in Lemma \ref{lemma531}, the function $f(n)$ behaves as if it were a convergent function with respect to the rules of infinitesimal calculus as discussed earlier. Such result can be generalized to oscillating sums of the form $\sum_{k=a}^n s_k\,g(k)$, in which $s_k$ is a periodic sign sequence that satisfies the condition $\sum_{k=a}^{a+p-1}s_k = 0$, where $s_k$ is decomposed into functions of the form $e^{i\theta k}$ using the Discrete Fourier Transform (DFT) as shown earlier in Lemma \ref{DFTs_k}. In fact, it can be generalized even further as the following theorem shows.\\ \hrule

\begin{theorem}\label{theorem531} 
Given a simple finite sum $\sum_{k=a}^n g(k)$, where $\sum_{k=a}^\infty g(k)$ exists in $\mathfrak{T}$, then we have: 
\begin{equation}\label{theorem531Eq1}
f_G^{(r)}(n) = \sum_{k=a}^n g^{(r)}(k) - \sum_{k=a}^\infty g^{(r)}(k)
\end{equation}
Similarly the indefinite integral is given by: 
\begin{equation}\label{theorem531Eq2}
\int^n f_G(t)\,dt = \sum_{k=a}^n \int^k g(t)\,dt + n\sum_{k=a}^\infty g(k) + c
\end{equation}
Here, $c$ is an arbitrary constant of integration, and all infinite sums are interpreted using the generalized definition $\mathfrak{T}$. 
\end{theorem}
\begin{proof}
We can rewrite $f(n)$ as $f(n)=\sum_{k=a}^n e^{i\theta k} (e^{-i\theta k} g(k))$. Note that this is always valid because both initial condition and recursive identity still hold for all $n$. Using results of Lemma \ref{lemma531}, Eq \ref{theorem531Eq1} also holds in the general case. The integral rule follows immediately, which is also analogous to the case of the convergent sums.
\end{proof} \hrule\vspace{12pt} 

One example of Theorem \ref{theorem531} was already discussed in Section \ref{Sect2dot2} in which we showed that the function $f(n)=\sum_{k=0}^n \sin{k}$ is given by Eq \ref{postTherm531_1}, where $\beta_1$ and $\beta_2$ can be determined using any two values of $f(n)$. 
\begin{equation}\label{postTherm531_1}
f_G(n)=\sum_{k=0}^n \sin{k} = \beta_1\sin{n} -\beta_2(1-\cos{n})
\end{equation}

For example, setting $n=1$ and $n=2$ yields $\beta_1=\frac{1}{2}$. However, we also showed in Section \ref{Sect2dot2} that $f_G'(0)=\beta_1$. Using Theorem \ref{theorem531}, we can also deduce the same value of $\beta_1$ as follows: 
\begin{equation}\label{postTherm531_2}
f_G'(n)=\sum_{k=0}^n \cos{k} -\sum_{k=0}^\infty \cos{k} \Rightarrow f_G'(0) = 1 -\sum_{k=0}^\infty \cos{k}
\end{equation}
However, $\sum_{k=0}^\infty \cos{k}$ is equal to $\frac{1}{2}$ in $\mathfrak{T}$. Therefore, $f_G'(0)=\beta_1 = \frac{1}{2}$, which is consistent with the earlier claim. 

\section{Summary of Results} 
In this chapter, we used the generalized definition of infinite sums $\mathfrak{T}$ given earlier in Chapter \ref{Chapter4} to deduce the analog of the Euler-Maclaurin summation formula for oscillating sums, which can be used to deduce asymptotic expressions of oscillating sums, accelerate series convergence, and even deduce exact analytic expressions of summable divergent sums as well. In addition, results of this chapter can also be used to deduce asymptotic expressions of non-oscillating finite sums as illustrated for the second factorials function. We have also shown the remarkable fact that, as far as the foundational rules of Summability Calculus are concerned, summable divergent series indeed behave as if they were convergent. 

\chapter{Direct Evaluation of Finite Sums} \label{Chapter6}
\epigraph{\emph{The future influences the present, just as much as the past.}}{F. Nietzsche (1844 -- 1900)}

So far, we have looked at performing infinitesimal calculus on finite sums and products, which led to the Euler-Maclaurin summation formula as well as a family of Euler-like summation formulas for oscillating sums. Formally speaking, a simple finite sum is given by the Euler-like summation formulas but using those infinite sums is often difficult since they typically diverge even for simple functions. However, as shown in Chapter \ref{Chapter2}, Summability Calculus can be used to deduce series expansion of finite sums which, in turn, can be used to evaluate such functions for all $n\in\mathbb{C}$ using an appropriate analytic summability method if needed.

Consequently, we so far have two methods of \emph{evaluating} convoluted finite sums of the form $\sum_{k=a}^n s_k\, g(k,n)$, where $s_k$ is an arbitrary periodic sign sequence. First, we can use the generalized Euler-like family of summation formulas if they converge. Second, we can compute the series expansion using Summability Calculus and use an appropriate summability method afterwards. Unfortunately, both methods are often impractical.

In this section, on the other hand, we provide a simple \emph{direct} method of computing $\sum_{k=a}^n s_k\, g(k,n)$ for all $n\in\mathbb{C}$. In the case of simple finite sums \simplefinitesum\, in which $g(n)$ is asymptotically of a finite differentiation order $m$, the fairly recent work of M\"uller and Schleicher \cite{Mueller2010} has captured the general principle. Here, and as discussed earlier in the Chapter \ref{ChaptIntro}, the basic idea is to evaluate the finite sum asymptotically using approximating polynomials and propagate results backwards using the recurrence identity $f(n)=g(n)+f(n-1)$. 

To be more specific, the M\"uller-Schleicher method can be described as follows. Suppose that the iterated function $g(k)$ is asymptotically given by a polynomial $P_s(k)$ with degree $s$, where the approximation error vanishes asymptotically, i.e. $\lim_{k\to\infty}\{g(k)-P_s(k)\} = 0$. Then, for any $N\in\mathbb{N}$ and any $n\in\mathbb{N}$, we have $\sum_{k=N}^{n+N}g(k) \sim \sum_{k=N}^{n+N} P_s(k)$, where the latter can be evaluated using the Bernoulli-Faulhaber formula, which is itself a polynomial with degree $(s+1)$. Because the latter formula holds for any $n\in\mathbb{N}$, it is natural to \emph{define} the sum for all $n\in\mathbb{C}$ by the same polynomial. Of course, this is merely an approximation as $N\to\infty$. However, we can use the approach to approximate the finite sum $\sum_{k=1}^n g(k)$ for all $n\in\mathbb{C}$ upon using: 
\begin{equation}\label{MSMethodEq1}
\sum_{k=N}^{n+N}g(k) = \sum_{k=1}^n g(k) + \sum_{k=n+1}^{n+N} g(k) -\sum_{k=1}^{N-1} g(k)    
\end{equation} 
Consequently, we have: 
\begin{equation}\label{MSMethodEq2}
\sum_{k=1}^n g(k) = \sum_{k=N}^{n+N}g(k) - \sum_{k=n+1}^{n+N} g(k) + \sum_{k=1}^{N-1} g(k)
\end{equation} 

Now, choosing $N\in\mathbb{N}$ allows us to evaluate the sums $\sum_{k=n+1}^{n+N} g(k)$ and $ \sum_1^{N-1} g(k)$ directly by definition. Similarly, taking $N\to\infty$ allows us to evaluate the fractional sum $\sum_{k=N}^{n+N}g(k)$ with arbitrary accuracy as discussed earlier. Consequently, we can evaluate the original fractional sum $\sum_{k=a}^n g(k)$ with an arbitrary degree of accuracy by using the translation invariance property $\sum_{k=a}^n g(k) = \sum_{k=1}^{n-a} g(k+a)$ and by applying equation above at  the limit $N\to\infty$. Obviously, such approach is not restricted to analytic functions. For instance, if we define $f(n)=\sum_{k=1}^n \frac{1}{\lceil k\rceil}$, where $\lceil x \rceil$ is the \emph{ceiling} function, then the natural generalization implied by the above method is given by the \emph{discrete harmonic function} $f(n) = H_{\lceil n \rceil}$. This is the essence of the M\"uller-Schleicher method. 

However, if $g(n)$ is analytic, then we can Taylor's theorem to deduce that if $g^{(m+1)}(n) \to 0$ as $n\to\infty$, then $g(n)$ is asymptotically given by its first $(m+1)$ terms of the Taylor series expansion. That is, $g(n)$ in the latter case is indeed asymptotically approximated by a polynomial. This can be easily established, for example, by using the Lagrange reminder form. In this section, we show that the aforementioned approach is merely a special case of a more general formal method. Here, we present the general statement that is applicable to simple finite sums, even those in which $g(n)$ is not asymptotically of a finite differentiation order, convoluted finite sums of the form \convolutedsum\, and even oscillating convoluted finite sums of the form $\sum_{k=a}^n s_k\, g(k,n)$ for some arbitrary periodic sign sequence $s_k$. In other words, we extend the approach to the general case in which $g(k)$ may or may not be approximated by polynomials asymptotically. In addition, the statement readily provides a method of \emph{accelearing} convergence speed. 

Finally, we will establish the earlier claim that the Euler-Maclaurin summation formula is the unique natural generalization to simple finite sums by showing that it is the \emph{unique} generalization that arises out of polynomial fitting. Similar to the approach of Chapter \ref{Chapter2}, we will begin our treatment with the case of semi-linear simple finite sums and generalize results afterwards. 

\section{Evaluating Semi-Linear Simple Finite Sums} \label{Section6dot1}
The case of semi-linear simple finite sums is summarized in the following theorem. \\
\hrule 
\begin{theorem}\label{theorem611} 
Given a \emph{semi-linear} simple finite sum of the form \simplefinitesum, i.e. where $g(n)$ is nearly-convergent, then its unique natural generalization $f_G(n)$ for all $n\in\mathbb{C}$ can be evaluated using the following expression:
\begin{equation}\label{theorem611Eq}
f_G(n) = \lim_{s\to\infty} \Big\{(n-a+1)\,g(s) + \sum_{k=0}^s g(k+a)-g(k+n+1)\Big\}
\end{equation}
\end{theorem}
\begin{proof}
The proof consists of three parts. First, we need to show that the initial condition holds. Second, we need to show that the required recurrence identity also holds. Third, we need to show that $f_G(n)$ given by Eq \ref{theorem611Eq} is equivalent to the unique natural generalization implied by Summability Calculus.

The proof that the initial condition holds is straightforward. Plugging $n=a-1$ into Eq \ref{theorem611Eq} yields $f(a-1)=0$. To show that the recurrence identity holds, we note that:
\begin{equation}\label{theorem611_1}
f_G(n) - f_G(n-1) = \lim_{s\to\infty} \big\{g(n)+g(s)-g(s+n+1)\big\}
\end{equation}
However, because $g(n)$ is nearly-convergent by assumption, then we have:
\begin{equation}\label{theorem611_2}
\lim_{s\to\infty} \big\{g(s)-g(s+n+1)\big\} = 0, \\ \text{ for all } n\in\mathbb{C}
\end{equation}
Thus, the function $f_G(n)$ given by Eq \ref{theorem611Eq} indeed satisfies the recurrence identity:
\begin{equation}\label{theorem611_3}
f_G(n)-f_G(n-1) = g(n)
\end{equation}
Finally, to show that $f_G(n)$ is identical to the unique natural generalization implied by Summability Calculus, we differentiate both sides of Eq \ref{theorem611Eq}, which yields: 
\begin{equation}\label{theorem611_4}
f_G'(n)= \lim_{s\to\infty} \big\{g(s)-\sum_{k=0}^s g'(k+n+a)\big\} = \sum_{k=a}^n g'(k) + \lim_{s\to\infty}\big\{g(s)-\sum_{k=a}^s g'(k)\big\}
\end{equation}
However, last equation is identical to Theorem \ref{theorem231}. Thus, $f_G(n)$ given by Eq \ref{theorem611Eq} is indeed the unique natural generalization to the simple finite sum.
\end{proof} \hrule
\begin{corollary}\label{corollary611} 
If $g(n)\to 0$ as $n\to\infty$, then the unique natural generalization to the simple finite sum $f(n)=\sum_{k=a}^n g(k)$ for all $n\in\mathbb{C}$ is given by: 
\begin{equation}\label{corollary611Eq} 
f_G(n) =  \sum_{k=0}^\infty \big(g(k+a)-g(k+n+1)\big)
\end{equation}
\end{corollary} 
\begin{proof}
Follows immediately by Theorem \ref{theorem611}.
\end{proof} \hrule\vspace{12pt} 

As discussed earlier, because a simple finite sum is asymptotically linear in any bounded region $W$, we can push $W$ to infinity such that the simple finite sum becomes exactly linear. In such case, we know that $\sum_{k=s}^{s+W} g(k) \to (W+1)\,g(s)$ for any fixed $0\le W < \infty$. This provides a convenient method of evaluating fractional sums. To evaluate simple finite sums of the form $\sum_{k=a}^n g(k)$, for $n\in\mathbb{C}$, we use the backward recurrence relation $f(n-1)=f(n)-g(n)$. Theorem \ref{theorem611} shows that such approach yields the unique natural generalization to simple finite sums as defined by Summability Calculus. Of course, since Summability Calculus on simple finite sums was built to exploit such semi-linear property of simple finite sums, we expect both approaches to be equivalent, which is indeed the case. We will illustrate the statement of Theorem \ref{theorem611} using three examples. 

First, we will start with the log-factorial function $\varpi(n) = \sum_{k=1}^n \log{k}$ . Because $\log{n}$ is nearly-convergent, we use Theorem \ref{theorem611}, which yields:
\begin{align*} 
\log{n!} &= \lim_{s\to\infty} \Big\{ n\log{s} + \sum_{k=0}^s \log{\big(\frac{k+1}{k+n+1}\big)}\\
&= \lim_{s\to\infty} \Big\{ n\sum_{k=1}^s \log{\big(1+\frac{1}{k}\big)} +\sum_{k=0}^s \log{\big(\frac{k+1}{k+n+1}\big)} \\ 
&= -\log{(1+n)} + \sum_{k=1}^\infty \log{\Big(\big(1+\frac{1}{k}\big)^n \frac{k+1}{k+n+1}\Big)} 
\end{align*} 
Here, we used the fact that $\lim_{s\to\infty} \log{(1+\frac{1}{s})} =0$. Therefore, we have: 
\begin{equation}\label{EulerProduct} 
n! = \prod_{k=1}^\infty \big(1+\frac{1}{k}\big)^n \frac{k}{k+n}
\end{equation} 

Eq \ref{EulerProduct} is the famous Euler infinite product formula for the factorial function. In addition, we know by Theorem \ref{theorem611} that Eq \ref{EulerProduct} is an alternative definition of $\Gamma(n+1)$, where $\Gamma$ is the Gamma function because the log-Gamma function is the unique natural generalization of the log-factorial function as proved earlier in Lemma \ref{factorial_gamma}. 
 
Our second example is the harmonic sum $\sum_{k=1}^n \frac{1}{k}$. Using Corollary \ref{corollary611}, we immediately have the well-known expression in Eq \ref{Eq61Harmonic}. Of course, the expression in Eq \ref{Eq61Harmonic} is simple to prove for all $n\in\mathbb{N}$ but we reiterate here that Theorem \ref{theorem611} shows it actually holds for all $n\in\mathbb{C}$. 
\begin{equation}\label{Eq61Harmonic}
\sum_{k=1}^n \frac{1}{k} = \sum_1^\infty \Big(\frac{1}{k} - \frac{1}{k+n}\Big) = \sum_{k=1}^\infty \frac{n}{k(k+n)}
\end{equation} 

Our third and last example is the sum of square roots function $\sum_{k=1}^n \sqrt{k}$. Previously, we derived its series expansion in Eq \ref{Example_261_5} whose radius of convergence was $|n|\le 1$. In Section \ref{Section4dot2_2}, we used the summability method $\Xi$ to evaluate the Taylor series expansion in a larger region over the complex plane $\mathbb{C}$. Now, we use Theorem \ref{theorem611} to evaluate the function for all $n\in\mathbb{C}$. The sum of square roots function is plotted in Figure \ref{NewSumSqrt}. The values highlighted in green are evaluated by definition. Clearly, the function that results from applying Theorem \ref{theorem611} to the sum of square roots function $\sum_{k=1}^n \sqrt{k}$ correctly interpolates the discrete points as expected \footnote{Here, however, it appears that the function $\sum_{k=1}^n \sqrt{k}$ becomes complex-valued for $n<-1$}. 

\begin{figure} [h]
\centering
\includegraphics[scale=0.4]{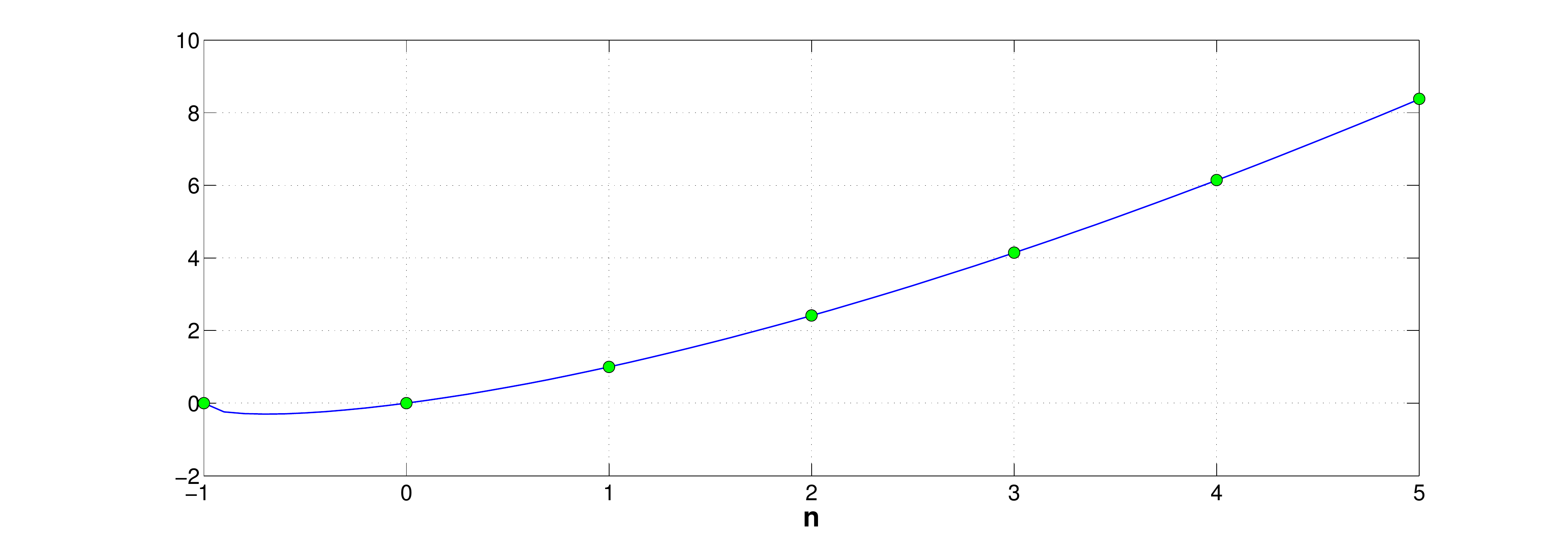}
\caption[The function $\sum_{k=1}^n \sqrt{k}$ plotted for $n \ge -1$]{The function $\sum_{k=1}^n \sqrt{k}$ evaluated using Theorem \ref{theorem611}. Values highlighted in green are exact and are evaluated directly by definition. }
\label{NewSumSqrt}
\end{figure}

\section{Evaluating Arbitrary Simple Finite Sums} \label{Section6dot2}
To evaluate arbitrary simple finite sums, we extend results of the previous chapter as the following theorem shows. \\ \hrule

\begin{theorem}\label{theorem621}
Given a simple finite sum of the form $f(n)=\sum_{k=a}^n g(k)$, then $f(n)$ for all $n\in\mathbb{C}$ is formally given by:
\begin{equation}\label{theorem621Eq1}
f(n) = \sum_{r=0}^\infty \frac{b_r(n)-b_r(a-1)}{r!}\,g^{(r)}(s) + \sum_{k=0}^s g(k+a)-g(k+n+1)
\end{equation} 
Here, Eq \ref{theorem621Eq1} holds \emph{formally} for any value $s$, where $b_r(n)=\sum_{k=1}^{n+1} k^r$, i.e. a polynomial given by the Bernoulli-Faulhaber formulaa. In the special case where $g(s)$ is asymptotically of finite differentiation order $m$, i.e. $g^{(m+1)}(s)\to 0$ as $s\to\infty$, then $f(n)$ can also be computed using the following limiting expression:
\begin{equation}\label{theorem621Eq2}
f(n) = \lim_{s\to\infty} \Big\{\sum_{r=0}^m \frac{b_r(n)-b_r(a-1)}{r!}\,g^{(r)}(s) + \sum_{k=0}^s g(k+a)-g(k+n+1)\Big\}
\end{equation} 
\end{theorem}
\begin{proof} 
Again, the proof consists of three parts. 

To show that the initial condition holds is straightforward. Plugging $n=a-1$ into Eq \ref{theorem621Eq1} yields $f(a-1)=0$. To show that the recurrence identity holds, we note,  as stated in the theorem, that the expression is formal so we will ignore the issue of convergence. Thus, starting from Eq \ref{theorem621Eq1}, we have:   
\begin{equation}\label{theorem621_2} 
f(n)-f(n-1) = \Big(\sum_{r=0}^\infty \frac{g^{(r)}(s)}{r!}(n+1)^r\Big)\, +g(n)-g(n+s+1)
\end{equation} 
The first term is a Taylor series expansion and it is formally equal to $g(n+s+1)$. Consequently, we indeed  have:  
\begin{equation}\label{theorem621_3} 
f(n)-f(n-1) =g(n)
\end{equation} 
So, the recurrence identity holds formally. Finally, to show that $f(n)$ given by Eq \ref{theorem621Eq1} is identical to the unique natural generalization implied by Summability Calculus, we differentiate both sides of Eq \ref{theorem621Eq1}, which yields: 
\begin{equation}\label{theorem621_4} 
f'(n) =  \sum_{r=0}^\infty \frac{b_r'(n)}{r!}\,g^{(r)}(s) - \sum_{k=0}^s g'(k+n+a)
\end{equation} 
Upon rearranging the terms, we have: 
\begin{equation}\label{theorem621_5} 
f'(n) = \sum_{k=a}^n g'(k) + \Big[ \sum_{r=0}^\infty  \frac{b_r'(n)}{r!}\,g^{(r)}(s) - \sum_{k=a}^{s+n+1} g'(k)\Big]
\end{equation} 
Comparing last equation with Rule 1 of Table \ref{TableRules} reveals that the term in square brackets must be independent of $n$. Therefore, we select $n=-1$, which yields: 
\begin{equation}\label{theorem621_6} 
f'(n) = \sum_{k=a}^n g'(k) + \Big[\sum_{r=0}^\infty  \frac{B_r}{r!}\,g^{(r)}(s) - \sum_{k=a}^{s} g'(k)\Big]
\end{equation} 
Here, $B_r$ are Bernoulli numbers. However, Eq \ref{theorem621_6} is identical to Theorem \ref{theorem251}. Therefore, $f(n)$ given by Eq \ref{theorem621Eq1} is indeed the unique natural generalization implied by Summability Calculus.
\end{proof} \hrule\vspace{12pt} 

One example of Theorem \ref{theorem621} is the hyperfactorial function $H(n)= \prod_{k=1}^n k^k$. Applying Theorem \ref{theorem621} to $\log{H(n)}$ yields the Euler-like infinite product formula:
\begin{equation}\label{theorem621_Ex_1}
\prod_{k=1}^n k^k = 4\,e^{n(n-1)/2-1} \prod_{k=2}^\infty \frac{(k-1)^n\,(1-\frac{1}{k})^{-k(n-2)-n(n-1)/2}\,(k+1)^{k+1}}{k(k-1)(k+n-1)^{k+n-1}}
\end{equation}
For example, setting $n=0$ in the the last equation yields: 
\begin{equation} \label{theorem621_Ex_2}
\prod_{k=2}^\infty (1+\frac{1}{k})(1-\frac{1}{k^2})^k = \frac{e}{4}
\end{equation}

Similarly, applying Theorem \ref{theorem621} to the superfactorial function $S(n)=\prod_{k=1}^n k!$  yields Eq \ref{theorem621_Ex_3}. Note that the same expression can be alternatively derived using Euler's infinite product formula for the factorial function.
\begin{equation} \label{theorem621_Ex_3}
\prod_{k=1}^n k! = \prod_{k=1}^\infty \big(1+\frac{1}{k}\big)^{n(n+1)/2}\;k^n\,\frac{k!}{\Gamma{(k+n+1)}}
\end{equation}
One immediate consequence of Eq \ref{theorem621_Ex_3} is given by: 
\begin{equation} \label{theorem621_Ex_4}
\frac{S'(0)}{S(0)} = \sum_{k=1}^\infty \Big(\frac{\log{k} + \log{(k+1)}}{2} \, + \lambda - H_k\Big)
\end{equation}

Using the series expansion for the log-superfactorial function that was derived earlier in Eq \ref{Eq263_12}, we have:
\begin{equation} \label{theorem621_Ex_5}
\sum_{k=1}^\infty \Big(\frac{\log{k} + \log{(k+1)}}{2} \, + \lambda - H_k\Big) = \frac{\log{(2\pi)}\, -1}{2} - \lambda 
\end{equation}

Eq \ref{theorem621_Ex_5} shows that the harmonic number $H_n$ converges to the average $(\log{k} + \log{(k+1)})/2$ much faster than its convergence to $\log{n}$. This is clearly shown in the fact that the sum of all approximation errors given by the earlier equation converges whereas the sum $\sum_{k=1}^\infty (\log{k} + \lambda - H_k)$ diverges. Note that, by contrast, we have Eq \ref{AlternHarmonicSum_2}. 

If $g(n)$ is asymptotically of a finite differentiation order $m$, we can use Eq \ref{theorem621Eq2}. However, adding more terms to the sum $\sum_{r=0}^m \frac{b_r(n)-b_r(a-1)}{r!}\,g^{(r)}(s)$ improves convergence speed. For example, if we return to the factorial function and choose $m=1$, we obtain: 
\begin{equation}\label{speedFactorial}
\log{n!} = \lim_{s\to\infty} \Big\{ n\log{s} + \frac{n^2+3n}{2s} + \sum_{k=1}^{s+1} \log{\frac{k}{k+n}} \Big\}
\end{equation}
By contrast, Euler's infinite product formula omits the term $\frac{n^2+3n}{2s}$. So, if we wish to compute a value such as $\Gamma(\pi+1)$, and choose $s=10^4$, we obtain numerically the two estimates  $\Gamma(\pi+1)\approx 1.9715$ and $\Gamma(\pi+1)\approx 1.9724$, which correspond to choosing $m=0$ and $m=1$ respectively. Here, the approximation error is $9\times 10^{-4}$ if $m=0$ but it is less than $1.6\times 10^{-7}$ if we use $m=1$. So, indeed, the formal expression in Theorem \ref{theorem621} yields a method of improving convergence speed.

Before we conclude this section, it is worth pointing out that Theorem \ref{theorem621} indicates that the Euler-Maclaurin summation formula is indeed closely related to Taylor's theorem. In fact, the expression in Eq \ref{theorem621Eq1} resembles Taylor series expansions, and we do obtain a Taylor series by simply subtracting $f(n-1)$ from $f(n)$. It should not be surprising, therefore, to note that analytic summability methods that are consistent with the $\mathfrak{T}$ definition of infinite sums can work reasonably well even with the Euler-Maclaurin summation formula. We have already illustrated this in Table \ref{Tab43Bern}, in which we showed that the sum $\sum_{r=0}^\infty B_r$ is \emph{nearly} summable using $\Xi$. In particular, using $\Xi$ with small $n$ yields reasonably accurate figures, and accuracy improves as we increase $n$ until we reach about $n=35$, after which the error begins to increase. This is also true for the sum $\sum_{r=1}^\infty \frac{B_r}{r}$, which is formally equivalent to Euler's constant $\lambda$ as discussed earlier. 

Such close correspondence between the Euler-Maclaurin summation formula and Taylor series expansion also shows that the Euler-Maclaurin formula is an asymptotic expansion. Here, if $g(n)$ is asymptotically of a finite differentiation order $m$, then by Taylor's theorem: 
\begin{equation}\label{EMisAsym}
\lim_{n\to\infty} \big\{g(n+h) - g(n) - \sum_{r=1}^m \frac{g^{(r)}(n)}{r!}\,h^r\} = 0
\end{equation}

Eq \ref{EMisAsym} is easily established upon using, for example, the Lagrange remainder form. Now, looking into Eq \ref{theorem621Eq2}, we see that the expression indeed satisfies initial condition and recurrence identity by referring to the earlier equation \ref{EMisAsym}. Consequently, if $g(n)$ is asymptotically of a finite differentiation order $m$, we have by differentiating both sides of Eq \ref{theorem621Eq2} the following result: 
\begin{equation}\label{EMisAsym_2}
\lim_{n\to\infty} \Big\{ \frac{d}{dn} \sum_{k=a}^n g(k) - \sum_{r=0}^m \frac{B_r}{r!}\,g^{(r)}(n)\Big\} = 0
\end{equation}

Last equation was the basis of many identities derived in Section \ref{sectionSFSGC}. This is a simple informal justification as to why the Euler-Maclaurin summation formula is more correctly interpreted as an asymptotic series expansion.

Finally, in Section \ref{Section_23}, we proved that Summability Calculus yields unique most natural generalization for semi-linear simple finite sums by reducing argument for unique natural generalization to the case of linear fitting since all semi-linear simple finite sums are asymptotically linear in bounded regions. Here, we can see that for the general case of simple finite sums, Summability Calculus yields unique natural generalization as well because it implicitly uses polynomial fitting. Because polynomial fitting is the simplest generalization if it can correctly interpolate discrete points, and given lack of additional special requirements, it is indeed the unique most natural generalization. So, where does polynomial fitting show up in Summability Calculus? 

To answer the question, we return to Theorem \ref{theorem621}. In the proof of Eq \ref{theorem621Eq1}, we note that the crucial property of the functions $b_r(x)$  that make $f_G(n)$ satisfy initial conditions and recurrence identity is the condition $b_r(x)-b_r(x-1)=(x+1)^r$. However, it is obvious that an infinite number of functions $b_r(x)$  exist that can satisfy such condition. Nevertheless, because the only condition required happens to be satisfied by the Bernoulli-Faulhaber formula, which is the unique \emph{polynomial solution} to the required condition, defining $b_r(x)=\sum_{k=0}^{x+1} k^r$ for fractional $x$ by the Bernoulli-Faulhaber formula is equivalent to polynomial fitting. Such choice of generalization coincides with the Euler-Maclaurin summation formula as proved in Theorem \ref{theorem621}. Consequently, the argument for unique natural generalization in Summability Calculus, in the general case, is reduced to the argument of polynomial fitting. 

\section{Evaluating Oscillating Simple Finite Sums}\label{Section6dot3}
In Chapter \ref{Chapter2}, we showed that a simple finite sum is formally given by the Euler-Maclaurin summation formula. In addition, we showed that simple finite sums of the form $\sum_{k=a}^n g(k)$  in which $g(n)$ is asymptotically of a finite differentiation order are easier to work with. For instance, instead of evaluating the entire Euler-Maclaurin summation formula, we can exploit the fact that $g(n)$ is asymptotically of a finite differentiation order and use  Eq \ref{corollary251_3Eq} instead if we wish to compute the derivative. Similarly, such advantage shows up again in Section \ref{Section6dot2} as stated in Eq \ref{theorem621Eq2}. 

On the other hand, if $g(k)$ is oscillating, i.e. of the form $g(k)=s_k\,z(k)$, where $s_k$ is an arbitrary non-constant periodic sign sequence, then $g(k)$ is not asymptotically of a finite differentiation order. Nevertheless, if $z(k)$ is, then we can use the Euler-like summation formula for alternating sums to exploit such advantage as well. In fact, this was the main motivation behind introducing such Euler-like family of summations, which allowed us to easily perform infinitesimal calculus and deduce asymptotic behavior as was done repeatedly throughout Chapter \ref{Chapter5}.

The objective of this section is to present the analog of Theorem \ref{theorem621} for oscillating simple finite sums. Similar to earlier results, this will provide us with a much simpler method of handling such sums. Because each function $\sum_{k=a}^n s_k\,z(k)$, where $s_k$ is periodic with period $p$, can be decomposed into a sum of a finite number of functions of the form $\sum_{k=a}^n e^{i \theta k} z(k)$ , as shown in Corollary \ref{corollary522}, we will focus on the latter class of functions. Again, the first component given by $\big(\frac{1}{p}\sum_{k=a}^{a+p-1} s_k\big) \sum_{k=a}^n g(k)$ is evaluated using Theorem \ref{theorem621} but all other components in which $\theta \neq 0$  are evaluated using the following Theorem. \\ \hrule 

\begin{theorem}\label{theorem631}
Given a simple finite sum of the form $f(n)=\sum_{k=a}^n e^{i\theta k} g(k)$, then its unique natural generalization $f_G(n)$ for all $n\in\mathbb{C}$ s formally given by Eq \ref{theorem631Eq1}, which holds for all $s$.
\small
\begin{equation}\label{theorem631Eq1} 
f_G(n) = e^{i\theta s} \sum_{r=0}^\infty \frac{\Omega_r(n) - \Omega_r(a-1)}{r!} g^{(r)}(s) + \sum_{k=0}^s e^{i\theta(k+a)}\,g(k+a)-e^{i\theta(k+n+1)}\,g(k+n+1)
\end{equation}
\normalsize
Here, $\Omega_r(n)=\sum_{k=1}^{n+1} e^{i\theta k}\,k^r$, i.e. an \lq\lq oscillating polynomial''. It is closed-form is given by Eq \ref{theorem631_4} below. 
\end{theorem}
\begin{proof} 
The proof that initial condition and recurrence identity hold is straightforward and is similar to the proof of Theorem \ref{theorem621}. To show that it is indeed the unique natural generalization implied by Summability Calculus, we differentiate both sides, which yields:
\begin{equation}\label{theorem631_1} 
f'(n) = e^{i\theta s} \sum_{r=0}^\infty \frac{\Omega'_r(n)}{r!}\,g^{(r)}(s) - \sum_{k=0}^s \frac{d}{dn} \big(e^{i\theta(k+n+1)}\,g(k+n+1)\big)
\end{equation}
Upon rearranging the terms: 
\begin{equation}\label{theorem631_2} 
f'(n) = \sum_{k=a}^n \frac{d}{dk}\big(e^{i\theta k}\,g(k)\big) + \Big[e^{i\theta s} \sum_{r=0}^\infty \frac{\Omega'_r(n)}{r!}\,g^{(r)}(s) - \sum_{k=a}^{s+n+1} \frac{d}{dk} \big(e^{i\theta k}\,g(k)\big)\Big]
\end{equation}
Again, comparing last equation with the differentiation rule of simple finite sums, i.e. Rule 1, we see that the second term must be independent of $n$. Therefore, choosing $n=-1$ yields: 
\begin{equation}\label{theorem631_3} 
f'(n) = \sum_{k=a}^n \frac{d}{dk}\big(e^{i\theta k}\,g(k)\big) + \Big[e^{i\theta s} \sum_{r=0}^\infty \frac{\Omega'_r(-1)}{r!}\,g^{(r)}(s) - \sum_{k=a}^{s} \frac{d}{dk} \big(e^{i\theta k}\,g(k)\big)\Big]
\end{equation}
Using Theorem \ref{theorem513}, we note that: 
\begin{equation}\label{theorem631_4}
\Omega_r(n) = \Theta_r - e^{i\theta(n+1)}\,\sum_{m=0}^r \binom{r}{m}\,\Theta_m\,(n+1)^{r-m}
\end{equation} 
Thus, Eq \ref{theorem631_3} can be rewritten as: 
\begin{equation}\label{theorem631_5} 
f'(n) = \sum_{k=a}^n \frac{d}{dk}\big(e^{i\theta k}\,g(k)\big) - \sum_{r=0}^\infty \frac{\Theta_r}{r!}\,\frac{d}{ds} (e^{i\theta s}\,g^{(r)}(s)) - \sum_{k=a}^{s} \frac{d}{dk} \big(e^{i\theta k}\,g(k)\big)
\end{equation}
Comparing this with Rule 1 reveals that: 
\begin{equation}\label{theorem631_6} 
f'(n) = - \sum_{r=0}^\infty \frac{\Theta_r}{r!}\,\frac{d}{ds} (e^{i\theta s}\,g^{(r)}(s))
\end{equation}
However, this is exactly what Theorem \ref{theorem513} states. Therefore, the function $f(n)$ given by Eq \ref{theorem631Eq1} is indeed the unique natural generalization $f_G(n)$ to the simple oscillating finite sum.
\end{proof} \hrule\vspace{12pt} 
As an example to Theorem \ref{theorem631}, suppose we would like to evaluate the alternating simple finite sum $\sum_{k=1}^\frac{1}{2} (-1)^k \log{k}$. Because $g(n)$ is asymptotically of a finite differentiation order zero, we have: 
\small
\begin{equation}\label{theorem631firstEx}
\sum_{k=1}^\frac{1}{2} (-1)^k \log{k} = \lim_{s\to\infty} \Big\{ (-1)^s (\Omega_0(\frac{1}{2}) - \Omega_0(0))\,\log{s} + \sum_{k=1}^s (-1)^{k}\log{k} - (-1)^{k+\frac{1}{2}} \log{(k+\frac{1}{2})}   \Big\} 
\end{equation}
\normalsize 
However, by know by definition that $\Omega_0(x) = \sum_{k=1}^{x+1} (-1)^k = -\frac{1}{2}((-1)^x+1)$. Thus, we have: 
\begin{equation}
\sum_{k=1}^\frac{1}{2} (-1)^k \log{k} = \lim_{s\to\infty} \Big\{ (-1)^s \frac{1-i}{2} \,\log{s} + \sum_{k=1}^s (-1)^{k}\log{k} - (-1)^{k+\frac{1}{2}} \log{(k+\frac{1}{2})}   \Big\} 
\end{equation}

Finally, this approach yields the value $\sum_{k=1}^\frac{1}{2} (-1)^k \log{k} \approx 0.2258 + i\,0.0450$. This is indeed correct as will be verified shortly. Here, the real part is given by $(\log{\pi}-\log{2})/2$. 

Whereas Theorem \ref{theorem631} is occasionally useful, it is often easier in practice to work with the analytic summability identity $\sum_{k=a}^n e^{i\theta k}\,g(k) = \sum_{k=a}^\infty e^{i\theta k}\,g(k) - \sum_{k=n+1}^\infty e^{i\theta k}\,g(k)$, where all infinite sums are interpreted using the $\mathfrak{T}$ definition. For example, suppose we would like to evaluate $\sum_{k=0}^{\frac{1}{2}} (-1)^k\,k$ directly, then we have: 
\small
\begin{equation}\label{theorem631Ex_1} 
\sum_{k=0}^\frac{1}{2} (-1)^k\,k = \sum_{k=0}^\infty (-1)^k\,k - \sum_{k=\frac{3}{2}}^\infty e^{i\pi k}\,k = \sum_{k=1}^\infty (-1)^k\,k - i \sum_{k=1}^\infty (-1)^k(k+\frac{1}{2}) = -\frac{1}{4} + i \frac{1}{2}
\end{equation} \normalsize

A second example would be the finite sum $\sum_{k=1}^\frac{1}{2} (-1)^k \log{k}$ addressed above. We can evaluate this expression directly using the analytic summability approach, which yields:
\begin{align*} 
\sum_{k=1}^\frac{1}{2} (-1)^k\,\log{k} &= \sum_{k=1}^\infty (-1)^k\,\log{k} - i \sum_{k=\frac{3}{2}}^\infty (-1)^{k-\frac{1}{2}} \log{k} \\ 
&= \frac{\log{\pi}-\log{2}}{2} + i \big(\log{\frac{3}{2}} -\log{\frac{5}{2}} + \log{\frac{7}{2}} - \dotsm\big) \\
&= \frac{\log{\pi}-\log{2}}{2} + i \big(\log{3} - \log{5} + \log{7} - \dotsm\big) - i \frac{\log{2}}{2}
\end{align*}

Now, we evaluate the sum $(\log{3} - \log{5} + \log{7} - \dotsm\big)$ using an analytic summability method such as $\Xi$, which yields a value of $0.3916$. Plugging this into last equation yields the desired result, which is also given by $0.2258 + i\,0.0450$. This is indeed identical to the value derived earlier by direct application of Theorem \ref{theorem631}. Of course, we can also calculate the same value by direct application of the summability method $\Xi$ to the two infinite sums $\sum_{k=1}^\infty (-1)^k\,\log{k}$ and $\sum_{k=\frac{3}{2}}^\infty (-1)^{k-\frac{1}{2}} \log{k}$, which yields the same final result. 

\section{Evaluating Convoluted Finite Sums} 
Former analysis can be generalized readily to the case of convoluted finite sums. We present the final results herein and omit the proofs to avoid repetitiveness. In general, the proofs are very similar to the proofs of earlier theorems. 
\begin{theorem}\label{theorem641}
Given a convoluted finite sum of the form $\sum_{k=a}^n e^{i\theta k}\,g(k,n)$ for any $\theta \in \mathbb{R}$, then its unique natural generalization is given formally by: 

\begin{equation}\label{theorem641Eq} 
f_G(n) = e^{i\theta s} \sum_{r=0}^\infty \frac{\Omega_r(n) - \Omega_r(a-1)}{r!}\frac{\partial^r}{\partial s^r} g(s, n) + \sum_{k=0}^s e^{i\theta(k+a)}g(k+a,n) - e^{i\theta(k+n)}g(k+n+1,n)
\end{equation} 

If $g(k, n)$ is asymptotically of a finite differentiation order $m$ with respect to the iteration variable $k$, then the infinite sum can be evaluated up to $m$ only and the overall expression is taken as a limiting expression where s tends to infinity. This is similar to earlier statements. In addition, if $\theta=2\pi n$ for some $n\in\mathbb{N}$, then $\Omega_r(x) = b_r(x)$, where the functions $b_r$ are defined in Theorem \ref{theorem621}. 
\end{theorem} 
\begin{proof} 
Similar to the proofs of earlier theorems. However, we note here that no recurrence identity holds so the proof consists of two parts: (1) showing that the initial condition holds, and (2) showing that the derivative of the expression in Eq \ref{theorem641Eq} is identical to the derivative of the unique most natural generalizations given in Chapter \ref{Chapter3} and Chapter \ref{Chapter5}. 
\end{proof} \hrule\vspace{12pt} 

Theorem \ref{theorem641} generalizes earlier results of this chapter to the broader class of convoluted sums. For example, suppose we wish to evaluate the convoluted finite sum $\sum_{k=1}^n \log{(1+\frac{k}{n})}$. Since $g(k,n)$ is asymptotically of a finite differentiation order zero with respect to $k$, we can evaluate $f_G(n)$ for all $n\in\mathbb{C}$ using the following expression:
\begin{equation}\label{thoerem641Ex1}
f_G(n) = \lim_{s\to\infty} \Big\{n \log{(1+\frac{s}{n})} + \sum_{k=0}^s \big[\log{(1+\frac{k+1}{n})} - \log{(2+\frac{k+1}{2})}\big]\Big\}
\end{equation} 

For example, if $n=2$, Eq \ref{thoerem641Ex1} evaluates to 1.0986, which we know is correct because $f_G(2)=\log{3}$. In general, we know in this particular example that $f_G(n)= \log{(2n)!} - \log{n!} - n\log{n}$. Therefore, we can also test Eq \ref{thoerem641Ex1} for fractional values of $n$. For instance, if $n=\frac{1}{2}$, then Eq \ref{thoerem641Ex1} evaluates to 0.4674, which is also correct because $f_G(\frac{1}{2}) = \frac{\log{2} - \log{\sqrt{\pi}}}{2}$. 

A second example would be the convoluted finite sum $\frac{1}{n} \sum_{k=1}^n k^{\frac{1}{n}}$. Because $\lim_{n\to\infty} n^\frac{1}{n} = 1$, we expect $f_G(n)\to 1$ as $n\to\infty$. In addition, $g(k)=k^\frac{1}{n}$ is asymptotically of a finite differentiation order zero if $n>1$ so the convoluted sum can be evaluated using: 
\begin{equation}\label{thoerem641Ex2} 
\frac{1}{n}\sum_{k=1}^n k^\frac{1}{n} = \lim_{s\to\infty} s^{\frac{1}{n}} + \frac{1}{n} \sum_{k=1}^{s+1} k^\frac{1}{n} - (k+n)^\frac{1}{n} \\ \text{ if } n>1
\end{equation} 
The function is plotted in Figure \ref{Chap6ConvEx2Figure}. In same figure, exact values of $f(n)$ for $n\in\mathbb{N}$ are indicated as well in green. Clearly, despite the apparent complexity of the convoluted finite sum, it is indeed a simple function to compute for $n\in\mathbb{C}$. In addition, we have $f_G(n) \to 1$ when $n\to\infty$ as expected. 
 
\begin{figure} [h]
\centering
\includegraphics[scale=0.4]{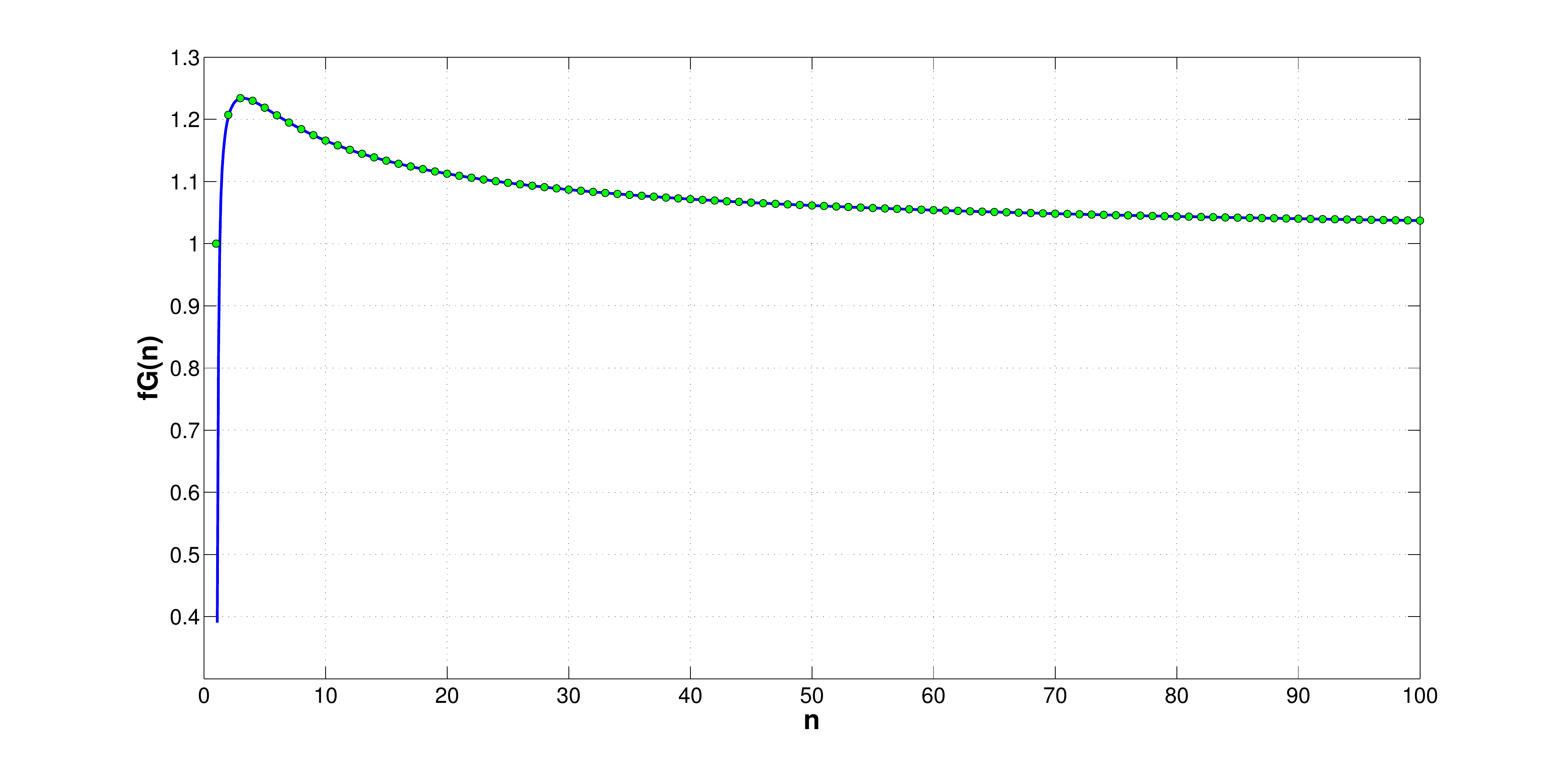}
\caption[The function $f_G(n)=\frac{1}{n}\sum_{k=1}^n k^\frac{1}{n}$ plotted for $n>1$]{The generalized definition of the convoluted finite sum $f_G(n)=\frac{1}{n}\sum_{k=1}^n k^\frac{1}{n}$ plotted against $n>1$ using Theorem \ref{theorem641}. }
\label{Chap6ConvEx2Figure}
\end{figure}

\section{Summary of Results} 
In this chapter, we presented a simple direct method for evaluating simple and convoluted finite sums. Whereas earlier analysis shows that Summability Calculus yields an elementary framework for handling finite sums such as performing infinitesimal calculus, deducing asymptotic expressions, as well as summing divergent series, all within a single coherent framework, the results of this chapter reveal that even the task of \emph{evaluating} convoluted finite sums of the general form $\sum_{k=a}^n s_k\,g(k,n)$ for $n\in\mathbb{C}$, where $s_k$ is an arbitrary periodic sign sequence, is straightforward as well. It is also shown that such method yields the same unique natural generalization that Summability Calculus implicitly operates on, and that uniqueness of natural generalization in Summability Calculus arises out of polynomial fitting. This is clearly a generalization to the earlier statement of linear fitting for semi-linear simple finite sums in Section \ref{Section_23}. In the following chapter, we expand results even more using the Calculus of Finite Differences, in which equivalent alternative results can be expressed using \emph{finite differences} as opposed to infinitesimal derivatives. 

\chapter{Summability Calculus and Finite Differences} \label{Chapter7}
\epigraph{\emph{What nature demands from us is not a quantum theory or a wave theory; rather, nature demands from us a synthesis of these two views.}}{Albert Einstein  (1879 -- 1955)}

In Chapter \ref{Chapter2} we stated that the unique natural generalization of finite sums must be inferred using information that is known to be correct with certainty. For example, using the differentiation rule of simple finite sums $f(n)=\sum_{k=a}^n g(k)$, we know with certainty that $f_G^{(r)}(n) -f_G^{(r)}(n-1)=g^{(r)}(n)$ holds for the unique most natural generalization $f_G(n)$ so we used this fact to derive the Euler-Maclaurin summation formula, which was, in turn, generalized to convoluted finite sums as well as oscillating convoluted finite sums in Chapters \ref{Chapter3} and \ref{Chapter5} respectively.

Nevertheless, there is also additional information that is known to be correct with certainty, which is the discrete values of finite sums $\sum_{k=a}^n g(k)$ for $(n-a)\in\mathbb{N}$. Of course, this information was also repeatedly used throughout earlier analysis but a valid question that arises is whether relying on this information \emph{solely} is sufficient to deduce unique natural generalization. The answer to the latter question is in the negative as the simple counter example $\sum_{k=0}^n \sin{(\pi k)}$ clearly illustrates. Here, $f(n)=0$ for all $n\in\mathbb{N}$ but the correct natural generalization is not the zero function since the zero function does not satisfy the recurrence identity $f(n)-f(n-1)=\sin{(\pi n)}$ for all $n\in\mathbb{C}$. Nevertheless, there are, in fact, common finite sums whose discrete samples are sufficient to deduce unique natural generalization such as the power sum function $\sum_{k=1}^n k^s$. So, why is it that discrete samples of the latter function carry complete information about global behavior whereas discrete samples of the former function do not? And if discrete samples are sufficient, can we use them to perform infinitesimal calculus? Answering these  questions brings us to the well-known subjects of the Sampling Theorem and the Calculus of Finite Differences.

In this chapter, we will use Summability Calculus to prove some of most foundational results in the Calculus of Finite Differences, and show how to perform infinitesimal calculus using discrete samples. Using such results, we will present the analog of all summation formulas presented throughout previous chapters using the language of \emph{finite differences} as opposed to infinitesimal derivatives. For example, we will re-establish the well-known result that the analog of the Euler-Maclaurin summation formula for discrete functions is given by Gregory quadrature formula. In addition, we will show that Summability Calculus presents a simple geometric proof to the well-known Shannon-Nyquist Sampling Theorem, and even prove a stronger statement, which we will call the Half Sampling Theorem. Moreover, we will use these results to deduce Newton's interpolation formula and show that it is intimately tied to the summability method $\Xi$ of Claim \ref{claim411}. All of these results will, in turn, be used to deduce many interesting identities related to fundamental constants such as $\pi$, $e$, and Euler's constant $\lambda$.

\section{Summability Calculus and the Calculus of Finite Differences} 
Our starting point into the Calculus of Finite Differences is to revisit the differentiation rule of simple finite sums given by Corollary \ref{corollary251_1}, a special case of which is written here in Eq \ref{Sec711DiffEq1}. Here, $\nabla$ stands for the backward difference operator and $B_1=1/2$. It is straightforward to express Eq \ref{Sec711DiffEq1} using the forward difference operator $\Delta$, which is given by the exact same expression in Eq \ref{Sec711DiffEq2} except that we now have $B_1=-1/2$.  
\begin{equation}\label{Sec711DiffEq1}
f'(n) = \sum_{r=0}^\infty \frac{B_r}{r!} \nabla f^{(r)} (n) \\, B_1=\frac{1}{2}
\end{equation} 
\begin{equation}\label{Sec711DiffEq2}
f'(n) = \sum_{r=0}^\infty \frac{B_r}{r!} \Delta f^{(r)} (n) \\, B_1=-\frac{1}{2}
\end{equation} 

Now, we note that  $\Delta^{(p)} f^{(m)}(n)$ can also be expressed using the exact same expression, namely: 
\begin{equation}\label{Sec711DiffEq3}
\Delta^{p} f^{(m)}(n) = \sum_{r=0}^\infty \frac{B_r}{r!} \Delta^{p+1} f^{(m+r-1)} (n) \\, B_1=-\frac{1}{2}
\end{equation} 

We also note in Eq \ref{Sec711DiffEq3} that $\Delta f^{(r-1)}(n)$ is used while evaluating $\Delta f^{(r)}(n)$ . However, $\Delta f^{(r-2)}(n)$  is, in turn, used in evaluating $\Delta f^{(r-1)}(n)$  and so on. Therefore, it is possible to express the derivative $f'(n) $ formally using \emph{solely} higher order differences $\Delta^p f(n)$. Tracing this operation term-by-term, suggests the following expression for the derivative:
\begin{align}\label{Sec711DiffEq4} 
f'(n) &= \frac{B_0}{0!} \Delta f(n) + \big(\frac{B_1}{1!}\frac{B_0}{0!}\big)\Delta^2 f(n) + \big(\frac{B_1}{1!}\frac{B_1}{1!}\frac{B_0}{0!} + \frac{B_2}{2!}\frac{B_0}{0!}\frac{B_0}{0!}\big) \Delta^3 f(n) \dotsm \\ 
\label{Sec711DiffEq5} &= \frac{\Delta f(n)}{1} - \frac{\Delta^2 f(n)}{2} + \frac{\Delta^3 f(n)}{3} - \dotsm 
\end{align} 

Eq \ref{Sec711DiffEq5} is well-known in the Calculus of Finite Differences (see for instance \cite{Jordan1965}). Upon making the substitution $n=h\,x$, for a small step size $h$, we arrive at a method of obtaining higher order approximations to first derivative that is given by Eq \ref{Sec711HigherDeriv_1}. Here, $\Delta_h = f(x+h)-f(x)$, and $\Delta_h^p = \Delta_h \Delta_h^{p-1}$. Note that in such notation, both the function $f$ and the independent variable $x$ are semantically identical for all symbols, hence they can be omitted.  
\begin{equation}\label{Sec711HigherDeriv_1} 
f'(x) = \frac{1}{h} \Big(\frac{\Delta_h}{1}-\frac{\Delta_h^2}{2} +\frac{\Delta_h^3}{3} - \frac{\Delta_h^4 }{4} + \dotsm \Big)
\end{equation}
Now, Eq \ref{Sec711HigherDeriv_1} can be used to obtain a similar expression for the second derivative $f^{(2)}(x)$: 
\begin{equation}\label{Sec711HigherDeriv_2} 
f^{(2)}(x) = \frac{1}{h} \Big(\frac{\Delta_h f'(x)}{1}-\frac{\Delta_h^2 f'(x)}{2} +\frac{\Delta_h^3 f'(x)}{3} - \frac{\Delta_h^4 f'(x)}{4} + \dotsm \Big)
\end{equation}
Plugging the expression Eq \ref{Sec711HigherDeriv_1} into Eq \ref{Sec711HigherDeriv_2}, we arrive at: 
\begin{equation}\label{Sec711HigherDeriv_3} 
f^{(2)}(x) = \frac{1}{h^2} \Big(\Delta_h^2-\Delta_h^3 +\frac{11}{12} \Delta_h^4 - \frac{5}{6} \Delta_h^5 + \dotsm \Big)
\end{equation}

Upon close inspection of the process used in the previous equation, we note that the coefficients in the second derivative are simply given by the series expansion of the function $\log^2{x}$ because composition behaves exactly as if it were algebraic multiplication. In general, we have the formal expression given in Eq \ref{Sec711HigherDeriv_4}. Here, the notation $\log^r(I+\Delta_h)$  means that the series expansion of the function $\log^r(1+x)$ is to be applied formally on the forward difference operator $\Delta_h$. In addition, $I$ stands for the \emph{unity symbol}, which is semantically given by $I=D^0=I^0=f(x)$.

\begin{equation}\label{Sec711HigherDeriv_4} 
f^{(r)}(x)=D^r = \frac{\log^r{(I+\Delta_h)}}{h^r}
\end{equation} 
So, using Eq \ref{Sec711HigherDeriv_4}, we have a simple method for obtaining higher-order approximations to infinitesimal derivatives.  It also readily yields the Cauchy definition of derivatives since we have $D^r = \frac{\log^r{(I+\Delta_h)}}{h^r} = \lim_{h\to 0} \{\frac{\log^r{(I+\Delta_h)}}{h^r}\} = \lim_{h\to 0}\big\{\Delta_h^r/h^r\big\}$. 

Now, because the unity symbol $I$ is semantically the same with respect to both the differentiation operator $D$ and the forward difference operator $\Delta$, taking the composition of both sides of Eq \ref{Sec711HigherDeriv_4} with respect to a function $z$ is valid. In other words, we have that for any function $z$, Eq \ref{Sec711HigherDeriv_5} holds. As will be shown shortly, Eq \ref{Sec711HigherDeriv_5} is a very incredible equation that summarizes almost all foundational results in calculus. Here, however, it is crucial to keep in mind that both sides of Eq \ref{Sec711HigherDeriv_5} are to be interpreted formally using Taylor series expansions. 
 
\begin{equation}\label{Sec711HigherDeriv_5}
z\big(h^r\,D^r\big) = z\big(\log^r{(I+\Delta_h)}\big)
\end{equation} 

Choosing difference functions $z$ yield new identities. For example, if $z(x)=e^x$ and $r=1$, then we have: 
\begin{equation}\label{DiffValuesOfZ_1}
e^{hD} = I + \Delta_h = f(x+h)
\end{equation} 

Eq \ref{DiffValuesOfZ_1} is basically a restatement of Taylor's theorem. However, we could have also selected $r=2$, which yields:
\begin{equation}\label{DiffValuesOfZ_2}
e^{h^2D^2} = (I + \Delta_h)^{\log{(I+\Delta_h)}}
\end{equation} 
Now, in order to make sense out of the last equation, we apply Taylor series expansion formally on both sides, which yields Eq \ref{DiffValuesOfZ_3}. Eq \ref{DiffValuesOfZ_3} can be confirmed numerically quite readily. 
\begin{equation}\label{DiffValuesOfZ_3}
I + \frac{h^2D^2}{1!} + \frac{h^4D^4}{2!} + \frac{h^6D^6}{3!} + \dotsm = I + \Delta_h^2 -\Delta_h^3+\frac{17}{4}\Delta_h^4 -\frac{11}{6}\Delta_h^5 + \dotsm
\end{equation} 

So far, we have shown that Eq \ref{Sec711HigherDeriv_5} generalizes the definition of higher-order derivatives and provides Taylor's theorem as well. Not surprisingly, there is much more to this equation. For instance, if we wish to find the anti-derivative, we select $r=-1$ in Eq \ref{Sec711HigherDeriv_4}, which yields:
\begin{equation}\label{DiffValuesOfZ_4}
D^{-1} = \frac{h}{\log{(I+\Delta_h)}} = h\,\Big(\Delta_h^{-1} + \frac{1}{2}\Delta_h^0 -\frac{1}{12}\Delta_h^1 + \frac{1}{24}\Delta_h^2 +\dotsm\Big)
\end{equation} 
Eq \ref{DiffValuesOfZ_4} generalizes the classical definition of anti-derivatives. We will return to it in the following section when we discuss Gregory's quadrature formula. Moreover, from Eq \ref{DiffValuesOfZ_1} we have: 
\begin{equation}\label{DiffValuesOfZ_5}
(I+\Delta_h)^\alpha = e^{\alpha hD} = f(x+\alpha h)
\end{equation} 
Eq \ref{DiffValuesOfZ_5} is a concise statement of Newton's interpolation formula. To see this, we select $\alpha = \frac{x-x_0}{h}$, which yields by the binomial theorem: 
\begin{equation}\label{DiffValuesOfZ_6}
f(x) = f(x) + \frac{\Delta_h}{1!} \Big(\frac{x-x_0}{h}\Big) + \frac{\Delta_h^2}{2!} \Big(\frac{x-x_0}{h}\Big)\Big(\frac{x-x_0}{h}-1\Big) + \dotsm 
\end{equation} 
Now, if we denote $n=\frac{x-x_0}{h}$, where $h$ is chosen such that $n$ is an integer, then Newton's interpolation formula in Eq \ref{DiffValuesOfZ_6} can be rewritten as: 
\begin{equation}\label{DiffValuesOfZ_7}
f(x) = \sum_{j=0}^n \chi_n(j) \frac{\Delta_h^j}{h^j} \frac{(x-x_0)^j}{j!} 
\end{equation} 
Interestingly, the function $\chi_n(j)$ shows up again. It is important to note that Eq \ref{DiffValuesOfZ_7} is an \emph{exact} expression that holds for all $n\in\mathbb{N}$, where again the sampling interval $h$ is selected such that $n$ is an integer. Choosing $h$ arbitrarily small yields Eq \ref{DiffValuesOfZ_8}. Note here that Newton's interpolation formula given by Eq \ref{DiffValuesOfZ_8} is the discrete mirror of the summability method $\Xi$ of Claim \ref{claim411}, where infinitesimal derivatives $f^{(j)}(x_0)$ are replaced with their approximations $\Delta_h^j /h^j$, thus bridging the two results.
\begin{equation}\label{DiffValuesOfZ_8}
f(x) = \lim_{n\to\infty} \Big\{\sum_{j=0}^n \chi_n(j) \frac{\Delta_h^j}{h^j} \frac{(x-x_0)^j}{j!}\Big\} 
\end{equation} 

So far, we have derived a generalized definition of derivatives and anti-derivatives as well as Newton's interpolation formula, which respectively are the discrete analog of differentiation, integration, and computing Taylor series expansion in infinitesimal calculus. Before we start applying these results, we need to return to the expression given by Eq \ref{Sec711HigherDeriv_4}, which was the building block of all subsequent analysis. In this equation, the well-known pitfall with \emph{undersampling} clearly arises. For example, if $f(x)=\sin{\pi x}$, and $h=1$, then $\Delta_h^j = 0$  and the right-hand side of Eq \ref{Sec711HigherDeriv_4} evaluates to the zero function even though the correct derivative is not. This arises because the simplest function that interpolates the samples $f(k)=0$ is itself the zero function so the Calculus of Finite Differences basically operates on the simplest of all possible generalizations. Thus, our goal is to know for which sampling interval $h$ does Eq \ref{Sec711HigherDeriv_4} hold? The answer is given by the following theorem. \\ \hrule

\begin{theorem}\label{theorem711} 
Given a bandlimited function $f(x)$ with bandwidth $B$ that is analytic over the domain $[x, \infty)$, and if the sampling interval $h$ satisfies the Nyquist criteria $h<\frac{1}{2B}$, then the $\mathfrak{T}$ definition of the infinite sum $\frac{1}{h} \sum_{j=0}^\infty (-1)^j \frac{\Delta_h^j}{j}$ is given by the function's derivative $f'(x)$. In general, Eq \ref{Sec711HigherDeriv_4} holds, where the right-hand infinite sum is interpreted using the generalized definition $\mathfrak{T}$. The strict inequality in the Nyquist criteria can be relaxed to an inequality $h\le \frac{1}{2B}$ if the function $f(x)$ has no bandwidth components at frequency $w=B$. 
\end{theorem} 
\begin{proof} 
Let the Fourier transform of the function $f(x)$ be denoted $\mathcal{F}\{f\} = \mathbf{F}(w)$. Also, let $c_{r,j}$ be the coefficients of the series expansion of the function $\log^r{(1+x)}$. For example, we have $c_{1,j} = \frac{(-1)^{j+1}}{j}$. As defined earlier, let $\xi_\delta(j)$ be a function chosen such that the summability method $\lim_{\delta\to 0} \big\{ \sum_{j=0}^\infty \xi_\delta(j) a_j\big\}$ correctly sums any analytic function in its Mittag-Leffler star. For example, we have $\xi_\delta(j) = j^{-\delta j}$ in Lindel\"of summability method. Then, we have:
\begin{equation}\label{theorem711_1}
\mathcal{F}\Big\{\lim_{\delta\to 0} \big\{\frac{1}{h} \sum_{j=0}^\infty \xi_\delta(j)\,c_{r,j}\, \Delta_h^j\big\}\Big\} = \frac{1}{h} \lim_{\delta\to 0} \big\{\sum_{j=0}^\infty \xi_\delta(j)\,c_{r,j}\, \mathcal{F}\{\Delta_h^j\}\big\}
\end{equation} 
However, if $\mathcal{F}\{f\} = \mathbf{F}(w)$, then we know that $\mathcal{F}\{\Delta_h^j\} = \mathbf{F}(w)\, (e^{iwh}-1)^j$. Plugging this into the earlier equation Eq \ref{theorem711_1} yields: 
\begin{equation}\label{theorem711_2}
\mathcal{F}\Big\{\lim_{\delta\to 0} \big\{\frac{1}{h} \sum_{j=0}^\infty \xi_\delta(j)\,c_{r,j}\, \Delta_h^j\big\}\Big\} = \frac{\mathbf{F}(w)}{h} \lim_{\delta\to 0} \big\{\sum_{j=0}^\infty \xi_\delta(j)\,c_{r,j}\, (e^{iwh}-1)^j\big\}
\end{equation} 
Because the summability method correctly sums any function in its Mittag-Leffler star, then we know that Eq \ref{theorem711_3} holds as long as the line segment $[1, e^{iwh}]$ over the complex plane $\mathbb{C}$ does not pass through the origin. 
\begin{equation}\label{theorem711_3}
 \lim_{\delta\to 0} \big\{\sum_{j=0}^\infty \xi_\delta(j)\,c_{r,j}\, (e^{iwh}-1)^j\big\} = \log^r(e^{iwh}) = (iwh)^r
\end{equation} 
Because $\mathbf{F}(w)=0$ for all $|w|>B$, we need to guarantee that the line segment $[1, e^{iwh}]$ does not pass through the origin for all $|w|\le B$. However, to guarantee that such condition holds, the sampling interval $h$ has to satisfy the Nyquist criteria $h < \frac{1}{2B}$. Geometrically, $e^{iwh}$ forms an arc when parameterized by $w$ and this arc forms a half-circle that passes through the origin at $w=B$ if $h=\frac{1}{2B}$. Selecting $h < \frac{1}{2B}$  ensures that the arc never forms a half-circle for all $|w| \le B$, and that the line segment $[1, e^{iwh} ]$  does not pass the singularity point at the origin.

Therefore, if $f(x)$ is bandlimited with bandwidth $B$ and if $h<\frac{1}{2B}$, then the Fourier transform of the right-hand side of Eq \ref{Sec711HigherDeriv_4} evaluates to $(iw)^r\mathbf{F}(w)$, which is also the Fourier transform of $f'(x)$. Since the Fourier transform is invertible, both sides of Eq \ref{Sec711HigherDeriv_4} must be equivalent under stated conditions, when interpreted using the generalized definition $\mathfrak{T}$. 
\end{proof} \hrule\vspace{12pt} 

In order to see why the generalized definition $\mathfrak{T}$ is important, we start from Eq \ref{Sec711HigherDeriv_5} and apply the function $z(x)=\frac{1}{x}$ to both sides, which yields:
\begin{equation}\label{theorem711Ex_1} 
\frac{1}{I+\Delta_h} = e^{-hD}
\end{equation}
If $f(x)=e^x$ and $h=\log{2}$, then for $x=0$ the left-hand side evaluates to: 
\begin{equation}\label{theorem711Ex_2} 
\frac{1}{I+\Delta_h} = I-\Delta_h+\Delta_h^2 - \Delta_h^3 + \dotsm = 1-1+1-1+1 - \dotsm
\end{equation}
Of course, this series is divergent but its $\mathfrak{T}$ value is $\frac{1}{2}$ as discussed earlier. On the other hand, the right-hand side of Eq \ref{theorem711Ex_1} evaluates to: 
\begin{equation}\label{theorem711Ex_3} 
e^{-hD} = I-hD+\frac{h^2D^2}{2!} - \frac{h^3D^3}{3!} + \dotsm = e^{-\log{2}} = \frac{1}{2}
\end{equation}
Therefore, both sides agree as expected when interpreted using the generalized definition of infinite sums $\mathfrak{T}$. Theorem \ref{theorem711} allows us to deduce a stronger statement than the classical Shannon-Nyquist Sampling Theorem. \\ \hrule 
\begin{theorem}{\textbf{(The Half Sampling Theorem)}}\label{theorem712}
Given a bandlimited function $f(x)$ with bandwidth $B$ that is analytic over a domain $(a, \infty)$, then $f(x)$ for $a<x<\infty$ can be perfectly reconstructed from its discrete samples if: 
\begin{enumerate} 
\item The discrete samples are taken in the interval $(a, \infty)$. 
\item The sampling rate satisfies the Nyquist criteria. 
\end{enumerate} 
\end{theorem}
\begin{proof} 
By Theorem \ref{theorem711}, if the sampling rate satisfies the conditions of this theorem, then all higher derivatives $f^{(r)}(x_0)$ for some $x_0\in(a, \infty)$ can be computed exactly. By rigidity of analytic functions, this is sufficient to reconstruct the entire function in the domain $(a, \infty)$.
\end{proof} \hrule\vspace{12pt} 

Historically, extensions of the Sampling Theorem have generally focused on two aspects: 
\begin{enumerate} 
\item Knowing how our \emph{apriori} information about the function can reduce required sampling rate. Examples include sampling non-bandlimited signals, which is commonly referred to as \emph{multiresolution} or \emph{wavelet sampling theorems} (see for instance \cite{Vaid2001}). 
\item Knowing how \emph{side information} that are transmitted \emph{along with} the discrete samples can reduce the required sampling rate. For example, whereas the bandwidth of $\sin^3(2\pi x)$ is three-times larger than the bandwidth of $\sin(2\pi x)$, we can transmit the function $\sin{(2\pi x)}$ along with some side information telling the receiver that the function should be cubed. Examples of this type of generalization is discussed in \cite{Zhu92}. 
\end{enumerate} 

In Theorem \ref{theorem712}, on the other hand, it is stated that statement of the Shannon-Nyquist Sampling Theorem itself can be strengthened. In particular, if $f(x)$ is analytic over a domain $(a, \infty)$, then $f(x)$ can be perfectly reconstructed from its discrete samples if the sampling rate satisfies conditions of of the theorem. Note that unlike the classical Shannon-Nyquist Sampling Theorem, which states that samples have to be taken from $-\infty$ to $+\infty$, this theorem states that only \lq\lq half'' of those samples are sufficient, hence the name. An illustrative example is depicted in Figure \ref{NewtonInterpExSin}, where reconstruction is performed using Newton's interpolation formula. In the figure, the samples are taken from the function $\sin(x)$ with sampling interval $h=\frac{1}{2}$ and starting point $x_0=0$. Clearly, all 5\textsuperscript{th} degree, 10\textsuperscript{th} degree, and 20\textsuperscript{th} degree approximations work extremely well for $x\ge x_0$ as expected, but even the function's past can be reconstructed perfectly as higher degree approximations are used. 
\begin{figure} [h]
\centering
\includegraphics[scale=0.4]{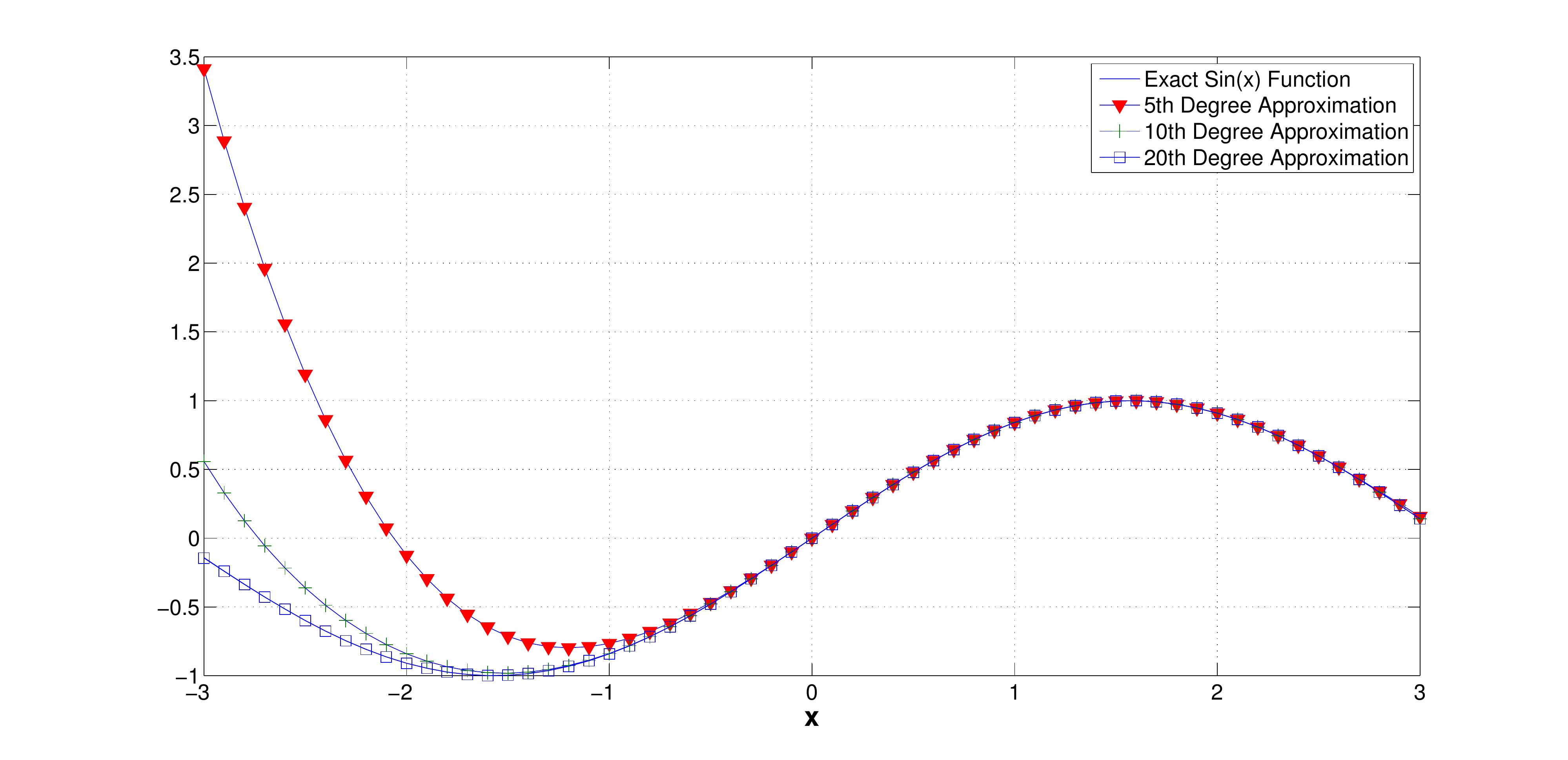}
\caption[Newton's interpolation formula applied to the discretized function $\sin{x}$]{Newton's interpolation formula applied to the discretized function $f(x)=\sin{x}$, where sampling interval is $h=\frac{1}{2}$.}
\label{NewtonInterpExSin}
\end{figure}

While  statement of Theorem \ref{theorem712} does provide the answer we were looking for, it is unfortunate that the classical definition of Fourier transforms cannot adequately capture the meaning of bandwidth. For instance, the function $\log{x}$, which is analytic over the domain $(0, \infty)$ is not Fourier transformable even though it is straightforward to prove that its bandwidth must, in fact, be equal to zero. To see this, we note here the following two well-known properties:
\begin{enumerate}
\item Adding a constant term to a function does not change its bandwidth. In other words, if $\mathfrak{B}f$ stands for the bandwidth of the function $f(x)$, then $\mathfrak{B}(f+c) = \mathfrak{B}(f)$. 
\item Rescaling units expands/shrinks bandwidth. In other words, $\mathfrak{B}f(c\,x) = c\,\mathfrak{B}f$ for all $c>0$. 
\end{enumerate} 

Of course, these two properties can be derived immediately using the definition of Fourier transformation. However, a simpler proof that does not rely on Fourier theory can also be inferred intuitively upon using the direct correspondence between the notion of bandwidth and the notion of minimum sampling rate. Here, adding a constant term to a function is immediately carried over to the discrete samples so the minimum sampling rate should not be altered by adding a constant term. Similarly,  defining $f(c\,x)$ is equivalent to \emph{rescaling} unit of the independent variable $x$ so the minimum sampling rate should be rescaled by the same proportion. Now, if we apply both properties to the logarithmic function $\log{x}$, we have: 
\begin{equation}\label{LogBWisZero_1} 
\mathfrak{B}\log{x} = \mathfrak{B}(\log{x} + c) = \mathfrak{B}\log{(e^c\,x)} = e^c\,\mathfrak{B}\log{x} 
\end{equation} 
However, because Eq \ref{LogBWisZero_1} must hold for all $c$, we must have $\mathfrak{B}\log{x}=0$. We can verify this by using Eq \ref{Sec711HigherDeriv_4}. Here, we note that regardless of the length of the sampling interval $h$, Eq \ref{Sec711HigherDeriv_4} seems to hold for the logarithmic function for all $r$.  For instance, choosing $h=1$ and $r=1$ yields the identity in Eq \ref{LogBWisZero_2}; a result that was also proved in \cite{Sondow08}. In \cite{GoldMartin2005}, the identity is both proved and generalized to $e^\frac{1}{x}$ for $x\in\mathbb{N}$ using probability theory. We will present similar other identities at a later section.
\begin{equation}\label{LogBWisZero_2} 
e=\Big(\frac{2}{1}\Big)^{\frac{1}{1}} \Big(\frac{2^2}{1\cdot 3}\Big)^{\frac{1}{2}}\Big(\frac{2^3\cdot 4}{1\cdot 3^3}\Big)^{\frac{1}{3}}\Big(\frac{2^4\cdot 4^4}{1\cdot 3^6\cdot 5}\Big)^{\frac{1}{4}}\dotsm 
\end{equation} 

\section{Discrete Analog of the Euler-like Summation Formulas} 
In this section, we present the discrete analog of the many Euler-like summation formulas deduced in earlier chapters. Our starting point is Eq \ref{DiffValuesOfZ_4}, which extends the classical Riemann sum and presents higher degree approximations to integrals. Choosing $h=1$ yields Gregory's quadrature formula given by Eq \ref{Sect72_Gregory}. Here, the constants $G_r=\{1, \frac{1}{2}, -\frac{1}{12}, \frac{1}{24}, -\frac{19}{720}, \frac{3}{160}, \dotsm\}$ are called Gregory coefficients (\begingroup\ttfamily OEIS A002206 (numerators) and A002207 (denominators)\endgroup \cite{OEISGreg1, OEISGreg2}). 

\begin{equation}\label{Sect72_Gregory} 
\sum_{k=a}^n g(k) = \int_a^n g(t)\,dt + \frac{g(n)+g(a)}{2} - \sum_{r=2}^\infty G_r \big[\Delta^{r-1}g(n) - \Delta^{r-1}g(a)\big]
\end{equation} 
Many methods exist for computing Gregory coefficients. One convenient recursive formula is given by \cite{OEISGreg1,Kluyver1924}: 
\begin{equation} \label{GregForm_1}
G_n=\frac{(-1)^{n+1}}{n+1}-\sum_{k=1}^{n-1} (-1)^k\,\frac{G_{n-k}}{k+1}, \\ G_0=1
\end{equation}

Clearly, Gregory's formula is analogous to the Euler-Maclaurin summation formula where finite differences replace infinitesimal derivatives. To deduce the analog of remaining Euler-like summation formulas, we note again that the discrete analog of Taylor's series expansion is given by Newton's interpolation formula: 
\begin{equation}\label{Sect72_NewtonInterp} 
f(x)= f(x_0) + \frac{\Delta}{1!} (x-x_0)_1 + \frac{\Delta^2}{2!} (x-x_0)_2 + \frac{\Delta^3}{3!} (x-x_0)_3 + \dotsm 
\end{equation}
Here, $\Delta=f(x_0+1)-f(x_0)$ and $(z)_n$ is the \emph{falling factorial} defined by $(z)_n = z(z-1)(z-2)\dotsm (z-n+1)$. One way of interpreting this equation is to note that it yields consistent results when the forward difference operator $\Delta$ is applied to both sides of the equation, which, in turn, follows because $\Delta(z)_n = n\,(z)_{n-1}$. 
Using analytic summability theory, and in a manner that is analogous to the development of Chapter \ref{Chapter5}, we have: 
\begin{equation}\label{Sect72_DivergSum_1}
\sum_{k=a}^\infty e^{i\theta k} g(k) = e^{i\theta a}\sum_{r=0}^\infty \frac{\Phi_r}{r!} \Delta^r g(a-1),\\ \text{ where } \Phi_r=\sum_{k=1}^\infty e^{i\theta k} (k)_r
\end{equation}
Here, we note that $\Phi_r$ is the $\mathfrak{T}$ value of an infinite sum that also involves the falling factorial function $(k)_r$. Because $(k)_r$ is a polynomial of degree $r$, the constants $\Phi_r$ can be computed from the constants $\Theta_r = \sum_{k=1}^\infty e^{i\theta k} k^r$. In particular, using Stirling numbers of the first kind $s(r,k)$, we have: 
\begin{equation}\label{PhiExpressionTheta}
\Phi_r = \sum_{k=0}^r s(r,k)\, \Theta_k
\end{equation}
Eq \ref{Sect72_DivergSum_1}, in turn, implies that: 
\begin{equation}\label{Sect72_DivergSum_2} 
\sum_{k=a}^n e^{i\theta k} g(k) = e^{i\theta} g(a) - \sum_{r=0}^\infty \frac{\Phi_r}{r!} \Big(e^{i\theta n}\,\Delta^r g(n)-e^{i\theta a}\,\Delta^r g(a)\Big)
\end{equation}
In addition, we have that the $\mathfrak{T}$ value of divergent sums is formally given by: 
\begin{equation}\label{Sect72_DivergSum_3} 
\sum_{k=a}^\infty e^{i\theta k} g(k) = \sum_{k=a}^n e^{i\theta k} g(k) + \sum_{r=0}^\infty \frac{\Phi_r}{r!} e^{i\theta n}\,\Delta^r g(n)
\end{equation}
Equally important, last equation is also a method of deducing asymptotic expressions to the oscillating sums $ \sum_{k=a}^n e^{i\theta k} g(k)$. Finally, to evaluate simple finite sums directly, in a manner that is analogous to the development of Chapter \ref{Chapter6}, we can use the following equation, which holds formally for all $s$: 
\begin{equation}\label{Sect72_DivergSum_4} 
\sum_{k=a}^n g(k) = \sum_{r=0}^\infty \frac{d_r(n)-d_r(a-1)}{r!} \Delta^r g(s) + \sum_{k=0}^s g(k+a)-g(k+n+1)
\end{equation}
Here, $d_r(n) = \sum_{k=1}^{n+1} (k)_r$. The proof of Eq \ref{Sect72_DivergSum_4} is analogous to the proof of Theorem \ref{theorem621}. If $\Delta^{m+1} g(s) \to 0$ as $s\to\infty$, Eq \ref{Sect72_DivergSum_4} can be used as an asymptotic expression and the infinite sum needs not be entirely evaluated. Similarly, we have the following evaluation method for oscillating finite sums:
\small
\begin{equation}\label{Sect72_EvalOscillatingSum_1} 
\sum_{k=a}^n e^{i\theta k} g(k) = e^{i\theta s} \sum_{r=0}^\infty \frac{\Upsilon_r(n) - \Upsilon_r(a-1)}{r!} \Delta^r g(s) + \sum_{k=0}^s e^{i\theta(k+a)}\,g(k+a) - e^{i\theta(k+n+1)}\,g(k+n+1)
\end{equation}
\normalsize
Here, $\Upsilon_r(x) = \sum_{k=1}^{x+1} e^{i\theta k}(k)_r$, which is available in analytic closed form. In particular, if $\Omega_r(x)$ is defined as in Theorem \ref{theorem631}, then we have: 
\begin{equation} 
\Upsilon_r(x) = \sum_{k=0}^r s(r,k)\, \Omega_r(x)
\end{equation}
Validity of the new formulas above is illustrated in the following section. 

\section{Applications to Summability Calculus} 
The rules derived in the previous section can be applied to arbitrary discrete functions. In the context of Summability Calculus, these rules present alternative formal methods of computing fractional sums using their unique most natural generalizations. For example, we can safely replace the Euler-Maclaurin summation formula with Gregory's formula to find boundary values of higher order derivatives so that the non-arbitrary constant in the differentiation rule of simple finite sums can be determined. In addition, we can use the methods given in Eq \ref{Sect72_DivergSum_4} and Eq \ref{Sect72_EvalOscillatingSum_1} to compute fractional sums directly instead of using the theorems of Chapter \ref{Chapter6}. 

However, we need to exercise caution when using any discretization of continuous functions to avoid under-sampling! Once discrete samples carry complete information of the original continuous functions, then the Calculus of Finite Differences can indeed become a useful tool for analyzing simple finite sums and products. In this section, we will present a few examples that illustrate the latter point.

\subsection{Example I: Binomial Coefficients and Infinite Products} 
In Eq \ref{LogBWisZero_2}, we derived an infinite product formula for the natural logarithmic base $e$ by applying Eq \ref{Sec711HigherDeriv_1} to the logarithmic function. Of course, the logarithmic function is merely one possibility. For example, if we apply Eq \ref{Sec711HigherDeriv_1} to the simple finite sum $\sum_{k=1}^n \log{k}$, we obtain Eq \ref{Ex731_1}. Eq \ref{Ex731_1} was first proved by Ser in 1926 and was later rediscovered by Sondow in 2003 using hypergeometric series \cite{Sondow03, Sondow08}. 
\begin{equation}\label{Ex731_1} 
e^\lambda = \Big(\frac{2}{1}\Big)^{\frac{1}{2}} \Big(\frac{2^2}{1\cdot 3}\Big)^{\frac{1}{3}}\Big(\frac{2^3\cdot 4}{1\cdot 3^3}\Big)^{\frac{1}{4}}\Big(\frac{2^4\cdot 4^4}{1\cdot 3^6\cdot 5}\Big)^{\frac{1}{5}}\dotsm 
\end{equation}

Moreover, we can use the same approach to derive additional interesting identities. For example, if we apply Eq \ref{Sec711HigherDeriv_1} to the log-superfactorial function, $\sum_{k=1}^n \log{k!}$, we obtain:
\begin{equation}\label{Ex731_2} 
\frac{e^{\lambda+\frac{1}{2}}}{\sqrt{2\pi}} = \Big(2\Big)^{\frac{1}{2}} \Big(\frac{2}{3}\Big)^{\frac{1}{3}}\Big(\frac{2\cdot 4}{3^2}\Big)^{\frac{1}{4}}\Big(\frac{2\cdot 4^3}{3^3\cdot 5}\Big)^{\frac{1}{5}}\dotsm 
\end{equation}
Again, the internal exponents are given by the binomial coefficients alternating in sign. Dividing Eq \ref{Ex731_1} by Eq \ref{Ex731_2} yields Eq \ref{Ex731_3}. This equation was derived in \cite{Sondow08} using double integrals via analytic continuation of Lerch's Transcendent.
\begin{equation}\label{Ex731_3} 
\sqrt{\frac{2\pi}{e}} = \Big(\frac{2}{1}\Big)^{\frac{1}{3}} \Big(\frac{2^2}{1\cdot 3}\Big)^{\frac{1}{4}}\Big(\frac{2^3\cdot 4}{1\cdot 3^3}\Big)^{\frac{1}{5}}\Big(\frac{2^4\cdot 4^4}{1\cdot 3^6\cdot 5}\Big)^{\frac{1}{6}}\dotsm 
\end{equation}

Finally, suppose we have $f(n)=\sum_{k=1}^n \log{\frac{4k-1}{4k-3}}$, and we wish to find its derivative at $n=0$. Using Summability Calculus, we immediately have by Theorem \ref{theorem231} that the following equation holds: 
\begin{equation}\label{Ex731_4}
f'(0) = 4\,\sum_{k=1}^\infty \Big(\frac{1}{4k-3} - \frac{1}{4k-1}\Big) = 4\,\Big(1-\frac{1}{3}+\frac{1}{5}-\frac{1}{7}+\dotsm\Big) = \pi
\end{equation} 
Using Eq \ref{Sec711HigherDeriv_1}, on the other hand, yields the infinite product formula:
\begin{equation}\label{Ex731_5}
f'(0) = \log \Big[ \big(\frac{3}{1}\big)^\frac{1}{1} \big(\frac{3\cdot 5}{1\cdot 7}\big)^\frac{1}{2} \big(\frac{3\cdot 5^2\cdot 11}{1\cdot 7^2\cdot 9}\big)^\frac{1}{3}\dotsm\Big]
\end{equation} 
Equating both equations yields Eq \ref{Ex731_6}. An alternative proof to such identity is given in \cite{Sondow08}. 
\begin{equation}\label{Ex731_6}
e^\pi =  \big(\frac{3}{1}\big)^\frac{1}{1} \big(\frac{3\cdot 5}{1\cdot 7}\big)^\frac{1}{2} \big(\frac{3\cdot 5^2\cdot 11}{1\cdot 7^2\cdot 9}\big)^\frac{1}{3}\dotsm
\end{equation}

\subsection{Example II: The Zeta Function Revisited} 
Suppose we start with the harmonic numbers $f(n)=\sum_{k=1}^n \frac{1}{k}$ and we wish to find its derivative at $n=0$. We know by Summability Calculus that the derivative is given by $\zeta_2$. Since we have for the harmonic numbers $\Delta^p = \frac{{(-1)}^{p+1}}{p}$, applying Eq \ref{Sec711HigherDeriv_1} also yields the original series representation for $\zeta_2$. However, such approach can also be applied to the generalized harmonic numbers $\sum_{k=1}^n \frac{1}{k^s}$, which would yield different series representations for the Riemann zeta function $\zeta_s$ as given by Eq \ref{Ex732_1}. Eq \ref{Ex732_1} follows immediately by Summability Calculus and Eq \ref{Sec711HigherDeriv_1}.  Interestingly, such identity is a globally convergent expression to the Riemann zeta function for all $s\in\mathbb{C}$. It was discovered by Helmut Hasse in 1930 \cite{MathWorldRiemannZeta}.  
\begin{equation}\label{Ex732_1} 
\zeta_s = \frac{1}{1-s} \sum_{k=1}^\infty \frac{1}{k} \sum_{j=1}^k \binom{k-1}{j-1} \frac{(-1)^{j}}{j^{s+1}}
\end{equation}
Note that both Eq \ref{Ex732_1} and infinite product representation for $e^\lambda$ in Eq \ref{Ex731_1} can both be used to establish the following well-known result: 
\begin{equation}\label{Ex732_2} 
\lim_{s\to 1} \big\{\zeta_s -\frac{1}{s-1}\big\} = \lambda
\end{equation}

\subsection{Example III: Identities Involving Gregory's Coefficients} 
Gregory's formula presents a method of numerical computation of integrals that is similar to the Euler-Maclaurin summation formula. In this example, we will use Gregory's formula to derive identites that relate Gregory coefficients with many fundamental constants. Our starting point would be the case of \emph{semi-linear} simple finite sums $\sum_{k=a}^n g(k)$ that have well-known asymptotic expansions $S_g(n)$ such that $\lim_{n\to\infty} \{S_g(n)-\sum_{k=a}^n g(k)\} = 0$. To recall earlier discussion, if a simple finite sum is semi-linear, then $g'(n)\to 0$ as $n\to\infty$, which implies that $\Delta g(n)\to 0$ as $n\to\infty$. Now, by Gregory's formula:
\begin{equation}\label{Ex733_1} 
\lim_{n\to\infty} \Big\{\sum_{k=a}^n g(k) - \frac{g(n)}{2} - \int_a^n g(t)\,dt\Big\}=  \sum_{r=1}^\infty G_r\,\Delta^{r-1} g(a)
\end{equation}

For example, if $g(k)=\frac{1}{k}$, and $a=1$, then $\Delta^{p} g(a) = \frac{(-1)^{p}}{p+1}$. Plugging this into the above formula yields: 
\begin{equation}\label{Ex733_2} 
\sum_{r=1}^\infty \frac{|G_r|}{r} = \lambda = \frac{1}{2} + \frac{1}{24} + \frac{1}{72} + \frac{19}{2880} + \frac{3}{800} + \frac{863}{362880} + \dotsm
\end{equation}
Eq \ref{Ex733_2} is one of the earliest expressions discovered that express Euler's constant $\lambda$ as a limit of rational terms. It was discovered by Kluyver in 1924 using the integral representation of the digamma function $\psi(n)$\cite{Kluyver1924}. Similarly, if we now let $g(k)=H_k$ and choose $a=0$, then we have: 
\begin{align*}\label{Ex733_3} 
\lim_{n\to\infty} \Big\{\sum_{k=0}^n H_k - \frac{H_n}{2} - &\int_0^n H_t\,dt\Big\}=  \lim_{n\to\infty} \Big\{(n+1)\big(H_{n+1}-1\big)-\frac{H_n}{2}-\lambda n- \log{n!}\Big\} \\ 
&= \lim_{n\to\infty} \Big\{n\,(H_n-\log{n}-\lambda)+\frac{H_n-\log{n}-\log{2\pi}}{2}\Big\} \\ 
&=\frac{1+\lambda-\log{2\pi}}{2} 
\end{align*}
Because $\Delta^{p} H_0 = \frac{(-1)^{p+1}}{p}$, we have: 
\begin{equation}\label{Ex733_4} 
\sum_{r=2}^\infty \frac{|G_r|}{r-1} =\frac{\log{2\pi}-1-\lambda}{2} =  \frac{1}{12} + \frac{1}{48} + \frac{19}{2160} + \dotsm
\end{equation}

Of course, such approach can be applied to many other functions. In addition, we have by Gregory's formula: 
\begin{equation}\label{Ex733_5} 
\sum_{k=a}^{a-1} g(k) = \int_a^{a-1} g(t)\,dt + \frac{g(a)+g(a-1)}{2} + \sum_{r=2}^\infty G_r \Delta^{r} g(a-1)
\end{equation}
Here, we have used the fact that $\Delta^{r-1} g(a)- \Delta^{r-1} g(a-1) = \Delta^{r} g(a-1)$. However, by the empty sum rule, Eq \ref{Ex733_5} implies that: 
\begin{equation}\label{Ex733_6} 
\int_{a-1}^{a} g(t)\,dt = \frac{g(a)+g(a-1)}{2} + \sum_{r=2}^\infty G_r\, \Delta^{r} g(a-1)
\end{equation}
Last equation allows us to deduce a rich set of identities. For example, if $g(k)=\frac{1}{k}$ and $a=2$, we have the following identity (compare it with Eq \ref{Ex733_2} and Eq \ref{Ex733_4}): 
\begin{equation}\label{Ex733_7} 
\sum_{r=1}^\infty \frac{|G_r|}{r+1} = 1-\log{2}
\end{equation}
Similarly, using $g(k) = \sin{\frac{\pi}{3}k}$ or $g(k) = \cos{\frac{\pi}{3}k}$ yield identities such as:
\begin{equation}\label{Ex733_8} 
|G_1| + |G_2| -|G_4| - |G_5| +|G_7|+|G_8| - |G_{10}| - |G_{11}| + \dotsm = \frac{\sqrt{3}}{\pi}
\end{equation}
\begin{equation}\label{Ex733_9} 
|G_2| + |G_3| -|G_5| - |G_6| +|G_8|+|G_9| - |G_{11}| - |G_{12}| + \dotsm = \frac{2\sqrt{3}}{\pi} -1
\end{equation}
\begin{equation}\label{Ex733_10} 
|G_1| - |G_3| -|G_4| + |G_6| +|G_7|-|G_9| - |G_{10}| + |G_{12}| + \dotsm = 1- \frac{\sqrt{3}}{\pi}
\end{equation}

Many other identities can be easily deduced using the series expansion of $\log^{-1}{(1+x)}$.
\section{Summary of Results} 
Throughout previous chapters, the focus of Summability Calculus has been on expressing unique natural generalization to simple and convoluted finite sums using infinitesimal derivatives. This manifested markedly in the Euler-like summation formulas. Starting from Chapter \ref{Chapter4}, however, the notion of summing or defining divergent series using a regular, stable, and linear method was introduced via a formal definition $\mathfrak{T}$. Such definition allowed us to expand the real of Summability Calculus quite naturally to oscillating finite sums, which simplified their analysis considerably, and allowed us to fill many gaps. 

While such approach is nearly complete, it is shown in this final chapter that Summability Calculus can still be expanded even further by expressing almost all its main results alternatively using the language of finite differences. In this chapter, Summability Calculus was used to derive foundational results in the Calculus of Finite Differences such as the discrete definition of derivatives, Newton's interpolation formula, and Gregory's integration formula. Such results were, in turn, used in conjunction with analytic summability theory to prove a stronger version of the celebrated Shannon-Nyquist Sampling Theorem, and to deduce the discrete analog of the many Euler-like summation formulas. In addition, it is illustrated how the Calculus of Finite Differences is indeed a valuable tool in analyzing finite sums; hence an integral part of Summability Calculus. 


\end{document}